\documentclass[a4paper,red,twoside]{Math_class}

\usepackage{float} 
\usepackage[
backend=biber,
style=alphabetic,
sorting=ynt
]{biblatex}
\usepackage[latin1]{inputenc}

\usepackage{cleveref}

\usepackage{thmtools}
\usepackage{thm-restate}

\usepackage{fancyhdr}
\hypersetup{
    colorlinks=true,
    linkcolor=blue,
    filecolor=magenta,      
    urlcolor=cyan,
    pdftitle={Inverse of the Gaussian multiplicative chaos: Lehto welding of Independent Quantum disks},
    pdfpagemode=FullScreen,
    }
\usepackage{blindtext}
\renewcommand{\dint}{\mathrm{d}}
\title{Inverse of the Gaussian multiplicative chaos: Lehto welding of Independent Quantum disks}
\author{Ilia Binder,  Tomas Kojar}
\begingroup
\lccode`\~=`\ %
\lowercase{%
  }%
\endgroup
\newcounter{word}

\makeatletter
\newenvironment{proofs}[1][\proofname]{\par
  \pushQED{\qed}%
  \normalfont \topsep0\p@\relax
  \trivlist
  \item[\hskip\labelsep\itshape
  #1\@addpunct{.}]\ignorespaces
}{%
  \popQED\endtrivlist\@endpefalse
}
\makeatother

\makeatletter
\let\old@rule\@rule
\def\@rule[#1]#2#3{\textcolor{blue}{\old@rule[#1]{#2}{#3}}}
\makeatother

\begin{document}
{\let\thefootnote\relax\footnote{{for feedback please contact kojartom@gmail.com and ilia@math.utoronto.ca}}}

\maketitle
\begin{abstract}
In this article, we use the framework of "Random conformal weldings" \cite{AJKS} to prove the existence of Lehto-welding for the inverse for $\gamma<0.1818$ and independent copies for $\gamma_{2}\leq\gamma_{1}<0.1818$. In particular, we obtain the existence of some conformaly invariant loop $\Gamma$ that glues two disks with boundary length given by independent copies of GMC on the unit circle. It is still unclear how to build a parallel proof to \cite{ang2021integrability} that those loops $\Gamma$ are in fact SLE-loops.
\end{abstract}

\tableofcontents
 
\par\noindent\rule{\textwidth}{0.4pt}

\newpage\part{Introduction}\label{part:introduction}
\section{Introduction }
SLE describes a curve growing in time: the original curve of interest (say a cluster boundary in a spin system) is obtained as time tends to infinity. In this paper we give a construction of random loops which is stationary, i.e. the probability measure on curves is directly defined without introducing an auxiliary time, that might be related to the SLE loop (see \cref{SLEloopt}).  \\
Our construction is based on the idea of conformal welding. Consider a Jordan curve $\Gamma$ bounding a simply connected region $\Omega$ in the plane. By the Riemann mapping theorem, there are conformal maps $h_{\pm}$ mapping the unit disc $\mathbb{D}$ and its complement to $\Omega$ and its complement. The map $h^{-1}_{+} \circ h_{-}$ extends continuously to the boundary $T=\partial \mathbb{D}$ of the disc, and defines a homeomorphism of the circle. Conformal welding is the inverse operation where, given a suitable homeomorphism of the circle, one constructs a Jordan curve on the plane. In fact, in our case the curve is determined up to a \Mobius transformation of the plane. Thus random curves (modulo \Mobius transformations) can be obtained from random homeomorphisms via welding.
In this paper we build on the framework of \cite{AJKS} to introduce a random scale invariant set of homeomorphisms $h_{\omega}:T\to T$ and construct the welding curves for their inverse $h_{\omega}^{-1}$. The model considered here has been proposed by Peter Jones, one of the authors in \cite{AJKS}. The construction depends on a real parameter $\gamma$ (the "inverse temperature") and the maps are a.s. in $\omega$ \Holder continuous for $\gamma<\gamma_{c}$. For this range of $\gamma$ the welding map will be a.s. well defined. For $\gamma>\gamma_{c}$ we expect the map $h_{\omega}$ not to be continuous and no welding to exist. \\
In the context of SLE, there have been various constructions of the \textit{SLE loops} , which are basically Jordan loops that locally are SLEs (i.e. after cutting a part, the remaining curve is an SLE curve). There is active work there so we list here are at least some references from Zhan's work \cite{zhan2021sle}, that constructed an SLE loop measure using a two-sided whole-plane SLE but also proved a criterion for uniqueness up to a constant to unify the various constructions; Sheffield and Werner CLE measure \cite{sheffield2012conformal}, Kemppainen and Werner CLE's Mobius-invariance \cite{kemppainen2016nested}, Miller and Schoug CLE for $\kappa\in (8/3, 8) $ \cite{miller2022existence}, Fields and Lawler "SLE loops rooted at an interior point"  (in preparation), Benoist, S., Dub\`{e}dat, J.: "Building $SLE_\kappa$ loop measures for $\kappa < 4$" (in preparation). Also, there was the construction in \cite{ang2021integrability, ang2020conformal} by Morris, Holden and Sun, who showed in \cite[theorem 1.3]{ang2021integrability} that the construction of SLE loops based on the quantum zipper satisfies Zhan's uniqueness. \\
Also, see the work \cite{kupiainen2023conformal} for a parallel route for the independent copies welding that actually directly builds on \cite{AJKS}. Due to the randomness of the annuli they too had to deal with obtaining disjoint annuli with a large enough overlap as we will in the section on independent copies. They found a beautiful algorithm that uses annuli of different scales for each copy. We came upon the same issue and the same need to use different scales. As we will see, we will scale our annuli by $b_{n}M_{n}=b_{n}^{\xi+U_{n}^{1}}\frac{1}{\tau(1)}$. So if the ratio of independent same scales $e^{U_{n}^1}/e^{\tilde{U}_{n}^1}$, then this ratio can be arbitrarily large due to the cancellation of constants. So there is a need for a difference in scales in order to extract the normalization drifts to control the ratio i.e. considering 
\begin{equation}
\frac{b_{n+m}M_{1,n+m}}{b_{n}M_{2,n}}.
\end{equation}
\paragraph*{Inverse measure and homeomorphism}
This article is the second part of the preliminary work needed for extending the work in \cite{AJKS}. In particular, in that work the main object is the Gaussian random field $H$ on the circle with covariance
\begin{equation}
\Expe{H(z)H(z')}=-\ln\abs{z-z'},    
\end{equation}
where $z, z'\in\mathbb{C}$ have modulus $1$ (see below for precise definitions and also \cite{AJKS,bacry2003log}).  The exponential $\expo{\gamma H}$ gives rise to a random measure $\tau$ on the unit circle $\T$, given by 
\begin{equation}
\tau(I):=\mu_{H}(I):=\liz{\e}\int_{I}e^{\gamma H_{\e}(x)-\frac{\gamma^{2}}{2}\Expe{\para{H_{\e}(x)}^{2}}}\dx,    
\end{equation}
for Borel subsets $I\subset \T=\mathbb{R}/ \mathbb{Z}=[0,1)$ and $H_{\e}$ is a suitable regularization. This measure is within the family of \textit{Gaussian multiplicative chaos} measures (GMC) (for expositions see the lectures \cite{robert2010gaussian,rhodes2014gaussian}). So finally, in \cite{AJKS} consider the random homeomorphism $h:[0,1)\to [0,1)$ defined as the normalized measure
\begin{equation} 
h(x):=\frac{\tau[0,x]}{\tau[0,1]}, x\in [0,1),    
\end{equation}
and prove that it gives rise to a Beltrami solution and a \textit{conformal welding } map. In the ensuing articles, we extend this result to its inverse $h^{-1}$ and then to the composition $h_{1}^{-1}\circ h_{2}$ where $h_{1},h_{2}$ are two independent copies. The motivation for that is of obtaining a parallel point of view of the beautiful work by Sheffield \cite{sheffield2016conformal} of gluing two quantum disks to obtain an SLE loop. \\
We let $Q_{\tau}(x):[0,\tau([0,1])]\to [0,1]$ denote the inverse of the measure $\tau:[0,1]\to [0,\tau([0,1])]$ i.e.
\begin{equation}
Q_{\tau}(\tau[0,x])=x\tand \tau[0,Q_{\tau}(y)]=y,
\end{equation}                                  
for $x\in [0,1]$ and $y\in [0,\tau([0,1])].$  The existence and continuity of the inverse $Q$ follows from the a)non-atomic nature of GMC \cite[theorem 1]{bacry2003log} and b)its continuity and strict monotonicity, which in turn follows from satisfying bi-\Holder over dyadic intervals \cite[theom 3.7]{AJKS}. We use the notation Q because the measure $\tau$ can be thought of as the "CDF function" for the "density" $\expo{\gamma H}$ and thus its inverse $\tau^{-1}=Q$ is the quantile (also using the notation $\tau^{-1}$ would make the equations less legible later when we start including powers and truncations). We will also view this inverse as a hitting time for the measure $\tau$
\begin{equation}
Q_{\tau}(x)=T_{x}:=\inf\set{t\geq 0: \tau[0,t]\geq x}.    
\end{equation}
(For the actual precise notation of the inverse used in this article, see the section on "notations".) The inverse homeomorphism map $h^{-1}:[0,1]\to [0,1]$ is defined as
\begin{equation}\label{inversehomeo}
h^{-1}(x):=Q_{\tau}(x\tau([0,1])\tfor x\in [0,1]     
\end{equation}
In the first part of this work in \cite{binder2023inverse}, we studied the moments of the inverse $Q$ and also ratios of it. We use those here to estimate the multipoint estimates 
\begin{equation}
\Expe{\prod_{k}\para{\frac{Q^{k}(a_{k},b_{k})}{Q^{k}(c_{k},d_{k})}}^{q_{k}} } ,   
\end{equation}
with $d_{k+1}<a_{k}<b_{k}<c_{k}<d_{k}$, using a conditional independence result for the inverse i.e. that $\frac{Q^{1}(a_{1},b_{1})}{Q^{1}(c_{1},d_{1})},\frac{Q^{2}(a_{2},b_{2})}{Q^{2}(c_{2},d_{2})}$ can be decoupled into separate conditional factors  given the event that they are well-separated $G_{1,2}:=\set{Q^{2}(d_{2})+\delta_{2} <Q^{1}(a_{1})}$ i.e. roughly
\begin{equation}
\Expe{\frac{Q^{1}(a_{1},b_{1})}{Q^{1}(c_{1},d_{1})}\frac{Q^{2}(a_{2},b_{2})}{Q^{2}(c_{2},d_{2})} }=\Expe{\Expe{\frac{Q^{1}(a_{1},b_{1})}{Q^{1}(c_{1},d_{1})}\mid \mathcal{F}}\ind{G_{1,2}} \frac{Q^{2}(a_{2},b_{2})}{Q^{2}(c_{2},d_{2})} },    
\end{equation}
see notations for detail on the filtration $\mathcal{F}$. These estimates are needed because they are the analogous of $\Expe{\prod_{k}\para{\frac{\eta^{k}(a_{k},b_{k})}{\eta^{k}(c_{k},d_{k})}}^{q_{k}} }$ showing up in estimating the dilatation in \cite[equation (89)]{AJKS}.
\subsection{Acknowledgements}
We thank Eero Saksman and Antti Kupiainen. We had numerous useful discussions over many years. Eero asked us to compute the moments of the inverse and that got the ball rolling seven years ago.  We also thank J.Aru, J.Junnila and V.Vargas for discussions and comments.
\section{Main result, SLE loop and Overview}
The main result of this article is the existence of a Beltrami solution for the homeomorphism $\phi(x)=Q_{\tau}(x\tau([0,1])$ and the  independent copies $\phi_{+}(x)=Q_{1}(x\tau_{1}([0,1]),\phi_{-}(x)=Q_{2}(x\tau_{2}([0,1])$. Namely the following is the analogue of \cite[theorem 5.2]{AJKS}.
\begin{theorem}\label{Beltinv}
For $\gamma<0.1818$, the inverse $h^{-1}(x):=Q_{\tau}(x\tau([0,1]),x\in [0,1]$  defines a \Holder continuous circle homeomorphism, such that the welding problem has a solution $\Gamma$, where $\Gamma$ is a Jordan curve bounding a domain $\Omega=f_{+}(\mathbb{D})$ with a \Holder continuous Riemann mapping $f_+$. Tthe solution is unique up to a \Mobius map of the plane.
\end{theorem}
\begin{remark}
In  \cite[theorem 5.2]{AJKS}, they also prove that the curve $\Gamma$ is continuous in $\gamma$. This required showing analyticity of GMC in the $\gamma$-variable. This is interesting to study for the inverse too.
\end{remark}
\noindent The following is the analogue of \cite[theorem 5.5]{AJKS}.
\begin{theorem}\label{Beltindc}
For every pair $\gamma_{+}, \gamma_{-}<\sqrt{\frac{2}{64}}=0.1818 $, the welding problem for the homeomorphism $\phi_{+}\circ \phi_{-}^{-1}$ has a solution $\Gamma=\Gamma_{\gamma_{+},\gamma_{-}}$ , where $\Gamma_{\gamma_{+},\gamma_{-}}$ is a Jordan curve bounding the domains $\Omega_{+} =f_{+}(\mathbb{D})$ and $\Omega_{+} =f_{+}(\mathbb{C}\setminus\mathbb{D})$, with \Holder continuous Riemann mappings $f_{\pm}$. The solution is unique up to a \Mobius map of the plane and the curves $\Gamma_{\gamma_{+},\gamma_{-}}$  are continous in $\gamma_{+}$ and $\gamma_{-}$.
\end{theorem}
\subsection{SLE loop and Beltrami welding } 
\noindent Next here is one natural follow-up problem. In the beautiful work in \cite{sheffield2016conformal}, one starts with a Hadamard-coupling of GFF and SLE to obtain a quantum zipper process and in turn the welding result. In the \cite{AJKS}-perspective one starts with the welding but it is still left with relating the resulting Jordan curve $\Gamma$ to an SLE loop as described in \cite[theorem 4.1]{zhan2021sle} and \cite[theorem 1.3]{ang2021integrability}. So the problem is to obtain an independent second proof of the SLE loop construction in \cite[theorem 1.3]{ang2021integrability} using the \cite{AJKS}-framework.\\
For example, some variation of the Hadamard coupling needs to get rediscovered in the Beltrami equation setting. Perhaps we need more development of partial-welding concepts (\cite{hamilton2002conformal}) and the distributional-extensions for the Dirichlet-energy-coupling in \cite{Viklund_2020}. It is unclear how to characterize the law of the resulting Jordan curve $\Gamma$ from the composition of $\phi_{+}(x)=Q_{1}(x\tau_{1}([0,1]),\phi_{-}(x)=Q_{2}(x\tau_{2}([0,1])$.
By achieving this the community can benefit from having multiple perspectives on the same welding-coupling phenomenon for GFF and SLE. 
\begin{conjecture}\label{SLEloopt}
Does the above Beltrami-welding yield an SLE loop measure and thus an independent second construction of the SLE loop in \cite[theorem 1.3]{ang2021integrability}?
\end{conjecture}
\noindent We don't have any developing work on this problem. It seems to require the development of new technology in the overlap of Beltrami equation, partial-welding and Hadamard-coupling. We tried for a few months to make progress but got stuck showing basic SLE properties for $\Gamma$ such as domain Markov property. It is unclear how to pass information
\begin{equation}
\text{ Inverse $Q$ Welding }\to \text{ Quasiconformal maps }\to \text{ Beltrami curve}.
\end{equation}
As we zoom on the boundary of the quasiconformal maps image $\Gamma[0,T]$, we locally depend on the inverse $Q^{1}$ and so there is a large correlation. 
\begin{remark}[Domain Markov property test]
Consider the map
\begin{equation}
g_{+,T}(x):=f_{+}(\expo{iQ(T,T+x)})    \tfor x\in [0,1],
\end{equation}
which is well-defined because of the periodicity $Q(x+1)=Q(x)+1$. The map $g_{+,T}$ still maps $[0,1]$ into the same curve $\Gamma$. For small $x<h\leq 1-T$, we have that it gets mapped into $\Gamma[T,T+h]$. On the other hand, if we have a strong-translation invariance $Q(T,T+x)\eqdis Q(0,x)$, then we have
\begin{equation}
g_{+,T}(x)\eqdis \wt{f}_{+}(\expo{i\wt{Q}(0,x)})    \tfor x\in [0,1].
\end{equation}
Now the RHS maps into a different sampled curve $\wt{\Gamma}$. So for deterministic $h>0$, the small curve-piece $\Gamma[T,T+h]$ has the same law as $\wt{\Gamma}[0,h]$.   However, as explained in  \cref{differencetermunshifted} there is likely no such strong-translation invariance.\\
Perhaps the \cref{deltaSMP} can be useful here 
\begin{equation}
\text{ $Q_{s}^{\delta}\bullet \para{Q_{a}^{\delta}+r}$ is independent of $Q_{a}^{\delta}$ }   
\end{equation}
for $r\geq \delta$ and so $Q_{s}^{\delta}\bullet \para{Q_{a}^{\delta}+r}\eqdis Q_{s}^{\delta}$. Meaning that we let $T:=Q_{a}^{\delta}+\delta$ and make some choices on $a,\delta$ to get that $\Gamma[T,T+h]$ has the same law as $\wt{\Gamma}[0,h]$. But of course the issue here is the need to get it for all arbitrarily small $h>0$. Perhaps we should consider a sequence of Lehto-loops with $\delta\to 0$. It is unclear how to proceed.
\end{remark}

\subsection{Overview}
\noindent Since the inverse of GMC didn't seem to appear in other problems, it was studied very little and so we had to find and build many of its properties. Our guide for much for the work in \cite{binder2023inverse} and here was trying to transfer the known properties of the GMC measure to its inverse, the Markovian structure for the hitting times of Brownian motions (such as the Wald's equation and the independent of the increments of hitting times) and then trying to get whatever property was required for the framework set up by \cite{AJKS} to go through successfully. This was a situation where a good problem became the roadmap for finding many interesting properties for the inverse of GMC and thus GMC itself. Here is an outline of each of the sections of this article.
\begin{itemize}

    \item In \nameref{part:introduction}, we go over the notations needed for this article.

    \item In \nameref{beltinversepart}, we go over the variation of the \cite{AJKS}-framework for the inverse circle homeomorphism $h^{-1}(x):=Q_{\tau}(x\tau([0,1])\tfor x\in [0,1]$. For example, in \cref{modifiedannuli}, we multiply the annuli with random constants that match their scale and the precomposed factors that show up when we change scales $Q^{n}\to Q^{n+1}$.

    \item In \nameref{part:deviationestLehto}, we go over each of the large deviations needed for the main deviation of the Lehto integral.

    \item In \nameref{Beltindcopies}, we go over the variation of the above framework for the independent copies.

    \item In \nameref{part:divlehtoindec}, we go over the large deviations needed for the independent copies version of the modified Lehto integral.

    \item In \nameref{part:researpappend}, we go over some research problems, some appendix results and an offshoot study of the tree-estimate of the dilatation in \cite[Theorem 24]{tecu2012random}.
   
\end{itemize}

\newpage \section{Notations}\label{notations}
\subsection{White noise expansion on the unit circle and real line }
Gaussian fields with logarithmic covariance indexed over time or even functions $f\in \mathbb{H}$,for some Hilbert space $\mathbb{H}$, have appeared in many places in the literature: for the \textit{Gaussian free field} which is a Gaussian field indexed over the $L^{2}$ Hilbert space equipped with the Dirichlet inner product \cite{sheffield2007gaussian,lodhia2016fractional,duplantier2017log} and for the \textit{Gaussian multiplicative chaos} which is about Gaussian fields with logarithmic covariance plus some continuous bounded function $\ln\frac{1}{\abs{x-y}}+g(x,y)$ \cite{robert2010gaussian,bacry2003log,rhodes2014gaussian,aru2017gaussian}. In this work we will work with the Gaussian field indexed over sets in the upper half-plane found in the works \cite{barral2002multifractal,bacry2003log} and used in the random welding context in \cite[section 3.2]{AJKS}. For the hyperbolic measure $\dlambda(x,y):=\frac{1}{y^{2}}\dx \dy$ in the upper half-plane $\uhp$ we consider a Gaussian process $\set{W(A)}_{A\in B_{f}(\uhp)}$ indexed by Borel sets of finite hyperbolic area: 
\begin{equation}
    B_{f}(\uhp):=\{A\subset \uhp: \lambda(A)<\infty; \supl{(x,y),(x',y')\in A}|x-x'|<\infty\}
\end{equation}
with covariance
\begin{equation}
    \Expe{W(A_{1})W(A_{2})}:=\lambda(A_{1}\cap A_{2}).
\end{equation}
Its periodic version $W_{per}$ has covariance 
\begin{equation}
\Expe{W(A_{1})W(A_{2})}:=\lambda(A_{1}\cap \bigcup_{n\in \Z} (A_{2}+n).    
\end{equation}
\subsubsection{Logarithmic field on the real line}Consider the triangular region
\begin{equation}
U:=\{(x,y)\in \uhp: x\in [-1/2,1/2]\tand y>2\abs{x}\}    
\end{equation}
and define the field on the real line
\begin{equation}
U(x):=W(U+x)\tfor x\in \mb{R}.    
\end{equation}
It has logarithmic covariance
\eq{\Exp[U(x)U(y)]=\ln\para{\frac{1}{\min(\abs{x-y},1)}}.}
This is the main field we will work with throughout this article. Because of the divergence along the diagonal, we upper and lower truncate by considering the field evaluated over shifts of the region
\begin{equation}
U_{\e}^{\delta}:=\{(x,y)\in \uhp: x\in [-\frac{\delta}{2},\frac{\delta}{2}]\tand \para{2\abs{x}}\vee \e<y\leq \delta\}.    
\end{equation}
The covariance of $U_{\e}^{\delta}(x)$ is the following (\cite{bacry2003log}).
\begin{lemma}\label{linearU}
The truncated covariance satisfies for $\delta\geq \e$ and all $x_{1},x_{2}\in \mathbb{R}$:
\begin{equation}
\Exp[U_{ \varepsilon}^{  \delta }(x_{1} )U_{ \varepsilon}^{  \delta }(x_{2} )  ]=R_{ \varepsilon}^{  \delta }(\abs{x_{1}-x_{2}}):=\left\{\begin{matrix}
\ln(\frac{\delta }{\varepsilon} )-\para{\frac{1}{\e}-\frac{1}{\delta}}\abs{x_{2}-x_{1}} &\tifc \abs{x_{2}-x_{1}}\leq \varepsilon\\ 
 \ln(\frac{\delta}{\abs{x_{2}-x_{1}}}) +\frac{\abs{x_{2}-x_{1}}}{\delta}-1&\tifc \e\leq \abs{x_{2}-x_{1}}\leq \delta\\
 0&\tifc  \delta\leq \abs{x_{2}-x_{1}}.
\end{matrix}\right.    .
\end{equation}
In the case of $\e=0$, we shorthand write $R^{  \delta }(\abs{x_{1}-x_{2}})=R^{  \delta }(\abs{x_{1}-x_{2}})$. 
\end{lemma}
\begin{remark}
This implies that we have a bound for its difference:
\begin{equation}
\Expe{\abs{U_{\varepsilon_{1}}^{\delta}(x_{1})-U_{\varepsilon_{2}}^{\delta}(x_{2})}^{2}}\leq 2 \frac{|x_{2}-x_{1}|+|\varepsilon_{1}-\varepsilon_{2}|}{\varepsilon_{2}\wedge \e_{1}}.     
\end{equation}
Therefore, being Gaussian this difference bound is true for any $a>1$
\begin{equation}
\Expe{|U_{\varepsilon_{1}}^{\delta}(x_{1})-U_{\varepsilon_{2}}^{\delta}(x_{2})|^{a}}\leq c\para{ \frac{|x_{2}-x_{1}|+|\varepsilon_{1}-\varepsilon_{2}|}{\varepsilon_{2}\wedge \e_{1}}     }^{a/2}.
\end{equation}
and so we can apply the Kolmogorov-Centsov lemma from \cite[Lemma C.1]{hu2010thick}:
\begin{equation}
|U_{\varepsilon_{1}}(x_{1})-U_{\varepsilon_{2}}(x_{2})|\leq M~ (\ln(\frac{1}{\varepsilon_{2}})^{\zeta}\frac{(|x_{2}-x_{1}|+|\varepsilon_{1}-\varepsilon_{2}|)^{\gamma}}{ \varepsilon_{2}^{\gamma+\varepsilon} },     
\end{equation}
where $\frac{1}{2}< \frac{\e_{1}}{\e_{2}}<2$, $0<\gamma<\frac{1}{2}$, $\varepsilon,\zeta>0 $ and $M=M(\varepsilon,\gamma,\zeta)$. 
\end{remark}

\subsubsection{Infinite cone field}
For the region
\begin{equation}
\mathcal{A}_{\e}^{\delta}:=\{(x,y)\in \uhp: x\in [-\frac{\delta}{2},\frac{\delta}{2}]\tand \para{2\abs{x}}\vee \e<y\}    
\end{equation}
we let $U_{\omega,\e}^{\delta}:=\omega_{\e}^{\delta}(x):=W(\mathcal{A}+x)$. This field was considered in \cite{bacry2003log}. It has the covariance
\begin{equation}
\Exp[\omega_{ \varepsilon}^{  \delta }(x_{1} )\omega_{ \varepsilon}^{  \delta }(x_{2} )  ]=\left\{\begin{matrix}
\ln(\frac{\delta }{\varepsilon} )+1-\frac{\abs{x_{2}-x_{1}} }{\e}&\tifc \abs{x_{2}-x_{1}}\leq \varepsilon\\ 
 \ln(\frac{\delta}{\abs{x_{2}-x_{1}}}) &\tifc \e\leq \abs{x_{2}-x_{1}}\leq \delta\\
  0&\tifc  \delta\leq \abs{x_{2}-x_{1}}
\end{matrix}\right.    .
\end{equation}
This field is interesting because it has "exact scaling laws" as we will see below.
\subsubsection{Trace of the Gaussian free field on the unit circle}
For the wedge shaped region 
\begin{equation}\label{eq:covarianceunitcircle}
H:=\{(x,y)\in \uhp: x\in [-1/2,1/2]\text{ and }y>\frac{2}{\pi}tan(\pi x)\},    
\end{equation}
we define the \textit{trace GFF} $H$ to be
\begin{equation}
H(x):=W(H+x), x\in \R/\Z.    
\end{equation}
A regularized version of $H$ comes from truncating all the lower scales. Let $A_{r,\varepsilon}:=\{(x,y)\in \uhp: r>y>\varepsilon\}$ and $H_{\varepsilon}^{r}:=H\cap A_{r,\varepsilon}$ then we set: 
\begin{equation}
H_{\varepsilon}^{r}(x):=W(H_{\varepsilon}^{r}+x)    
\end{equation}
and $H_{\varepsilon}:=h^{\infty}_{\varepsilon}$. Its covariance is as follows: For points $x,\xi\in [0,1)$, with 0 and 1 identified,  and $\varepsilon\leq r$ we find for $y:=|x-\xi|$:
\begin{eqalign}
\Exp[H_{\varepsilon}(x)H_{r}(\xi)]:=&\branchmat{2\log(2)+\log\frac{1}{2sin(\pi y)}  &\tifc y> \frac{2}{\pi} arctan(\frac{\pi}{2}\e)\\ \log(1/\e)+(1/2)\log(\pi^{2}\varepsilon^{2}+4)+\frac{2}{\pi}\frac{arctan(\frac{\pi}{2}\e)}{\varepsilon}&\tifc y\leq \frac{2}{\pi} arctan(\frac{\pi}{2}\e)\\
-\log(\pi)- y \varepsilon^{-1}-\log(cos(\frac{\pi}{2}y)&}.
\end{eqalign}
From the above covariance computations, we find the difference bound for $1/2\leq \frac{\e_{1}}{\e_{2}}\leq 2$:
\begin{equation}
\Exp[|H_{\varepsilon_{1}}(x_{1})-H_{\varepsilon_{2}}(x_{2})|^{2}]\lessapprox~\frac{|x_{2}-x_{1}|+|\varepsilon_{1}-\varepsilon_{2}|}{\varepsilon_{2}\wedge \e_{1} }    
\end{equation}
Therefore, by Kolmogorov-Centsov \cite[Lemma C.1]{hu2010thick}, there is a continuous modification with the modulus:
\begin{equation}
|H_{\varepsilon_{1}}(x_{1})-H_{\varepsilon_{2}}(x_{2})|\leq M (\ln(\frac{1}{\varepsilon_{1}})^{\zeta}\frac{(|x_{2}-x_{1}|+|\varepsilon_{1}-\varepsilon_{2}|)^{\gamma}}{ \varepsilon_{1}^{\gamma+\varepsilon} },      
\end{equation}
where $0<\gamma<\frac{1}{2}$, $\varepsilon,\zeta>0 $ and $M=M(\varepsilon,\gamma,\zeta)$. A similar modulus is true for the field $U(x)$ on the real line constructed above.
\subsubsection{Measure and Inverse notations}
\par In general for the field $X$, we will denote its measure as $\eta_{X}(I):=\mu_{X}(I):=\liz{\e}\int_{I}e^{\overline{X}_{\e}(x)}\dx$ for $I\subset [0,1)$.  Since we will work with the real-line field $U^{\delta}$ we will shorthand denote $\eta^{\delta}(A):=\liz{\e}\int_{A}e^{\gamma U_{\e}^{\delta}(x)-\frac{\gamma^{2}}{2}\ln\frac{1}{\e}}\dx$ for $A\subset \mathbb{R}$.\\
\noindent Its inverse $Q^{\delta}:\Rplus\to \Rplus$ is defined as
\begin{equation}
Q^{\delta}_{\eta}=Q^{\delta}_{x}=Q^{\delta}_{U,x}=Q^{\delta}_{U}(x):=\infp{t\geq 0 : \eta^{\delta}_{U}\spara{0,t}\geq x}    
\end{equation}
and we will also consider increments of the inverse over intervals $I:=(y,x)$
\begin{equation}
Q^{\delta}(I)=Q^{\delta}(y,x):=Q^{\delta}(x)-Q^{\delta}(y).    
\end{equation}
For the lower-truncated inverse of $\eta_{\e}^{\delta}$, we will write
\begin{equation}
Q^{\delta}_{\e,x}=Q^{\delta}_{\e}(x).    
\end{equation}
We fix a highest scale $\delta_{1}>0$ and we will shorthand write $U^{1}:=U^{\delta_{1}}_{0},\eta:=\eta^{\delta_{1}}$ and $Q=Q^{\delta_{1}}$. \\
In the ensuing articles we will work with a sequence of truncations $\delta_{n}\to 0$, and so we will shorthand denote $U^{n}:=U^{\delta_{n}}$, $\eta^{n}$and $Q^{n}$. Similarly, when lower truncating $U_{n}:=U_{\e_{n}}$, we will write $\eta_{n}$ and $Q_{n}$. The Gaussian fields $X_{\e}$ often appear with their exponential normalization, so in short we write
\begin{equation}
\overline{X_{\e}}:=\gamma X_{\e}-\frac{\gamma^{2}}{2}\Expe{X_{\e}^{2}}.
\end{equation}
By slight abuse of notation, we denote by $U(s)\cap U(t)$
\begin{equation}
U(s)\cap U(t):=W(U(s)\cap U(t)    
\end{equation}
and $U(s)\setminus U(t):=W(U(s)\setminus U(t)  $. In particular, $U(\ell+s)\cap U(\ell)$ also show up naturally when studying the GMC-measure $\eta(0,t)$ as a function of time. For example, in this work we repeatedly come across events of the form
\begin{equation}
\set{Q(x)\geq t_{1}, \eta(Q(x),Q(x)+L)\geq t_{2}}    
\end{equation}
and their correlation happens exactly via the field $U(Q(x)$ and so it is natural to study the decomposition
\begin{equation}
U(Q(x)+s)=U(Q(x)+s)\cap U(Q(x)+U(Q(x)+s)\setminus U(Q(x).     
\end{equation}
\subsubsection{Semigroup formula}\label{not:semigroupformula}
In the spirit of viewing $Q_{a}=Q(a)$ as a hitting time we have the following semigroup formula and notation for increments
\begin{equation}
Q^{\delta}(y,x)=Q^{\delta}(x)-Q^{\delta}(y)=\infp{t\geq 0: \etamu{Q^{\delta}(y),Q^{\delta}(y)+t}{\delta}\geq x-y }=:Q_{x-y}^{\delta }\bullet Q_{y}^{\delta}.
\end{equation}
Generally for any nonnegative $T\geq 0$ we use the notation
\begin{equation}
Q_{x}^{\delta} \bullet T:=\infp{t\geq 0: \etamu{T,T+t}{\delta}\geq x-y }.   
\end{equation}
This formula will be essential in studying moments of the ratio of increments: in the case $a<b<c<d$ we can isolate the $Q_{a}$
\begin{equation}
\Expe{\para{\dfrac{Q^{\delta}(a,b)}{Q^{\delta}(c,d)}}^{p}}=\Expe{\para{\dfrac{Q^{\delta}_{b-a}\bullet Q_{a}^{\delta}}{Q^{\delta}_{d-c}\bullet Q^{\delta}_{c} }}^{p}}    =\Expe{\para{\dfrac{Q^{\delta}_{b-a}\bullet Q_{a}^{\delta}}{Q^{\delta}_{d-c}\bullet \para{Q^{\delta}_{c-b}\bullet Q_{a}^{\delta}+Q_{a}^{\delta} } }}^{p}}    
\end{equation}
and then by decomposing $Q_{a}$, we can apply scaling laws that we will go over in properties of inverse. To be clear the notation $Q_{x}^{\delta }\bullet Q_{y}^{\delta}$ is \textit{not} equal to the composition $Q^{\delta }\bullet Q_{y}^{\delta}=Q^{\delta }(Q_{y}^{\delta})=Q^{\delta }(a),$ which is the first time that the GMC $\eta$ hits the level $a:=Q_{y}^{\delta}$.

\subsubsection{Approximation}
We will also work with a discrete canonical approximation $T_{n}(a)\downarrow Q(a)$:
\begin{equation}
T_{n}(a):=\frac{m+1}{2^{n}}\twhen \frac{m}{2^{n}}\leq Q(a) < \frac{m+1}{2^{n}}     
\end{equation}
and so for fixed $n$ the range of $T_{n}(a)$ is $D_{n}(0,\infty)$, all the dyadics of nth-scale in the open ray $(0,\infty)$.
\subsubsection{Scaling laws}
Here we study the scaling laws for the truncated field $\omega^{\delta}$ and the exact scaling law field $U^{\delta}$. We will need the following fields. For the exactly scaled field $\omega$ we have the following scaling covariance \cite[theorem 4]{bacry2003log}, \cite[proposition 3.3]{robert2010gaussian}.
\begin{proposition}\label{exactscaling}
For $\lambda\in (0,1)$ and fixed $x_{0}$ and all measurable sets $A\subset B_{\delta/2}(x_0)$ we have
\begin{equation}
 \set{\eta^{\delta}_{\omega}(\lambda A)}_{A\subset B_{\delta/2}(x_0)}\eqdis \set{\lambda e^{\overline{\Omega_{\lambda}}}\eta^{\delta,\lambda}_{\omega}(A)}_{A\subset B_{\delta/2}(x_0) },   
\end{equation}
where  $\overline{\Omega_{\lambda}}:=\gamma\sqrt{ \ln(\frac{1}{\lambda})}N(0,1)-\frac{\gamma^{2}}{2}\ln(\frac{1}{\lambda})$ and the measure $\eta_{\omega_{\e}}^{\delta,\lambda}$ has the underlying field $\omega^{\delta,\lambda}_{\e}(x)$ where in the covariance $\Expe{\omega^{\delta,\lambda}_{\e}(x_{1})\omega^{\delta,\lambda}_{\e}(x_{2})}$ we actually constrain $\abs{x_{2}-x_{1}}\leq \delta$ in order to have a positive definite covariance. 
\end{proposition}
\noindent For shorthand we let 
\begin{eqalign}\label{logfield}
&G_{\lambda}:=\frac{1}{\lambda}\expo{-\overline{\Omega_{\lambda}}},\\
&\tfor \overline{\Omega_{\lambda}}:=\gamma N(0,\ln\frac{1}{\lambda})-\frac{1}{2}\Expe{\para{\gamma N(0,\ln\frac{1}{\lambda})}^{2}}=\gamma N(0,\ln\frac{1}{\lambda})-\frac{\gamma^{2}}{2}\para{\ln(\frac{1}{\lambda})}.
\end{eqalign}
For $U^{\delta}$, we need the field $U_{ \varepsilon}^{\delta, \lambda}$ with covariance $R_{\e}^{\delta,\lambda}(\abs{x_{1}-x_{2}}):=\Expe{U_{ \varepsilon}^{  \delta,\lambda }(x_{1} )U_{ \varepsilon}^{  \delta,\lambda }(x_{2} )  }$
\begin{equation}\label{eq:truncatedscaled}
R_{\e}^{\delta,\lambda}(\abs{x_{1}-x_{2}})=\left\{\begin{matrix}
\ln(\frac{\delta }{\varepsilon} )-\para{\frac{1}{\e}-\frac{1}{\delta}}\abs{x_{2}-x_{1}}+(1-\lambda)(1-\frac{\abs{x_{2}-x_{1}}}{\delta})&\tifc \abs{x_{2}-x_{1}}\leq \varepsilon\\ ~\\
 \ln(\frac{\delta}{\abs{x_{2}-x_{1}}})-1+\frac{\abs{x_{2}-x_{1}}}{\delta}+(1-\lambda)(1-\frac{\abs{x_{2}-x_{1}}}{\delta}) &\tifc \e\leq \abs{x_{2}-x_{1}}\leq \frac{\delta}{\lambda}\\~\\
  0&\tifc \frac{\delta}{\lambda}\leq \abs{x_{2}-x_{1}}
\end{matrix}\right.    .
\end{equation}
\begin{remark}\label{rem:negativecov}
We note that after $\abs{x_{2}-x_{1}}\geq \delta$ this covariance is negative and thus discontinuous at $\abs{x_{2}-x_{1}}=\frac{\delta}{\lambda}$ and so the GMC $\eta^{\delta,\lambda}$ cannot be defined (so far as we know in the literature, GMC has been built for positive definite covariances \cite{allez2013lognormal}). So we are forced to evaluate those fields only over sets $A$ with length $\abs{A}\leq \delta$ (see \cite[lemma 2]{bacry2003log} where they also need this restriction to define the scaling law.).
\end{remark}
\noindent This is a useful field because it relates to the truncated $U^{\delta}$ in the following way
\begin{equation}\label{eq:truncatedscalinglawGMC}
U^{\delta}_{\lambda\e}(\lambda x)\eqdis N(0,\ln\frac{1}{\lambda}-1+\lambda)+  U_{ \varepsilon}^{  \delta,\lambda }(x),  
\end{equation}
where $N(0,r_{\lambda})$ ,with $r_{\lambda}:=\ln\frac{1}{\lambda}-1+\lambda$, is a Gaussian is independent of $ U_{ \varepsilon}^{  \delta,\lambda }$. So we let for $\lambda\in(0,1)$ the corresponding lognormal be defined as
\begin{eqalign}\label{logfieldshifted}
&c_{\lambda}:=\frac{1}{\lambda}\expo{-\overline{Z_{\lambda}}},\\
&\tfor\overline{Z_{\lambda}}:=\gamma N(0,\ln\frac{1}{\lambda}-1+\lambda)-\frac{1}{2}\Expe{\para{\gamma N(0,\ln\frac{1}{\lambda}-1+\lambda)}^{2}}=\gamma N(0,\ln\frac{1}{\lambda}-1+\lambda)-\frac{\gamma^{2}}{2}\para{\ln(\frac{1}{\lambda})-1+\lambda}.
\end{eqalign}
We summarize this in the following lemma.
\begin{lemma}\label{lem:scalinglawudelta}[Scaling transformation for $U^{\delta}_{\e}$]
For $\lambda\in (0,1)$ and fixed $x_{0}$ and all measurable sets $A\subset B_{\delta/2}(x_0)$ we have
\begin{equation}
 \set{\eta^{\delta}_{U}(\lambda A)}_{A\subset B_{\delta/2}(x_0)}\eqdis \set{\lambda e^{\overline{Z_{\lambda}}}\eta^{\delta,\lambda}_{U}(A)}_{A\subset B_{\delta/2}(x_0) },   
\end{equation}
where the measure $\eta_{U}^{\delta,\lambda}$ has the underlying field $U^{\delta,\lambda}_{\e}(x)$.
\end{lemma}
\noindent For any $p\in \R$ we have the moment formula
\begin{equation}
\Expe{\expo{p\overline{Z_{\lambda}}}  }=\expo{p\beta r_{\lambda} (p-1)}.    
\end{equation}
More generally, for positive continuous bounded function $g_{\delta}:\Rplus\to \Rplus$ with $g_{\delta}(x)=0$ for all $x\geq \delta$, we let
\begin{equation}
\Expe{U_{ \varepsilon}^{\delta,g }(x_{1} )U_{ \varepsilon}^{  \delta,g }(x_{2} )  }=\left\{\begin{matrix}
\ln(\frac{\delta }{\varepsilon} )-\para{\frac{1 }{\e}-\frac{1}{\delta}}\abs{x_{2}-x_{1}}+g_{\delta}(\abs{x_{2}-x_{1}})&\tifc \abs{x_{2}-x_{1}}\leq \varepsilon\\ 
 \ln(\frac{\delta}{\abs{x_{2}-x_{1}}})-1+\frac{\abs{x_{2}-x_{1}}}{\delta}+g_{\delta}(\abs{x_{2}-x_{1}}) &\tifc \e\leq \abs{x_{2}-x_{1}}\leq \delta\\
  0&\tifc \delta\leq \abs{x_{2}-x_{1}}
\end{matrix}\right.  .
\end{equation}

\subsubsection{Singular integral}
Using the Girsanov/tilting-lemma (\cite[Lemma 2.5]{berestycki2021gaussian}) we have 
\begin{equation}
\Expe{\ind{\eta(x)\geq t}e^{\thickbar{U}(a)}}=\Proba{\eta_{R}(x)\geq t},    
\end{equation}
where this is the singular/fusion integrals for GMC measure (\cite[lemma A.1]{david2016liouville})
\begin{equation}
\eta_{R}(A):=\int_{A}e^{\gamma^{2}\Expe{U(s)U(a)}}e^{\gamma U(s)}\ds,  \end{equation}
where $A\subset \Rplus$ and $a\in\Rplus$. We similarly, let $Q_{R}(x)$ to denote its inverse.

\subsubsection{Filtration notations}
As in \cite[section 4.2]{AJKS}, for a Borel set $S\subset \mathbb{H}$ let $\mathcal{B}_{S}$ be the $\sigma$-algebra generated by the randoms variables $U(A)$, where $A$ runs over Borel subsets $A\subseteq S$. For short if $X\in \CB_{S}$ we will call such a variable as \textit{measurable} $X\in \CU(S)$.\\
We fix sequence $\delta_{n}>\delta_{n+1}\to 0$ for $n\geq 1$. To denote the $\sigma-$algebra of upper truncated fields we write $\CU^{n}$ or $\CU^{n}_{\infty}$ generated by strictly decreasing $U^{n}:=U^{\delta_{n}}_{0}$. Also, we can consider strictly increasing larger heights i.e. $U^{h}_{0}$ denotes the field of height $h\geq \delta_{1}$. The infinitely large field  $U^{\infty}_{0}$ does not exist because it has infinite variance.\\
All the lower truncations are measurable with respect to the filtration of the larger scales
\begin{equation}
U^{n}(s)=U^{\delta_{n}}_{0}(s)\in \CU^{k}([a,b]) \tfor n\geq k,s\in [a,b],    
\end{equation}
and the same is true for the measure and its inverse
\begin{eqalign}
\set{\eta^{n}(a,b)\geq t}\in \CU^{k}([a,b])\tand \set{Q^{n}(0,x)\geq t}\in \CU^{k}([0,t]).  \end{eqalign}
We also define measurability with respect to random times since we will use the perspective of stopping times. For single time $Q^{k}(b)$ we have the usual definition of the stopping time filtration 
\begin{equation}
\mathcal{F}\para{[0,Q^{k}(b)]}:=\left\{A\in \bigcup_{t\geq 0}  \CU^{k}([0,t]): A\cap\{Q^{k}(b) \leq t\}\in  \CU^{k}([0,t]), \forall t\geq 0\right\}.    
\end{equation}
For example, $\CU^{n}\para{Q^{k}(b)-s}\subset \mathcal{F}\para{[0,Q^{k}(b)]}$ for all $n\geq k$ and $s\in [0,Q^{k}(b)]$. We also have
\begin{equation}
\set{Q^{k}(b)\geq Q^{n}(a)+s}=\bigcup_{t\geq 0}\set{Q^{k}(b)\geq t}\cap \set{t\geq Q^{n}(a)+s} \in  \mathcal{F}\para{[0,Q^{k}(b)]}, \end{equation}
for $n\geq k$. \\
This perspective helps us achieve a decoupling. For two interval increments $Q^{k}(a_{k},b_{k}),Q^{n}(a_{n},b_{n})$ from possibly different scales $n\geq k$ , we we will need their \textit{gap event} 
\begin{equation}
G_{k,n}:=\set{Q^{k}(a_{k})-Q^{n}(b_{k})\geq \delta_{n}}.    
\end{equation}
As we just saw this gap is measurable with respect to the larger scale $G_{k,n}\in  \mathcal{F}\para{[0,Q^{k}(a_{k})]} $ so we have by the tower property:
\begin{eqalign}\label{eq:towerproperty}
 &\Expe{\ind{Q^{k}(a_{k},b_{k})\in A_{1}}\ind{Q^{n}(a_{n},b_{n})\in A_{2}}\ind{G_{k,n}}  }\\
 &=\Expe{\Expe{\ind{Q^{k}(a_{k},b_{k})\in A_{1}}\conditional   \mathcal{F}\para{[0,Q^{k}(a_{k})]} }\ind{Q^{n}(a_{n},b_{n})\in A_{2}}\ind{G_{k,n}}  }.   
\end{eqalign}
In other words, we separated $Q^{k}(a_{k},b_{k})$ from $Q^{n}(a_{n},b_{n})$.
\newpage \subsection{Choices for parameters }\label{def:exponentialchoiceparamannul}\label{def:exponentialchoiceparam}
In this section we compile the parameters that we will use and their constraints, and prove their compatibility.
\begin{definition}[Choice for annuli]
We set 
\begin{eqalign}
&a_{k}^{P}:=\rho_{a}e^{P_{k}}\rho_{*}^{k}, b_{k}^{P}:=\rho_{b}e^{-P_{k}}\rho_{*}^{k}\quad \text{for pre-annuli }A^{0}_{k} \\
&a_{k}:=\rho_{a}\rho_{*}^{k}, b_{k}:=\rho_{b}\rho_{*}^{k} \tand \delta_{k}:=\rho_{*}^{k}\quad \text{for annuli }A_{k}, \\   &d_{k}:=\rho_{d}\rho_{*}^{k}
\end{eqalign}
for constants 
\begin{eqalign}\label{eqro*constants}
\rho_{*}:=2^{-p_{*}}, p_{*}>0,\rho_{a}:=\rho_{*}^{r_{a}},r_{a}>0,\rho_{b}:=\rho_{*}^{r_{b}},r_{b}>0,\rho_{d}:=\rho_{*}^{r_{d}},r_{d}>0,\rho_{P}:=\rho_{*}^{r_p},r_{p}>0,P_{k}\in (0,1)    
\end{eqalign}
We also need a constant for the deviations. We let
 \begin{eqalign}\label{eq:cstarconstant0}
c_{*}:= \frac{\min(\e_{ratio}-\e,1)}{2},    
 \end{eqalign}
for the $\e_{ratio}>0$ from \cref{prop:Multipointunitcircleandmaximum} and arbitrarily small $\e>0$. And then define $c_{idb}\in (0,1)$ to be optimized below and deviation-constants
\begin{eqalign}
c_{gap}:=c_{*}, c_{ut}:=\frac{c_{*}(1-c_{idb})}{2}, c_{comp}:=\frac{c_{*}(1-c_{idb})}{4},    c_{us}:=\frac{c_{*}(1-c_{idb})}{8},c_{trun}:=\frac{c_{*}(1-c_{idb})}{16},c_{\sigma}:=\frac{c_{*}(1-c_{idb})}{32},
\end{eqalign}
and 
\begin{eqalign}
c_{0}:=&(1-c_{idb}),c_{1}:= c_{gap}c_{0}, c_{2}:=c_{1}-c_{ut}, c_{3}:=c_{2}-c_{comp},\\
c_{4}:=& c_{3}- c_{us}, c_{5}:=c_{4},c_{6}:=c_{5}- c_{\sigma},
\end{eqalign}
that we will return again in \nameref{sec:proofofLehtodivergence}.
\end{definition}
\subsubsection{Constraints for Lehto welding of the inverse}
\begin{enumerate}

\item Disjoint annuli $\rho_{b}>\rho_{a}>\rho_{b}\rho_{*}$ and a condition needed later eg. in \cref{eq:constracgamma}:
\begin{eqalign}
\frac{\beta}{2}>r_{a}>r_{b}.
\end{eqalign}

\item Non-empty annuli:
\begin{align}
\rho_{b}e^{-P_{k}}>  \rho_{a}e^{P_{k}}\doncl (r_{a}-r_{b})\ln\frac{1}{\rho_{*}}>2P_{k}.  
\end{align}
This follows from previous constraint plus taking $\rho_{*}$ small.

\item Containment of perturbed pre-annuli $A_{n}^{0}$ inside annuli $A_{n}$
\begin{eqalign}
  u_{n,comp}+u_{n,us}=c\gamma\frac{\rho_{*}^{1/2}}{1-\rho_{*}^{1/2}}  \para{b_{n}+d_{n}}^{1/6}+c\gamma(\rho_{b}+\rho_{d})<P_{n}<1.
\end{eqalign}
This follows by taking small $\rho_{*}$.

\item Disjoint random annuli $A_{n,M}$ in \cref{eq:disjointrandomannuli2} for integers $n_{i+1}\geq n_{i}+1$
\begin{eqalign}
&u_{n_{i},comp}+u_{n_{i},us}+u_{n_{i},n_{i+1},trun}-\beta\ln\frac{1}{\rho_{*}}\\
=&c\gamma\frac{\rho_{*}^{1/2}}{1-\rho_{*}^{1/2}}  \para{b_{n_{i}}+d_{n_{i}}}^{1/6}+c\gamma(\rho_{b}+\rho_{d})+(n_{i+1}-n_{i}-(r_{a}-r_{b}))\ln\frac{1}{\rho_{*}}\\
\leq &(n_{i+1}-n_{i}-(r_{a}-r_{b}))\ln\frac{1}{\rho_{*}}+P_{n_{i+1}}+P_{n_{i}}. 
\end{eqalign}
Here we simply require $r_{a}-r_{b}< 1$.

\item Upper truncated $Q^{n}$ deviation: from \cref{eq:constraintondn} we constrain
    \begin{eqalign}
\rho_{b}> \rho_{d}\doncl  r_{d}>r_{b}.      
    \end{eqalign}

\item Truncated increments $U^{k_{i}}_{k_{i+1}}$.  The constraint in \cref{eq:ukkconstraint}
    \begin{eqalign}\label{eq:constracgamma}
 \frac{(1-\e_{1})}{2}c_{\gamma}c_{5}> 1+\e_{*}  
\end{eqalign}
for 
\begin{eqalign}\label{eq:constracgammac0}
c_{\gamma}:=&\frac{1}{2\gamma^{2}}\para{\beta+1-(r_{a}-r_{b})}^{2}(1-\e_{0})^{2}    
\end{eqalign}
and small $\e_{0},\e_{1}$,where as mentioned in \cref{rem:notwointhelehtoinverse} we don't need to include the $2(r_{a}-r_{b})$ here but we do below in the independent copies. Recall that $c_{5}=\frac{c_{*}(1-c_{idb})}{8}=\frac{(1-c_{idb})}{8} \frac{\e_{ratio}-\e}{2}$. So if we take $r_{a}-r_{b}<\beta$, we then request
\begin{eqalign}\label{eq:constracgamma2}
\eqref{eq:constracgamma}&\doncl \frac{1}{\beta} \frac{(1-\e_{1})(1-\e_{0})^{2}(\e_{ratio}-\e))(1-c_{idb})}{128} >1+\e_{*}\\
&\doncl \frac{1}{\beta} \frac{\e_{ratio}(1-c_{idb})}{128} >1+\e_{*}+o(\e).
\end{eqalign}
In the proof of \cref{prop:Multipointunitcircleandmaximum} in \cite{binder2023inverse} we also had the following constraint for $\e_{ratio}$
\begin{eqalign}\label{eq:constraratio}
\beta^{-1}\frac{p_{1}-1}{p_{1}}>(1+\e_{ratio})\para{1+\beta\spara{p_{1}(1+\e_{ratio})+1}}+o(\e),    
\end{eqalign}
for $p_{1}>1$.

\item In the decoupling existence \cref{thm:gapeventexistence}, we have the constraint in \cref{eq:constraintdecouplingbeta}
\begin{eqalign}\label{eq:constraintra}
&(1-\e)\para{\frac{(\beta+1)^{2}}{4\beta}-\frac{(1+\e_{*})}{c_{*}(1-c_{idb})}}-\frac{\beta^{-1}+1}{2}r_{a}> \frac{(1+\e_{*})}{c_{*}(1-c_{idb})}\\
&\frac{(\beta+1)^{2}}{4\beta}-\frac{\beta^{-1}+1}{2}r_{a}> 2\frac{(1+\e_{*})}{c_{*}(1-c_{idb})}+o(\e),
\end{eqalign}
for $r_{a}\in (0,\beta)$.

\item Number of annuli $N$ and lower bound $N_{0}$.
\begin{itemize}
    \item In \cref{prop:scalescomparedtobN+1} we require  
\begin{eqalign}
&\frac{1}{c_{ov}}R_{N}<\frac{\rho_{a}}{\rho_{b}}  \doncl N_{0}\geq \frac{1}{c_{R}}\para{r_{a}-r_{b}+\para{\ln\frac{1}{\rho_{*}}}^{-1}\ln \frac{c_{ov}}{c}}.
\end{eqalign}
By taking large enough $\rho_{*}$, we ask for    
\begin{eqalign}
N_{0}\geq \frac{1}{c_{R}}\para{r_{a}-r_{b}+1}.    
\end{eqalign}

\end{itemize}

\end{enumerate}

\subsubsection{Borel Cantelli estimate}Here we collect all the constraints that are needed in the Borel-Cantelli estimate in \cref{eq:borelcantelliexponents} (and thus for showing existence of Beltrami solution for the inverse). We include all the above constraints (\cref{eq:constraintra}, \cref{eq:constracgamma2} and \cref{eq:constraratio}). We include the constraints from the proof of \cref{prop:scalescomparedtobN+1} in \cref{eq:determinsticlowerboundcon2}. And the positivity constraint for the second constraint from \cref{eq:borelcantelliexponents}
\begin{eqalign}
&-pc_{R}+\para{\frac{\zeta(p(1+\e))}{(1+\e)}-1}(1+\e_{**}(1-\lambda))>0,    
\end{eqalign}
for $p\in (0,\beta^{-1}).$ In summary, for $y:=1+\e_{**}(1-\lambda)$ we require
\begin{eqalign}\label{eq:determinsticlowerboundcon0}
1)&\frac{1}{\beta} \frac{\e_{ratio}(1-c_{idb})}{128} >1,\\
2)&\beta^{-1}\frac{p_{1}-1}{p_{1}}>(1+\e_{ratio})\para{1+\beta\para{p_{1}(1+\e_{ratio})+1}}, \\
3)&\frac{(\beta+1)^{2}}{4\beta}-\frac{\beta^{-1}+1}{2}r_{a}> 4\frac{1}{\e_{ratio}(1-c_{idb})},\\
4)&-pc_{R}+\para{(p(1-\beta(p-1))-1}>0,    \\
5)&\beta^{-1}c_{R}\lambda_{0}>1,\\
6)&\frac{(1-\lambda_{0})c_{R}-(1+\alpha)c_{idb}}{\sqrt{\beta (1+\alpha)c_{idb}  }   }-\sqrt{\beta (1+\alpha) c_{idb}}   >\sqrt{32y},\\
\tand 7)&\frac{(\beta+1)^{2}}{4\beta}\alpha c_{idb}>1,
\end{eqalign}
where we removed the $o(\e)$-errors and kept only sharp inequalities. Here the variables $(\beta,c_{R},c_{idb},p_{1},r_{a},\e_{ratio},\alpha,\lambda_{0},y,p)$ are contained in the following intervals
\begin{eqalign}
&\beta,\lambda_{0},\e_{ratio},c_{idb}\in (0,1)    \\
&p\in (0,\beta^{-1}),r_{a}\in (0,\beta)\\
&\alpha,c_{R}>0\\
&p_{1}>1\\
&y\in (1,1.0005).
\end{eqalign}
For simplicity, we drop all the $o(\e)$ terms by always keeping strict inequalities. For the (1)-constraint, we take $c_{idb}\in (0,\frac{1}{2})$ and so require the lower bound
\begin{eqalign}\label{eq:erationchoice}
\e_{ratio}>\beta 256.    
\end{eqalign}
So we set $\e_{ratio}:=\beta 256.1$. That turns the (3)-constraint into
\begin{eqalign}
&\frac{(\beta+1)^{2}}{4\beta}-\frac{\beta^{-1}+1}{2}r_{a}> \frac{8}{256.1 \beta} \\   
\doncl &\frac{1}{\beta}\para{\frac{1}{4}+\frac{\beta-r_{a}}{2}-\frac{8}{256.1}}+\frac{\beta}{4}-\frac{r_{a}}{2}> 0, 
\end{eqalign}
which is true for $r_{a}<\frac{\beta}{2}$. The (2)-constraint turns into
\begin{eqalign}
 \beta^{-1}\frac{p_{1}-1}{p_{1}}>(1+\beta 256.1)\para{1+\beta\para{p_{1}(1+\beta 256.1)+1}}.   
\end{eqalign}
Solving this system for $p_{1}=2$, we obtain a solution for each $\beta\in (0, 0.32594)$. For the (4)-constraint we maximize the LHS over $p\in (0,\beta^{-1})$ to bound by
\begin{eqalign}
(\sqrt{\beta}-1)^{2}>c_{R} .   
\end{eqalign}
And so combined with (5) and (6), we ask
\begin{eqalign}
&(\sqrt{\beta}-1)^{2}>\beta\lambda_{0}^{-1}    \\
&(\sqrt{\beta}-1)^{2}>\frac{1}{(1-\lambda_{0})}\para{(1+\alpha)c_{idb}(1+\beta)+\sqrt{32}\sqrt{\beta (1+\alpha) c_{idb}}} .
\end{eqalign}
Using the (7)-constraint we set
\begin{eqalign}
\alpha:=\frac{4.1\beta}{(\beta+1)^{2}c_{idb}},    
\end{eqalign}
to get
\begin{eqalign}
&(\sqrt{\beta}-1)^{2}>\beta\lambda_{0}^{-1}    \\
&(\sqrt{\beta}-1)^{2}>\frac{1}{(1-\lambda_{0})}\para{(c_{idb}+\frac{4.1\beta}{(\beta+1)^{2}})(1+\beta)+\sqrt{32}\sqrt{\beta (c_{idb}+\frac{4.1\beta}{(\beta+1)^{2}}) }} .
\end{eqalign}
Here we use $c_{idb}\leq \frac{1}{2}$ to get the system
\begin{eqalign}
&(\sqrt{\beta}-1)^{2}>\beta\lambda_{0}^{-1}    \\
&(\sqrt{\beta}-1)^{2}>\frac{1}{(1-\lambda_{0})}\para{(\frac{1}{2}+\frac{4.1\beta}{(\beta+1)^{2}})(1+\beta)+\sqrt{32}\sqrt{\beta (\frac{1}{2}+\frac{4.1\beta}{(\beta+1)^{2}}) }}.
\end{eqalign}
Here we can insert this system for $\beta,\lambda_{0}\in (0,1)$ in Mathematica to get a solution for each $\beta\in (0,0.00602498)$. But a close enough answer can be obtained more directly by setting $\lambda_{0}=\frac{1}{100}$. The first constraint forces
\begin{eqalign}
(\sqrt{\beta}-1)^{2}>\beta 100\Rightarrow\beta<\frac{1}{121}\approx 0.00826446
\end{eqalign}
The second constraint now is only terms of $\beta$
\begin{eqalign}
(\sqrt{\beta}-1)^{2}>\frac{100}{99}\para{(\frac{1}{2}+\frac{4\beta}{(\beta+1)^{2}})(1+\beta)+\sqrt{32}\sqrt{\beta (\frac{1}{2}+\frac{4\beta}{(\beta+1)^{2}}) }},    
\end{eqalign}
and we can get solution for each $\beta<0.00596838$, which is the smallest lower bound of the above. If we instead take $c_{idb}$ to be $o(\e)$, then we get
\begin{eqalign}
(\sqrt{\beta}-1)^{2}>\frac{100}{99}\para{(\frac{4\beta}{(\beta+1)^{2}})(1+\beta)+\sqrt{32}\sqrt{\beta (\frac{4\beta}{(\beta+1)^{2}}) }},    
\end{eqalign}
and a solution for $\beta<0.042477$.
\subsubsection{Constraints for Lehto of independent copies}\label{def:exponentialchoiceparamindcop}
\begin{enumerate}
   
   \item We take $\beta_{1}\geq \beta_{2}$

\item We need
\begin{eqalign}
r_{b}+\frac{1}{2}\minp{1,\beta}>r_{a}>r_{b}
\end{eqalign}
and to have non-empty overlapped annuli we need
\begin{eqalign}
\rho_{b}>\rho_{a}+\rho_{P}.    
\end{eqalign}

\item From the overlapped-annuli section we have \cref{eq:overlappednannuliconstraint}
\begin{eqalign}\label{eq:overlappednannuliconstraintindcp}
\frac{\para{\lambda_{6}\e_{0}(1+\beta_{1})}^{2}}{2\beta_{1}}q_{5}c_{*}=\frac{1}{64\beta_{1}}(1+\beta_{1})^{2} \frac{\e_{ratio}-\e}{2}>1+\e_{**}(1-\lambda).   
\end{eqalign}
So using the choice from \cref{eq:erationchoice} $\e_{ratio}=\beta_{1} 256.1+o(\e)$, we are left with
\begin{eqalign}
2(1+\beta_{1})^{2}>1+o(\e),    
\end{eqalign}
which is true.

\item For the unique-pairs sections we use the results on truncated increments $U^{k_{i}}_{k_{i+1}}$, but instead of \cref{eq:constracgamma2}, we instead have
\begin{eqalign}
c_{\gamma}:=&\frac{1}{2\gamma^{2}}\para{\beta+1-2(r_{a}-r_{b})}^{2}(1-\e_{0})^{2}.    
\end{eqalign}
That's why we are forced to take $r_{a}<\frac{1}{2}\beta$ instead. 
   
\item Overlap constraint in \cref{eq:coverlapconstraint}
    \begin{eqalign}
\frac{1-c_{ov}\frac{\rho_{a}}{\rho_{b}}}{1+c_{ov}\frac{\rho_{a}}{\rho_{b}}}> \frac{\rho_{a}+\rho_{P}}{\rho_{b}}.
\end{eqalign}
Since we already assumed that $\rho_{b}>\rho_{a}+\rho_{P}$, this constraint is true by taking $\rho_{*}$ small enough so that the upper bound is close to $1$.

\end{enumerate}

\newpage\part{Beltrami solution for the inverse }\label{beltinversepart} 
\section{Conformal welding and the Lehto condition}\label{weldLeh}
For exposition and references on the Lehto approach see \cite[section 2]{AJKS} and \cite[section 20.9]{astala2008elliptic}. Here we just define the objects and state some of the results that we will need.
\subsection{Existence in the degenerate case: the Lehto condition}
Here we are forced to work in the degenerate case of Beltrami coefficient $\norm{\mu}_{\infty}=1$. One approach is built on the \textit{Lehto integral}
\begin{equation}\label{Lehtoint}
\mathcal{L}_{K}(z,r,R):=\int_{r}^{R}\frac{1}{\int_{0}^{2\pi}K(z+\rho e^{i\theta})\dtheta}\frac{\drho}{\rho}. 
\end{equation}
For other approaches in the degenerate case see \cite[chapter 4]{gutlyanskii2012beltrami}. This integral comes to be useful because it controls the modulus of a function $f$ with dilatation $K_{f}=K$. Given a bounded (topological) annulus $A\subset \C$, with $E$ being the bounded component of $\C\setminus A$, we denote by $D_O(A):=diam(A)$ the outer diameter, and by $D_I(A):=diam(E)$ the inner diameter of $A$.
\begin{lemma}\label{Lehtoequic}(\cite[Lemma 2.3]{AJKS})
Let $f$ be a quasiconformal mapping on the annulus $A=A(w, r, R)$, with
distortion function $K_f$ . It then holds that
\begin{equation}
D_{I}(f(A)\leq 16 D_{O}(f(A) e^{-2\pi^{2}L_{K_{f}}(w,r,R) },   
\end{equation}
where $L_{K_{f}}$ is the Lehto integral for the dilatation of $f$.
\end{lemma}
\noindent Therefore, for $R=1$ and $r\to 0$ if one can obtain that $L_{K_{f}}(w,r,1)\to +\infty$, we get $D_{I}(f(A)\to 0$, which is a form of equicontinuity that is in turn used to obtain existence of solutions for the Beltrami equation. For the following formulation of Lehto's theorem see \cite[p. 584]{astala2008elliptic}
\begin{theorem}\label{Lehtosolutions}
Suppose the following.
\begin{enumerate}
    \item the Beltrami coefficient $\mu$ is measurable and compactly supported with $\abs{\mu(z)}<1$ for almost every $z\in\C$.
    
    \item The distortion function is locally integrable $K(z):=\frac{1+\abs{\mu(z)}}{1-\abs{\mu(z)}}\in L^{1}_{loc}(\C)$.
    
    \item (Lehto condition) for some $R_0>0$, the Lehto integral diverges:
    \begin{equation}
    \mathcal{L}_{K}(z,r,R_{0})\to \infty \tas r\to 0^+,~~ \quad\forall z\in \C.    
    \end{equation}
    
\end{enumerate}
Then the degenerate Beltrami equation  
\begin{equation}\label{Belteq}
\frac{\partial f}{\partial \thickbar{z}}=\mu(z) \frac{\partial f}{\partial z}   , \tfor~a.e.~ z\in \Omega,
\end{equation}
admits a homeomorphic solution $f\in W^{1,1}_{loc}(\C)$.
\end{theorem}
\noindent As a consequence, the welding extends beyond the class of quasisymmetric functions.
\begin{corollary}(\cite[Corollary 2.2.]{AJKS})
Suppose that $\phi:\T\to\T$ extends to a homeomorphism $f: \C\to\C$ satisfying the above assumptions together with $K(z)\in L^{\infty}(\C)$. Then $\phi$ admits a welding: there are a Jordan curve $\Gamma\subset\hat{\C}$ and conformal mappings $f_{\pm}$ onto the complementary domains of $\Gamma$ such that
\begin{equation}
\phi(z)=f^{-1}_{+}\circ f_{-}(z),z\in \uc.  
\end{equation}
\end{corollary}
For uniqueness one needs a good control on the modulus. One result requires \Holder.
\begin{corollary}(\cite[Corollary 2.5]{AJKS})\label{cor:uniquenessholder}
Suppose that $\phi:\T\to\T$ admits welding
\begin{equation}
\phi(z)=f^{-1}_{+}\circ f_{-}(z),z\in \uc.  
\end{equation}
If $f_{-}\tor f_{+}$ is \Holder continuous on the boundary $\mathbb{T}$. Then the welding is unique i.e. any other welding pair $(g_{+},g_{-})$ for $\phi$ is of the form
\begin{eqalign}
g_{\pm}=\Phi\circ f_{\pm},   
\end{eqalign}
where $\Phi:\hat{\mathbb{C}}\to \hat{\mathbb{C}}$ is a \Mobius transformation.
\end{corollary}
In the inverse setting we will study the inverse homeomorphism $\phi^{-1}:\T\to \T$ defined as $\phi^{-1}(e^{2\pi i x})=e^{2\pi i h^{-1}(x)}$ where $h^{-1}:\R\to \R$ is the homeomorphism of the real line defined as $h^{-1}(x):=Q_{\tau}(x\tau(0,1))$ for $x\in[0,1)$, with periodicity $h^{-1}(x+1)=h^{-1}(x)+1$ and $h^{-1}(0)=0$. Therefore, we will apply all the above results to the homeomorphism $\phi^{-1}$. We denote the dilatation as $K_{h^{-1}}=K_{Q}$.
\subsection{Extension of the homeomorphism \sectm{$\phi^{-1}$}}\label{ABext}
As in \cite[section 2.4]{AJKS}, we use the same extension maps but for the inverse. We define the Ahlfors-Beurling extension $F:\mb{H}\to \mb{H}$ (\cite{beurling1956boundary}) as follows:
\begin{equation}
F(x+iy):=\left\{\begin{matrix}
 \frac{1}{2}\int_{[0,1]}h^{-1}(x+ty)+h^{-1}(x-ty)\dint t+i\frac{1}{2}\int_{[0,1]}h^{-1}(x+ty)-h^{-1}(x-ty)\dint t
&,0<y<1 \\ \\
x+i+c_{0} &, y=1 \\\\ 
x+iy+(2-y)c_{0} &, 1\leq y\leq 2\\\\
 x+iy&, y\geq 2 
\end{matrix}\right.    
\end{equation}
where $c_{0}:=\int_{[0,1]}h^{-1}(t)\dint t-\frac{1}{2}$. The periodicity $h^{-1}(x+1)=h^{-1}(x)+1$ implies periodicity $F(z+k)=F(z)+k$ for $k\in \mathbb{Z}$. 
 \begin{figure}[ht]
 \centering
 \includegraphics[scale=.2]{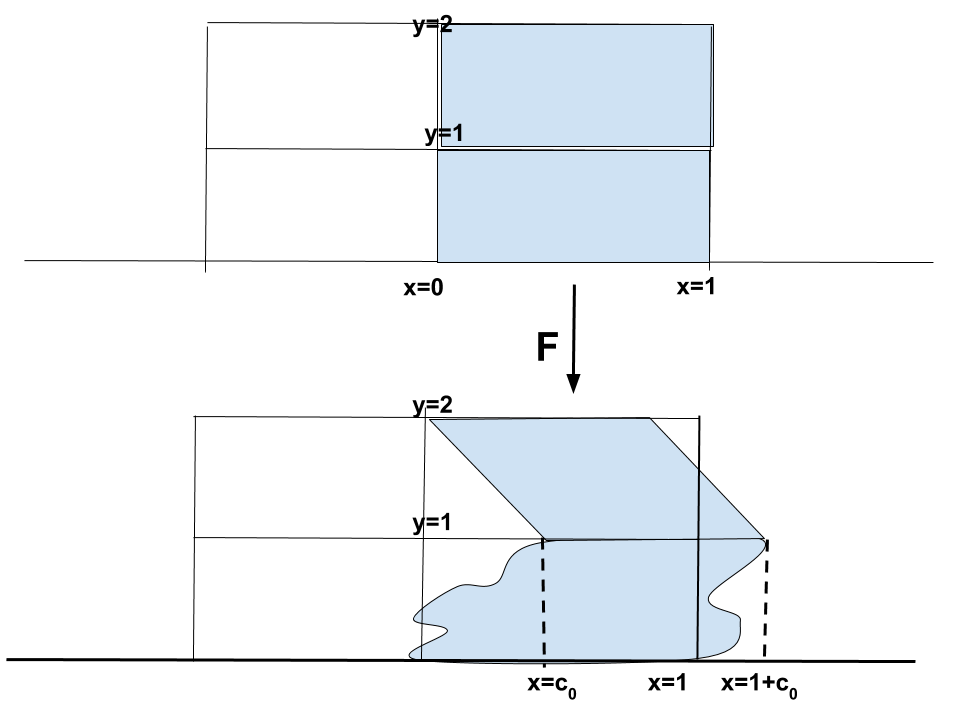}
 \caption{Image of $F$.}
 \end{figure}
~\\
\noindent The original circle mapping admits a natural extension to the disc,
\begin{equation}
\Psi(z):=\expo{2\pi i F\para{\frac{\ln(z)}{2\pi i}} },z\in \ud.    
\end{equation}
From above we see that this is a well-defined homeomorphism of the disc with
$\Psi\mid_{\T}=\phi$ and $\Psi(z)\equiv z$ for $\abs{z}\leq e^{-4\pi}$.
\subsection{Boundary correspondence}\label{sec:boundarycorrespondence}
The distortion properties are not altered under this locally conformal change of variables,
\begin{equation}
K(z,\Psi)=K(w,F)    ,z=e^{2\pi i w}, w\in \mb{H},
\end{equation}
so we will reduce all distortion estimates for $\Psi$ to the corresponding ones for $F$. Since $F$ is conformal for $y>2$, it suffices to restrict the analysis to the strip $S=\R\times [0,2]$. To estimate $K(w, F)$ we introduce some notation. Let $\mathcal{D}_{n}:=\set{[\frac{k}{2^{n}},\frac{k+1}{2^{n}}]: k\in \Z}$ be the set of all dyadic intervals of length $2^{-n}$ and write $\mathcal{D}=\bigcup_{n\geq 0 }\mathcal{D}_{n}$ for their union.\\
Consider the measure $Q_{\tau}\para{(a\tau(0,1),b\tau(0,1))}:=h^{-1}(b)-h^{-1}(a)=Q_{\tau}\para{(0,b\tau(0,1))}-Q_{\tau}\para{(0,a\tau(0,1))}$ and for a pair of intervals $\mathbf{J}=(J_{1},J_{2})$ in $\R$ the ratio 
\begin{equation}
\delta_{\tau}(\mathbf{J}):=\frac{Q_{\tau}\para{(J_{1})\tau(0,1)}}{Q_{\tau}\para{(J_{2})\tau(0,1)}}+ \frac{Q_{\tau}\para{(J_{2})\tau(0,1)}}{Q_{\tau}\para{(J_{1})\tau(0,1)}}.
\end{equation}
If $J_1$ and $J_2$ are the two halves of an interval $I$, then $\delta_{\tau}(\mathbf{J})$ measures the local doubling properties of the measure $\tau$ . In such a case we define $\delta_{\tau}(I):=\delta_{\tau}(\mathbf{J})$. In particular, the quasisymmetric condition \cite[Lemma 3.11.3 and Theorem 5.8.1]{astala2008elliptic}
\begin{eqalign}
K(\phi):=\sup_{s,t\in \mathbb{R}}\frac{\abs{\phi(e^{2\pi i (s+t)})-\phi(e^{2\pi i (s)})}}{\abs{\phi(e^{2\pi i (s-t)})-\phi(e^{2\pi i (s)})}}<\infty    
\end{eqalign}
holds for the circle homeomorphism $\phi^{-1}(e^{2\pi i x})=e^{2\pi i h^{-1}(x)}$ \textit{if and only if} the quantities $\delta_{\tau}(I)$ are uniformly bounded, for all (not necessarily dyadic) intervals $I$.\\
The local distortion of the extension $F$ will be controlled by sums of the expressions
$\delta_{\tau}(\mathbf{J}):$ in the appropriate scale. For this, let us pave the strip $S=\R\times [0,2]$ by \textit{Whitney cubes} $\set{C_I }_{I\in\mathcal{D}}$ defined by
\begin{equation}
C_{I}:=\set{(x,y)\in \mb{H}: x\in I\tand 2^{-n-1}\leq y\leq 2^{-n}}    
\end{equation}
for $I\in \mathcal{D}_{n},n>0$ and for $I\in \mathcal{D}_{0}$, we let $C_{I}:=I\times[\frac{1}{2},2]$. Given an $I\in\mathcal{D}_n$ let $j_{0}(I)$ be the union of $I$ and its neighbors in $\mathcal{D}_n$ and
\begin{equation}
j_{5}(I):=\set{ \mathbf{J}=(J_{1},J_{2}): J_{1},J_{2}\in \mathcal{D}_{n+5}\tand  J_{1},J_{2}\subset j_{0}(I)}.    
\end{equation}
We then define the boundary correspondent term
\begin{equation}
K_{Q}(I):=\sum_{\mathbf{J}\in j_{5}(I) }  \delta_{\tau}(\mathbf{J})=\sum_{\mathbf{J}\in j_{5}(I) } \frac{Q_{\tau}(J_{1})\tau(0,1)}{Q_{\tau}(J_{2})\tau(0,1)}+ \frac{Q_{\tau}(J_{2})\tau(0,1)}{Q_{\tau}(J_{1})\tau(0,1)}. 
\end{equation}
With these notions we have the basic geometric estimate for the distortion function,
in terms of the boundary homeomorphism $h^{-1}$.
\begin{theorem}\label{boundarycorres}(\cite[theorem 2.6]{AJKS})
Let $h^{-1}: \R\to\R$ be a $1$-periodic homeomorphism and let $F: \mb{H}\to \mb{H}$ be its
above extension. Then, for each $I\in \mathcal{D}$, we have the estimate for its dilatation over Whitney cubes in terms of the boundary interval
\begin{equation}
\supl{z\in C_{I}}K(z,F)\leq C_{0}K_{Q}(I)=C_{0}\sum_{\mathbf{J}\in j_{5}(I) }   \delta_{\tau}(\mathbf{J}),
\end{equation}
with a universal constant $C_0$.
\end{theorem}
\noindent So for the remaining work we will study the above boundary bound in order to satisfy the conditions in the Lehto \cref{Lehtosolutions}.
\section{Definition of the random annuli}\label{sec:randomannuliinverse}
In lower bounding the Lehto integral, one needs a suitable subsequence choice of annuli whose sum of moduli diverges. We study the Lehto integral from \cref{Lehtoint}
\begin{equation}
\mathcal{L}_{K}(z,r,R):=\int_{r}^{R}\frac{1}{\int_{0}^{2\pi}K_{\tau}(z+\rho e^{i\theta})\dtheta}\frac{\drho}{\rho}, z\in \C
\end{equation}
for the random homeomorphism $h^{-1}:[0,1]\to [0,1]$ defined as
\begin{equation}
h^{-1}(x):=Q_{\tau}(x\tau([0,1])\tfor x\in [0,1],     
\end{equation}
where $\tau([a,b]):=\int_{a}^{b}\liz{\e}e^{\thickbar{H}_{\e}(s)}ds$. As we will see in detail in \cref{Beltinvsection}, the most delicate situation is studying $\mathcal{L}_{K}(z,r,R)$ on the unit circle, $z\in \T$. In particular,  for each $n\geq 1$ we cover the unit circle $\T$ by some number $D_{n}$ of disks $B_{R_{n}}(e^{2\pi i y_{k}})$ of radius $R_{n}:=cR^{-n}$ for $R>1$ obtained in \cref{sec:holderequicontinuitybounds} and centered on $y_{k}:=\frac{\tau(0,x_{k,D_{n}})}{\tau(0,1)}$ for some deterministic sequence $x_{k,D_{n}}\in [0,1]$ that we will choose in order that the disks indeed cover $\T$. Finally, we consider the event that the Lehto integral doesn't diverge
\begin{equation}
E_{n,k}:=\set{ L_{K_{\tau}}(y_{k},R_{n},2)<n\delta   },    
\end{equation}
for small $\delta>0$. In \cref{part:deviationestLehto} we will show the following summability.
\begin{restatable}[Lehto integral divergence]{theorem}{lehtodivergence}
\label{Lehtodivergent}
There exists constant $\e_{*}>0$ and $\delta_{0}$ such that for all $0<\delta<\delta_{0}$, the Lehto integral satisfies the estimate
\begin{equation}
\Proba{ L_{K_{\tau}}(y_{0},R_{n},2)<\delta n  } \leq c\rho_{*}^{(1+\e_{**})n},\forall n\geq n_{0}
\end{equation}
for some large enough $n_{0}>0$ (that does \textit{not} depend on $\rho_{*}$) and small parameter $\e_{**}>0$.
\end{restatable}
\begin{remark}
Thus the Lehto integral satisfies the summability
\begin{equation}
\sum_{n\geq 1}\Proba{\bigcup_{k=0}^{D_{n}}E_{n,k}}<\infty
\end{equation}
and so by Borel-Cantelli obtain $\Proba{\bigcup_{m\geq 1}\bigcap_{n\geq m} \bigcap_{k=0}^{D_{n}}E_{n,k}^{c} }=1$ i.e. for each $\omega$ there exists $m(\omega)$ such that for all $n\geq m(\omega)$ and all $k=0,...,N_n$ we have divergent Lehto integrals
\begin{equation}
 L_{K_{\tau}}(y_{k},R_{n},2)>n\delta.   
\end{equation}
\end{remark}
\subsection{Modified annuli}\label{modifiedannuli}
In this section we go over the choice of the randomly-centered and randomly-scaled annuli that will accommodate the conditional independence results.  The main part is the need to rescale them in order that we can eventually decouple the quasisymmetric ratios and thus capitalize on the conditional independence of the inverse GMC.  Why not just stay with $Q_{H}$ and use deterministic annuli? One would still have to handle the upper-scale fields $H_{n}^{1}$ somehow. If these are not removed, then there is no decoupling because to get the existence of a gap event, we essentially use that the overlap $\set{Q^{k}(a_{k})-Q^{k+n}(b_{k+n})\leq \delta_{k+n}}$ decays in $n$.  
\pparagraph{Pre-annuli $A_{n}^{0}$}
Before we do any randomization, we first start with a deterministic choice of annuli to randomize. We will use the \nameref{def:exponentialchoiceparamannul}, namely we start with the square annuli around the origin for each $n=1,...,N$ and $N\geq 1$
\begin{equation}
A_{n}^{0}:=\set{(x,y)\in \mathbb{H}: x\in [-b_{n}^{0},-a_{n}^{0}]\cup [a_{n}^{0},b_{n}^{0}] \tand y\in[-b_{n}^{0}, b_{n}^{0}] \tor x\in [-a_{n}^{0},a_{n}^{0}]\tand y\in  [a_{n}^{0},b_{n}^{0}]\cup[-b_{n}^{0},-a_{n}^{0}]}.   
\end{equation}
We next shift their centers and base $[a_{n}^{0},b_{n}^{0}]$. We will also sometimes use the notation $A([a,b])$ to denote an annulus with base given by $[a,b]$ as described in the picture.
\begin{figure}[ht]
    \centering
\begin{tikzpicture}
    \draw[very thick] (-3,-3) rectangle (3, 3) [fill=blue!10]; 
    \draw[very thick] (-1, -1) rectangle (1,1) [fill=white]; 
   \draw[->,ultra thick] (-4,0)--(4,0) node[right]{$x$};
    \draw[->,ultra thick] (0,-4)--(0,4) node[above]{$y$};    
    
    \draw (3,0) node[above]{$(b_{n}^{0},0)$}; 
    \draw (1,0) node[above]{$(a_{n}^{0},0)$}; 
    \node at (1,0)[circle,fill,inner sep=2pt]{};
    \node at (3,0)[circle,fill,inner sep=2pt]{};

\draw (0,3) node[above]{$(0,b_{n}^{0})$}; 
    \draw (0,1) node[above]{$(0,a_{n}^{0})$}; 
    \node at (0,1)[circle,fill,inner sep=2pt]{};
    \node at (0,3)[circle,fill,inner sep=2pt]{};

\draw (0,3) node[above]{$(0,b_{n}^{0})$}; 
    \draw (0,1) node[above]{$(0,a_{n}^{0})$}; 
    \node at (0,1)[circle,fill,inner sep=2pt]{};
    \node at (0,3)[circle,fill,inner sep=2pt]{};

   \draw (-3,0) node[above]{$(-b_{n}^{0},0)$}; 
    \draw (-1,0) node[above]{$(-a_{n}^{0},0)$}; 
    \node at (-1,0)[circle,fill,inner sep=2pt]{};
    \node at (-3,0)[circle,fill,inner sep=2pt]{};

\draw (0,-3) node[above]{$(0,-b_{n}^{0})$}; 
    \draw (0,-1) node[above]{$(0,-a_{n}^{0})$}; 
    \node at (0,-1)[circle,fill,inner sep=2pt]{};
    \node at (0,-3)[circle,fill,inner sep=2pt]{};

\end{tikzpicture}
    \caption{Pre-annuli $A_{n}^{0}$}
    \label{fig:preannuli}
\end{figure}
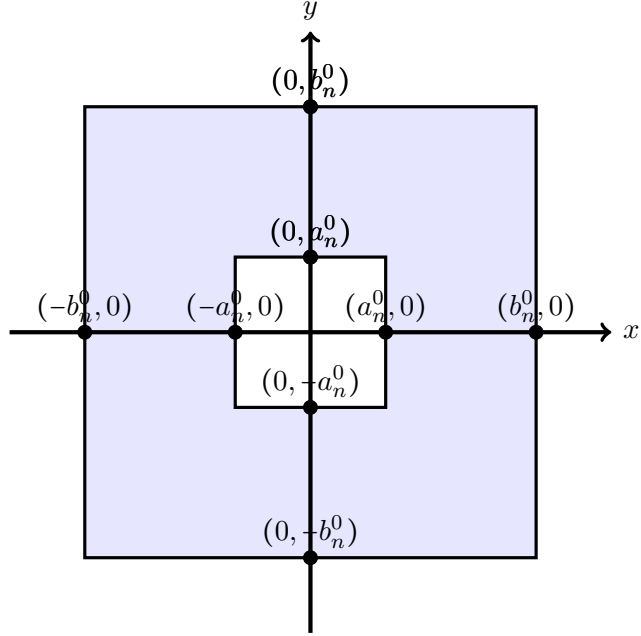
~\\
\pparagraph{Centers of the annuli $A_{n}^{0}$} We start with their centers. In the \cite{AJKS} work, they used the translation invariance of GMC in order to focus on getting estimates for the annuli around the origin only, as opposed to going around the unit circle. It would have been useful to have translation invariance across different scales for the inverse i.e. for deterministic $y_{0},a,b$ to have
\begin{equation}
\para{Q^{n}(y_{0},y_{0}+a),    Q^{n+1}(y_{0},y_{0}+b)}\eqdis \para{Q^{n}(0,a),    Q^{n+1}(0,b)}
\end{equation}
however as explained in \cref{nontranslationinvar} the inverse likely doesn't have translation invariance. So instead we must work with randomly shifted annuli. We will use the shift relation from \cref{it:shiftscalingdiff}:  fix $x_{0}\in [0,1]$ and let $Q_{x_{0},\tau}$ be the inverse of the shifted measure $\tau_{x_{0}}(0,x):=\int_{0}^{x}e^{H(x_{0}+s)}\ds$, then
\begin{equation}
Q_{\tau}\para{\tau(0,x_{0})+a,\tau(0,x_{0})+b}= Q_{x_{0},\tau}(a,b).
\end{equation}
However, since we have stationarity $H(x_{0}+s)\eqdis H(s)$, we get
\begin{equation}
  Q_{x_{0},\tau}(a,b)\eqdis Q_{\tau}(a,b).  
\end{equation}
That means that the natural center that allows for "translation" invariance is $\tau(0,x_{0})$. Therefore, in the proof of \cref{Lehtoinverse}, for each $N$ we will choose a sequence $\set{y_{m,D}=\frac{\tau[0,x_{m,D}]}{\tau[0,1]}}_{m=0}^{D}\subset [0,1]$ ,for some $D=D(N)$, that are evenly spread out to serve as centers for the annuli in the Lehto integral. So by stationarity it again suffices to study only one of those centers namely the one at location zero. We return to this in \cref{sec:Coveringtheannulus}. In summary we take our pre-annuli $A_{n}^{0}$ and shift them
\begin{equation}
A_{n,m}^{0}:=A_{n}^{0}+y_{m,D}    
\end{equation}
for $N\geq 1$, $n=1,...,N$ and $m=0,...,D$.
\pparagraph{Scaling of the annuli $A_{n}^{0}$}
\noindent Next, we choose the random scaling for $A_{n}^{0}$. First, the inverse homeomorphism $h^{-1}(x):=Q_{\tau}(x\tau([0,1])\tfor x\in [0,1]$ contains the factor $\tau[0,1]$ which depends on all the scales and so if we hope to have any kind of decoupling between distant annuli analogous to \cite[section 4.1]{AJKS}, it needs to be scaled out.\\
 Secondly, when moving from the inverse $Q_{\tau}$ of $\tau$ to the inverse $Q:=Q_{\eta}$ of $\eta$ , as explained in \cref{it:doubleboundinv} we also gain a comparison field that depends on all the scales and it will need need to be scaled out: let $[c,c+y]\subset [0,1]$ then
\begin{equation}
Q_{\tau}([c,c+y ])=Q_{\eta}\para{\frac{1}{e^{\xi(\theta_{0,c}))}}\spara{c,c+ye^{\xi(\theta_{0,c})-\xi(\theta_{c,c+y})}}}.    
\end{equation}
We recall
\begin{equation}
\xi(s):=\lim_{\e\to 0}H_{\e}(s )-U_{\e}(s)  
\end{equation}
and $\theta_{a,b}\in [Q_{\tau}(a),Q_{\tau}(b)] $.\\
Finally, again en route of decoupling in the spirit of \cite[section 4.1]{AJKS}, we need to be removing scales $Q\to Q^{n}$ so that at each annulus $A_{n}^{0}$ with bases $A_{n}^{0}\cap \mb{R}:=[-b_{n}^{0},-a_{n}^{0}]\sqcup [a_{n}^{0},b_{n}^{0}]$, the measures $Q^{n}[a_{n}^{0},b_{n}^{0}],Q^{n+1}[a_{n+1}^{0},b_{n+1}^{0}]$ will satisfy some form of conditional independence. This means we will need to use the following scale-relations from \cref{it:shiftscalingdiff}
\begin{equation}
Q([c,c+y ])=Q^{n}\para{\frac{1}{G_{n}(c)}\spara{c,c+y\frac{G_{n}(c)}{G_{n}(c,c+y)}}},    
\end{equation}
where
\begin{equation} 
G_{n}(c,c+y)=\expo{\overline{U_{n}}(\theta_{c,c+y} )  }\tforsome \theta_{c,c+y}\in [Q(c),Q(c+y)].     
\end{equation}
Here we note that the factor $\frac{1}{G_{n}(c)}$ can be arbitrarily large or small and that it depends on all the scales. So again if we hope to decouple, we must scale out that factor. \\
We will need a sequence of slightly larger intervals
\begin{equation}
[-b_{n},-a_{n}]\sqcup [a_{n},b_{n}]    
\end{equation}
for sequences $a_{n},b_{n}$ recorded in \nameref{def:exponentialchoiceparamannul} that will contain the ratios. In summary, we consider the following annuli $A^{0}_{n,k}$.
\begin{definition}\label{def:scalingfactornthannulus}(The center and Scaling factor for the nth-scale annulus)
The center for $A^{0}_{n,k}$ is $y_{k,D}=\frac{\tau[0,x_{k,D}]}{\tau[0,1]}$ for $x_{k,D}:=\frac{k}{D_{n}}$ and $D_{n}:=\ceil{\rho_{*}^{-(1+\e_{**}(1-\lambda)) }}$ for some $\lambda\in (0,1)$ and $\e_{**}>0$ from \cref{Lehtodivergent}. The scaling factor for the annuli $A^{0}_{n,k}$ is
\begin{eqalign}\label{imageendpoints}
M_{n,k}:=&\frac{1}{b_{n}}\frac{\tau([x_{k,D},x_{k,D}+Q^{n}(\eta^{n}(x_{k,D}),\eta^{n}(x_{k,D})+b_{n})]}{\tau([0,1])}\\
=&e^{\xi(x_{k,D}+\theta_{b_{n}}^{H-U})}e^{\thickbar{U}^{1}_{n}(x_{k,D}+\theta_{b_{n}})}\frac{1}{\tau([0,1])},   
\end{eqalign}
where for some
\begin{eqalign}
\theta_{b_{n}}^{H-U},\theta_{b_{n}}\in [0,Q^{n}(\eta^{n}[0,x_{k,D}],\eta^{n}[0,x_{k,D}]+b_{n})].    
\end{eqalign}
This is interesting because it will fit with the image of gluing two disks along their boundaries guided by stopping times $Q^{n}$. The gluing happens in the image of the homeomorphism $h(x)=\frac{\tau(0,x)}{\tau(0,1)}$. We also consider a version of the constants without $\tau([0,1])$ 
\begin{equation}
\wt{M}_{n,k}:=M_{n,k}\tau([0,1])=e^{\xi(x_{k,D}+b_{n})}e^{\thickbar{U}^{1}_{n}(x_{k,D}+\theta_{b_{n}})}   
\end{equation}
because  $\tau([0,1])$ is uniformly applied to all annuli and so we remove it in the beginning. That means that we will use the annuli
\begin{eqalign}
A_{n,M,k}:=&\big\{(x,y)\in \mathbb{H}: x\in [y_{k,D}-b_{n,M,k},y_{k,D}-a_{n,M,k}] \cup    [y_{k,D}+a_{n,M,k},y_{k,D}+b_{n,M,k}]\\
&\tand y\in [-b_{n,M,k},b_{n,M,k}] \tor x\in \spara{y_{k,D}-a_{n,M,k},y_{k,D}+a_{n,M,k}}\tand y\in  \spara{a_{n,M,k},b_{n,M,k}} \big\},   
\end{eqalign}
where $a_{n,M,k}:=a_{n}^{0}M_{n,k}\tand b_{n,M,k}:=b_{n^{0}}M_{n,k}$. For the case $k=0$, we will drop the subscript $k$ i.e. $M_{n}:=M_{n,0}$.
\end{definition}
\noindent We will explain in \cref{sec:annulidecoupl} how to use this factor to get decoupled ratios but the main idea is that they assist in forming ratios which can then be controlled with concentration estimates in the spirit of Borel-TIS inequalities \cref{lem:moduest2d}.
\subsection{Disjoint modified annuli \sectm{$A_{n,M}$}}\label{sec:disjointannuli}
In \cite[section 4.1]{AJKS}, once they select disjoint annuli with base $[\rho^{n},2\rho^{n}]$, they lower bound the Lehto integral
\begin{equation}
\mathcal{L}_{K}(0,\rho^{N},2\rho)\geq c \sum_{n=1}^{N}    \mathcal{L}_{K}(0,\rho^{n},2\rho^{n}).
\end{equation}
So here too we would need those random annuli to be disjoint. By the relations \cref{imageendpoints} the bases are
\begin{eqalign}
A_{n,M}\cap \mb{R}&:=[-b_{n,M},-a_{n,M}]\sqcup [a_{n,M},b_{n,M}]\\
&=\spara{-\frac{\tau[0,Q^{n}(b_{n}^{0})]}{\tau([0,1])},-\frac{a_{n}^{0}}{b_{n}^{0}}\frac{\tau[0,Q^{n}(b_{n})]}{\tau([0,1])}}\sqcup \spara{\frac{a_{n}^{0}}{b_{n}^{0}}\frac{\tau[0,Q^{n}(b_{n})]}{\tau([0,1])},\frac{\tau[0,Q^{n}(b_{n})]}{\tau([0,1])}},
\end{eqalign}
and so to obtain disjointness we are asking for the event $E_{n,dis}$
\begin{equation}\label{eq:disjointannuli}
 \frac{M_{n+1}}{M_{n}}<\frac{a_{n}^{0}}{b_{n+1}^{0}}\Leftrightarrow \frac{ \tau[0,Q^{n+1}(b_{n+1})]}{\tau[0,Q^{n}(b_{n})]}<\frac{a_{n}^{0}}{b_{n+1}^{0}},
\end{equation}
where the latter condition can be seen as a regularity constraint for the quasisymmetric ratio. In terms of fields the condition \cref{eq:disjointannuli} can be written as follows
\begin{equation}\label{eq:disjointnesscondition}
 \frac{M_{n+1}}{M_{n}}= \expo{\xi(b_{n+1})-\xi(b_{n})}\expo{\thickbar{U}^{1}_{n}(\theta_{b_{n+1}})-\thickbar{U}^{1}_{n}(\theta_{b_{n}})}   \expo{\thickbar{U}^{n}_{n+1}(\theta_{b_{n+1}})} <\frac{a_{n}^{0}}{b_{n+1}^{0}}.
\end{equation}
We will return to this disjointness in \cref{eq:disjointrandomannuli}. 
\begin{remark}
 Given that the gap event existence gives some ordering
\begin{equation}
    Q^{n_{i+1}}(b_{n_{i+1}})+\delta_{n_{i+1}}\leq Q^{n_{i}}(a_{n_{i}}),
\end{equation}
it would be interesting if we could use this to directly get disjointness also, but unfortunately we would still need to get with small probability error that:
\begin{eqalign}
&\frac{b_{n_{i}}}{a_{n_{i}}}\tau\para{Q^{n_{i+1}}(b_{n_{i+1}})}\leq \tau\para{Q^{n_{i+1}}(b_{n_{i+1}})+\delta_{n_{i+1}}}.   
\end{eqalign}    
\end{remark}
\subsection{Equicontinuity and Uniqueness}
In the proof of \cref{Lehtoinverse} to obtain uniqueness we will need \Holder equicontinuity \cref{cor:uniquenessholder}. 

To use \cref{cor:uniquenessholder}, the innermost radii need to have a deterministic lower bound of the form $R_{N}:=cR^{-N}$ for some $R>1$. We need to extract a sequence of disjoint annuli $A_{k_{1},M},...,A_{k_{L},M}\subset A([R_{N},2])$ for some $k_{i}\in [N]$ and $L\geq cN$ such that 
\begin{eqalign}\label{eq:deterministicballcontained}
b_{k_L}M_{k_{L}}\geq R_{N}\tand b_{k_{1},M}\leq 2.     
\end{eqalign}
We prove the first inequality in \cref{sec:holderequicontinuitybounds} and the second inequality in \cref{cor:uppertruncatedlebconv}. We also need the radius to be large enough to contain consecutive centers in order that the disks $B_{R_{N}}(y_{k,N})$ cover the unit circle i.e.
\begin{eqalign}\label{eq:deterministicballcontainingcenters}
R_{N}\geq \frac{\tau(x_{k,N},x_{k+1,N})}{\tau(1)}    
\end{eqalign}
and we will prove this in \cref{proofofmaintheoreminverse}.
\subsection{Decoupling via modulus estimates}\label{sec:annulidecoupl}
In this subsection, we give a heuristic overview of the method. We describe how the choice of our random annuli leads to estimates for the dilatation involving terms that have conditional independence.  

\noindent For any two essentially disjoint intervals $(p_{n,M},p_{n,M}+x_{n,M}),(q_{n,M},q_{n,M}+y_{n,M})\subset [a_{n,M},b_{n,M}]$, the quasisymmetric ratios for the nth annulus $A_{n,M}$ will be of the form
\begin{equation}\label{eq:mainratiotodecouple}
\frac{Q_{\tau}([y_{0}+p_{n,M},y_{0}+p_{n,M}+x_{n,M}]\tau([0,1])}{Q_{\tau}([y_{0}+q_{n,M},y_{0}+q_{n,M}+y_{n,M}]\tau([0,1])},    
\end{equation}
where $x_{n,M}:=x_{n}M_{n}$ and $y_{0}:=\frac{\tau(0,x_{0})}{\tau(0,1)}$. We focus on the numerator because the denominator is similar.
\pparagraph{Shift center and remove total mass} First we rewrite it in terms of the shifted inverse $Q_{x_0,\tau}$:
\begin{eqalign}
&Q_{\tau}\para{\spara{\frac{\tau(0,x_{0})}{\tau(0,1)}+p_{n,M},\frac{\tau(0,x_{0})}{\tau(0,1)}+p_{n,M}+x_{n,M}}\tau([0,1])}\\ &=Q_{\tau}\para{\spara{\tau(0,x_{0})+\tau([0,1])p_{n,M},\tau(0,x_{0})+\tau([0,1])p_{n,M}+\tau([0,1])x_{n,M}}}\\
&=Q_{x_0,\tau}\para{\spara{\tau([0,1])p_{n,M},\tau([0,1])p_{n,M}+\tau([0,1])x_{n,M}}}
\end{eqalign}\label{eq:transiationtoeta0}
and remove the $\tau([0,1])$
\begin{eqalign}
Q_{x_0,\tau}\para{e^{\xi(b_{n})}e^{\thickbar{U}^{1}_{n}(\theta_{b_{n}})}\cdot \spara{p_{n},p_{n}+x_{n}} }.  
\end{eqalign}
No pertubation was needed here because we removed this from all annuli. So we moved 
\begin{equation}
[p_{n,M},p_{n,M}+x_{n,M}]\mapsto  e^{\xi(b_{n})}e^{\thickbar{U}^{1}_{n}(\theta_{b_{n}})}\cdot [p_{n},p_{n}+x_{n}].   
\end{equation}
\pparagraph{Transition to real line measure $\eta$}  Secondly, we transition to $Q:=Q_{x_0,\eta}$
\begin{equation}\label{eq:transiationtoeta1}
 \eqref{eq:transiationtoeta0}=\displaystyle  Q\para{\frac{e^{\xi(b_{n})}}{e^{\xi(z_{n,M})}}e^{\thickbar{U}^{1}_{n}(\theta_{b_{n}})}\cdot\spara{p_{n},p_{n}+x_{n}\expo{\xi(z_{n,M})-\xi(z_{n,M},z_{n,M}+r_{n,M})}}},
\end{equation}
where we shorthand wrote $r_{n,M}:=x_{n}\expo{\xi(b_{n})+\thickbar{U}^{1}_{n}(\theta_{b_{n}})}$. As mentioned in \cref{it:doubleboundinv}, the $\xi$ fields above are evaluated over the domain
\begin{equation}
[0,Q_{H}(z_{n,M})] \tand [Q_{H}(z_{n,M}),Q_{H}(z_{n,M}+r_{n,M})].    
\end{equation}
Using the comparison relations in reverse, these two intervals are both contained in $I_{n}:=[0,Q^{n}(b_{n})]$ because $p_{n}\leq b_{n}$ and so
\begin{equation}
Q_{H}(z_{n,M}+r_{n,M})\leq Q_{H}(b_{n,U_{n},\xi}=Q^{n}_{U}(b_{n}).    
\end{equation}
So here we will use concentration estimates for the modulus of the stationary continuous field $\xi$ over the interval $I_{n}$
\begin{eqalign}
&\Proba{\abs{\xi(b_{n})-\xi(p_{n,U_{n}})}\geq u_{n}  }  \leq \Proba{ \supl{(x,y)\in [0,Q^{n}(b_{n})]^{2}}\abs{\xi(x)-\xi(y)}\geq u_{n}  }\lessapprox e^{-c_{n} },
\end{eqalign}
for some constants $u_{n}\to 0$ and $c_{n}\to +\infty$ that we will specify later. This is used to bound the differences  
\begin{equation}\label{eq:eventH-U}
X\in \spara{-u_{n}, u_{n} }=: \spara{\ln d_{n,\xi},\ln D_{n,\xi}}=:I_{n,\xi}
\end{equation}
for all four fields fields in numerator and denominator $X\in \set{\xi(b_{n})-\xi(A),\xi(A)-\xi(A,A+S)}$ for $A\in \set{p_{n,U_{n}},q_{n,U_{n}}},S\in \set{x_{n,U_{n}},y_{n,U_{n}}}$. Therefore, we could upper bound using a maximum
\begin{eqalign}\label{eq:maineventlowerscales}
\eqref{eq:transiationtoeta1}&\leq \maxls{y_{i,j}\in I_{n,\xi}\\ i,j=1,2}\frac{Q\para{e^{y_{1,2}}e^{\thickbar{U}^{1}_{n}(\theta_{b_{n}})}\cdot\spara{p_{n},p_{n}+x_{n}e^{y_{1,2}}}}}{Q\para{e^{y_{2,1}}e^{\thickbar{U}^{1}_{n}(\theta_{b_{n}})}\cdot\spara{q_{n},q_{n}+y_{n}e^{y_{2,2}}}}}.
\end{eqalign}
But in fact, we will need to take this $\max$ after we do certain decompositions of the inverses (for details see proof of \nameref{partmultipointestimates} for the ratio).
\pparagraph{Transition to lower scales} Finally, we apply the scale-relation for different scales $Q\to Q^{n}$ to write:
\begin{eqalign}
Q\para{e^{y_{1,1}}e^{\thickbar{U}^{1}_{n}(\theta_{b_{n}})}\cdot\spara{p_{n},p_{n}+x_{n}e^{y_{1,2}}}}&=Q^{n}\para{\frac{\expo{\overline{U_{n}}(\theta_{b_{n}} )  }}{G_{n}(T_{U})}e^{y_{1,1}}\spara{p_{n},p_{n}+te^{y_{1,2}}\frac{G_{n}(T_{U})}{G_{n}(T_{U},T_{U}+t_{U})}}},
\end{eqalign}
where 
\begin{eqalign}
T_{U}:= e^{y_{1,1}}e^{\thickbar{U}^{1}_{n}(\theta_{b_{n}})}\cdot p_{n}   \tand t_{U}:=e^{y_{1,1}}e^{\thickbar{U}^{1}_{n}(\theta_{b_{n}})}\cdot x_{n}e^{y_{1,2}}.
\end{eqalign}
Here again we can take supremum over the same interval $I_{n}=[0,Q^{n}(b_{n})]$ because we will restrict the perturbations to satisfy 
\begin{equation}
e^{y_{1,1}}\cdot p_{n}+ e^{y_{1,1}}\cdot x_{n}e^{y_{1,2}}<b_{n}. 
\end{equation}
That means that we will use concentration estimates for the stationary fields $U_{n}$ over the interval $I_{n}:=[0,Q^{n}(b_{n})]$, to bound the differences
\begin{equation}\label{eq:eventU}
\spara{-r_{n}, r_{n} }=: \spara{\ln d_{n,U},\ln D_{n,U}}=:I_{n,U}
\end{equation}
for all four fields in numerator and denominator. One could upper bound by
\begin{eqalign}
\eqref{eq:maineventlowerscales}\leq &   \maxls{u_{i,j}\in I_{n,\xi},i,j=1,2\\ u_{i,j}\in I_{n,U},i,j=3,4}\frac{Q^{n}\para{e^{u_{1,1}+u_{3,3}}\spara{p_{n},p_{n}+x_{n}e^{u_{1,2}+u_{3,4}}}}}{Q^{n}\para{e^{u_{2,1}+u_{4,3}}\spara{q_{n},q_{n}+u_{n}e^{u_{2,2}+u_{4,4}}}}}\\
= & \maxls{u_{i}\in I_{n,\xi}+I_{n,U}\\i=1,2,3,4}\frac{Q^{n}\para{e^{u_{1} }\spara{p_{n},p_{n}+x_{n}e^{u_{2}}}}}{Q^{n}\para{e^{u_{3}}\spara{q_{n},q_{n}+y_{n}e^{u_{4}}}}}
\end{eqalign}
but as mentioned already the proof is a bit more complicated in that we have first decompose $Q_{H}(a_{n}^{0})$ before we take maximum of the fields $\xi,U^{1}_{n}$. \noindent By having the constants one $d_{n,\xi},D_{n,\xi},d_{n,U},D_{n,U}$ be not too big, we can then hope to have a decoupling between distinct quasisymmetric ratios from distant annuli using the gap event existence for $[a_{k},b_{k}]$.\\
In summary, we started with the pre-annuli $a_{k}^{P}$, we randomize them, we apply the scaling relations and the deviations estimates, and finally end up with deterministic annuli again but pertubed by the not-too big constants $d_{n},D_{n}$ so that we are still within the main annuli $A_{k}$ with base $[a_{k},b_{k}]$.

\begin{remark}[Constraints]\label{rem:randomconstraints}
We need a subsequence pairs $(a_{n}^{0},b_{n}^{0}),(a_{n},b_{n})$ that satisfy the following conditions.
\begin{itemize}

\item A deviation estimate for the transition from unit circle to the real line
    \begin{eqalign}
& \Proba{ \supl{(x,y)\in [0,Q^{n}(b_{n}^{0})]^{2}}\abs{\xi(x)-\xi(y)}\geq u_{n,comp}  }\leq e^{-c_{n,\xi} }.
\end{eqalign}

\item A deviation estimate for the transition to lower scales
    \begin{eqalign}
& \Proba{ \supl{(x,y)\in [0,Q^{n}(b_{n}^{0})]^{2}}\abs{U_{n}(x)-U_{n}(y)}\geq u_{n,us}  }\leq e^{-c_{n,U} }.
\end{eqalign}

\item Disjointness of the modified annuli in \cref{eq:disjointnesscondition} and the requirement to contain a deterministic ball also require to study the increments
\begin{equation}
\Proba{\supl{s\in [0,Q^{k_{i}}(b_{k_{i}})]}\thickbar{U}^{k_{i-1}}_{k_{i}}(s)\geq u_{k_{i-1},k_{i},trun}  }\leq e^{-c_{n,trun} }.  
\end{equation}

\item An existence of a gap event for a subsequence of the intervals $[Q^{n}(a_{n}),Q^{n}(b_{n})]$. Then the perturbed pre-annuli need to be contained in them
\begin{eqalign}
\para{M_{n}a_{n}^{0},M_{n}b_{n}^{0}}\subseteq \spara{e^{-r_{n}-u_{n}} a_{n}^{0},e^{r_{n}+u_{n}} b_{n}^{0} } \subseteq [a_{n},b_{n}]. 
\end{eqalign}
This is needed because in the Lehto integral we use $[M_{n}a_{n}^{0},M_{n}b_{n}^{0}]$ and so to decouple the quasisymmetric ratios, we need the Whitney squares and their boundary intervals to still be within $[a_{n},b_{n}]$ after perturbation..

\end{itemize}
\end{remark}

\newpage\section{Beltrami solution for the Inverse  }\label{Beltinvsection}
In this section we give a precise formulation to our main result \cref{Beltinv} as a theorem (this is analogous to \cite[theorem 5.1]{AJKS}). 
\begin{theorem}\label{Lehtoinverse}
Assume that $\gamma<0.11289...$ and that $\phi:\T\to \T$ is the random homeomorphism $\phi(e^{2\pi i x}):=\expo{2\pi iQ_{\tau}(x\tau(0,1)}$ and $\Psi$ its extension defined in the \cref{ABext}. Define the Beltrami coefficient 
\begin{equation}
\mu(z):=\chi_{\ud}(z)\frac{\partial_{\thickbar{z}}\Psi(z)}{\partial_{z}\Psi(z)},~~\quad z\in \C.    
\end{equation}
Then there exists  a random homeomorphic solution $f\in W^{1,1}_{loc}(\C)$  to the degenerate Beltrami equation $\frac{\partial f}{\partial \thickbar{z}}=\mu(z) \frac{\partial f}{\partial z}$, a.e. $\C$ with the hydrodynamic normalization $f(z)=z+o(1)$ as $z\to \infty$. Then there exists an exponent $\alpha>0$ such that the restriction $f:\T\to\C$ is a.s. $\alpha$-\Holder continuous.
\end{theorem}

\begin{remark}
The existence of such a solution $f$ is established by an argument similar to Lehto's \cref{Lehtosolutions}.
\end{remark}
\subsection{Integrability of \sectm{$K=K_{h^{-1}}$} }\label{sec:integrabilityofdilatation}
We establish the $L^1$-integrability of $K=K_{h^{-1}}=\frac{1+\abs{\mu}}{1-\abs{\mu}}$ for the $\mu$ from above and the homeomorphism $h^{-1}(x):=Q_{H}(x\mu_{H}(0,1))$, which is the analogous result of \cite[Lemma 4.5]{AJKS}. 
\begin{lemma}\label{lem:integrabilityofdilatation}
Fix  $\beta\in (0,\frac{\sqrt{11}-3}{2})$ (i.e. $\gamma\in (0,\sqrt{2(\sqrt{11}-3)})\approx (0,1.211...)$). We have that $K_{h^{-1}}\in L^{1}([0,1]\times [0,2])$.
\end{lemma}
\begin{proof}
First we do a change variables to clarify the effect of the total mass $M_{1}:=\mu_{H}(0,1)$
\begin{equation}
\int_{[0,1]\times [0,2]}K_{h^{-1}}(x+iy)\dx\dy=\frac{1}{M_{1}^{2}}\int_{[0,M_{1}]\times [0,M_{1}]}K_{Q_{H}}(x+iy)\dx\dy,     
\end{equation}
where $K_{Q_{H}}(x+iy):=K_{h^{-1}}(\frac{1}{M_{1}}(x+iy))$. Next we decompose the total mass so that we can use the ratio estimates (\cref{prop:inverse_ratio_moments}). We set 
\begin{equation}
 E_{k,M}:=\set{M_{1}\in [a_{k},a_{k+1}]}\tfor k\in \Z\setminus\set{0},   
\end{equation}
for some bi-sequence $a_{k}$ going to zero and plus infinity, to be determined later. We apply those bounds and separate this decomposition event by \Holder
\begin{eqalign}
&\Expe{\frac{1}{M_{1}^{2}}\int_{[0,M_{1}]\times [0,M_{1}]}K_{Q_{H}}(x+iy)\dx\dy}    \\
=&\sum_k\Expe{\ind{E_{k,M}}\frac{1}{M_{1}^{2}}\int_{[0,M_{1}]\times [0,M_{1}]}K_{Q_{H}}(x+iy)\dx\dy}  \\
\leq &\sum_k\frac{1}{a_{k}^{2}} \int_{[0,a_{k+1}]\times [0,a_{k+1}]}\Expe{\ind{E_{k,M}}K_{Q_{H}}(x+iy)}\dx\dx  \\
\leq &\sum_k\frac{1}{a_{k}^{2}} \para{\Proba{E_{k,M}}}^{1/p_{1}}\int_{[0,a_{k+1}]\times [0,a_{k+1}]}\para{\Expe{\para{K_{Q_{H}}(x+iy)}^{p_{2}}}}^{1/p_{2}}\dx\dx,
\end{eqalign}
for $p_{1}^{-1}+p_{2}^{-1}=1$. 
\pparagraph{Dilatation integral}
In order to use \cref{prop:inverse_ratio_moments}, we require $p_{2}=1+\e$ for small $\e>0$ and thus its \Holder-conjugate $p_{1}$ to be arbitrarily large. Therefore, by tiling $[0,a_{k+1}]\times [0,a_{k+1}]$ by Whitney squares we have by \cref{boundarycorres}
\begin{eqalign}
 \int_{[0,a_{k+1}]\times [0,a_{k+1}]}\para{\Expe{\para{K_{Q_{H}}(x+iy)}^{p_{2}}}}^{1/p_{2}}\dx\dx \leq& c(a_{k+1})^{1+\epsilon}\sum_{n=1}^{\infty}\sum_{m=1}^{2^{n}}\int_{\frac{m-1}{2^{n}}}^{\frac{m}{2^{n}}}\frac{1}{x^{\e}}\frac{1}{2^{n+1}}\dx\\
=&c(a_{k+1})^{1+\epsilon}\sum_{n=1}^{\infty}\int_{0}^{1}\frac{1}{x^{\e}}\frac{1}{2^{n+1}}\dx\\
 =&c\frac{(a_{k+1})^{1+\epsilon}}{2(1-\epsilon)}.   
\end{eqalign}
\pparagraph{Sum over $a_{k}$}
We set
\begin{equation}
a_{k}:=\branchmat{ k^{\lambda_{1}}& k\geq 1 \\ \abs{k}^{-\lambda_{2}}& k\leq -1 \\},   
\end{equation}
for $\lambda_{1},\lambda_{2}>0$. For $k\leq -1$, we use the negative moments for $q_{neg}>0$
\begin{eqalign}
\sum_{k\leq -1}\frac{a_{k+1}^{1+\epsilon}}{a_{k}^{2}} \para{\Proba{E_{k,M}}}^{1/p_{1}}\leq\sum_{k\leq -1}\frac{a_{k+1}^{1+\epsilon}}{a_{k}^{2}} \para{\Proba{M_{1}\leq a_{k+1}}}^{1/p_{1}} \leq \sum_{k\leq -1}\frac{a_{k+1}^{1+\epsilon}}{a_{k}^{2}} a_{k+1}^{q_{neg}/p_{1}}\leq \sum_{k\geq 1} k^{-\lambda_{2}\para{1+\epsilon+\frac{q_{neg}}{p_{1}}-2} },    
\end{eqalign}
and so we ask $\lambda_{2}\para{1+\epsilon+\frac{q_{neg}}{p_{1}}-2}>1$, which is possible by taking $q_{neg}>0$ large enough. For $k\geq 1$, we use the positive moments for $q_{pos}\in [1,\beta^{-1})$
\begin{equation}
\sum_{k\geq 1}\frac{a_{k+1}^{1+\epsilon}}{a_{k}^{2}} \para{\Proba{E_{k,M}}}^{1/p_{1}}\leq \sum_{k\geq 1}\frac{a_{k+1}^{1+\epsilon}}{a_{k}^{2}} a_{k}^{-q_{pos}/p_{1}}\leq \sum_{k\geq 1} k^{-\lambda_{1}\para{2+\frac{q_{pos}}{p_{1}}-1-\epsilon} },    
\end{equation}
here we have finiteness because $2>1+\epsilon$ and by taking $\lambda_{1}>0$ large enough.
\end{proof}
\subsection{Proof of \sectm{\cref{Lehtoinverse}} } \label{proofofmaintheoreminverse}\label{sec:Coveringtheannulus}
\noindent We write the proof along the lines of \cite[theorem 5.1]{AJKS},\cite[Theorem 20.9.4]{astala2008elliptic}, to which we refer for further details and background.\\
Let $D_{n}:=\ceil{\rho_{*}^{-(1+\e_{**}(1-\lambda))n}}$ for $\rho_{*}$ in \nameref{def:exponentialchoiceparamannul}, some $\lambda\in (0,1)$ and $\e_{**}>0$ from \cref{Lehtodivergent}. For each $n\geq n_{0}$ we cover the unit circle $\T$ by $D_{n}$-number of disks $B_{R_{n}}(e^{2\pi i y_{k}})$ ,for $R_{n}:=cR^{-n}$, and centers $y_{k}:=\frac{\tau(0,\frac{k}{D_{n}})}{\tau(0,1)}$ and $k=0,...,D_{n}$. Write also $G_{n}:=\set{e^{2\pi i y_{k}}: k=0,...,D_{n}}$.\\
We consider the event that the Lehto integral doesn't diverge
\begin{equation}
E_{n,k}:=\set{ L_{K_{\tau}}(y_{k},R_{n},2)<n\delta   }\tfor n\geq 1,  k=0,...,D_{n},
\end{equation}
for any $\delta<\delta(\rho_*)$. In \cref{Lehtodivergent} we show the following bound
\begin{equation}\label{eq:lehtotermestimate}
\Proba{E_{n,k}}\leq c \rho_{*}^{(1+\e_{**})n}
\end{equation}
for $\e_{**}>0$.  
\pparagraph{Covering random disks}
\noindent Here one difference with the proof of \cite[theorem 5.1]{AJKS} is that we need to ensure that the random centers are close enough and that the disks still cover the unit circle
\begin{eqalign}\label{eq:coveringunitcircle0}
C_{n,k}:=C_{1,n,k}\cap C_{2,n,k}:=\set{\frac{\tau(\frac{k}{D_{n}},\frac{k+1}{D_{n}})}{\tau(0,1)}\leq R_{n}}\cap \set{ R_{n}\leq a_{m_{n},M}  }, 
\end{eqalign}
where the first event $C_{1,n,k}$ is about the centers being $R_{n}$-close and the second event $C_{2,n,k}$ is about the subsequence of concentric annuli contains a deterministic ball of size $R_{n}$.
For the first event from a simple Markov estimate for $p\in (0,\beta^{-1})$ we get
\begin{eqalign}\label{eq:coveringunitcircle1}
\Proba{C_{1,n,k}^{c}}=\Proba{\frac{\tau(\frac{k}{D_{n}},\frac{k+1}{D_{n}})}{\tau(0,1)}\geq R_{n}}\leq cR^{-p}_{n}D_{n}^{-\frac{1}{(1+\e_{1})}\zeta(p(1+\e_{1}))},  
\end{eqalign}
where $\e_{1}>0$ can be taken arbitrarily small because we applied \Holder-inequality and used the existence of all negative moments for $\tau(0,1)$. For the second event in \cref{prop:scalescomparedtobN+1}, we prove the following for $C_{2,n,k}^{c}$
\begin{eqalign}\label{eq:c2nkc}
    \Proba{C_{2,n,k}^{c}}\leq c\rho_{*}^{(1+\e_{R})n},
\end{eqalign}
for some $\e_{R}>0$.
    
\pparagraph{Borel-Cantelli}Therefore, using the estimates \cref{eq:lehtotermestimate}, \cref{eq:c2nkc} and \cref{eq:coveringunitcircle1}, the analogous Borel-Cantelli bound is
\begin{eqalign}\label{eq:borelcantelliexponents}
\sum_{n=1}^{\infty}\Proba{\bigcup_{k=1}^{D_{n}} E_{n,k}\cup C_{n,k}^{c}  }\leq &c \sum_{n=1}^{\infty}D_{n}\rho_{*}^{(1+\e_{**})n} +D_{n} R^{pn}D_{n}^{-\frac{\zeta(p(1+\e_{1}))}{(1+\e_{1})}}+D_{n}\rho_{*}^{(1+\e_{R})n}\\
\leq &c \sum_{n=1}^{\infty}\rho_{*}^{\e_{**}\lambda n} + \rho_{*}^{\para{-pc_{R}+\para{\frac{\zeta(p(1+\e_{1}))}{(1+\e_{1})}-1}(1+\e_{**}(1-\lambda)) } n}+\rho_{*}^{\para{(1+\e_{R})-(1+\e_{**}(1-\lambda))}n},
\end{eqalign}
and so for this sum to be finite we require positive exponents. The first term is summable and by taking $\e_{**}$ small enough so that $\e_{R}-\e_{**}>0$, the third term is summable too. So we study the second exponent being strictly positive in \nameref{def:exponentialchoiceparamannul} in the \textbf{Borel-Cantelli estimate} part in \cref{eq:determinsticlowerboundcon0}.

\pparagraph{Approximating sequence}
We next proceed as in the standard proof of Lehto's theorem by approximating $\mu$ by the sequence 
\begin{equation}
\mu_{\ell}:=\frac{\ell}{\ell+1}\mu   ,\ell\in \N.  
\end{equation}
Since $\norm{\mu_{\ell}}_{\infty}<1$, we are in the uniformly elliptic setting and so as explained in \cref{weldLeh} there is a Beltrami $\mu_{\ell}$-solution $f_{\ell}$ with the hydrodynamic normalization $f_{\ell}(z)=z+o(1)$ as $z\to \infty$.  Then every $f_\ell$
is a quasiconformal homeomorphism of $\C$ . \\
We will show that the $f_{\ell}$ and their inverses $g_{\ell}=f^{-1}_{\ell}$ form an equicontinuous family and thus their limit $f$ is a homeomorphism and a Beltrami $\mu$-solution. We start with equicontinuity for the inverse. For the inverse \cite[Lemma 20.2.3]{astala2008elliptic} we have the modulus
\begin{equation}
\abs{g_{\ell}(z)-g_{\ell}(w)}\leq \frac{16\pi^{2}}{\log(e+1/\abs{z-w})}\para{\abs{z}^{2}+\abs{w}^{2}+\int_{\ud}\frac{1+\abs{\mu_{\ell}(\zeta)}}{1-\abs{\mu_{\ell}(\zeta)}}\dzeta  }, \quad z,w\in \C.    
\end{equation}
Here the integrals are uniformly bounded because $\frac{1+\abs{\mu_{\ell}(\zeta)}}{1-\abs{\mu_{\ell}(\zeta)}}=K_{\ell}(\zeta)\leq K(z)  $ where $\zeta:=e^{2\pi i z}$ and $K\in L^{1}\cap L^{\infty}_{loc}([0,1]\times [0,2])$. Thus the inverse maps $g_{\ell}=f^{-1}_{\ell}$ form an equicontinuous family.
\pparagraph{Equicontinuity of $f_{\ell}$}
Next we check equicontinuity of $f_\ell$. First, we go over the unit disk $\ud$. We will show that for $z\in \ud$ we  have $diam(f_{\ell}(B(z,u))\to 0$ as the radius $u\to \zplus$. By \cref{Lehtoequic} we have
\begin{equation}
diam(f_{\ell}(B(z,u))\leq diam(f_{\ell}(B(z,1))16e^{-2\pi^{2}L_{K_{f_{\ell}}}(z,u,1) }.    
\end{equation}
By Koebe's theorem we have $f_{\ell}(2\ud)\subset 5\ud$ and so $diam(f_{\ell}(B(z,1))\leq 5$. For the Lehto integral for $a:=\frac{1-\abs{z}}{2}$ we have the following diverging lower bound
\begin{equation}
  L_{K_{f_{\ell}}}(z,u,1)\geq \mathcal{L}_{K}(z,u,1)\geq \frac{1}{\norm{K}_{L^{\infty}(B(z,a)  }  }\ln\frac{a}{u}\to +\infty, \tas u\to \zplus, 
\end{equation}
where we used that $\norm{K}_{L^{\infty}(B(z,a)  } <\infty$ since it is evaluated away from the unit circle. This was uniformly in $\ell$.\\
Secondly equicontinuity in the exterior of the unit disk $\ud^{c}$ follows from Kobe's theorem \cite[Corollary 2.10.2]{astala2008elliptic} because for $z\in \ud^{c} $ we set $\mu(z):=0$ and so $\mu_{\ell}(z):=0$ which implies that $f_{\ell}$ are conformal.\\
Thirdly we check equicontinuity on the unit circle $\T$. It suffices to prove local equicontinuity on points of $[0, 1]$ for the family
\begin{equation}
F_{\ell}(z):=f_{\ell}(e^{2\pi i z}), \ell\in \N.
\end{equation}
We first estimate the diameter of the image $F_\ell\para{B\para{\frac{k}{D_{n}},R_{n}} }$, assuming that $n>n_0(\omega)$. Again by \cite[Corollary 2.10.2]{astala2008elliptic}, we have that
\begin{equation}
diam\para{F_\ell\para{B\para{\frac{k}{D_{n}},2}}}\leq diam(f_\ell(B\para{\zeta_{n,k}, 2})\leq 10.
\end{equation}
Using this together with \cref{Lehtoequic} we obtain
\begin{eqalign}\label{ellequic}
diam \para{F_\ell\para{B\para{\frac{k}{D_{n}},R_{n}}}}&\leq diam( F_{\ell}(B\para{z,2})16e^{-2\pi^{2}L_{K}(\frac{k}{D_{n}},R_{n},2) }\leq C(\omega) e^{-2\pi^{2}n\delta}.       
\end{eqalign}
From these estimates we get the required equicontinuity. Namely, working now on the circle $\T$, the balls $B\para{\zeta_{n,k}, R_{n}}$ cover a $R_{n+1}$-neighborhood of $\T$ in such a way that any two points that are in this neighborhood, with distance not exceeding $R_{n+1}$ lie in the same ball. Since this holds for every $n>n_0(\omega)$, we infer from \cref{ellequic} that there are $\e_0>0$ and $\alpha>0$ such that, uniformly in $\ell$,
\begin{equation}\label{Holdermod}
\abs{f_{\ell}(z)-f_{\ell}(w)}\leq C(\omega)\abs{z-w}^{\alpha}     
\end{equation}
when $\abs{z}=1$, $1-\e_0\leq \abs{w}\leq 1+\e_0$ and $\abs{z-w}\leq \e_0$. In particular, one may actually take $\alpha=\frac{2\pi^{2}\delta}{-1+\ln R}$, where by taking large enough $n$ we have $\abs{\frac{\ln \para{c}}{n}}\leq 1$ and so we replaced it by $1$. This clearly yields equicontinuity at the points
of $\T$, and hence on $\hmC$. We may now pass to a limit and one obtains $W^{1,1}$-homeomorphic solution $f(z)=\lii{\ell} f_\ell(z)$ to the Beltrami equation as in \cite[p. 585]{astala2008elliptic}.\\
At the same time the modulus estimate \cref{Holdermod} shows that $f:\T\to \C$ is \Holder continuous. Since $f$ is analytic outside the disc, with $f(z)=z+o(1)$ at infinity, in fact it follows that $f$ is \Holder continuous on $\C\setminus \ud$.
\newpage\part{Deviation Estimates for the Lehto integral of inverse }\label{part:deviationestLehto}\label{decouplingLehto}
\noindent This part is the parallel to the part of \cite[section 4.2]{AJKS}. We use the above random annuli to obtain decoupled estimates for the dilatation. In particular, it is devoted to the proof of \cref{Lehtodivergent}
i.e. the estimate for the Lehto integral to not diverge.
\lehtodivergence*
\section{Decoupling of Lehto integral}\label{sec:decoupLehtoint}
~~~~~~~~~~~~~~~~~~~~~~~~~~~~~~~~~~~~~~~~~~~~~~~~~~~
Due to the above deviation estimates on $\xi,U^{1}_{n}$ and the truncation $U^{n_{i}}_{n_{i+1}}$ in \cref{rem:randomconstraints}, we have that the random annuli $A_{n,M_{n}}$ with $n\in I$, for some subset $I\subset [N]$, are disjoint. Indeed in order to get \cref{eq:disjointnesscondition}  for consecutive annuli $A_{n_{i},M_{n_{i}}},A_{n_{i+1},M_{n_{i+1}}}$
\begin{equation}\label{eq:disjointnessconditionpert}
 \frac{M_{n_{i+1}}}{M_{n_{i}}}= \expo{\gamma(\xi(b_{n_{i+1}})-\xi(b_{n_{i}}))}\expo{\gamma(U^{1}_{n_{i}}(\theta_{b_{n_{i+1}}})-U^{1}_{n_{i}}(\theta_{b_{n_{i}}}))}   \expo{\thickbar{U}^{n_{i}}_{n_{i+1}}(\theta_{b_{n_{i+1}}})} <\frac{a_{n_{i}}^{0}}{b_{n_{i+1}}^{0}},
\end{equation}
we require
\begin{eqalign}\label{eq:disjointrandomannuli}
&u_{n_{i},comp}+u_{n_{i},us}+u_{n_{i},n_{i+1},trun}-\beta\ln\frac{1}{\rho_{*}}\\
<&\ln\frac{a_{n_{i}}^{0}}{b_{n_{i+1}}^{0}}=(n_{i+1}-n_{i}-2(r_{a}-r_{b}))\ln\frac{1}{\rho_{*}}+P_{n_{i+1}}+P_{n_{i}}=:B_{n_{i},n_{i+1}}.   
\end{eqalign}
This we will see that it indeed follows from the deviations \cref{eq:compconstaevent},\cref{eq:upperscalemoddecay} and \cref{eq:trunconstraindev} to get
\begin{eqalign}\label{eq:disjointrandomannuli2}
&u _{n_{i},comp}+u_{n_{i},us}+u_{n_{i},n_{i+1},trun}-\beta\ln\frac{1}{\rho_{*}}\\
=&c\gamma\frac{\rho_{*}^{1/2}}{1-\rho_{*}^{1/2}}  \para{b_{n_{i}}+d_{n_{i}}}^{1/6}+c\gamma(\rho_{b}+\rho_{d})+((n_{i+1}-n_{i})-2(r_{a}-r_{b}))\ln\frac{1}{\rho_{*}} \\\leq& B_{n_{i},n_{i+1}},   
\end{eqalign}
which is a constraint added in \nameref{def:exponentialchoiceparam}. Therefore, as in \cite[eq. (61)]{AJKS}, we use this disjointness and the boundary correspondence \cref{sec:boundarycorrespondence}, to lower bound the Lehto integral as follows
\begin{equation}\label{eq:disjointness}
\mathcal{L}(0,R^{-n},2)\geq\sum_{n\in I}\int_{M_{n}a_{n}^{0}}^{M_{n}b_{n}^{0}}\frac{1}{\int_{0}^{2\pi}K_{h^{-1}}(r e^{i\theta})\dtheta}\frac{\dr}{r}\geq \sum_{n\in I}\int_{a_{n}^{0}}^{b_{n}^{0}}\frac{1}{K_{Q_{\mu,M_{n}}}(r)}\dr,    
\end{equation}
for the above mentioned $I$ and
\begin{eqalign}\label{eq:definidisjoinlow}
&K_{Q_{\mu,M_{n}}}(r):=\suml{I: C_{I}\cap S_{r}\neq \varnothing}\abs{C_{I}}K_{Q_{\mu,M_{n}}}(I)\tand K_{Q_{\mu,M_{n}}}(I):=\sum_{(J_{1},J_{2})\in j_{5}(I)}\delta_{Q}(M_{n}\cdot \mathbf{J}),\\
&\delta_{\tau}(\mathbf{J})=\frac{Q_{\tau}(J_{1})\tau(0,1)}{Q_{\tau}(J_{2})\tau(0,1)}+ \frac{Q_{\tau}(J_{2})\tau(0,1)}{Q_{\tau}(J_{1})\tau(0,1)},\\
&j_{5}(I):=\set{ \mathbf{J}=(J_{1},J_{2}): J_{1},J_{2}\in \mathcal{D}_{n+5}\tand  J_{1},J_{2}\subset j_{0}(I)},\\  
&S_{r}:=\set{(x,y):\abs{(x,y)}=r}.
\end{eqalign}
We further lower bound it similarly to \cite[eq, (82)]{AJKS}.
 We set
\begin{eqalign}\label{eq:intermediatedisks}
&B_{n}:=\set{(x,y): \abs{x},\abs{y}\leq b_{n,M} }.
\end{eqalign}
Because of the gap event, the decoupling here needs to involve the random annuli. In particular, we obtained a sequence $A_{n_{1},M_{n_{i}}},...,A_{n_{L},M_{n_{L}}}$ for $L\geq c_{gap}N$ such that if we have intervals $I\subset A_{n_{i},M_{n_{i}}},J\subset A_{n_{i+1},M_{n_{i+1}}}$, then we can decouple the $Q_{H}(I)$ from $Q_{H}(J)$. The only difference with \cite{AJKS} happens at \cite[eq. (80)]{AJKS}. For each $j\in [1,M-1]$ we consider intervals for each $\ell\in [1,2(M-j)]$
\begin{eqalign}
\tilde{A}_{n_{j},1}:=[\wt{M}_{n_{j+1}}b_{n_{j+1}},\wt{M}_{n_{j}}a_{n_{j}}]\tand \tilde{A}_{n_{j},\ell}:=\branchmat{\spara{ \wt{M}_{n_{j+\ell}}a_{n_{j+\ell}}^{0}, \wt{M}_{n_{j+\frac{\ell}{2}}}b_{n_{j+\frac{\ell}{2}}}^{0}} & \ell~\even\\
\spara{ \wt{M}_{n_{j+\ceil{\frac{\ell}{2}}}}b_{n_{j+\ceil{\frac{\ell}{2}}}}^{0}, \wt{M}_{n_{j+\floor{\frac{\ell}{2}}}}a_{n_{j+\floor{\frac{\ell}{2}}}}^{0}} & \ell~\odd},
\end{eqalign}
where we used the floor and ceiling notation in the odd case i.e. in $b_{n_{j+\ceil{\frac{\ell}{2}}}}^{0}$ we used $\ceil{\frac{\ell}{2}}\geq \frac{\ell}{2}$ and in $a_{n_{j+\floor{\frac{\ell}{2}}}}^{0}$ we used $\floor{\frac{\ell}{2}}\leq \frac{\ell}{2}$. In particular $\tfor j=M$ we set $\tilde{A}_{n_{M}}:=[0,\wt{M}_{n_{M}}a_{n_{M}}]$. By fixing a subset $A:=\set{n_{1},...,n_{M}}$, we will also need the notations
\begin{eqalign}
o[n_{j}]:=j\tand s[n_{j}]=n_{j+1}, 
\end{eqalign}
the first function gives the order of an element $n\in A$ and the second function gives the successor of $n\in A$. From here we proceed as in \cite[eq, (82)]{AJKS} to get the following lower bound.
\begin{lemma}\label{lem:decestimLehto}
Going over all the subsets $A\subset \set{1,...,N}$ that satisfy the deviation events, we have
\begin{equation}\label{eq:Lnmmnsigman}
\mathcal{L}(0,R^{-n},2)\geq\sum_{n\in A}\int_{a_{n}^{0}}^{b_{n}^{0}}\frac{1}{K_{Q_{\mu,M_{n}}}(r)}\dr\geq \sum_{n\in A}\sigma_{n}\cdot m_{n}.
\end{equation}
where we set for $n\in A=\set{n_{1},...,n_{M}}$
\begin{eqalign}\label{eq:lnmdecoupledboundPhiphin}
\sigma_{n}:=&\para{1+\sumls{2o[n]<m\leq 2M }L_{n,m}}^{-1},\\
m_{n}:=&\int_{a_{n}^{0}}^{b_{n}^{0}}\frac{1}{1+L_{n}(r)+L_{n,n}} \frac{\dr}{b_{n}^{0}-a_{n}^{0}},\\
\delta_{n,Q_{H}}(J_{1},J_{2}):=&\frac{Q_{H}(\wt{M}_{n} J_{1})}{Q_{H}(\wt{M}_{n} J_{2}) }+ \frac{Q_{H}(\wt{M}_{n} J_{2})}{Q_{H}(\wt{M}_{n} J_{1}) },\\
L_{n,m}:=&\sumls{(J_{1},J_{2})\in j_{5}(I) \\I\in C_{n}}\spara{\frac{Q_{H}\para{ \tilde{M}_{n}J_{1}\cap \tilde{A}_{n,m}} }{Q_{H}(\tilde{M}_{n}\para{ J_{2}\setminus B_{s[n]}}) }+\frac{Q_{H}\para{\tilde{M}_{n}J_{2}\cap \tilde{A}_{n,m}} }{Q_{H}(\tilde{M}_{n}\para{ J_{1}\setminus B_{s[n]}}) }},
,\\
L_{n,n}:= & \sumls{(J_{1},J_{2})\in j_{5}(I) \\I\in C_{n}}\delta_{n,Q_{H}}(J_{1}\setminus B_{s[n]}, J_{2}\setminus B_{s[n]}),\\
L_{n}(r):=&\sum_{I\in C_{r,n}^{dec}}\frac{\abs{C_{I}}}{b_{n}^{0}-a_{n}^{0}}  \sum_{(J_{1},J_{2})\in j_{5}(I)}\delta_{Q_{H}}(\tilde{M}_{n}\cdot \mathbf{J}),\\
C_{n}:=&\set{I\subset B_{n}\mid C_{I}\cap A_{n}^{0}\neq \varnothing \tand  j_{5}(I)\cap B_{n+1}\neq \varnothing  },\\  
\tand C_{r,n}^{dec}:=&\set{ I\subset B_{n}\mid C_{I}\cap S_{r}\neq \varnothing\tand j_{5}(I)\cap B_{n+1}=\varnothing }.
\end{eqalign}
We have that
\begin{eqalign}\label{eq:cardinalitycn}
\abs{C_{n}}\leq    (r_{a}-r_{b})p_{*}. 
\end{eqalign}
\end{lemma}
\begin{proof}
The proof follows the same logic as for \cite[eq, (82)]{AJKS}. The only difference is that we don't have any post-factors upper-scale terms $t_{n,k}$ as in \cite[eq. (73)]{AJKS}. They instead show up as pre-factors that we have to handle in the multipoint ratio estimates described in \cref{sec:multipointmaximum}. \\
Next we study the cardinality of $C_{n}$. The intervals that satisfy $ C_{I}\cap A_{n}^{0}$ involve dyadic scales $2^{-m}$ with
\begin{eqalign}
2^{-m}\in [a_{n},b_{n}]\doncl (n+r_{b})p_{*}  \leq  m \leq (n+r_{a})p_{*}.   
\end{eqalign}
So there are proportionally $(r_{a}-r_{b})p_{*}$-number of intervals in $C_{n}$. 

\end{proof}
\section{The list of deviations}\label{sec:proofofLehtodivergence}
\label{LLNstep}
~~~~~~~~~~~~~~~~~~~~~~~~~~~~~~~~~~~~~~~~~~~~~~~~~~~~~~~~~~~~~~~~~~~
Here we study the probability that the Lehto integral doesn't diverge i.e. $\Proba{\mathcal{L}(0,R^{-N},2)<\delta N }$ for some fixed $\delta>0$. As explained above, to capitalize on the conditional independence we need to insert a couple of estimates to decouple them.\\
We first start with the random annuli, and among them, we pick those that give small-enough perturbations. Then further among those perturbed deterministic annuli, we pick those with large enough gaps. And then finally among those we pick the ones that satisfy the $\sigma_{n}$ and $m_{n}$ deviations.\\
We start with the deviation for the Lehto integral and then split over the events in \cref{rem:randomconstraints}
\begin{equation}\label{eq:mainLehtodeviation}
\Proba{\mathcal{L}(0,R^{-N},2)<\delta N }.    
\end{equation}
We need to insert the following deviations.

\begin{enumerate}

\item The deviation for the upper truncated inverse $Q^{n}$
\begin{eqalign}\label{eq:uppertruncateddeviation}
D_{ut}(I):=&\set{\sum_{n}\ind{Q^{n}(b_{n})\geq b_{n}+d_{n}}\geq c_{ut}N}\\
=&\set{\sum_{n}\ind{Q^{n}(b_{n})< b_{n}+d_{n}}\leq \abs{I}-c_{ut}N},
\end{eqalign}
for $c_{ut}\in (0,1)$ and $d_{n}:=\rho_{d}\rho_{*}^{n}$. Given $\para{D_{ut}(I)}^{c}$ there is a set $S\subset I$ with $\abs{S}\geq \abs{I}-c_{ut}N$
\begin{eqalign}
E_{ut}(S):=\bigcap_{n\in S}\set{Q^{n}(b_{n})< b_{n}+d_{n}}. \end{eqalign}

\item The deviation for the field $\xi$
\begin{equation}\label{eq:compconstaevent}
D_{comp}(I):=\set{\sum_{n\in I}\ind{\supl{(s,t)\in [0,Q^{n}(b_{n})]^{2}}\gamma\abs{\xi(t)-\xi(s)}\geq \gamma u_{n,comp}}\geq  c_{comp}N},   
\end{equation}  
for $c_{comp}\in (0,1)$ and
\begin{eqalign}\label{eq:u_{n,comp}choice}
u_{n,comp}=c\frac{\rho_{*}^{1/2}}{1-\rho_{*}^{1/2}}  \para{b_{n}+d_{n}}^{1/6}, \end{eqalign}
for some particular $c>0$ such that we will obtain decay estimate for \cref{eq:compconstaevent} in \cref{prop:comparisonfieldxi}. The complement yields a set $S\subset I$ with $\abs{S}\geq \abs{I}-c_{comp}N$
\begin{eqalign}
E_{comp}(S):=\bigcap_{n\in S}\set{\supl{(s,t)\in [0,Q^{n}(b_{n})]^{2}}\gamma\abs{\xi(t)-\xi(s)}\leq \gamma u_{n,comp}}. \end{eqalign}

\item The deviation for the field $U_{n}$
\begin{eqalign}\label{eq:upperscalemoddecay}
&D_{us}(I):=\set{\sum_{n\in I }\ind{\supl{(s,t)\in [0,Q^{n}(b_{n})]^{2}}\gamma\abs{U_{n}(t)-U_{n}(s)}\geq \gamma u_{n,us}}\geq c_{us}N},
\end{eqalign}
for $c_{us}\in (0,1)$ and $u_{n,us}=c(\rho_{b}+\rho_{d})$. The complement yields a set $S\subset I$ with $\abs{S}\geq \abs{I}-c_{us}N$
\begin{eqalign}
E_{us}(S):=\bigcap_{n\in S}\set{\supl{(s,t)\in [0,Q^{n}(b_{n})]^{2}}\gamma\abs{U_{n}(t)-U_{n}(s)}\leq \gamma u_{n,us}}. \end{eqalign}

\item The deviation for the truncated fields.
\begin{equation}\label{eq:trunconstraindev}
D_{trun}(I):=\set{\bigcapls{S\subset I\\ \abs{S}\geq c_{trun}N}\bigcup_{k_{i}\in S }\ind{\supl{s\in [0,b_{k_{i}}+d_{k_{i}}]}\gamma U^{k_{i-1}}_{k_{i}}(s)\geq \gamma u_{k_{i},k_{i-1}}  }},   
\end{equation}
for $c_{trun}\in (0,1)$ and $u_{k,m}:=(m-k)r_{u,m,k}\ln\frac{1}{\rho_{*}}$, $r_{u,m,k}:=\beta+1-\frac{r_{a}-r_{b}}{m-k}>0$ and $m>k$. The completement yields a set $S\subset I$ with $\abs{S}\geq c_{trun}N$
\begin{eqalign}
E_{trun}(S):=\bigcap_{k_{i}\in S }\set{\supl{s\in [0,b_{k_{i}}+d_{k_{i}}]}\gamma U^{k_{i-1}}_{k_{i}}(s)\leq \gamma u_{k_{i},k_{i-1}}  }. 
\end{eqalign}

\item The gap event $D_{gap}=\set{\alpha(G)\leq c_{gap}N }$ for $c_{gap}\in (0,1)$. The complement event yields a set $S$ with $\abs{S}\geq c_{gap}N $
\begin{eqalign}
E_{gap}(S):=&\bigcapls{k,m\in S\\ m>k}\set{Q^{k}(a_{k})-Q^{m}(b_{m})\geq \delta_{m}}.
\end{eqalign}

\item The deviation for an inner deterministic radius $R_{N}$
\begin{eqalign}\label{eq:holderequicontinuityestimate}
D_{idb}(I):=\set{\sum_{k\in I}\ind{a_{k,M}\leq cR_{N}  }\geq c_{idb}N},    
\end{eqalign}
for $c_{idb}\in (0,1)$. The completement yields a set $S$ with $\abs{S}\geq \abs{I}-c_{idb}N $
\begin{eqalign}
E_{idb}(S):=&\bigcapls{k\in S}\set{ a_{k,M}\geq  cR_{N}}.
\end{eqalign}

\item The deviation for the $\sigma_{n}$ that was defined in \cref{eq:lnmdecoupledboundPhiphin} 
\begin{equation}\label{eq:sigmadecay}
D_{\sigma}(I):=\set{\sum_{n\in I_{\sigma} }\ind{ \sigma_{n}\geq g(\delta) }\geq  c_{\sigma}N},   
\end{equation}
for $c_{\sigma}\in (0,1)$. The completement yields a set $S$ with $\abs{S}\geq \abs{I}-c_{\sigma}N $
\begin{eqalign}
E_{\sigma}(S):=&\bigcapls{n\in S}\set{ \sigma_{n}\leq g(\delta)}
\end{eqalign}
for $g(\delta):=\delta^{1/2}$.

\end{enumerate}
For each of the above deviations, the complement event yields a set $S$ of cardinality $\abs{S}\geq \abs{I}-cN$. As we insert more deviations, the lower bound on the cardinality will decrease. The manner with which we insert them matters because some deviations are used in estimating other ones. 

Coming back to \cref{eq:mainLehtodeviation}, we insert all the deviations
\begin{eqalign}\label{eq:mainLehtodeviationxiUtrungapsigma}
\eqref{eq:mainLehtodeviation}\leq \sumls{\vec{S}} & \Proba{\mathcal{L}(0,R^{-N},2)<\delta N,E_{idb}(S_{0}),E_{gap}(S_{1}),E_{ut}(S_{2}),E_{comp}(S_{3}),E_{us}(S_{4}),E_{trun}(S_{5}),E_{\sigma}(S_{6})}\\
&+\Proba{D_{idb}([N])}+\Proba{D_{gap}(S_{0})}+\Proba{D_{ut}(S_{1})}+\Proba{D_{comp}(S_{2}),E_{idb}(S_{0}),E_{gap}(S_{1}),E_{ut}(S_{2})}\\
&+\Proba{D_{us}(S_{3}),E_{idb}(S_{0}),E_{gap}(S_{1}),E_{ut}(S_{2})}\\
&+\Proba{D_{trun}(S_{4}),E_{idb}(S_{0}),E_{gap}(S_{1}),E_{ut}(S_{2})}\\
&+\Proba{D_{\sigma}(S_{5}),E_{idb}(S_{0}),E_{gap}(S_{1}),E_{ut}(S_{2}),E_{comp}(S_{3}),E_{us}(S_{4}),E_{trun}(S_{5})},  
\end{eqalign}
where we sum over
\begin{eqalign}
&\bigg\{\vec{S}\subset [N]^{7} : S_{i}\subseteq S_{i-1} , \twith \abs{S_{0}}\geq  N, \abs{S_{1}}\geq c_{1}N, \abs{S_{2}}\geq c_{2}N,\abs{S_{3}}\geq c_{3}N,\\
& \abs{S_{4}}\geq c_{4}N,\abs{S_{5}}\geq c_{5}N,\abs{S_{6}}\geq c_{6}N \bigg\}, 
\end{eqalign}
and the constants are 
\begin{eqalign}
c_{0}:=&(1-c_{idb}),c_{1}:= c_{gap}c_{0}, c_{2}:=c_{1}-c_{ut}, c_{3}:=c_{2}-c_{comp},\\
c_{4}:=& c_{3}- c_{us}, c_{5}:=c_{4},c_{6}:=c_{5}- c_{\sigma}.
\end{eqalign}
Similarly to the choice of $\alpha$ in \cite[eq.(88)]{AJKS}, we use the $\e_{ratio}>0$ from \cref{prop:Multipointunitcircleandmaximum} to choose
 \begin{eqalign}\label{eq:cstarconstant}
c_{*}:= \frac{\min(\e_{ratio}-\e,1)}{2},    
 \end{eqalign}
for arbitrarily small $\e>0$. The $c_{0}\in (0,1)$ is optimized in \nameref{def:exponentialchoiceparam} and then set
\begin{eqalign}\label{eq:alltheconstants}
c_{gap}:=c_{*}, c_{ut}:=\frac{c_{*}c_{0}}{2}, c_{comp}:=\frac{c_{*}c_{0}}{4},    c_{us}:=\frac{c_{*}c_{0}}{8},c_{trun}:=c_{4}=\frac{c_{*}c_{0}}{8},c_{\sigma}:=\frac{c_{*}c_{0}}{32}.
\end{eqalign}
We already estimated the event $E_{gap}(S)$ in \cref{thm:gapeventexistence}. Therefore, to prove \cref{Lehtodivergent}, we need to obtain an exponential bound that can be made arbitrarily small for each of those six terms. In particular, we will bound each of the probabilities by $c^{N}\rho_{*}^{(1+\e_{*})N}$ and thus 
\begin{eqalign}\label{eq:mainLehtodeviationxiUtrungapsigma2}
\eqref{eq:mainLehtodeviationxiUtrungapsigma}\leq \tilde{c}^{N} \rho_{*}^{(1+\e_{*})N}.   
\end{eqalign}
By extracting a small $\e>0$ and making $\rho_{*}$ smaller, we further bound
\begin{eqalign}
\eqref{eq:mainLehtodeviationxiUtrungapsigma2}\leq c\rho_{*}^{(1+\e_{*}-\e)N}.   
\end{eqalign}
So in the rest of the sections we estimate each of the remaining deviations in \cref{eq:mainLehtodeviationxiUtrungapsigma}. It is analogous to \cite[section 4.3-4.4]{AJKS}. 
\newpage\section{The Lehto term }
Here we study the first term from \cref{eq:mainLehtodeviationxiUtrungapsigma}. Here we have the seven events 
\begin{equation}
E_{7}(S_{6}):=\set{E_{idb}(S_{0}),E_{gap}(S_{1}),E_{ut}(S_{2}),E_{comp}(S_{3}),E_{us}(S_{4}),E_{trun}(S_{5}),E_{\sigma}(S_{6})},    
\end{equation}
imposed on the intervals $I$ with length $\abs{S_{6}}\geq (1-c_{\sigma})N$. 
\begin{proposition}\label{prop:Lehtotermlemmalegs}
For interval $I$ with length $\abs{I}\geq (1-c_{\sigma})N$ we have
\begin{equation}
\Proba{\mathcal{L}(0,R^{-N},2)<\delta N, E_{7}(I) }\leq C \expo{-C_{Leh}N},
\end{equation}
where the constant $C_{Leh}=C_{Leh}(\delta)$ can be made arbitrarily large by making $\delta>0$ sufficiently small.
\end{proposition}
\noindent Therefore, we get the exponential decay we requested for the first term in \cref{eq:mainLehtodeviationxiUtrungapsigma} above. The strategy is the same as in \cite[eq.(89)]{AJKS} with the key exception that to estimate the product of the expectations, we need the gap event.\\
~~~~~~~~~~~~~~~~~~~~~~~~~~~~~~~~~~~~~~~~~~~~~~~~~~~
\subsection{Proof of \sectm{\cref{prop:Lehtotermlemmalegs}}}
\begin{proofs}[Proof of \cref{prop:Lehtotermlemmalegs}]
We start with applying the \cref{lem:decestimLehto} to further upper bound by the probability
\begin{equation}\label{eq:mainequationlehtotermmnsigman}
\Proba{\mathcal{L}(0,R^{-N},2)<\delta N, E_{7}(I) }\leq \Proba{\sum_{n\in I }m_{n}\sigma_{n} <\delta N, E_{7}(I ) }.  
\end{equation}
We also use the bound for $\sigma_{n}$ and the Chernoff-bound to upper bound by
\begin{equation}\label{eq:mainequationlehtotermmnsigmanChernoff}
\eqref{eq:mainequationlehtotermmnsigman}\leq \Proba{\sum_{n\in I }m_{n}\sigma_{n} <\delta N, E_{7}(I ) }\leq \Proba{\sum_{n\in I}m_{n}<\frac{\delta}{g(\delta)} N, E_{7}(I ) }\leq e^{t\frac{\delta}{g(\delta)}N }\Expe{\prod_{n\in I}e^{-tm_{n}}\ind{E_{7}(I )}},
\end{equation}
for $t>0$. We need the following lemma, whose proof we postpone.
\begin{lemma}\label{mndeviation}
For the joint moment of $m_{n}$ and $A\subset\set{1,...,N}$ we have the following upper bound
\begin{equation}
 \Exp[\prod_{n\in A}e^{-t_{n}m_{n}}\ind{E_{7}(A)}] \leq  \sumls{S\subseteq A}e^{-\sum_{i\in S^{c}}t_{i}x_{i}}\prod_{k\in S}x_{k}^{q_{k}}B,
\end{equation}
for parameters $t_{k},x_{k}>0$ and a fixed constant $B>0$.
\end{lemma} 
\noindent We set $x_{k}=x>0$ and $q_{k}=q>0$. The lemma gives the following upper bound 
\begin{equation}\label{eq:mainequationlehtotermmnsigmanChernofflemma}
 \eqref{eq:mainequationlehtotermmnsigmanChernoff}\leq e^{t\frac{\delta}{g(\delta)}N }
 \sumls{S\subset I}e^{-\abs{S^{c}}tx }\prod_{k\in S}x^{q}B.
\end{equation}
By solving the equation $B x^{q}=e^{-tx}$ for $x(t)$ we obtain the upper bound 
\begin{eqalign}\label{eq:mainequationlehtotermmnsigmanChernofflemma2}
\eqref{eq:mainequationlehtotermmnsigmanChernofflemma}\leq e^{t\frac{\delta}{g(\delta)}N }2^{\abs{I}} e^{-\abs{I}tx}.    
\end{eqalign}
It is easy to check that  $x(t)\to 0 $ and $tx(t)\to +\infty $ as $t\to +\infty$. In particular, for each small enough $\delta>0$, we choose our $t(\delta)$ s.t. 
\begin{eqalign}
x(t(\delta))=2\delta^{\frac12}\dfrac{ 1}{(1-c_{\sigma})},    
\end{eqalign}
Observe that  $\lim_{\delta\to 0^{+}}t(\delta) \delta^{\frac12}= \infty.$
Therefore, we obtain the bound
\begin{equation}\label{eq:mainequationlehtotermmnsigmanChernofflemmadec}
\eqref{eq:mainequationlehtotermmnsigmanChernofflemma2}\leq  2^{\abs{I}}\expo{-tx(1-c_{\sigma})N+t\delta^{1/2}N}\leq 2^{(1-c_{\sigma})N}e^{-t \delta^{1/2}(1-c_{\sigma})N},
\end{equation}
where the upper bound can be made sharper by taking smaller $\delta$ and thus have it be summable in $N$.
\end{proofs}
\subsection{Proof of \sectm{\cref{mndeviation}}}
\begin{proofs}

\noindent We start with
\begin{equation}\label{eq:multipointmn}
\Expe{e^{-t_{i_{1}}m_{i_{1}}}\cdots e^{-t_{i_{\abs{A}}}m_{i_{\abs{A}}}}\ind{E_{7}(A)}}    
\end{equation}
and use the following inequality that is true for any bounded random variable $X\geq 0$
\begin{equation}
\Expe{e^{-t_{k}m_{k}}X}=\Expe{e^{-t_{k}m_{k}}X 1_{m_{k}\geq x_{k}}}+\Expe{e^{-t_{k}m_{k}}X 1_{m_{k}\leq x_{k}}} \leq   e^{-t_{k}x_{k}} \Expe{X }+\Expe{X 1_{m_{k}\leq x_{k}}},
\end{equation}
where $x_{k}>0$, each time in order to get the upper bound
\begin{eqalign}\label{eq:multipointmnsummation}
\eqref{eq:multipointmn}& \leq \sumls{S\subseteq A}e^{-\sum_{i\in S^{c}}t_{i}x_{i}}\Proba{\bigcap_{k\in S}\set{m_{k}\leq x_{k}},E_{7}(A)},
\end{eqalign}
where the sum is over all possible subsets $S\subseteq A$.\\
Next we lower bound each $m_{k}=\int_{a_{k}^{P}}^{b_{k}^{P}}\frac{1}{1+L_{k,k}+L_{k}(r)} \frac{\dr}{b_{k}^{P}-a_{k}^{P}}$ by ratios of the inverse $Q_{H}$ as in the proof of \cite[Proposition 4.3.]{AJKS}. 
The denominator has the following deterministic upper bound 
\begin{equation}
1+L_{k,k}+L_{k}(r)\leq 1+L_{k,k}+ \sum_{m\geq 0}c_{k,m}\sumls{I\in \mathcal{D}_{k,m}\cap\mc{L}_{r}}K_{k,Q_{H} }(I),    
\end{equation}
where we set
\begin{eqalign}
\mc{L}_{r}:=&\set{I:  C_{I}\cap S_{r}\neq \varnothing,J_{2}^{k}(I)\cap B_{k+m}=\varnothing},\\    
c_{k,m}:=&\frac{\abs{C_{I}} }{b_{k}^{P}-a_{k}^{P}}\tfor I\in \mathcal{D}_{k,m}=\set{J\subset \spara{a_{k}^{P},b_{k}^{P}}: \abs{J}=\para{b_{k}^{P}-a_{k}^{P}}2^{-m}},\\
\tand K_{k,Q_{H} }(I):=&\sum_{(J_{1}^{k},J_{2}^{k})\in j_{5}(I)}\delta_{k,Q_{H}}(J_{1}^{k},J_{2}^{k})=\sum_{(J_{1}^{k},J_{2}^{k})\in j_{5}(I)}\frac{Q_{H}(\wt{M}_{k} J_{1}^{k})}{Q_{H}(\wt{M}_{k} J_{2}^{k}) }+\frac{Q_{H}(\wt{M}_{k} J_{2}^{k})}{Q_{H}(\wt{M}_{k} J_{1}^{k}) }.
\end{eqalign}
The internal summation domain $\set{I\in \mathcal{D}_{k,m}\cap\mc{L}_{r}}$ has at most four terms for fixed $m$ because at most four same-scale intervals satisfy $ C_{I}\cap S_{r}\neq \varnothing$.\\
\noindent Observe that if for pair $(J_{1}^{k},J_{2}^{k})$
\begin{eqalign}
\delta_{k,Q_{H}}(J_{1}^{k},J_{2}^{k})\leq \frac{c_{k,m}^{-\rho_{k}}}{x_{k}}     
\end{eqalign}
for some $\rho_{k}\in (0,1)$ ,with $\rho_{k} q_{k}>1$, then 
\begin{equation}\label{eq:finitenessiwithdyadics}
L_{k}(r)\leq  \sumls{m\geq 0}c_{k,m}\sumls{I\in \mathcal{D}_{k,m}\cap \mc{L}_{r}}\sumls{(J_{1},J_{2})\in j_{5}(I)}\frac{1}{x_{k}} c_{k,m}^{-\rho_{k}} \leq 4\frac{1}{x_{k}} \sum_{m\geq 0}c_{k,m}^{1-\rho_{k}}\leq C\frac{1}{x_{k}}.    
\end{equation}
We want this to be true for all $r\in (a_{k}^{P},b_{k}^{P})$ and so we estimate the event that it fails for one $r_{*}\in (a_{k}^{P},b_{k}^{P})$:
\begin{eqalign} \label{eq:lehtotermerrortail}
 \set{L_{k}(r_{*})> \frac{1}{x_{k}} , ~\tforsome~ r_{*}\in (a_{k}^{P},b_{k}^{P})}\subset &\bigcup_{m\geq 0}\bigcup_{\substack{I\in \mathcal{D}_{k,m}\\ I\in \mc{L}_{r}}}\bigcup_{(J_{1}^{k},J_{2}^{k})\in j_{5}(I)}\set{\delta_{k,Q_{H}}(J_{1}^{k},J_{2}^{k})>\frac{c_{k,m}^{-\rho_{k}}}{x_{k}} }.
\end{eqalign}
We get the same event estimate for $L_{k,k}$ and so we just multiply by $2$.
\pparagraph{Estimate for \eqref{eq:multipointmnsummation}}
We return to the $m_{k}$ in \cref{eq:multipointmnsummation} to upper bound 
\begin{eqalign}\label{eq:intersectionprobabimkinS}
&\Proba{\bigcap_{k\in S}\set{m_{k}\leq x_{k}},E_{7}(A)}    \\
\leq &\Proba{\bigcap_{k\in S}\bigcup_{m\geq 0}\bigcup_{\substack{I\in \mathcal{D}_{k,m}\\ I\in \mc{L}_{r}}}\bigcup_{(J_{1}^{k},J_{2}^{k})\in j_{5}(I)}\set{\delta_{k,Q_{H}}(J_{1}^{k},J_{2}^{k})>\frac{1}{x_{k}} c_{k,m}^{-\rho_{k}}},E_{7}(A)}.
\end{eqalign}
Here we "swap" intersection and unions and apply union bounds by using iterated summations
\begin{eqalign}\label{eq:intersectionprobabimkinSunion}
\eqref{eq:intersectionprobabimkinS}\lessapprox& \sumls{m_{i}\geq 0\\ i\in S}\quad\sumls{I_{i}\in \mathcal{D}_{k_{i},m_{i}}\cap \mc{L}_{r}\\ i\in S}\quad\sumls{(J_{1}^{k_{i}},J_{2}^{k_{i}})\in j_{5}(I_{i})\\ i\in S}\\
&\qquad\Proba{\bigcap_{i\in S}\set{\frac{Q_{H}(\wt{M}_{k_{i}}J_{1}^{k_{i}})}{Q_{H}(\wt{M}_{k_{i}}J_{2}^{k_{i}})}>\frac{c_{k_{i},m_{i} }^{-\rho_{k_{i}}}}{2x_{k_{i}}}}\cup \set{\frac{Q_{H}(\wt{M}_{k_{i}}J_{2}^{k_{i}})}{Q_{H}(\wt{M}_{k_{i}}J_{1}^{k_{i}})}>\frac{c_{k_{i},m_{i} }^{-\rho_{k_{i}}}}{2x_{k_{i}}}},E_{7}(A)}\\
\lessapprox& \sumls{m_{i}\geq 0\\ i\in S}\quad\sumls{I_{i}\in \mathcal{D}_{k_{i},m_{i}}\cap \mc{L}_{r}\\ i\in S}\quad\sumls{(J_{1}^{k_{i}},J_{2}^{k_{i}})\in j_{5}(I_{i})\\ i\in S}2^{\abs{S}}\Proba{\bigcap_{i\in S}\set{\frac{Q_{H}(\wt{M}_{k_{i}}J_{1}^{k_{i}})}{Q_{H}(\wt{M}_{k_{i}}J_{2}^{k_{i}})}>\frac{c_{k_{i},m_{i} }^{-\rho_{k_{i}}}}{2x_{k_{i}}}},E_{7}(A)}, 
\end{eqalign}
where we split the two ratios and kept only one of the ratios at the cost of multiplying by $2^{\abs{S}}$ due to swapping intersection and union. Here in the probability term, we apply Markov and \cref{prop:Multipointunitcircleandmaximum}(as discussed in remark in the case of equal lengths for exponents of the form $q_{k}=1+\e_{k}$ for small enough $\e_{k}>0$ with $\frac{\abs{J_{1}^{k_{i}}} }{\delta_{k_{i}}}=\para{\rho_{b}e^{-P_{k_{i}}}-\rho_{a}e^{P_{k_{i}}}}2^{-m_{i}}<1$

\begin{eqalign}
&\Proba{\bigcap_{i\in S}\set{\frac{Q_{H}(\wt{M}_{k_{i}}J_{1}^{k_{i}})}{Q_{H}(\wt{M}_{k_{i}}J_{2}^{k_{i}})}>\frac{c_{k_{i},m_{i} }^{-\rho_{k_{i}}}}{2x_{k_{i}}}},E_{7}(A)}\leq c^{\abs{S}} \prod_{i\in S} \para{x_{k_{i}}c_{k_{i},m_{i}}^{\rho_{k_{i}}}}^{q_{k_{i}}} \para{\frac{\abs{J_{1}^{k_{i}}} }{\delta_{k_{i}}}}^{-\e_{ratio}(q_{k_{i}})}.   
\end{eqalign}
Finally, since$\mathcal{D}_{k,m}$ contains at most $\floor{\para{b^{0}_{k_{i}}-a^{0}_{k_{i}}}^{-1}2^{m_{i}}}$-number of intervals in the $k_{i}$-annulus, we can upper bound 
\begin{eqalign}\label{eq:completeboundwithallsums}
\eqref{eq:intersectionprobabimkinSunion}\leq &c^{\abs{S}} \prod_{k\in S}  x_{k}^{q_{k}} \sumls{m_{i}\geq 0\\ i\in S}\para{b^{0}_{k_{i}}-a^{0}_{k_{i}}}^{-1}2^{m_{i}}\para{\para{b^{0}_{k_{i}}-a^{0}_{k_{i}}}^{-1}2^{m_{i}}}^{-q_{k_{i}}\rho_{k_{i}}}\para{\frac{\delta_{k_{i}}2^{-m_{i}} }{\delta_{k_{i}}}}^{-\e_{ratio}(q_{k_{i}})}\\
&\leq c^{\abs{S}} \prod_{k\in S}  x_{k}^{q_{k}} \sumls{m_{i}\geq 0\\ i\in S}\para{\para{\rho_{b}e^{-P_{k_{i}}}-\rho_{a}e^{P_{k_{i}}}}\rho_{*}}^{k_{i}\para{q_{k_{i}}\rho_{k_{i}}-1}}2^{-m_{i}\para{q_{k_{i}}\rho_{k_{i}}-1-\e_{ratio}(q_{k_{i}})}} \\    
&\leq  c^{\abs{S}} \prod_{k\in S} x_{k}^{q_{k}},
\end{eqalign}
where due to $q_{k_{i}}\rho_{k_{i}}-1>0$, we absorbed all the comparison constants into $c^{\abs{S}}$ and they are uniform in $\rho_{*}$ becoming smaller.\\
So now we return to \cref{eq:multipointmnsummation} to upper bound by
\begin{eqalign}\label{eq:multipointmnsummation2}
\eqref{eq:multipointmnsummation}& \leq \sumls{S\subseteq A}e^{-\sum_{i\in S^{c}}t_{i}x_{i}} c^{\abs{S}} \prod_{k\in S} x_{k}^{q_{k}},
\end{eqalign}
\end{proofs}
\newpage \section{Upper scale \sectm{$U^{1}_{n}$}} \label{sec:upperscalesdeviation}
Here we study the deviation term involving the upper scales
\begin{eqalign}
&\Proba{S_{U}(I):=\sum_{n\in I }\ind{\supl{(s,t)\in [0,Q^{n}(b_{n})]^{2}}\abs{U_{n}(t)-U_{n}(s)}\geq u_{n}}\geq c_{us}N,E_{4}(I)},
\end{eqalign}  
where we include the four events 
\begin{equation}
E_{4}(S_{3}):=\set{E_{idb}(S_{0}),E_{gap}(S_{1}),E_{ut}(S_{2}),E_{comp}(S_{3})}.    
\end{equation}
The situation is analogous to the postfactor terms $t_{n,k}$ in \cite[proposition 4.2]{AJKS}. Fix constant sequence $u_{n}:=c\rho_{b,d}$. We use the deviation estimate of the upper truncated event in \cref{eq:uppertruncateddeviation} to simplify. 
\begin{proposition}\label{prop:upperscalesdeviation}
For interval $I$ with length $\abs{I}\geq cN=c_{us}N$ we have    
\begin{eqalign}
&\Proba{S_{U}:=\sum_{n\in I }\ind{\supl{(s,t)\in [0,b_{n}+d_{n}]^{2}}\abs{U_{n}(t)-U_{n}(s)}\geq u_{n}}\geq cN}\leq c\rho_{*}^{BN},
\end{eqalign}    
for constant $B$ that be chosen arbitrarily large by taking $\rho_{*}$ smaller.
\end{proposition}
\noindent Therefore, we get the exponential decay we requested for the third term in \cref{eq:mainLehtodeviationxiUtrungapsigma}.

\begin{proofs}[proof of \cref{prop:upperscalesdeviation}]
\noindent First we remove the terms $Q^{n}(b_{n})$ from the domain of the supremum so that we can return to concentration estimates for Gaussian fields. We will use the same strategy of splitting into independent increments $U_{k}^{k-1}$ as done for the $t_{nk}$-terms in \cite[eq.(71)]{AJKS}.
\pparagraph{Independent increments}
We use the same trick as in \cite[section 4.3]{AJKS} of inserting $1=\prod_{n\in I}\para{\chi_{n}+\chi^{c}_{n}}$. Let
\begin{eqalign}
A_{n}:=&\set{\supl{(s,t)\in \spara{0,b_{n}+d_{n}}^{2}}\abs{U_{n}(t)-U_{n}(s)}\geq u_{n}}\\
E_{n,k}:=&\set{\supl{(s,t)\in \spara{0,b_{n}+d_{n}}^{2}}\abs{U_{k}^{k-1}(t)-U_{k}^{k-1}(s)}\leq u_{n,k} }   \\
\tand \chi_{n}:=&\ind{E_{n}}:=\set{\bigcap_{k=1}^{n} E_{n,k}}=\prod_{k=1}^{n}\ind{E_{n,k}}\tand  \chi_{n}^{c}:=1-\chi_{n},  \end{eqalign}
where we choose the $u_{n,k}$ s.t.
\begin{eqalign}
 \sum_{k=1}^{n}u_{n,k}<u_{n}.   
\end{eqalign}
In particular we take
\begin{equation}\label{eq:unkchoiceforu1n}
u_{n,k}= KD_{n,k}\ln\frac{1}{D_{n,k}}+D_{n,k}^{\alpha},  \end{equation}
for some $\alpha\in (0,1/2)$ and the $D_{n,k}$ is from
 \cref{lem:moduest2d}
\begin{equation}
 D_{n,k}=\sqrt{\frac{2(1-\frac{\delta_{k}}{\delta_{k-1}} )}{\delta_{k}}}\sqrt{b_{n}+d_{n}}\leq c\sqrt{\frac{b_{n}+d_{n}}{\delta_{k}}}, 
\end{equation}
which implies
\begin{eqalign}
 \sum_{k=1}^{n}u_{n,k}<c_{K}\rho_{b,d}\sum_{k=1}^{n}\rho_{*}^{n-k}<2c_{K}\rho_{b,d}.
\end{eqalign}
This means that we have the inclusions
\begin{eqalign}
&\ind{ A_{n} }   \leq \ind{ A_{n},E_{n} }+\ind{E_{n}^{c}}=0+\ind{E_{n}^{c}}=\ind{E_{n}^{c}}=\chi_{n}^{c}\\
&\Rightarrow \set{S_{U}\geq cN}\subseteq \set{\sum_{n}\chi_{n}^{c}\geq cN }.
\end{eqalign} 
We furthermore let
\begin{eqalign}
\chi_{A} :=\prod_{n\in A}\chi_{n}   \tand \chi_{c,A} :=\prod_{n\in A}\chi_{n}^{c} 
\end{eqalign}
and so we insert $1=\prod_{n\in I}\para{\chi_{n}+\chi^{c}_{n}}=\sum_{A\subset I} \chi_{A}\chi_{c,A^{c}}$ to write
\begin{eqalign}\label{eq:maindeviationuppertr}
\Proba{S_{U}\geq cN}=\sum_{A\subset I}\Expe{\ind{S_{U}\geq cN} \chi_{A}\chi_{c,A^{c}}}.    
\end{eqalign}
Here we use that on the event $\set{S_{U}\geq cN}$ we are left only with the subsets $A^{c}$ with cardinality $\abs{A^{c}}\geq cN$ to give the upper bound
\begin{eqalign}
\eqref{eq:maindeviationuppertr}\leq \sumls{A\subset I\\ \abs{A^{c}}\geq cN}\Expe{\chi_{c,A^{c}}}.    
\end{eqalign}
So we next study the event 
\begin{eqalign}
\chi_{c,A^{c}}:=\prod_{n\in A^{c}}\chi_{n}^{c}= \prod_{n\in A^{c}}\ind{\bigcup_{\ell= 1}^{n} E_{n,\ell}^{c}}=\ind{\bigcap_{n\in A^{c}}\bigcup_{\ell=1}^{n} E_{n,\ell}^{c}},
\end{eqalign}
for $A^{c}:=\set{n_{1},...,n_{p}}$ with $p\geq cN$. As done in \cite[eq.(102)]{AJKS}, we pull-out the union bound for the largest scale $n_{p}$ 
\begin{equation}\label{eq:firstunionboundups}
\Expe{\chi_{c,A^{c}}}\leq \sum_{\ell_{p}=1}^{n_{p}}\Proba{E_{n_{p},n_{p}-\ell_{p}}^{c}\cap\bigcap_{n\in A^{c}\setminus\set{n_{p}}}\bigcup_{\ell=1}^{n} E_{n,\ell}^{c}},    
\end{equation}
where we switched $\ell$-indices to write $n-\ell$ to match the \cite[eq.(102)]{AJKS}. For each $\ell_{p}$, we drop all the events $E_{n,\ell}^{c}$ that correspond to scales that are larger than it $n\geq n_{p}- \ell_{p}$ because they are correlated. And we pick the largest integer $n_{i_{p-1}}\leq n_{p}-\ell_{p}$ and apply the union bound for that scale
\begin{equation}\label{eq:firstsecondunionboundups}
\sum_{\ell_{p}=0}^{n_{p}-1}\sum_{\ell_{p-1}=0}^{n_{i_{p-1}}-1}\Proba{E_{n_{p},n_{p}-\ell_{p}}^{c}\cap E_{n_{i_{p-1}},n_{i_{p-1}}-\ell_{p-1}}^{c}\cap\bigcap_{n\in A^{c}\setminus\set{n_{p},n_{i_{p-1}}}}\bigcup_{\ell=1}^{n} E_{n,\ell}^{c}}.    
\end{equation}
The first two events here are independent. We repeat for the rest of the scales to get
\begin{eqalign}\label{eq:u1ndecoupledestimate}
\sum_{r=0}^{p-1}\sum_{(\ell_{p},...,\ell_{p-r})}\prod_{j=0}^{r} \Proba{ E_{n_{i_{p-j}},n_{i_{p-j}}-\ell_{p-j}}^{c}},   
\end{eqalign}
where we let $i_{p}:=p$ and the $\ell_{p-j}$ satisfy $n_{i_{p-j}}-\ell_{p-j}\geq n_{i_{p-j-1}}$ and in particular the covering condition
\begin{equation}\label{eq:coveringconduppertrun}
\sum_{j=0}^{r}(\ell_{p-j}+1)\geq p. \tag{*}    
\end{equation}
For the probability term in \cref{eq:u1ndecoupledestimate} we use a deviation estimate over $S^{2}=[0,b_{n}+d_{n}]^{2}$ \cref{lem:moduest2d} for 
\begin{equation}
 D_{n,k}=\sqrt{\frac{2(1-\frac{\delta_{k}}{\delta_{k-1}} )}{\delta_{k}}}\sqrt{b_{n}+d_{n}}\leq c\sqrt{\frac{b_{n}+d_{n}}{\delta_{k}}},  \end{equation}
to obtain
\begin{eqalign}\label{eq:concentrationestimateforincrement}
\Proba{\supl{(s,t)\in [0,b_{n}+d_{n}]^{2}}\abs{U_{k}^{k-1}(t)-U_{k}^{k-1}(s)}\geq u_{n,k} }\leq \expo{-\frac{1}{K^{2}D_{n,k}}\para{u_{n,k}-KD_{n,k}\ln\frac{1}{D_{n,k}} }^{2}  },    
 \end{eqalign}
where we require $u_{n,k}>KD_{n,k}\ln\frac{1}{D_{n,k}}$, $K>0$ is fixed and now the constant $D_{n,k}$ is bounded uniformly in $k,n$ since $k\leq n$.  So we study
\begin{eqalign}
\sum_{r=0}^{p-1}\sum_{(\ell_{p},...,\ell_{p-r})} \expo{-\sum_{j=0}^{r} \frac{1}{K^{2}D_{n_{i_{p-j}},n_{i_{p-j}}-\ell_{p-j}}}\para{u_{n_{i_{p-j}},n_{i_{p-j}}-\ell_{p-j}}-KD_{n_{i_{p-j}},n_{i_{p-j}}-\ell_{p-j}}\ln\frac{1}{D_{n_{i_{p-j}},n_{i_{p-j}}-\ell_{p-j}}} }^{2}  }.    
 \end{eqalign}
Here we took
\begin{equation}
u_{n,k}= KD_{n,k}\ln\frac{1}{D_{n,k}}+D_{n,k}^{\alpha},  \end{equation}
for $\alpha<1\frac{1}{2}$ to get
\begin{eqalign}\label{eq:mainsuminuppertruncated}
\sum_{r=0}^{p-1}\sum_{(\ell_{p},...,\ell_{p-r})} \expo{-c\sum_{j=0}^{r}\para{\frac{\delta_{n_{i_{p-j}}-\ell_{p-j}}}{b_{n_{i_{p-j}}}+d_{n_{i_{p-j}}}}}^{\frac{1-2\alpha}{2}}   }.    
 \end{eqalign}
\proofparagraph{Using $\rho_{\delta}=\rho_{*}$}
Up to this point we allowed even the case $\rho_{\delta}<\rho_{*}$. But now we will use the assumption $\rho_{\delta}=\rho_{*}$ 
\begin{eqalign}\label{eq:mainsuminuppertruncated2}
\eqref{eq:mainsuminuppertruncated}=\sum_{r=0}^{p-1}\sumls{(\ell_{p},...,\ell_{p-r})\\\sum_{j=0}^{r}\ell_{p-j}\geq p-r-1} \expo{-c\sum_{j=0}^{r}\para{\frac{1}{\rho_{b,d}\rho_{*}^{\ell_{p-j}}}}^{\frac{1-2\alpha}{2}}   },    
 \end{eqalign}
where for simplicity we wrote $\rho_{b,d}:=\rho_{b}+\rho_{d}$. Here we need to use the covering condition \cref{eq:coveringconduppertrun}. We split over subsets $S$ of $[r]$ for $\ell>0,\ell=0$
\begin{eqalign}\label{eq:mainsuminuppertruncated3}
\eqref{eq:mainsuminuppertruncated2}=\sumls{r=0}^{p-1}\sumls{S\subset [r]}\sumls{\ell_{p-j}\geq 1,j\in S,\\ \sum_{j=0}^{r}\ell_{p-j}\geq p-r-1} \expo{-c\frac{1}{\rho_{b,d}^{\frac{1-2\alpha}{2}} }\abs{S^{c}}  } \expo{-c\sum_{j\in S}\para{\frac{1}{\rho_{b,d}\rho_{*}^{\ell_{p-j}}}}^{\frac{1-2\alpha}{2}}   }.    
 \end{eqalign}
We split the factor into $\frac{1}{4},\frac{3}{4}$ and then for the first part we use the Bernouli inequality $\rho^{-\ell}\geq \ell \para{\frac{1}{\rho}-1}$
\begin{eqalign}\label{eq:mainsuminuppertruncated4}
\eqref{eq:mainsuminuppertruncated3}\leq \sumls{r=0}^{p-1}\sumls{S\subset [r]}e^{-c_{1}(p-\abs{S})}\sumls{\ell_{p-j}\geq 1,j\in S,\\ \sum_{j=0}^{r}\ell_{p-j}\geq p-r-1} \expo{-c_{2}\abs{S^{c}}  } \expo{-c_{3}\sum_{j\in S}\para{\frac{1}{\rho_{*}^{\ell_{p-j}}}}^{\frac{1-2\alpha}{2}}   },
 \end{eqalign}
 where the factor $e^{-c_{1}(p-\abs{S})}$ came from the covering condition \cref{eq:coveringconduppertrun} and
\begin{equation}
c_{1}:= \para{\frac{1}{\rho_{*}^{\frac{1-2\alpha}{2}}}-1}\frac{c}{4}\rho_{b,d}^{-\frac{1-2\alpha}{2}}, c_{2}:=c\rho_{b,d}^{-\frac{1-2\alpha}{2}}\tand c_{3}:=\frac{3c}{4}\rho_{b,d}^{-\frac{1-2\alpha}{2}}.    
\end{equation}
We see here that $c_{3}\frac{1}{\rho_{*}^{\frac{1-2\alpha}{2}}}>c_{1}$. Next we evaluate the $\ell$-sums.
\begin{lemma}\label{lemma:doublyexponential}
For the doubly-exponential series we have the upper bound
\begin{eqalign}
\sum_{\ell= 1}^{n}  e^{-b_{1}e^{b_{2}\ell}}\leq 2   \expo{-b_{1} \expo{b_{2}}}, 
\end{eqalign}
for $b_{1},b_{2}\geq 1$.
\end{lemma}
\begin{proof}
This follows immediately from the following inequality
\begin{eqalign}
\ell+b_{1}e^{b_{2}}\leq b_{1}e^{b_{1}\ell}, \forall \ell\geq 2,   
\end{eqalign}
which in turns follows from setting $f(x):=b_{1}b_{2}e^{b_{1}x}-b_{1}b_{2}e^{b_{1}}-x$, and computing $f'(x)= b_{1}b_{2}e^{b_{1}x}-1\geq 0$ and that $f(2)\geq 0 \doncl e(e-1)\geq 2$.
\end{proof}
So by \cref{lemma:doublyexponential}, for $b_{1}:=c_{3}$ and $b_{2}:=\ln\para{\frac{1}{\rho_{*}^{\frac{1-2\alpha}{2}}}}$ we get
\begin{eqalign}\label{eq:mainsuminuppertruncated5}
\eqref{eq:mainsuminuppertruncated4}\leq& \sumls{r=0}^{p-1}\sumls{S\subset [r]}e^{-c_{1}(p-\abs{S})}\expo{-c_{2}\abs{S^{c}}  } \expo{-c_{3}\frac{1}{\rho_{*}^{\frac{1-2\alpha}{2}}}\abs{S}}\\
=& \sumls{r=0}^{p-1}\sumls{m=0}^{r}\sumls{S\subset [r]\\\abs{S}=m}e^{-c_{1}(p-\abs{S})}\expo{-c_{2}\abs{S^{c}}  } \expo{-c_{3}\frac{1}{\rho_{*}^{\frac{1-2\alpha}{2}}}\abs{S}}\\
\leq & e^{-c_{1}p}\sumls{r=0}^{p-1}\sumls{m=0}^{r}\expo{-c_{2}(p-m)  } \expo{-\para{c_{3}\frac{1}{\rho_{*}^{\frac{1-2\alpha}{2}}}-c_{1}}m}2^{r}.
 \end{eqalign}
For $\alpha_{1}:=\expo{-c_{2}  }$ and $\alpha_{2}:=\expo{-\para{c_{3}\frac{1}{\rho_{*}^{\frac{1-2\alpha}{2}}}-c_{1}\ln\frac{1}{\rho_{*}}}}$ we get
\begin{eqalign}
\eqref{eq:mainsuminuppertruncated5}=& e^{-c_{1}p}\sumls{r=0}^{p-1}2^{r}\alpha_{1}^{p-r}(\alpha_{1}^{r+1}-\alpha_{2}^{r+1})\frac{1}{\alpha_{1}-\alpha_{2}}<ce^{-c_{1}p},
 \end{eqalign}
 where we used $c_{3}\frac{1}{\rho_{*}^{\frac{1-2\alpha}{2}}}>c_{1}$ to get finiteness. Since we also have $p>cN$ and doubly exponential $e^{-c\rho_{*}^{-\frac{1-2\alpha}{2}}}$, we indeed get a bound of the form $\rho_{*}^{BN}$ for constant $B$ that can be chosen arbitrarily large by making $\rho_{*}$ smaller.

\end{proofs}
\newpage
\section{Deviation of upper truncated }
Here we study the deviation for the upper truncated field $Q^{n}(b_{n})-b_{n}$.  As mentioned in \cite{binder2023inverse}, as $n$ grows this quantity goes to zero. However, because we need to extract a uniform exponential rate, we simply try to extract a subsequence of annuli where the following bound is true
\begin{eqalign}
Q^{n}(b_{n})\leq b_{n}+d_{n},    
\end{eqalign}
for some error $d_{n}:=\rho_{d}\rho_{*}^{n}$. That will help us estimating the supremum quantities for $\xi,U^{1}_{n}$ that involve $Q^{n}(b_{n})$. 
\begin{proposition}\label{prop:uppertruncatedlebconv}\label{prop:deviationuppertruncatedinverse} We set $d_{n}:=\rho_{d}\rho_{*}^{n}$  and $\rho_{b,d}:=\rho_{b}+\rho_{d}$. For $I\subset [N]$ with $\abs{I}>c_{ut}N$, we have the exponential decay
\begin{eqalign}\label{eq:maindeviationuppertdevizero}
\Proba{\sum_{n\in I}\ind{Q^{n}(b_{n})\geq b_{n}+d_{n}}\geq c_{ut}N}\leq c2^{N}\rho_{*}^{(1+\e_{*})N},
\end{eqalign}
for some $\e_{*}>0$.
\end{proposition}
As a result we get the exponential error in the Lehto estimate. We also include here a corollary that the subsequence of annuli that satisfy this inequality, will also be contained within the ball of radius $2$.
\begin{corollary}\label{cor:uppertruncatedlebconv}
Let 
\begin{eqalign}
E(S):=\set{\bigcap_{n\in S}\set{Q^{n}(b_{n})\leq b_{n}+d_{n}}},    
\end{eqalign}
for $S:=\set{k_{1},...,k_{L}}$. Then we have that
\begin{eqalign}
b_{k_{1},M}=\frac{\tau([x_{m,D},x_{m,D}+Q^{k_{1}}(\eta^{k_{1}}(x_{m,D}),\eta^{k_{1}}(x_{k,D})+b_{k_{1}})])}{\tau([0,1])}\leq   2.
\end{eqalign}
\end{corollary}
\begin{proofs}
We will use the deviation estimate from \cref{prop:uppertruncatedlebconv} to get
\begin{eqalign}
b_{k_{1},M}\leq \frac{\tau([x_{m,D},x_{m,D}+b_{k_{1}}+d_{k_{1}}])}{\tau([0,1])}\leq \frac{\tau([0,2])}{\tau([0,1])}=2,  
\end{eqalign}
by periodicity.    
\end{proofs}
\begin{proofs}[Proof of \cref{prop:uppertruncatedlebconv}]
\proofparagraph{Case $\rho_{\delta}<\rho_{*}$}
\noindent 
 We study the following complimentary event
\begin{equation}
\bigcapls{S\subset I\\ \abs{S}\geq c_{ut}N}\bigcup_{n\in S} \ind{Q^{n}(b_{n})-b_{n}\geq d_{n}}.    
\end{equation}
As in the proof of \cref{indnumbertheorem}, we restrict to the subset
\begin{equation}
S_{*}:=\set{N,...,(1-c)N-1}.    
\end{equation}
By the union bound
\begin{eqalign}\label{eq:maindeviationeventuions*}
\Proba{\bigcapls{S\subset I\\ \abs{S}\geq c_{ut}N}\bigcup_{n\in S} \ind{Q^{n}(b_{n})-b_{n}\geq d_{n}}}\leq& \sum_{n\in S_{*}}\Proba{\abs{Q^{n}(b_{n})-b_{n}}\geq d_{n}}\\
=&\sum_{n\in S_{*}}\Proba{b_{n}+d_{n}-\eta^{n}(b_{n}+d_{n})\geq d_{n}}.
\end{eqalign}
We apply \cref{lem:l2boundforGMC} for $\delta_{n}<b_{n}+d_{n}$ to bound
\begin{eqalign}
\eqref{eq:maindeviationeventuions*}\leq \sum_{n\in S_{*}}\frac{2(b_{n}+d_{n})\delta_{n}}{d_{n}^{2}}.
\end{eqalign}
In the  case $\rho_{\delta}<\rho_{*}$, this estimate gives summabilty and  exponential decay in $n$, since $\frac{\delta_{n}}{d_{n}}=\frac{1}{\rho_{d}}\para{\frac{\rho_{\delta}}{\rho_{*}}}^{n}$. Now we turn to the harder case $\rho_{\delta}=\rho_{*}$.
\proofparagraph{Case $\rho_{\delta}=\rho_{*}$}
Let
\begin{eqalign}
S^{>}:=&\sum_{n}\ind{Q^{n}(b_{n})\geq b_{n}+d_{n}}.
\end{eqalign}
\pparagraph{Bound with Product of independent increments} The strategy is to reduce to independent increments as it was done for the terms $L_{n,m}$ in \cite[section 4.3.]{AJKS} as defined in \cref{eq:Lnmmnsigman}. We divide into small intervals 
\begin{equation}
[0,b_{n}+d_{n}]=\bigcup_{m\geq 0}I_{n,m}\tfor I_{n,m}:=[z_{n,m},z_{n,m}+w_{n,m}] 
\end{equation}
for the sequences
\begin{eqalign}
w_{n,m}:=&(1-\rho_{*})\para{b_{n}+d_{n}}\rho_{*}^{m}    \tand z_{n,m}:=\para{b_{n}+d_{n}}\rho_{*}^{m+1},  \twith n\geq 1,m\geq 0.
\end{eqalign}
We use this partition to define the following event: for $L_{n,m}:=\eta^{n}\para{I_{n,m}}$ consider the indicator
\begin{eqalign}
\chi_{n}:=&\ind{E_{n}}:=\ind{\bigcap_{m\geq 0}\set{ L_{n,m}\geq u_{n,m}}}=\prod_{m=0}^{\infty}\ind{L_{n,m}\geq u_{n,m}}\tand  \chi_{n}^{c}:=1-\chi_{n},       
\end{eqalign}
where  $u_{n,m}:=b_{n}\rho_{u}^{m}$ ,for $\rho_{u}:=\rho_{*}^{r_{u}}$ and $r_{u}>0$. Since
\begin{eqalign}
 \sum_{m\geq 0}u_{n,m}=\frac{b_{n}}{1-\rho_{u}}>b_{n},
\end{eqalign}
 on the event $E_{n}$ we have $\eta^{n}\para{0,b_{n}+d_{n}}=\sum_{m=0}^{\infty}\eta^{n}\para{I_{n,m}}>b_{n}$ and thus the inequality
\begin{eqalign}
&\sum_{n=1 }^{N}\ind{Q^{n}(b_{n})-b_{n}\geq d_{n}} \leq \sum_{n=1 }^{N}\chi_{n}^{c}.
\end{eqalign}
So next we will replace the deviation by the indicators $\chi_{n}$. For $A\subset I$ we use that the indicators $\chi_{A} :=\prod_{n\in A}\chi_{n}   \tand \chi_{c,A} :=\prod_{n\in A}\chi_{n}^{c}$ satisfy the identity
\begin{eqalign}
1=\prod_{n\in I}\para{\chi_{n}+\chi^{c}_{n}}=\sum_{A\subset [1,N]} \chi_{A}\chi_{c,A^{c}}    
\end{eqalign}
 in order to insert it and write
\begin{eqalign}\label{eq:maindeviationuppertdevi}
\Proba{S^{>}\geq c_{ut}N}=\sum_{A\subset I}\Expe{\ind{S^{>}\geq c_{ut}N} \chi_{A}\chi_{c,A^{c}}}.    
\end{eqalign}
Here we use that on the event $\set{S^{>}\geq c_{ut}N}$ we are left only with the sets $A^{c}$ that satisfy $\abs{A^{c}} \geq c_{ut}N$ to give the upper bound
\begin{eqalign}
\eqref{eq:maindeviationuppertdevi}\leq \sumls{A\subset I\\ \abs{A^{c}}\geq c_{ut}N}\Expe{\chi_{c,A^{c}}}.    
\end{eqalign}
We next study the event in the above expectation
\begin{eqalign}
\chi_{c,A^{c}}:=\prod_{n\in A^{c}}\chi_{n}^{c}= \prod_{n\in A^{c}}\ind{\bigcup_{\ell\geq 0} \set{L_{n,n+\ell}\leq u_{n,\ell}}}=\ind{\bigcap_{n\in A^{c}}\bigcup_{\ell\geq 0} \set{L_{n,n+\ell}\leq u_{n,\ell}}},
\end{eqalign}
for $A^{c}:=\set{n_{1},...,n_{M}}$ with $M\geq c_{ut}N$. 
\begin{remark}
This event can also be written as a min-max
\begin{equation}
\Proba{\bigcap_{n\in A^{c}}\bigcup_{\ell\geq 0} \set{L_{n,n+\ell}\leq u_{n,\ell}}}=\Proba{\min_{n\in A^{c}}\max_{\ell\geq 0}\frac{u_{n,\ell}}{\eta^{n}(I_{n,n+\ell})}\geq 1}.    
\end{equation}
However, it is unclear if there is some way from here to get an exponential decay in $N$ for this probability. Perhaps it involves using some variation of the Gordon's comparison min-max inequality for Gaussian fields \cite{gordon1988milman},\cite[section 6.2]{van2014probability} since $\eta$ is an increasing function of the underlying Gaussian field. However, it is unclear if there is any field to dominate by in order to make the computation easier. 
\end{remark}
\pparagraph{Decoupling} Here we again follow the spirit of the argument from \cite[eq.(98)]{AJKS}. We apply the union bound for the first $i_{1}:=n_{1}$-scale
\begin{eqalign}\label{eq:n1scalesum}
\sum_{\ell_{1}\geq 0}\Expe{\chi_{i_{1},i_{1}+\ell_{1}}^{c} \prod_{m\in A^{c}}\chi_{n_{m}}^{c}},    
\end{eqalign}
where we generally let $\chi_{n,n+\ell}^{c}:=\set{\frac{u_{n,\ell}}{\eta^{n}(I_{n,n+\ell})}\geq 1}$. Next we pick the smallest $i_{2}\geq i_{1}$ such that
\begin{eqalign}
& \delta_{i_{2}}+b_{i_{2}}+d_{i_{2}}\leq z_{i_{1},i_{1}+\ell_{1}+1}
\end{eqalign}
and if no such integer exists, then we have $\delta_{n_p}+b_{n_{M}}+d_{n_{M}}\geq z_{i_{1},i_{1}+\ell_{1}+1}$. So we upper bound
\begin{eqalign}\label{eq:n1scalesumcases}
\eqref{eq:n1scalesum}\leq \sum_{\ell_{1}\geq 0}\Expe{\chi_{i_{1},i_{1}+\ell_{1}}^{c} \prod_{m\in A^{c}\setminus[i_{1},i_{2}]}\chi_{n_{m}}^{c}}.   
\end{eqalign}
We keep repeating this process for each of the remaining factors. There are $L=1,...,M$ possible iterated sums
\begin{equation}\label{eq:n1scalesumcasesecondallprod}
\eqref{eq:n1scalesumcases}\leq\sumls{L=1}^{M}\sumls{\ell_{i}\geq 0\\ i\in [1,L]}\prod_{j=1}^{L}\Expe{\chi_{n_{i_{j}},n_{i_{j}+\ell_{j}}}^{c} },  
\end{equation}
where we set $k_{0}:=1$, and the $\ell,k$ are restricted to satisfy the inequality constraints above. Next we translate those constraints into concrete inequalities for $\ell,k$.
\pparagraph{Covering condition}
Using the exponential-choices, we write
\begin{eqalign}
& \delta_{i_{2}}+b_{i_{2}}+d_{i_{2}}\leq z_{i_{1},i_{1}+\ell_{1}+1}
\doncl\rho_{*}^{i_{2}}+\rho_{b}\rho_{*}^{i_{2}}+\rho_{d}\rho_{*}^{i_{2}}\leq \para{\rho_{b}\rho_{*}^{i_{1}}+\rho_{d}\rho_{*}^{i_{1}}}\rho_{*}^{\ell_{1}+1}\\ \doncl&i_{1}+\ell_{1}+c_{1,\rho}< i_{2}, 
\end{eqalign}
for 
\begin{eqalign}
c_{1,\rho}:=&\para{\ln\frac{1}{\rho_{*}}}^{-1}\para{\ln\para{\rho_{b}+\rho_{d}}^{-1} +\ln\para{\rho_{b}+\rho_{d}+1} }.    
\end{eqalign}
Similarly for any $r<i_{2}$ we get
\begin{eqalign}
i_{1}+\ell_{1}+c_{1,\rho}\geq r,    
\end{eqalign}
By taking $\rho_{*}$ small enough and $\rho_{b}>\rho_{d}$ we bound
\begin{eqalign}\label{eq:constraintondn}
c_{1,\rho}=r_{b}+\para{\ln\frac{1}{\rho_{*}}}^{-1}\para{\ln\para{1+\frac{\rho_{d}}{\rho_{b}}}^{-1} +\ln\para{\rho_{b}+\rho_{d}+1} }<1   
\end{eqalign}
and thus
\begin{eqalign}
i_{2}> i_{1}+\ell_{1}\geq r.    
\end{eqalign}
That means we have the following covering condition
\begin{equation}
A^{c}\subseteq [i_{1},i_{1}+\ell_{1}]\cup     \cdots \cup [i_{L},i_{L}+\ell_{L}]
\end{equation}
and thus 
\begin{equation}\label{eq:coveringconditionuppertruncadev0}
\sum_{i=1}^{L}(\ell_{i}+1)\geq M \tag{*}.  
\end{equation}
So if we only have a subset $ S\subset [1,N]$ with nonzero $\ell_{i}$, we get
\begin{equation}\label{eq:coveringconditionuppertruncadev}
\sum_{i\in S}(\ell_{i}+1)\geq M \tag{*}.  
\end{equation}
\pparagraph{Small ball}
Here for the small ball probability term we apply \cref{cor:smallballestimateexpocho} for the \textit{case $n=k,\ell\geq 1$} and
\begin{eqalign}
&r:=\frac{1}{u_{n,\ell}}=\rho_{*}^{-n-r_{b}}\rho_{u}^{-\ell} \tand  t:=\abs{I_{n,n+\ell}}=(1-\rho_{*})\para{b_{n}+d_{n}}\rho_{*}^{\ell}=\rho_{*}^{n+\ell+r_{b}+\e},   \\
&\rho_{1}:=\rho_{u}=\rho_{*}^{r_{u}}, \rho_{2}=\rho_{*},z:=r_{b}\tand \e:=\frac{\ln\para{(1-\rho_{*})(1+\frac{\rho_{d}}{\rho_{b}})}}{\ln \rho_{*}},\\
&a_{1}:=r_{u}\tand a_{2}:=1,
\end{eqalign}
where we require
\begin{eqalign}\label{eq:c1uppertruncexpoconst0}
r_{u}>(1+\e)(1+\beta)+\beta r_{b}.  
\end{eqalign}
Therefore, we get an upper bound
\begin{equation}\label{eq:smallballdeviationupp}
\Proba{1\geq \frac{\eta^{n}(I_{n,n+\ell})}{u_{n,\ell}}}\leq e\Expe{\expo{-\frac{\eta^{n}(I_{n,n+\ell})}{u_{n,\ell}}}}\leq  c\rho_{*}^{c_{1}\ell },
\end{equation}
where
\begin{eqalign}
c_{1}:= \frac{1}{32}\para{\frac{r_{u}-(1+\e) }{\sqrt{\beta\para{(1+\e)+r_{b}}  }   }-\sqrt{\beta\para{(1+\e)+r_{b}}  }   }^{2},  
\end{eqalign}
for arbitrarily small $\e>0$. We require
\begin{eqalign}\label{eq:c1uppertruncexpoconst}
c_{1}>\frac{1}{c_{ut}}2(1+2\e_{*}),    
\end{eqalign}
for $\e_{*}>0$. This implies
\begin{eqalign}
 \eqref{eq:smallballdeviationupp}\leq c\rho_{*}^{\frac{1+2\e_{*}}{c_{ut}}(\ell+1)},
\end{eqalign}
where we used that $\ell\geq 1$ to get the extra $+1$. Both constraints \cref{eq:c1uppertruncexpoconst0} and \cref{eq:c1uppertruncexpoconst} follow by taking $r_{u}$ large.
\pparagraph{Conclusion}
Therefore, after we split over subsets $S\subset [L]$ for $\ell\geq 1,\ell=0$ we bound by
\begin{equation}\label{eq:n1scalesumcasesecondallprodest2}
\eqref{eq:n1scalesumcasesecondallprod}\leq\sumls{L=1}^{M}\sumls{S\subset [L]}\sumls{\ell_{i}\geq 1, i\in S\\\sum_{i\in S}(\ell_{i}+1)\geq M}   c\rho_{*}^{\frac{1+2\e_{*}}{c_{ut}}\sum_{i\in S}(\ell_{i}+1)}. 
\end{equation}
By splitting
\begin{eqalign}
\rho_{*}^{\frac{1+2\e_{*}}{c_{ut}}(\ell+1)}=\rho_{*}^{\frac{1+\e_{*}}{c_{ut}}(\ell+1)} \rho_{*}^{\frac{\e_{*}}{c_{ut}}(\ell+1)},   
\end{eqalign}
 we bound 
\begin{eqalign}
\eqref{eq:n1scalesumcasesecondallprodest2}\leq\sumls{L=1}^{M}\sumls{S\subset [L]}\rho_{*}^{\frac{(1+\e_{*})}{c_{ut}}M}\para{c\sumls{\ell\geq 1}\rho_{*}^{\frac{\e_{*}}{c_{ut}}\ell}}^{\abs{S}}=&\rho_{*}^{\frac{(1+\e_{*})}{c_{ut}}M}\sumls{L=1}^{M}\sumls{m=0}^{L}\para{c\frac{\rho_{*}^{\frac{\e_{*}}{c_{ut}}}}{1-\rho_{*}^{\frac{\e_{*}}{c_{ut}}}}}^{m}\binom{L}{m}\\
=&\rho_{*}^{\frac{(1+\e_{*})}{c_{ut}}M}\sumls{L=1}^{M}\para{1+c\frac{\rho_{*}^{\frac{\e_{*}}{c_{ut}}}}{1-\rho_{*}^{\frac{\e_{*}}{c_{ut}}}} }^{L}\\
\leq &\rho_{*}^{(1+\e_{*})N}\para{1+c\frac{\rho_{*}^{\frac{\e_{*}}{c_{ut}}}}{1-\rho_{*}^{\frac{\e_{*}}{c_{ut}}}}}^{M}\\
\leq &2^{N}\rho_{*}^{(1+\e_{*})N},
\end{eqalign}
where we used $M\geq c_{ut}N$ and we took $\rho_{*}$ small enough to bound $\frac{1}{1-\rho_{*}^{\frac{\e_{*}}{c_{ut}}}}\leq c$.

\end{proofs}

\newpage\section{Comparison field \sectm{$\xi$} }\label{sec:comparisonfieldxi}
Here we study the deviation term of the comparison continuous field $\xi$. The proof is very close to that of the upper scales in \cref{sec:upperscalesdeviation}. Fix a sequence 
\begin{eqalign}
u_{n}=c\frac{\rho_{*}^{1/2}}{1-\rho_{*}^{1/2}}  \para{b_{n}+d_{n}}^{1/6}.    
\end{eqalign}
\begin{proposition}\label{prop:comparisonfieldxi}
For interval $I$ with length $\abs{I}\geq cN$ we have      
\begin{equation}
\Proba{\sum_{n\in I }\ind{\supl{(s,t)\in [0,Q^{n}(b_{n})]^{2}}\abs{\xi(t)-\xi(s)}\geq u_{n}}\geq  cN}\leq c \rho_{*}^{c_{\xi}N},
\end{equation}
for constant $c_{\xi}$ that can be taken arbitrary large by selecting smaller $\rho_{*}$.
\end{proposition}
\noindent This gives the exponential decay we need for the third term in \cref{eq:mainLehtodeviationxiUtrungapsigma}.
\begin{proofs}[proof of \cref{prop:comparisonfieldxi}]
We again separate out the $Q^{n}$ deviation to be left with studying
\begin{eqalign}\label{eq:maindeviationxifield}
\eqref{prop:comparisonfieldxi}\leq &\Proba{\sum_{n\in I }\ind{\supl{(s,t)\in [0,b_{n}+d_{n}]^{2}}\abs{\xi(t)-\xi(s)}\geq u_{n}}\geq  cN}\\
=&\Proba{\bigcupls{S\subset I\\ \abs{S}\geq cN}\bigcapls{n\in S}\set{\supl{(s,t)\in [0,b_{n}+d_{n}]^{2}}\abs{\xi(t)-\xi(s)}\geq u_{n}}}.
\end{eqalign}
Since we are dealing with the same random field, we simply bound by the min
\begin{align}\label{eq:maindeviationxifield2}
\eqref{eq:maindeviationxifield}\leq&\sumls{S\subset I\\ \abs{S}\geq cN}\Proba{\supl{(s,t)\in [0,b_{m_{S}}+d_{m_{S}}]^{2}}\abs{\xi(t)-\xi(s)}\geq u_{m_{S}}},
\end{align}
where $m_{S}:=\max(S)$. For the field $\xi$ we use deviation estimate from \cref{lem:moduest2d} and in particular \cref{eq:fieldmodulusdeviation} for diameter 
\begin{equation}
D_{n}:=c\para{b_{n}+d_{n}}^{2/3}    
\end{equation}
computed in \cref{prop:xicovariancebound}. So we have the modulus bound
 \begin{equation}\label{eq:ximoddeviation}
\Proba{\supl{(s,t)\in [0,b_{n}+\delta_{n}]^{2}}\abs{\xi(t)-\xi(s)}\geq u_{n} }\leq \expo{-\frac{1}{K^{2}D_{n}}\para{u_{n}-KD_{n}\ln\frac{1}{D_{n}^{1/2}} }^{2}  },     
 \end{equation}
for $u_{n}>KD_{n,k}\ln\frac{1}{D_{n}^{1/2}}$. Therefore, we have the upper bound
\begin{align}\label{eq:maindeviationxifield3}
\eqref{eq:maindeviationxifield2}\leq&2^{cN}\expo{-c\para{\frac{1}{\rho_{*}^{1/a}}}^{m_{S}} }.
\end{align}
Since $\abs{S}\geq cN$, we get that $m_{S}\geq c_{0}N$ for some $c_{0}$. This bound is doubly-exponential and so it gives a decay faster than $\rho_{*}^{c_{\xi}N}$ for $c_{\xi}$ that be arbitrarily large by taking smaller $\rho_{*}$.
\end{proofs}
\newpage \section{The inner deterministic ball }\label{sec:holderequicontinuitybounds}
Here we study the deviation 
\begin{eqalign}
\Proba{\sum_{k\in [N]}\ind{a_{k,M}\leq c_{ov}^{-1}R_{N}=\frac{1}{c_{ov}}cR^{-N}  }\geq c_{idb}N  },      
\end{eqalign}
for $a_{k,M}:=\frac{\rho_{a}}{\rho_{b}}\frac{\tau(Q^{k}(b_{k}))}{\tau(1)}$ and some $c_{ov}<\frac{1}{2}$. We will use a parameter $c_{idb}\in (0,1)$ that we will optimize in \nameref{def:exponentialchoiceparam}.
\begin{proposition}\label{prop:scalescomparedtobN+1}
For the deviation we have the bound
\begin{eqalign}
\Proba{\sum_{k\in [N]}\ind{\frac{\rho_{a}}{\rho_{b}}\frac{\tau(Q^{k}(b_{k}))}{\tau(1)}\leq \frac{1}{c_{ov}}R_{N}  }\geq c_{idb}N  } \leq  c\rho_{*}^{(1+\e_{R})N}, \forall N\geq N_{0}.  
\end{eqalign}
and some $\e_{R}>0$. We require  $\frac{1}{c_{ov}}R_{N}<\frac{\rho_{a}}{\rho_{b}}$ and included it in \nameref{def:exponentialchoiceparam}.
 \end{proposition}

\begin{proofs}
We have
\begin{eqalign}\label{eq:innerradius3}
&\Proba{\sum_{k\in I}\ind{a_{k,M}\leq c_{ov}^{-1}R_{N}   }\geq c_{idb}N  }=\Proba{\bigcapls{S\subset I\\ \abs{S}\geq c_{idb}N}\bigcup_{k\in S} \set{a_{k,M}\leq c_{ov}^{-1}R_{N}   } },
\end{eqalign}
and then we use the set $S_{**}:=\set{1,...,c_{idb}N}$ and the union bound
\begin{eqalign}\label{eq:innerradius4}
\eqref{eq:innerradius3}\leq \sum_{k\in S_{**}}\Proba{a_{k,M}\leq cR^{-N}}.
\end{eqalign}
If $E:=\set{Q^{k}(b_{k}))\geq 1}$, then we get
\begin{eqalign}
\frac{\rho_{a}}{\rho_{b}}\leq \frac{\rho_{a}}{\rho_{b}}\frac{\tau[0,Q^{k}(b_{k})]}{\tau([0,1])}\leq c_{ov}^{-1}R_{N}.    
\end{eqalign}
But due to $c_{ov}^{-1}R_{N}<\frac{\rho_{a}}{\rho_{b}}$ that means that we only have to study the event $E^{c}=\set{Q^{k}(b_{k})\leq 1}$.  We first separate out the total mass
\begin{eqalign}
\Proba{\frac{\rho_{a}}{\rho_{b}}\frac{\tau[0,Q^{k}(b_{k})]}{\tau([0,1])} \leq c_{ov}^{-1}R_{N} ,E^{c}} \leq    \Proba{\tau[0,Q^{k}(b_{k})]  \leq R^{-(1-\lambda_{0})N},E^{c} }+\Proba{ \tau(1) \geq \frac{\rho_{a}}{\rho_{b}}c_{ov}^{-1}cR^{\lambda_{0}N} },
\end{eqalign}
for some $\lambda_{0}\in (0,1)$.  The second term is bounded by Markov
\begin{eqalign}\label{eq:determinlowerboundest5}
\Proba{ \tau(1) \geq \frac{\rho_{a}}{\rho_{b}}c_{ov}^{-1}cR^{\lambda_{0}N} }\leq cR^{-p_{1}\lambda_{0}N} \tfor p_{1}\in (0,\beta^{-1}).    
\end{eqalign}
We return to this estimate below. For the first term we use the event $E^{c}$ to switch from measure $\tau$ to measure $\eta$ and separate out the field $\xi$
\begin{eqalign}\label{eq:splitoffirstterm}
&\Proba{ \tau(Q^{k}(b_{k}))  \leq R^{-(1-\lambda_{0})N},E^{c} }\\
\leq &\Proba{ \eta(Q^{k}(b_{k}))  \leq R^{-(1-\lambda_{0}-\e_{1})N},E^{c} }+\Proba{ e^{\sup_{s\in [0,1]}\abs{\xi(s)}}\geq R^{\e_{1}N} },
\end{eqalign}
for some $\e_{1}\in (0,1)$. The second term is bounded by $R^{-\lambda_{1}\e_{1}N}$ for arbitrarily large $\lambda_{1}>0$ since $\sup_{s\in [0,1]}\abs{\xi(s)}$ has all exponential moments by \cref{it:doubleboundinv}. We next study the first term in \cref{eq:splitoffirstterm}.
\pparagraph{Decompose}We decompose the event $E^{c}=\set{Q^{k}(b_{k})\in (0,1)}$ using $a_{m}:=\rho_{*}^{r_{sum}m}$ for $r_{sum}>0$ and $m\geq 0$
\begin{eqalign}\label{eq:detlowerbound1}
&\Proba{ \eta(Q^{k}(b_{k}))  \leq R^{-(1-\lambda_{0}-\e_{1})N},E^{c} }\\
=&\sum_{m\geq 0}\Proba{ \eta(a_{m+1})  \leq R^{-(1-\lambda_{0}-\e_{1})N},Q^{k}(b_{k})\in [a_{m+1},a_{m}] }.    
\end{eqalign}
We will only focus on the slowest term $k=M:=c_{idb}N$. We split cases for $a_{m}$ using some parameter $\alpha>0$.
\pparagraph{Case $m\geq \ceil{(1+\alpha)\frac{M}{r_{sum}}}$}
Here we just keep the second probability and we scale
\begin{eqalign}\label{eq:detlowerbound2}
&\sum_{m\geq \ceil{(1+\alpha)\frac{M}{r_{sum}}}}\Proba{ \eta(a_{m+1})  \leq R^{-(1-\lambda_{0}-\e_{1})N},Q^{M}(b_{M})\in [a_{m+1},a_{m}] }\\
\leq&\sum_{m\geq\ceil{(1+\alpha) \frac{M}{r_{sum}}}}\Proba{b_{M}\leq \eta^{M}(a_{m})}=\sum_{m\geq \ceil{(1+\alpha)\frac{M}{r_{sum}}}}\Proba{\rho_{b}\leq \eta^{1}\para{\frac{a_{m}}{\delta_{M}}}}.
\end{eqalign}
In this case we have $\frac{a_{m}}{\delta_{M}}<1$ and by applying Markov for $p_{2}\in (0,\beta^{-1})$, we bound by
\begin{eqalign}\label{eq:detlowerbound3}
\eqref{eq:detlowerbound2}\leq c&\sum_{m\geq (1+\alpha)\frac{M}{r_{sum}}-1}\rho_{*}^{\zeta(p_{2})(r_{sum}m-M)}\leq c\rho_{*}^{\zeta(p_{2})(\alpha-\e) M},
\end{eqalign}
for arbitrarily small $\e>0$ to give summability. We return to this estimate below.
\pparagraph{Case $0\leq m< \ceil{(1+\alpha)\frac{M}{r_{sum}}}$}
Here we apply \cref{cor:smallballestimateexpocho0} for the maximal $m_{*}:=(1+\alpha)\frac{M}{r_{sum}}-1$. We write
\begin{eqalign}
 \Proba{ \eta(a_{m_{*}+1})  \leq R^{-(1-\lambda_{0}-\e_{1})N}}=\Proba{ \eta(t)  \leq r},   
\end{eqalign}
\begin{eqalign}
r:=&R^{-\para{1-\lambda_{0}-\e_{1}+\frac{\ln c}{N\ln R} 
 }N}=\rho_{*}^{\para{1-\lambda_{0}-\e_{1}+\frac{\ln c}{N\ln R}}c_{R}N}=\rho_{*}^{r_{1}N},    \\
 t:=&a_{m_{*}+1}=\rho_{*}^{r_{sum}\frac{m_{*}+1}{N}N}=\rho_{*}^{r_{2}N},\\
 r_{1}:=&(1-\lambda_{0}-\e_{1}+\frac{\ln c}{N\ln R})c_{R},\\
 r_{2}:=&(1+\alpha)c_{idb}.
\end{eqalign}
We require $r_{1}>r_{2}>0,$ to get
\begin{eqalign}\label{eq:detlowerbound4}
  \Expe{\expo{-r\eta(t)}}\leq  c\rho_{*}^{B_{R}N},
\end{eqalign}
where  
\begin{eqalign}
B_{R}:=& \frac{1}{32}\para{\frac{(r_{1}-r_{2}) }{\sqrt{\beta\para{r_{2}-\frac{1}{N}}  }   }-\sqrt{\beta\para{r_{2}-\frac{1}{N}}  }   }^{2}.
\end{eqalign}
So from here we get a bound 
\begin{eqalign}\label{eq:detlowerbound10}
&\sum_{0\leq m\leq \ceil{(1+\alpha)\frac{M}{r_{sum}}}}\Proba{ \eta(a_{m+1})  \leq R^{-(1-\lambda_{0}-\e_{1})N},Q^{M}(b_{M})\in [a_{m+1},a_{m}] }\leq cN\rho_{*}^{B_{R}N}\leq c\rho_{*}^{(B_{R}-\e)N},
\end{eqalign}
for small $\e>0$. In \nameref{def:exponentialchoiceparam} in \cref{eq:determinsticlowerboundcon0}, we go over the following constraints of the above estimates in \cref{eq:determinlowerboundest5},\cref{eq:detlowerbound3},\cref{eq:detlowerbound4} respectively
\begin{eqalign}\label{eq:determinsticlowerboundcon2}
&\beta^{-1}c_{R}\lambda_{0}>y+o(\e),\\
&\frac{(\beta+1)^{2}}{4\beta}\alpha c_{idb}>y+o(\e),\\
\tand &\frac{(1-\lambda_{0})c_{R}-(1+\alpha)c_{idb}}{\sqrt{\beta (1+\alpha)c_{idb}  }   }-\sqrt{\beta (1+\alpha) c_{idb}}   >\sqrt{32y}+o(\e),
\end{eqalign}
for $y:=1+\e_{R}$. These give the desired estimate for our proposition \cref{prop:scalescomparedtobN+1}.
\end{proofs}

\newpage \section{The term with consecutive \sectm{$\bar{U}^{k_{i}}_{k_{i+1}}$}}\label{sec:trunanddisjointannuli}
\noindent Here we study the deviation of the truncated Gaussian fields $U^{k_{i-1}}_{k_{i}}$ using the same approach as for the Gap event in \cref{indnumbertheorem} and for the upper scales in \cref{sec:upperscalesdeviation}.
\begin{proposition}\label{prop:truncateddeviation}
Recall $c_{4},c_{5}$ from \cref{eq:alltheconstants}. For interval $I$ with length $\abs{I}\geq c_{5}N$ we have  
\begin{equation}
\Proba{\bigcapls{S\subset I\\ \abs{S}\geq c_{4}N}\bigcup_{k_{i}\in S }\set{\supl{s\in [0,b_{k_{i}}+d_{k_{i}}]}\abs{U^{k_{i-1}}_{k_{i}}(s)}\geq u_{k_{i},k_{i-1}}  }}\leq c\rho_{*}^{(1+\e_{*})N},
\end{equation}    
for $u_{k,m}:=(m-k)r_{u,m,k}\ln\frac{1}{\rho_{*}}$, $r_{u,m,k}:=\beta+1-2\frac{r_{a}-r_{b}}{m-k}>0$ and $m>k$.
\end{proposition}
\begin{remark}\label{rem:notwointhelehtoinverse}
For the disjointness we actually only need for $u_{k,m}:=(m-k)r_{u,m,k}\ln\frac{1}{\rho_{*}}$, $r_{u,m,k}:=\beta+1-\frac{r_{a}-r_{b}}{m-k}>0$ and $m>k$. But in the independent copies section \cref{eq:matchuniquepairstop-copy} we need the one above. So in \nameref{def:exponentialchoiceparam} we constrain $r_{a}-r_{b}<r_{a}<\frac{1}{2}\minp{\beta,1}$.
\end{remark}

\noindent The proposition implies the exponential decay we requested for the fourth term in \cref{eq:mainLehtodeviationxiUtrungapsigma}.
\begin{proofs}
Similar to the proof of the gap event in \cref{thm:gapeventexistence}, we will consider a graph structure $G$ here too. We say that two nodes $v_{k},v_{m}$ connect by an edge when
\begin{eqalign}
\supl{s\in [0,b_{m}+d_{m}]}\abs{U^{k}_{m}(s)}\geq u_{k,m},    
\end{eqalign}
and so the goal is to extract a subset $S$ with size $\abs{S}\geq c_{5}N$ where none of the vertices connect. We write
\begin{eqalign}\label{eq:independencenumberevents}
 \Proba{\alpha(G)\leq c_{5} N   }= \Proba{\bigcapls{S\subset I\\ \abs{S}\geq c_{*}N}\bigcup_{k_{i}\in S }\set{\supl{s\in [0,b_{k_{i}}+d_{k_{i}}]}\abs{U^{k_{i-1}}_{k_{i}}(s)}\geq u_{k_{i},k_{i-1}}  }},  
\end{eqalign}
for $\alpha(G)$ the independence number of this graph $G$. We start with inserting all possibilities of positive/zero for the sums of edge-weights with $m>k$
\begin{eqalign}\label{eq:avch1}
\eqref{eq:independencenumberevents}=\sum_{S\subset I}    \Proba{\alpha(G)\leq (1-c_{*}) N, \bigcap_{i\in S}\set{Y_{i}>0} , \bigcap_{i\in S^{c}}\set{Y_{i}= 0}   },
\end{eqalign}
for 
\begin{eqalign}
Y_{k}:=\sum_{m>k,m\in I}\ind{O_{k,m}}:=\sum_{m>k, m\in I}\ind{\supl{s\in [0,b_{m}+d_{m}]}\abs{U^{k}_{m}(s)}\geq u_{k,m}  }.    
\end{eqalign}
It remains to study the probability term
\begin{equation}\label{eq:mainprobabilityterminc}
\Proba{\bigcap_{i\in S} \set{\sum_{m>i, m\in S }\ind{O_{i,m}} > 0}},    
\end{equation}
for  $S:=\set{s_{1},...,s_{M}}\subset I$ with $s_{i}\leq s_{i+1}$ and $M\geq  c_{5} N$. By the union bound we have
\begin{eqalign}\label{eq:Ykdecompseventprobinc}
&\eqref{eq:mainprobabilityterminc}  \leq  \Proba{\bigcapls{j\in S}\bigcup_{m_{j}=j+1\in S}O_{j,m_j} }.
\end{eqalign}
\pparagraph{Decoupling}
We pull the first scale $i_{1}:=s_{1}$ using the union bound
\begin{eqalign}\label{eq:Ykdecompseventprobinc2}
\eqref{eq:Ykdecompseventprobinc}\leq~~\quad&\sum_{m_{i_{1}}=i_{1}+1,m_{i_{1}}\in S }\Proba{O_{i_{1},m_{i_{1}}}\cap \bigcapls{j\in S\setminus\{i_{1}\}}O_{j,m_{j}}}.
\end{eqalign}
We next pick the smallest scale in $i_{2}\in S$ such $i_{2}\geq m_{i_{1}}$.  By repeating the process, for each $L\in [N]$,  we get a subsequence $S_{L}:=\set{i_{1},...,i_{L}}\subset S$ such that we have the bound 
\begin{eqalign}\label{eq:avch3inc}
\eqref{eq:Ykdecompseventprobinc2}\leq &\sumls{L=1}^{M}~~~\sumls{m_{i}=i+1,m_{i}\in S_{L}\\i\in S_{L}}~\quad\prod_{j\in S_{L}}\Proba{O_{j,m_{j}}},
\end{eqalign}
where $i_{k+1}\in S$ is the smallest integer with $i_{k+1}\geq m_{i_{k}}$. Due to this constraint we have the usual covering condition
\begin{eqalign}\label{eq:coveringconditiontruncatedincre}
\sum_{j\in S_{L}} \para{m_{j}-j+1}\geq M.   
\end{eqalign}
We next estimate the probability term.
\pparagraph{Estimates} 
We insert $U^{j}_{m}(0)$ 
\begin{eqalign}\label{eq:truncatedsupgaussiantailinc}
&\Proba{\supl{s\in [0,b_{m}+d_{m}]}\gamma \abs{U^{j}_{m}(s)}\geq u_{j,m} }\leq \Proba{\supl{s\in [0,b_{m}+d_{m}]}\gamma \abs{U^{j}_{m}(s)-\gamma U^{j}_{m}(0)}\geq\e_{0}u_{j,m} }+2\Proba{\gamma U^{j}_{m}(0)\geq (1-\e_{0})u_{j,m} },
 \end{eqalign}
 for some arbitrarily small $\e_{0}>0$. For the first term, we use \cref{thm:suprconcen} and for the second term we use the Gaussian tail estimate
\begin{eqalign}\label{eq:suptwotermsinc}
&  \expo{-\frac{1}{D_{j,m}}\para{\e_{0}u_{j,m}}^{2}  }+ \expo{-\frac{1}{2\sigma_{j,m}^{2}}\para{(1-\e_{0})u_{j,m}}^{2}  },   
 \end{eqalign}
with constant
\begin{equation}
 D_{j,\ell}=\sqrt{\frac{2(1-\frac{\delta_{m}}{\delta_{j}} )}{\delta_{m}}}\sqrt{b_{m}+d_{m}}\leq c\sqrt{\rho_{b}+\rho_{d}}
\end{equation}
and $\sigma_{j,m}^{2}:=\gamma^{2}(m-j)\ln(\frac{1}{\rho_{*}})$. Using the particular choices of parameters we bound 
\begin{eqalign}
\eqref{eq:suptwotermsinc}\leq &\expo{-c\para{\frac{1}{\rho_{*}}}\para{(m-j)r_{u,m,k}\e_{0}\ln\frac{1}{\rho_{*}}}^{2} }+\expo{-\frac{1}{2\gamma^{2}}r_{u,m,k}^{2}(1-\e_{0})^{2}\ln(\frac{1}{\rho_{*}})(m-j)}.    
\end{eqalign}
The first term has a lot faster decay and so we only keep the second term and bound the exponent by
\begin{eqalign}
\frac{1}{2\gamma^{2}}r_{u,m,k}^{2}(1-\e_{0})^{2}>c_{\gamma}:=\frac{1}{2\gamma^{2}}\para{\beta+1-2(r_{a}-r_{b})}^{2}(1-\e_{0})^{2}.    
\end{eqalign}
To use the covering condition \cref{eq:coveringconditiontruncatedincre}, we start with splitting in half to get the $+1$
\begin{eqalign}
\rho_{*}^{c_{\gamma}(m-j)}=\rho_{*}^{\frac{1}{2}c_{\gamma}(m-j)}\rho_{*}^{\frac{1}{2}c_{\gamma}(m-j)}\leq \rho_{*}^{\frac{1}{2}c_{\gamma}(m-j+1)},    
\end{eqalign}
since $m-j\geq 1$. Then we further split it into $(1-\e_{1})$ and $\e_{1}$ for some small $\e_{1}>0$
\begin{eqalign}\label{eq:avch4inc}
\eqref{eq:avch3inc}\leq &\rho_{*}^{\frac{(1-\e_{1})}{2}c_{\gamma}M}\sumls{L=1}^{M}~~~\sumls{m_{i}=i+1,m_{i}\in S_{L}\\i\in S_{L}}~\quad\prod_{j\in S_{L}}c\rho_{*}^{\frac{\e_{1}}{2}c_{\gamma} (m_{j}-j+1)}.
\end{eqalign}
The $m-$sums are bounded by $\prod_{j\in S_{L}}c\rho_{*}^{\frac{\e_{1}}{2}c_{\gamma}}\leq c^{L}\rho_{*}^{\frac{\e_{1}}{2}c_{\gamma}L}$ and so the $L$-sum is finite by taking $\rho_{*}$ small enough. We include the constraint
\begin{eqalign}\label{eq:ukkconstraint}
 \frac{(1-\e_{1})}{2}c_{\gamma}\frac{M}{N}\geq \frac{(1-\e_{1})}{2}c_{\gamma}c_{5}> 1+\e_{*}.  
\end{eqalign}
in \nameref{def:exponentialchoiceparam}.
\end{proofs}
\newpage\section{The deviation of \sectm{$\sigma_{n}$}}
\noindent Here we study the deviation term of $\sigma_{n}:=\para{1+\sumls{2o[n]<m\leq 2M }L_{n,m}}^{-1}$. As mentioned  in \cref{eq:lnmdecoupledboundPhiphin}, we have $n_{j}\in S=\set{n_1,...,n_M}$, $M\geq c_{6} N$, and $m\in [2o[n_{j}]+1,2M]$ for \begin{eqalign}
o[n_{j}]:=j\tand s[n_{j}]=n_{j+1}.
\end{eqalign} 
We considered
\begin{equation}
L_{n,m}=\sumls{(J_{1},J_{2})\in j_{5}(I) \\I\in C_{n}}\spara{\frac{Q_{H}\para{ \tilde{M}_{n}J_{1}\cap \tilde{A}_{n,m}} }{Q_{H}(\tilde{M}_{n}\para{ J_{2}\setminus B_{s[n]}}) }+\frac{Q_{H}\para{\tilde{M}_{n}J_{2}\cap \tilde{A}_{n,m}} }{Q_{H}(\tilde{M}_{n}\para{ J_{1}\setminus B_{s[n]}}) }},
\end{equation}
where $\wt{M}_{n}:=e^{\xi(b_{n})}e^{\thickbar{U}^{1}_{n}(\theta_{b_{n}})}$ and for each $j\in [1,M-1]$ we consider intervals for each $\ell\in [1,2(M-j)]$
\begin{eqalign}\label{eq:paritysplitintervals}
\tilde{A}_{n_{j},1}:=[\wt{M}_{n_{j+1}}b_{n_{j+1}},\wt{M}_{n_{j}}a_{n_{j}}]\tand \tilde{A}_{n_{j},\ell}:=\branchmat{\spara{ \wt{M}_{n_{j+\ell}}a_{n_{j+\ell}}^{0}, \wt{M}_{n_{j+\frac{\ell}{2}}}b_{n_{j+\frac{\ell}{2}}}^{0}} & \ell~\even\\
\spara{ \wt{M}_{n_{j+\ceil{\frac{\ell}{2}}}}b_{n_{j+\ceil{\frac{\ell}{2}}}}^{0}, \wt{M}_{n_{j+\floor{\frac{\ell}{2}}}}a_{n_{j+\floor{\frac{\ell}{2}}}}^{0}} & \ell~\odd},
\end{eqalign}
where we used the floor and ceiling notation in the odd case i.e. in $b_{n_{j+\ceil{\frac{\ell}{2}}}}^{0}$ we used $\ceil{\frac{\ell}{2}}\geq \frac{\ell}{2}$ and in $a_{n_{j+\floor{\frac{\ell}{2}}}}^{0}$ we used $\floor{\frac{\ell}{2}}\leq \frac{\ell}{2}$.  In particular $\tfor j=M$ we set $\tilde{A}_{n_{M}}:=[0,\wt{M}_{n_{M}}a_{n_{M}}]$. Here we use the events 
\begin{equation}
E(I):=\set{D_{ut},D_{comp},D_{us},D_{trun},G_{gap}}    
\end{equation}
imposed on the intervals $I$ with length $\abs{I}\geq c_{6}N$, where we let $c_{6}:=1-\frac{\min(\e_{ratio}-\e,1)}{2} $, the $\e_{ratio}>0$ from \cref{prop:Multipointunitcircleandmaximum} and arbitrarily small $\e>0$. 
\begin{proposition}\label{prop:deviationofsigmas}
For interval $I$ with length $\abs{I}\geq c_{6}N$ we have    
\begin{equation}\label{eq:maindeviationsigmamain}
\Proba{\sum_{n\in I }\ind{ \sigma_{n}\leq g(\delta) } \geq c_{6}N   ,E(I)}\leq c \rho_{*}^{(1+b)N},
\end{equation}    
for $1+b:=1+\frac{\e_{ratio}-\e}{2}\para{1-\e_{ratio}+\e}$, the $\e_{ratio}>0$ from \cref{prop:Multipointunitcircleandmaximum} and arbitrarily small $\e>0$. 
\end{proposition}
\noindent As before, the proposition implies exponential decay we requested for the sixth term in \cref{eq:mainLehtodeviationxiUtrungapsigma}.
\begin{proofs}[proof of \cref{prop:deviationofsigmas}]
The proof is analogous to the proof of the same estimate for the $\sigma_{n}$ in \cite[eq. (92)-(99)]{AJKS}. 
\pparagraph{Shrinking intervals}
First we use the same trick of inserting $1=\prod_{n\in I_{4}}\para{\chi_{n}+\chi^{c}_{n}}$.
Let
\begin{equation}
\chi_{n}:=\ind{E_{n}}:=\ind{\bigcap_{2o[n]<m\leq 2M } L_{n,m}\leq u_{n,m}}=\prod_{2o[n]<m\leq 2M }\ind{L_{n,m}\leq u_{n,m}}\tand  \chi_{n}^{c}:=1-\chi_{n},   
\end{equation}
where we set
\begin{eqalign}
u_{n,m}:=2^{-(\ceil{\frac{m}{2}}-o[n])}\para{2g(\delta)}^{-1}    
\end{eqalign}
so that they satisfy 
\begin{eqalign}
 \para{1+\sum_{2o[n]<m\leq 2M }u_{n,m}}^{-1}>g(\delta).    
\end{eqalign}
This means that if $\chi_{n}=1$ we get
\begin{eqalign}
g(\delta)\geq \sigma_{n}\geq  \para{1+\sum_{2o[n]<m\leq 2M }L_{n,m}}^{-1}\geq \para{1+\sum_{2o[n]<m\leq 2M }u_{n,m}}^{-1}> g(\delta),   
\end{eqalign}
and so we have the inclusion
\begin{eqalign}
&\ind{ \sigma_{n}\leq g(\delta) }   \leq \ind{ \sigma_{n}\leq g(\delta),E_{n} }+\ind{E_{n}^{c}}=&0+\ind{E_{n}^{c}}=\ind{E_{n}^{c}}=\chi_{n}^{c}\\
&\Rightarrow \set{S_{\sigma}\geq cN}\subseteq \set{\sum_{n}\chi_{n}^{c}\geq cN }.
\end{eqalign} 
We furthermore let for any $A\subseteq I$
\begin{eqalign}
\chi_{A} :=\prod_{n\in A}\chi_{n}   \tand \chi_{c,A} :=\prod_{n\in A}\chi_{n}^{c}, 
\end{eqalign}
and so we insert $1=\prod_{n\in I}\para{\chi_{n}+\chi^{c}_{n}}=\sum_{A\subseteq I} \chi_{A}\chi_{c,A^{c}}$ to write
\begin{eqalign}\label{eq:maindeviationsigma}
\Proba{S_{\sigma}\geq cN,E(I)}=\sum_{A\subset I}\Expe{\ind{S_{\sigma}\geq cN,E(I)} \chi_{A}\chi_{c,A^{c}}}.    
\end{eqalign}
Here we use that on the event $\set{S_{\sigma}\geq cN}$ we are left only with the sets $A^{c}$ with cardinality $\abs{A^{c}}\geq cN$ to give the upper bound
\begin{eqalign}\label{eq:maindeviationsigmaintersunion}
\eqref{eq:maindeviationsigma}\leq \sumls{A\subset I\\ \abs{A^{c}}\geq cN}\Expe{\chi_{c,A^{c}} \ind{E(I)}}.    
\end{eqalign}
So we next study the event 
\begin{eqalign}\label{eq:taileventofLnm}
\chi_{c,A^{c}}:=\prod_{n\in A^{c}}\chi_{n}^{c}= \prod_{n\in A^{c}}\ind{\bigcup_{2o[n]<m\leq 2M } L_{n,m}\geq u_{m}}=\ind{\bigcap_{n\in A^{c}}\bigcup_{2o[n]<m\leq 2M } L_{n,m}\geq u_{m}},
\end{eqalign}
for $A^{c}:=\set{n_{1},...,n_{p}}$ with $p\geq cN$. 
\pparagraph{Bound for the intersection of unions}
 In the sum of $L_{n,m}$  we have finitely many terms
\begin{equation}
\abs{\set{(J_{1},J_{2})\in j_{5}(I): ~ I\in C_{n}}}\leq   C_{card,corr}  ,  
\end{equation}
for some $C_{card,corr}$ that is independent of $n$. This is due to $\abs{C_{n}}\leq    (r_{a}-r_{b})p_{*}$ (\cref{eq:cardinalitycn} ). Therefore, it suffices to study the event
\begin{equation}
R_{n,\ell}:=\set{\frac{Q_{H}\para{\tilde{A}_{n,\ell}}}{Q_{H}\para{I_{n}}}\geq u_{n,\ell}},    
\end{equation} 
for $\ell\in [1, 2(M-o[n])]$ and some interval $I_{n}:=J\setminus B_{s[n]}$ of size $\abs{I_{n}}=c\delta_{n}$ and $s[n]$ is the successor of $n\in A$. For the set $S=\set{n_{1},...,n_{M}}$, we consider the events
\begin{eqalign}
E_{S}:=\bigcapls{i\in [1,M]}&\left\{Q^{n_{i}}(a_{n_{i}})\geq Q^{n_{i+1}}(b_{n_{i+1}})+\delta_{n_{i+1}},\supl{(s,t)\in [0,b_{n_{i}}+d_{n_{i}}]^{2}}\abs{\xi(t)-\xi(s)}\leq u_{n_{i},\xi},\right.\\
&,\left.\supl{(s,t)\in [0,b_{n_{i}}+d_{n_{i}}]^{2}}\abs{U^{1}_{n_{i}}(t)-U^{1}_{n_{i}}(s)}\leq u_{n_{i},us}\right\},
\end{eqalign}
where we already included the deviation event about $Q^{n}$ and the sequences $u_{n_{i},\xi},u_{n_{i},us}$ that are determined in \nameref{def:exponentialchoiceparam}.We show the following lemma.
\begin{lemma}\label{prop:mixofdecouplingandsmall}
Fix set $S:=\set{n_{1},...,n_{M}}\subset [1,N]$ with $M\geq c_{6}N$. We have the estimate
\begin{equation}
\Proba{E_{S},\bigcap_{n\in S}\bigcup_{2o[n]<m\leq 2M }R_{n,m}}\leq  c \rho_{*}^{bN}.
\end{equation}
\end{lemma}
We will use the decoupling given by the gap event and the decay of the $x_{n,\ell}$ to get an exponential decay in $N$. The strategy here is to further expand the union
 \begin{equation}
 \bigcap_{n\in S}\bigcup_{2o[n]<m\leq 2M }R_{n,m}=\bigcupls{o[n]<m_{n}\leq M}\bigcap_{n\in S}R_{n,m_{n}},
\end{equation}
where out of any collection $\set{(n_{i},n_{i+\ell_{i}})}_{i=1}^{p}$ ,for $\ell_{i}\in [1,M-i]$, we only keep the scales that yield a large enough gap event and thus conditional independence.  \cite[eq. (98)]{AJKS} and its follow-up were possible because the decoupling can be achieved by concrete distant annuli whereas here we are required to work with a random choice and with gap events. However, a similar logic of decomposition over the unions of $\ell_{i}$ will work.     
\begin{proofs}[proof of \cref{prop:mixofdecouplingandsmall}]
For ease of notation, we write $1,...,M$. We will describe a bootstrap process starting from the first scale
\begin{equation}
\bigcap_{n\in S}\bigcupls{2o[n]<m\leq 2M }R_{n,m}=\bigcup_{\ell_{1}\in [1,2(M-n_{1})]}R_{n_{1}\ell_{1}}\cap  \bigcap_{n\in S\setminus\set{n_{1}}}\bigcup_{2o[n]<m\leq 2M }R_{n,m} 
\end{equation}
of keeping or discarding scales from $S$ depending on whether it decouples with $Q^{1}$ or not.
\pparagraph{Gap event and Covering}
We start with applying the union bound for the first scale 
\begin{eqalign}\label{eq:maineventLnmtodecouple}
&\Proba{E_{S},\bigcap_{n\in S}\bigcup_{2o[n]<m\leq 2M }R_{n,m}}\leq \sum_{\ell_{1}}  \Proba{R_{n_{1},\ell_{1}},E_{S},\bigcap_{n\in S_{1}}\bigcup_{2o[n]<m\leq 2M }R_{n,m}},
\end{eqalign}
where $S_{1}:=S\setminus\set{n_{1}}$. 
Due to the gap event, we only retain the scales that are strictly larger than $n_{1+\ceil{\frac{\ell_{1}}{2}}}$
\begin{eqalign}\label{eq:maineventLnmtodecouple2}
\eqref{eq:maineventLnmtodecouple}&\leq\sum_{\ell_{1}\in [1,2(M-n_{1})}  \Proba{R_{n_{1},\ell_{1}},E_{S},\bigcap_{n\in S_{[n_{1},n_{1+\ceil{\frac{\ell_{1}}{2}}}}}\bigcup_{2o[n]<m\leq 2M }R_{n,m}},
\end{eqalign}
for $S_{[n_{1},n_{1+\ceil{\frac{\ell_{1}}{2}}}]}:=S\setminus [n_{1},n_{1+\ceil{\frac{\ell_{1}}{2}}}]$. This ensures that $R_{n_{1},\ell_{1}}$ can be decoupled from all the lower scales. We take $n_{i_{2}}$ to be the smallest such integer in $S$ that is strictly larger $n_{i_{2}}>n_{1+\ceil{\frac{\ell_{1}}{2}}}$. Then we apply the comparison $E_{S}$ event to get 
\begin{eqalign}
&\tilde{R}_{n_{1},\ell_{1}}=:\set{\frac{Q^{n_{1+\ceil{\frac{\ell_{1}}{2}}}}\para{J_{n_{1+\ceil{\frac{\ell_{1}}{2}}}}}}{Q^{n_{1}}\para{I_{n_{1}}}}\geq u_{n_{1+\ceil{\frac{\ell_{1}}{2}}}}}\tfor J_{n_{1+\ceil{\frac{\ell_{1}}{2}}}}\subset [a_{n_{1+\ceil{\frac{\ell_{1}}{2}}}},1]\tand I_{n_{1}}\subset [b_{n_{2}},1],\\
\tand &\tilde{R}_{n_{i_{2}},m}:=\set{\frac{Q^{n_{i_{2}}}\para{J_{m}}}{Q^{n_{i_{2}}}\para{I_{n_{i_{2}}}}}\geq u_{n_{i_{2}},m}}\in  \CF\spara{0,Q^{n_{i_{2}}}(b_{n_{i_{2}}})}\tfor J_{m},I_{n_{i_{2}}}\subset [0,b_{n_{i_{2}}}].    
\end{eqalign} 
These two are decoupled because of the gap event
\begin{eqalign}
Q^{n_{i_{2}}}(b_{n_{i_{2}}})+    \delta_{n_{i_{2}}}<&Q^{n_{1+\ceil{\frac{\ell_{1}}{2}}}}(a_{n_{1+\ceil{\frac{\ell_{1}}{2}}}})\leq Q^{n_{2}}(a_{n_{2}}) 
\end{eqalign}
We repeat this operation $r$-times for some $r\in [1,M]$ to get
\begin{eqalign}\label{eq:maineventLnmtodecouple3}
\eqref{eq:maineventLnmtodecouple2}&\leq\sum_{r=1}^{M}\sum_{\ell_{1},...,\ell_{r}}  \Proba{\bigcap_{j=1}^{r}R_{n_{i_{j}},\ell_{j}},E_{S}},
\end{eqalign}
where $n_{i_{j+1}}\in S$ is the smallest element of $S$ strictly larger than $n_{s_{j}}$. Let $s_{j}:=i_{j}+\ceil{\frac{\ell_{j}}{2}}$. Since these intervals $\spara{n_{i_{j}},n_{s_{j}}}$ cover the set $S$, we get a covering condition as in \cite[eq. (99)]{AJKS}
\begin{eqalign}\label{eq:coveringconditionLnm}
S\subseteq \bigsqcup_{j=1}^{r} \spara{n_{i_{j}},n_{s_{j}}}\doncl \sum_{j=1}^{r}(i_{j}+\ceil{\frac{\ell_{j}}{2}}-i_{j}+1)\geq M.
\end{eqalign}
\pparagraph{Decoupling} Using the gap event and the comparison events in $E_{S}$ we get the following bound for the probability term in \cref{eq:maineventLnmtodecouple3} by using the \cref{prop:Multipointunitcircleandmaximum} for $p=1+\e_{0}$ and to get $q=1+\e_{1}$
\begin{eqalign}\label{eq:maineventLnmtodecouple4}
\Proba{\bigcap_{j=1}^{r}R_{n_{i_{j}},\ell_{j}},E_{S}}\leq c^{r} g(\delta)^{r}2^{(1+\e_{0})\sum_{j=1}^{r}\ceil{\frac{\ell_{j}}{2}}}\prod_{\ell_{j}~\even}\para{\frac{b_{n_{s_{j}}}^{0}}{\delta_{n_{i_{j}}}}}^{q}\prod_{\ell_{j}~\odd}\para{\frac{a_{n_{(s_{j}-1)}}^{0}}{\delta_{n_{i_{j}}}}}^{q},
\end{eqalign}
where we split cases depending on the parity of $\ell$ because of \cref{eq:paritysplitintervals}. Using the \nameref{def:exponentialchoiceparamannul} we bound
\begin{eqalign}\label{eq:maineventLnmtodecouple5}
\eqref{eq:maineventLnmtodecouple4}\leq  c^{r} g(\delta)^{r}2^{(1+\e_{0})\sum_{j=1}^{r}\ell_{j}}\prod_{\ell_{j}~\even}\para{\rho_{b}\rho_{*}^{n_{s_{j}}-n_{i_{j}}}}^{q}\prod_{\ell_{j}~\odd}\para{\rho_{a}\rho_{*}^{n_{(s_{j}-1)}-n_{i_{j}}}}^{q},
\end{eqalign}
where we absorbed the comparison constants $P_{n}\in (0,1)$ in the $c^{r}$. Since $n_{s_{j}}-n_{i_{j}}\geq s_j-i_j$, we have
\begin{eqalign}\label{eq:maineventLnmtodecouple6}
\eqref{eq:maineventLnmtodecouple5}\leq  c^{r} g(\delta)^{r}2^{(1+\e_{0})\sum_{j=1}^{r}\ceil{\frac{\ell_{j}}{2}}}\prod_{\ell_{j}~~\odd}\para{\rho_{b}\rho_{*}^{\ceil{\frac{\ell_{j}}{2}}}}^{q}\prod_{\ell_{j}~\odd}\para{\frac{\rho_{a}}{\rho_{*}}\rho_{*}^{\ceil{\frac{\ell_{j}}{2}}}}^{q}.
\end{eqalign}
We split the power $q=1+\e_{1}$ into $1+\e_{1}-\e$ and $\e$, for small $\e>0$ to bound by
\begin{eqalign}\label{eq:maineventLnmtodecouple7a}
\eqref{eq:maineventLnmtodecouple6}\leq  \para{cg(\delta)\frac{\rho_{a}}{\rho_{*}}}^{r} \expo{-\sum_{j=1}^{r}\ceil{\frac{\ell_{j}}{2}}\para{(1+\e_{1}-\e)\ln\frac{1}{\rho_{*}} } }  \rho_{*}^{\e\sum_{j=1}^{r}\ceil{\frac{\ell_{j}}{2}}},
\end{eqalign}
and then use \cref{eq:coveringconditionLnm} to bound
\begin{eqalign}\label{eq:maineventLnmtodecouple7}
\eqref{eq:maineventLnmtodecouple7a}\leq   \para{cg(\delta)\frac{\rho_{a}}{\rho_{*}}}^{r} \expo{-(M-r)\para{(1+\e_{1}-\e)\ln\frac{1}{\rho_{*}}-(1+\e_{0})\ln 2 } }  \rho_{*}^{\e\sum_{j=1}^{r}\ceil{\frac{\ell_{j}}{2}}},
\end{eqalign}
We take $\rho_{*}$ small enough to ignore the $\ln 2$ term at the cost of small $\e'>\e$. Returning to the sums we obtain
\begin{eqalign}\label{eq:maineventLnmtodecouple8}
\eqref{eq:maineventLnmtodecouple3}&\leq \rho_{*}^{(1+\e_{1}-\e')M}\sum_{r=1}^{M} \para{cg(\delta)\frac{\rho_{a}}{\rho_{*}}\rho_{*}^{\e}}^{r},
\end{eqalign}
and here we get a finite sum uniformly bounded in $\rho_{*}$ by taking $g(\delta)$ small enough. We chose
\begin{eqalign}
c_{6}=1-\frac{\min(\e_{1}-\e',1)}{2}    
\end{eqalign}
in \cref{eq:maindeviationsigma} and so we get the bound
\begin{eqalign}\label{eq:maineventLnmtodecouple80}
\eqref{eq:maineventLnmtodecouple8}&\leq c\rho_{*}^{(1+b)N},
\end{eqalign}
for $1+b:=1+\frac{\e_{1}-\e'}{2}\para{1-\e_{1}+\e'}$.
\end{proofs}

\end{proofs}
\newpage\part{Beltrami solution for the independent copies}\label{Beltindcopies}
\noindent In this part, we study the existence of the welding for the composition of two independent copies of random homeomorphisms on the unit circle. More specifically, we take two independent instances of Liouville measure $\tau_1$, $\tau_2$ and study the unit circle homeomorphism $\phi(e^{2\pi ix}):=\phi_{1}^{-1}\circ \phi_{2}(e^{2\pi ix})$, where $x\in [0,1)$, $\phi_{i}(e^{2\pi ix})=e^{2\pi i h_{i}(x)}$ and $h_{i}(x):=\frac{\tau_{i}[0,x]}{\tau_{i}[0,1]}\tfor i=1,2$. Solving a Beltrami equation problem (see below) for this homeomorphism will output a random Jordan curve $\Gamma\subset \hmC$ and two conformal maps $f_{+}:\ud\to int(\Gamma)$ and $f_{-}:\hmC\setminus  \thickbar{\ud}\to ext(\Gamma)$ that on the unit circle $\T$ satisfy 
\begin{equation}\label{boundarylengths}
\phi(e^{2\pi ix})=f_{+}^{-1}\circ f_{-}(e^{2\pi ix}),\tfor x\in [0,1)\doncl \phi_{1}(f_{+}^{-1}(\zeta)=\phi_{2}(f_{-}^{-1}(\zeta), \tfor \zeta\in \Gamma.
\end{equation}
Geometrically, this agreement of the unit circle measures $\phi_{1},\phi_{2}$ means that we take two unit disks with two independent Liouville measures and we "glue" them along their boundary to obtain a curve $\Gamma\subset \hmC$. Note that this still does not mean that the resulting curve is the SLE loop from \cite{sheffield2016conformal} (see \cref{SLEloopt} for further discussion). 
\section{Beltrami equation for independent copies of the inverse}
Here we carry out the analogous formulation of the Beltrami equation problem for the independent copies $\phi_{i}^{-1}(e^{2\pi ix})=\expo{2\pi i Q_{\tau_{i}}(x \tau_{i}(1)},i=1,2,x\in [0,1)$ as done in \cite[Theorem 5.5]{AJKS}. First, we extend the inverses of the maps $\phi_{i}^{-1}$. For the $\phi_{1}^{-1}$  we extend to the unit disk using the same Ahlfors-Beurling construction $\Psi_{1}$ as in section \eqref{ABext}. For the $\phi_{2}^{-1}$ we extend it to the exterior disk $\hmC\setminus  \thickbar{\ud} $ by first extending it to the unit disk to get $\Phi_{2}$ and then doing a Schwartz reflection $\Psi_{2}(z):=\overline{1/\Phi_{2}(1/\thickbar{z}) }$ for $z\in \hmC\setminus  \thickbar{\ud} $. Also, let $F_{1}:\uhp\to \uhp$ and $F_{2}:\uhp\to\uhp$ be the corresponding Ahlfors-Beurling extensions in the upper half-plane.\\
Second, we define the Beltrami coefficients $\mu_{i}:=\frac{\partial_{\thickbar{z}}\Psi_{i}}{\partial_{z}\Psi_{i}}$ and combine them together into one global coefficient over $\hmC$:
\begin{equation}\label{twosidedBeltcoef}
\mu_{1,2}(z):=\branchmat{\mu_{1}(z)& ,\tfor z\in \ud \\ \mu_{2}(z)& ,\tfor z\in \hmC\setminus  \thickbar{\ud}}.    
\end{equation}
Therefore, we study the degenerate Beltrami equation $\frac{\partial F}{\partial \thickbar{z}}=\mu_{1,2}(z) \frac{\partial F}{\partial z}$, a.e. in $\C$ with the hydrodynamic normalization $F(z)=z+o(1)$ as $z\to \infty$. Once solved, it will yield a random Jordan curve $\Gamma:=F(\T)$ and by the uniqueness of the solutions up to conformal mapping, two conformal maps $f_{+}:\ud\to \C$ and $f_{-}:\hmC\setminus  \thickbar{\ud}\to \hmC$ that satisfy
\begin{equation}
F(z)=f_{+}\circ \Psi_{1}(z), \tfor z\in \ud \tand     F(z)=f_{-}\circ \Psi_{2}(z), \tfor z\in \hmC\setminus  \thickbar{\ud}. 
\end{equation}
This implies the welding relation on $\T$ in the above equation \eqref{boundarylengths}
\begin{equation}
f_{+}\circ \phi_{1}^{-1}(x)=f_{-}\circ \phi_{2}^{-1}(x) ,x\in \T   \doncl \phi_{1}(f_{+}^{-1}(\zeta)=\phi_{2}(f_{-}^{-1}(\zeta),\zeta\in \Gamma.
\end{equation}
To prove existence for the above degenerate Beltrami equation, we will need to use a variation of the Lehto integral for the dilatation $K_{1,2}(z):=\frac{1+\abs{\mu_{1,2}(z)}}{1-\abs{\mu_{1,2}(z)}}$ for the two-sided Beltrami coefficient in \eqref{twosidedBeltcoef}. It has the same branching
\begin{equation}\label{eq:brancheddilatationinc}
K_{1,2}(z)=\branchmat{K(z,\Psi_{1})&, z\in \ud \\ K(z,\Psi_{2})& ,\tfor z\in \hmC\setminus  \thickbar{\ud}}=\branchmat{K(w,F_{1})&, w\in \uhp \\ K(\thickbar{w},F_{2})& ,\tfor w\in \hmC\setminus\thickbar{\uhp}=\uhp^{-} },
\end{equation}
In the second identity, we again used the fact that the dilatation is not altered under locally conformal change of variables. As in the section \ref{ABext}, we study the terms:
\begin{equation}
K_{\tau_{i}}(I):=\sum_{\mathbf{J}\in \CJ(I) } \delta_{\tau_{i}}(\mathbf{J})=\sum_{\mathbf{J}\in \CJ(I) }\frac{Q_{\tau_{i}}(J_{1})\tau_{i}(0,1)}{Q_{\tau_{i}}(J_{2})\tau_{i}(0,1)}+ \frac{Q_{\tau_{i}}(J_{2})\tau_{i}(0,1)}{Q_{\tau_{i}}(J_{1})\tau_{i}(0,1)},~~\tfor i=1,2, \text{ interval }I\subset \Rplus.
\end{equation}
In terms of the Lehto integral over the unit circle $z=e^{2\pi i x}\in \T, \tfor x\in \R,$ this translates into the following split
\begin{equation}\label{Lehtointbranch}
\mathcal{L}_{K_{1,2}}(z,r,R):=\int_{r}^{R}\frac{1}{\int_{0}^{2\pi}K_{1,2}(z+\rho e^{i\theta})\dtheta}\frac{\drho}{\rho}=\int_{r}^{R}\frac{1}{\int_{0}^{\pi}K(x+\rho e^{i\theta},F_{1})\dtheta+\int_{\pi}^{2\pi}K(x+\rho e^{i\theta},F_{2})\dtheta}\frac{\drho}{\rho}. 
\end{equation}
However, we need to consider a modification of the Lehto integral, because we need to do a change of variables for the $\rho$-integral as in \cref{eq:disjointness}, but now with different constants $M_{1},M_{2}$.\\
After that, we repeat the step of lower bounding by a summation of the modified Lehto integral over disjoint random annuli $A_{k}$ with a divergent sum of modulus. Here we need to make a choice of overlapping annuli that works for both dilatations $K_{i}(\cdot):=K(\cdot,F_{i}),i=1,2$. We will prove an existence theorem analogous to \cref{Lehtoinverse}.
\begin{theorem}\label{Lehtoinverseindcop}
There exists a random homeomorphic solution $f\in W^{1,1}_{loc}(\C)$ to the degenerate Beltrami equation $\frac{\partial f}{\partial \thickbar{z}}=\mu_{1,2}(z) \frac{\partial f}{\partial z}$, a.e.. Moreover, there exists $\alpha>0$ such that the restriction $f:\T\to\C$ is a.s. $\alpha$-\Holder continuous.
\end{theorem}
\subsection{Modified Lehto-integral for branched dilatation \sectm{\cref{eq:brancheddilatationinc}}}
The following lemma is a modification of the statement and proof of \cite[Lemma 20.9.1]{astala2008elliptic} using a branched dilatation.
\begin{lemma}\label{lem:lehtointebranchdil}
Consider two annuli
\begin{eqalign}
A=\set{z: r<\abs{z}<R}\tand    A'=\set{z: r'<\abs{z}<R'}
\end{eqalign}
and write $A=A_{+}\cup A_{-}$, where $A_{+}=A\cap \mathbb{H}_{+}$ and $A_{-}=A\cap \mathbb{H}_{-}$. For every homeomorphism $f : A \to A'$ of finite distortion $K=K_{1,2}=K_{1}\chi_{\ud}+K_{2}\chi_{ \hmC\setminus  \thickbar{\ud}}$ in \cref{eq:brancheddilatationinc}, we have
\begin{eqalign}
(2\pi)^{2}\mathcal{L}_{K_{1,2}}(r,R):=&\para{2\pi (R-r)}^{2} \para{\sqrt{\int_{A_{+}}\rho K_{1}(\rho e^{i\theta},f)\dint A(\rho e^{i\theta}) }+\sqrt{\int_{A_{-}}\rho K_{2}(\rho e^{i\theta},f)\dint A(\rho e^{i\theta}) }}^{-2}\\
\leq & 2\pi \log\frac{R'}{r}.   
\end{eqalign}
\end{lemma}
\begin{remark}
The proof demonstrates that this estimate is not sharp because we upper bounded an integral over $A_{+},A_{-}$ by the integral over $A$.    
\end{remark}
\begin{proofs}
We start from the step \cite[eq. (20.110)]{astala2008elliptic}
\begin{eqalign}
2\pi\leq \rho \int_{0}^{2\pi}\frac{\sqrt{K(\rho e^{i\theta},f)J(\rho e^{i\theta},f)}}{\abs{f(\rho e^{i\theta})}} \dtheta.   
\end{eqalign}
We now need to integrate over $\rho\in [r,R]$
    \begin{eqalign}\label{eq:mainintegralLehtobran1}
2\pi (R-r)\leq \int_{r}^{R}\rho \int_{0}^{2\pi}\frac{\sqrt{K(\rho e^{i\theta},f)J(\rho e^{i\theta},f)}}{\abs{f(\rho e^{i\theta})}} \dtheta\drho.  \end{eqalign}
 Then we split the inner integral and apply \Holder
\begin{eqalign}\label{eq:mainintegralLehtobran2}
\eqref{eq:mainintegralLehtobran1}\leq& \int_{r}^{R} \int_{0}^{\pi}\frac{\sqrt{\rho K_{1}(\rho e^{i\theta},f)\rho J(\rho e^{i\theta},f)}}{\abs{f(\rho e^{i\theta})}} \dtheta\drho+\int_{r}^{R}\int_{\pi}^{2\pi}\frac{\sqrt{\rho K_{2}(\rho e^{i\theta},f)\rho J(\rho e^{i\theta},f)}}{\abs{f(\rho e^{i\theta})}} \dtheta\drho\\
\leq &\sqrt{\int_{A_{+}}\rho K_{1}(\rho e^{i\theta},f)\dint A(\rho e^{i\theta}) \int_{A_{+}} \frac{\rho J(\rho e^{i\theta},f)}{\abs{f(\rho e^{i\theta})}^{2}} \dint A(\rho e^{i\theta}) }\\
&+\sqrt{\int_{A_{-}}\rho K_{2}(\rho e^{i\theta},f)\dint A(\rho e^{i\theta}) \int_{A_{-}} \frac{\rho J(\rho e^{i\theta},f)}{\abs{f(\rho e^{i\theta})}^{2}} \dint A(\rho e^{i\theta}) }\\
\leq &\para{\sqrt{\int_{A_{+}}\rho K_{1}(\rho e^{i\theta},f)\dint A(\rho e^{i\theta}) }+\sqrt{\int_{A_{-}}\rho K_{2}(\rho e^{i\theta},f)\dint A(\rho e^{i\theta}) }}\sqrt{\int_{A} \frac{\rho J(\rho e^{i\theta},f)}{\abs{f(\rho e^{i\theta})}^{2}} \dint A(\rho e^{i\theta}) }.
\end{eqalign}
The second factor was estimated by change of variables by $\sqrt{2\pi \log\frac{R'}{r}}$. Therefore, we get
\begin{eqalign}
\para{2\pi (R-r)} \para{\sqrt{\int_{A_{+}}\rho K_{1}(\rho e^{i\theta},f)\dint A(\rho e^{i\theta}) }+\sqrt{\int_{A_{-}}\rho K_{2}(\rho e^{i\theta},f)\dint A(\rho e^{i\theta}) }}^{-1}\leq  \sqrt{2\pi \log\frac{R'}{r}}.   
\end{eqalign}

\end{proofs}

\newpage\section{Definition of overlapped Random Annuli}\label{sec:randomannuliinverseindc}
In this section we define the overlapped random annuli.
\subsection{Centers}\label{sec:centersalgor}
In \cref{def:scalingfactornthannulus} we have the centers be $y_{k,D}=\frac{\tau[0,x_{k,D}]}{\tau[0,1]}$ for $x_{k,D}:=\frac{k}{D_{n}}$ and $D_{n}:=\ceil{\rho_{*}^{-(1+\e_{**}(1-\lambda)) }}$ for some $\lambda\in (0,1)$ and $\e_{**}>0$ from \cref{Lehtodivergent}. Doing this for both copies, we have
\begin{eqalign}
X_{k}:= y_{k,D}^{1}:= \frac{\tau_{1}[0,x_{k,D}]}{\tau_{1}[0,1]}\tand   Y_{k}:= y_{k,D}^{2}:= \frac{\tau_{2}[0,x_{k,D}]}{\tau_{2}[0,1]}.
\end{eqalign}
We now describe a process to pick out a subset of centers $(X_{k_{s}},Y_{m_{s}})$ such that
\begin{eqalign}\label{eq:radiuscenters}
0\leq X_{k_{s}}-Y_{m_{s}}\leq R_{N}, 0< X_{k_{s+1}}-X_{k_{s}}\leq 2R_{N}    \tand 0< Y_{m_{s+1}}-Y_{m_{s}}\leq 2R_{N},\forall s\in [0,L_{N}],
\end{eqalign}
for some random $L_{N}\leq D_{N}$, and the $2R_{N}$-disks centered at $X_{k_{s}}$ cover the interval $[0,1]$.
We start with setting $Y_{m_{1}}=Y_{0}:=0=X_{0}=:X_{k_{1}}$. We then select $X_{k_{2}}$ with
\begin{eqalign}
k_{2}:=\min\set{k> k_{1}: X_{k}\geq Y_{m_{1}} }    
\end{eqalign}
and then $Y_{m_{2}}$ with
\begin{equation}
m_{2}:=\max\set{m\geq 1: Y_{m}\leq X_{k_{2}} }.    
\end{equation}
Generally, we let
\begin{eqalign}\label{eqdefincenters}
k_{s+1}:=\min\set{ k\in [k_{s}+1,D_{N}]: X_{k}\geq Y_{m_{s}+1} }\tand m_{s+1}:=\max\set{m\in[m_{s}+1,D_{N}]: Y_{m}\leq X_{k_{s+1}} }. 
\end{eqalign}
When it terminates at $L_{N}\leq D_{N}$, we have $X_{k_{L}}=1$. The \cref{eqdefincenters} implies
\begin{eqalign}
Y_{m_{s}}\leq X_{k_{s}}<Y_{m_{s}+1}\tand X_{k_{s}-1}<Y_{m_{s}}\leq X_{k_{s}},
\end{eqalign}
so together with 
\begin{eqalign}
0< X_{k_{s}}-X_{k_{s}-1}\leq R_{N} \tand 0< Y_{m_{s}+1}-  Y_{m_{s}}\leq R_{N}   
\end{eqalign}
we get \cref{eq:radiuscenters}.
\subsection{Overlapped annuli}\label{eq:overlappedannulibyP}
We will use the same constants $M_{n}$ for each copy as defined in \cref{imageendpoints}. Here we study the overlapped annuli  $A_{n,M_{1}}(x)\cap A_{m,M_{2}}(y)$, where as mentioned in \cref{sec:centersalgor}, we will assume 
\begin{eqalign}
0\leq x-y\leq R_{N}    
\end{eqalign}
We will denote their overlapped annuli for a pair of centers $x=X_{k_{s}}$ and $y=Y_{m_{s}}$ and pair of scales $(n_{i},m_{i})_{i=1}^{cN}$ as
\begin{eqalign}
A^{1,2}_{i,k_{s}}:=A^{1}_{n_{i}}(X_{k_{s}})\cap A_{2,m_{i}}(Y_{m_{s}}).    
\end{eqalign}
Denote the overlap as
\begin{eqalign}
P_{n,M_{1},m,M_{2}}:=\maxp{P_{n,M_{1}},P_{m,M_{2}}},    
\end{eqalign}
where 
\begin{eqalign}
P_{n,M_{1}}:=\rho_{*}^{nr_{P}} M_{i,n},   
\end{eqalign}
for $r_{P}$ constrained in \nameref{def:exponentialchoiceparamindcop}. The desired event of the bases of the annuli $A_{n,M_{1}}(x)$ $A_{m,M_{2}}(y)$ intersecting on both their right base is
\begin{eqalign}\label{eq:overlappedannulievent000}
O_{n,m,right}^{0}:=&\set{\minp{y+b_{m,M_{2}}, x+b_{n,M_{1}}}-\maxp{y+a_{m,M_{2}}, x+a_{n,M_{1}}}> P_{n,M_{1},m,M_{2}}   },
\end{eqalign}
and left base is
\begin{eqalign}
O_{n,m,left}^{0}:=&\set{\minp{y-a_{m,M_{2}}, x-a_{n,M_{1}}}-\maxp{y-b_{m,M_{2}}, x-b_{n,M_{1}}}> P_{n,M_{1},m,M_{2}}   },
\end{eqalign}
and let 
\begin{eqalign}\label{eq:overlapeventleftrightbaseint}
O_{n,m}^{0}:=O_{n,m,right}^{0}\cap O_{n,m,left}^{0}    
\end{eqalign}
In \cref{sec:overlapevent000} we obtain the existence of $cN$ annuli pairs $\para{A^{1}_{n_{i}}(X_{k_{s}})\cap A_{2,m_{i}}(Y_{m_{s}})}_{i=1}^{cN}$ with the above overlap $P_{n_i,M_{1},m_i,M_{2}}$.

\subsection{Lehto estimate}
In \cref{part:divlehtoindec} we will show the following summability estimate.
\begin{restatable}[Lehto integral divergence]{theorem}{lehtodivergenceindc}
\label{Lehtodivergentindc}
There exist constants $\e_{*}>0$ and $\delta_{0}>0$ such that for all $0<\delta<\delta_{0}$ the modulus of a map $\Phi$ with the dilatation $K_{1,2}$ from \cref{eq:brancheddilatationinc} satisfies the estimate
\begin{equation}
\Proba{\bigcupls{k,\ell\in [1,D_{N}]}\set{mod\para{\Phi\para{A^{1,2}(X_{k},Y_{\ell},R_{N},2)} } \leq 2\pi\delta N, 0\leq X_{k}-Y_{\ell}\leq R_{N}}}\leq c\rho_{*}^{(1+(\e_{*}-\e))N},\forall N\geq N_{0},
\end{equation}
for some large enough $N_{0}>0$ (that does \textit{not} depend on $\rho_{*}$) and any  $0<\e<\e_{*}$, as discussed in \nameref{def:exponentialchoiceparamannul}.
\end{restatable}
Next using this theorem we will prove the existence of a Beltrami solution for the independent copies.
\section{Beltrami solution for the independent copies of the inverse}\label{sec:proofLehtoinverseindcop}
\begin{proofs}[proof of \cref{Lehtoinverseindcop}]
The integrability of $K_{1,2}$ follows as in \cref{lem:integrabilityofdilatation} because we study two separate integrals
\begin{equation}
\int_{[0,1]\times [0,2]}K_{h_{1}^{-1}}(x+iy)\dx\dy+ \int_{[0,1]\times [-2,0]}K_{h_{2}^{-1}}(x+iy)\dx\dy.
\end{equation}
From here we continue similarly to \cref{sec:Coveringtheannulus}. Let $D_{N}:=\ceil{\rho_{*}^{-(1+\e_{**}(1-\lambda))N}}$ for $\rho_{*}$ in \nameref{def:exponentialchoiceparamannul}, some $\lambda\in (0,1)$ and $\e_{**}>0$ from \cref{Lehtodivergent}. For each $n\geq n_{0}$ we cover the unit circle $\T$ by $L_{N}$ disks $B_{3R_{N}}(e^{2\pi i X_{k_{s}}})$, for $R_{N}:=cR^{-N}$, and centers $X_{k_{s}}:=\frac{\tau_{1}(0,\frac{k_{s}}{D_{n}})}{\tau_{1}(0,1)}$ and $k_{s}=0,...,L_{N}\leq D_{N}$ as described in \cref{sec:centersalgor}. \\
In \cref{sec:centersalgor}, we showed that we can obtain successive centers that are $2R_{N}$-close $0\leq X_{k_{s+1}}-X_{k_{s}}\leq 2R_{N}$ and cover $[0,1]$ and satisfy $0\leq X_{k_{s}}-Y_{m_{s}}\leq R_{N}$. Then in \cref{eq:overlappedannulibyP} and \cref{sec:overlappedannulisection} we obtain pairs of annuli $(A_{1,n_{i}},A_{2,m_{i}})_{i=1}^{O_{N}}$, for some $O_{n}\geq cN$, that overlap. So in analogy to \cref{eq:coveringunitcircle0}, we have the events
\begin{eqalign}\label{eq:coveringunitcircle0indc}
C_{N,k}:=C_{1,N,k}\cap C_{2,N,k}:=\set{\frac{\tau(\frac{k}{D_{N}},\frac{k+1}{D_{N}})}{\tau(0,1)}\leq 3R_{N}}\cap \set{ 3R_{N}\leq \maxp{a_{n_{O_{N}}}M_{1,n_{O_{N}}},a_{m_{O_{N}}}M_{2,m_{O_{N}}}}  }, 
\end{eqalign}
where we changed $R_{N}\to 3R_{N}$ and wrote a maximum since we consider overlap (and we also used the closeness of the centers $0\leq X_{k_{s}}-Y_{m_{s}}\leq R_{N}$ to simplify the $ C_{2,N,k}$-event). The completements of these event are estimated as before (for the maximum we simply split cases). Therefore, we get the \Holder-equicontuity as in the proof of \cref{Lehtoinverse} for the radius $3R_{N}$.
\end{proofs}

\cref{sec:holderequicontinuitybounds}

\newpage \part{Deviation Estimates for the Lehto integral of independent copies}\label{part:divlehtoindec}
\noindent This part is parallel to \cref{part:deviationestLehto}.
\lehtodivergenceindc* 
\begin{proofs}
Here we proceed similarly to \cref{sec:proofofLehtodivergence} by inserting deviations. 
\pparagraph{The Lehto-inverse events}We first insert all the deviation events we had for the inverse in \cref{eq:mainLehtodeviationxiUtrungapsigma} and \cref{eq:coveringunitcircle0}
\begin{eqalign}\label{eq:Lehtoeachcopyevents}
E_{idb}(S_{0})\cap E_{gap}(S_{1})\cap E_{ut}(S_{2})\cap E_{comp}(S_{3})\cap E_{us}(S_{4})\cap E_{trun}(S_{5})\cap E_{\sigma}(S_{6})  \cap C_{n,k},  
\end{eqalign}
but for each copy by denoting them as $E_{Lehto,1}(A), E_{Lehto,2}(B)$ respectively where the sets $A,B\subset [N]$ satisfy the above events for each copy separately and with the same size $\abs{A}=\abs{B}$. These were shown to be $N$-summable for each center $y_{k}:=\frac{\tau(0,\frac{k}{D_{n}})}{\tau(0,1)}$, so we will insert the events
\begin{eqalign}\label{eq:combinedLehtoevent}
E_{Lehto,1,2}(A,B):=\bigcap_{k\in [1,D_{N}]} E_{Lehto,1}(y_{1,k},A)\cap \bigcap_{\ell\in [1,D_{N}]} E_{Lehto,2}(y_{2,\ell},B).   
\end{eqalign}
where we also added the center dependence. As shown in \cref{sec:centersalgor} the event $C_{n,k}$ implies that we get centers-pairs $(X_{k_{s}},Y_{m_{s}})$  that are $R_{N}$-close enough.
\pparagraph{The overlapped annuli events}
In the next section we study the summability of the complement of an event $O_{overlap}$ that imply that there is some $S\subseteq A$ of cardinality $\abs{S}\geq cN$ so that every top-copy annulus $A_{1,n}(x),n\in S$ intersects with some bottom-copy annulus $A_{2,m}(y)$ for $m\in I$, with centers satisfying $0\leq x-y\leq R_{N} $ as mentioned in \cref{sec:centersalgor}. This will be shown in \cref{eq:overlapeventBorelCantellisummable}. The issue is that the pairing might not be unique i.e. $A_{1,s_{1}}(x),A_{1,s_{2}}(x)$ getting matched with the same $A_{2,m_{1}}(y)$. So in \cref{sec:uniquepairingforoverlap}, we will show that the estimate in \cref{sec:trunanddisjointannuli} implies that there is some further subset $\tilde{S}\subset S$ where each top-copy $A_{1,s_{i}}(x)$ is paired with a different bottom-copy.
\pparagraph{Summarized event}
The overall event we estimate in the next section is
\begin{eqalign}\label{eq:Lehtotermestimate}
\bigcupls{k,\ell\in [1,D_{N}]}\set{mod\para{\Phi\para{A^{1,2}(X_{k},Y_{\ell},R_{N},2)} } \leq 2\pi\delta N, 0\leq X_{k}-Y_{\ell}\leq R_{N} } \cap E_{Lehto,1,2}(A,B)\cap O_{overlap}.
\end{eqalign}

\end{proofs}
\newpage\section{The Lehto term }
\begin{proposition}\label{prop:Lehtotermlemmalegsindc}
For interval $I$ with length $\abs{I}\geq c_{6}N$ we have
\begin{eqalign}
&\Proba{\set{\bigcupls{k,\ell\in [1,D_{N}]}mod\para{\Phi\para{A^{1,2}(X_{k},Y_{\ell},R_{N},2)} } \leq 2\pi\delta N, 0\leq X_{k}-Y_{\ell}\leq R_{N} } \cap E_{Lehto,1,2}(A,B)\cap O_{overlap} }\\
&\leq C \expo{-C_{Leh}N},
\end{eqalign}
where the constant $C_{Leh}=C_{Leh}(\delta)$ can be made arbitrarily large for sufficiently small $\delta>0$.
\end{proposition}
\noindent This implies exponential decay for the first term. Since the constant is arbitrarily large, it is smaller than the claimed estimate in \cref{Lehtodivergentindc}.
\subsection{Decoupling of Lehto integral}\label{sec:decoupLehtointindc}
We first obtain an analogue of the \cref{lem:decestimLehto}. 
\begin{lemma}\label{lem:decestimLehtoindc}
For any two subsets of $[N]$ of the same size $L$, $A=\set{n_{1},...,n_{L}},B:=\set{r_{1},...,r_{L}}$  that satisfy $E_{Lehto}^{1,2}(A,B)$  the following holds
\begin{eqalign}
&mod\para{\Phi\para{A^{1,2}(X_{k},Y_{\ell},R_{N},2)} }\\
&\geq 2\pi\sum_{i\in [L]}\para{\sqrt{m_{1,n_{i}}\para{\sigma_{1,n_{i}}}^{-1}}+\sqrt{m_{2,r_{i}}\para{\sigma_{2,r_{i}}}^{-1}}}^{-2},
\end{eqalign}
where for each copy we set
\begin{eqalign}\label{eq:lnmdecoupledboundPhiphinindc}
\sigma_{n}:=&\para{1+\sum_{2o[n]<m\leq 2L}L_{n,m}}^{-1},\\
m_{n}:=&\int_{a_{n}^{0}}^{b_{n}^{0}} 1+L_{n}(r)+L_{n,n} \frac{\dr}{\rho_{*}^{nr_{P}}},\\
\delta_{n,Q_{H}}(J_{1},J_{2}):=&\frac{Q_{H}(\wt{M}_{n} J_{1})}{Q_{H}(\wt{M}_{n} J_{2}) }+ \frac{Q_{H}(\wt{M}_{n} J_{2})}{Q_{H}(\wt{M}_{n} J_{1}) },\\
L_{n}(r):=&\sum_{I\in C_{r,n}^{dec}}\frac{\abs{C_{I}}}{\rho_{*}^{nr_{P}}}  \sum_{(J_{1},J_{2})\in j_{5}(I)}\delta_{Q_{H}}(\tilde{M}_{n}\cdot \mathbf{J}),\\
\end{eqalign}
and the same $L_{n,m},L_{n,n},C_{n},C_{r,n}^{dec}$ as in \cref{eq:lnmdecoupledboundPhiphin}.
\end{lemma}
 \begin{proof}
We have a sequence of overlapped disjoint annuli $A_{i}^{1,2}:=A_{1,n_{i}}\cap A_{2,r_{i}},i\in [L]$ with base $[a_{i,M}^{1,2},b_{i,M}^{1,2}]$ where
\begin{eqalign}
a_{i,M}^{1,2}:=\maxp{a_{1,n_{i},M},a_{2,r_{i},M}}    \tand b_{i,M}^{1,2}:=\maxp{b_{1,n_{i},M},b_{2,r_{i},M}}.    
\end{eqalign}
We first lower bound the modulus by the disjoint annuli
\begin{equation}\label{eq:decoupindc0}
mod\para{\Phi\para{A^{1,2}(X_{k},Y_{\ell},R_{N},2)} } \geq \sum_{i\in I}mod(\Phi(A_{i}^{1,2})),     
\end{equation}
and we further apply \cref{lem:lehtointebranchdil}
\begin{eqalign}\label{eq:decoupindc1}
\eqref{eq:decoupindc0}\geq& 2\pi\sumls{i\in [L]}\mathcal{L}_{K_{1,2}}(a_{i,M}^{1,2},b_{i,M}^{1,2})\\
=&2\pi\sumls{i\in [L]}\Bigg(\sqrt{\int_{a_{i,M}^{1,2}}^{b_{i,M}^{1,2}}\int_{0}^{\pi}\rho K_{1}(\rho e^{i\theta},\Phi) \dtheta \frac{\drho}{(b_{i,M}^{1,2}-a_{i,M}^{1,2})^{2}} }\\
&+\sqrt{\int_{a_{i,M}^{1,2}}^{b_{i,M}^{1,2}}\int_{\pi}^{2\pi}\rho K_{2}(\rho e^{i\theta},\Phi) \dtheta \frac{\drho}{(b_{i,M}^{1,2}-a_{i,M}^{1,2})^{2}} }\Bigg)^{-2}.
\end{eqalign}
Since the $\rho-$integrals are separate we can do separate change of variables
\begin{eqalign}\label{eq:decoupindc2}
\eqref{eq:decoupindc1}\geq &2\pi \sumls{i\in [L]}\Bigg(\sqrt{\int_{a_{n_{i}}^{0}}^{b_{n_{i}}^{0} }\int_{0}^{\pi}\rho K_{1}(\rho e^{i\theta},\Phi) \dtheta \frac{\drho}{(P_{n_{i}}^{0})^{2}} }\\
&+\sqrt{\int_{a_{r_{i}}^{0}}^{b_{r_{i}}^{0} }\int_{\pi}^{2\pi}\rho K_{2}(\rho e^{i\theta},\Phi) \dtheta \frac{\drho}{(P_{r_{i}}^{0})^{2}} }
\Bigg)^{-2}\\
\end{eqalign}
These integrals are similar to those in \cref{eq:disjointness} and so we proceed as in \cref{lem:decestimLehto}.
\end{proof}

\subsection{Proof of \sectm{\cref{prop:Lehtotermlemmalegsindc}}}
\begin{proofs}[Proof of \cref{prop:Lehtotermlemmalegsindc}]
For simplicity we just index over a set $I\subset [N]$ for the pairs of scales $(n_{i},r_{i})\in A\times B$ and $i\in I$. We start in the \cref{lem:decestimLehtoindc} to further upper bound by the probability
\begin{eqalign}\label{eq:mainequationlehtotermmnsigmanindc}
&\Proba{mod\para{\Phi\para{A^{1,2}(X_{k},Y_{\ell},R_{N},2)} }<2\pi\delta N, E_{Lehto,1,2}(A,B)\cap O_{overlap} }\\
\leq &\Proba{\sum_{i\in I}\para{\sqrt{m_{1,n_{i}}\sigma_{1,n_{i}}}+\sqrt{m_{2,r_{i}}\sigma_{2,r_{i}}}}^{-2} <\delta N, E_{Lehto,1,2}(A,B)\cap O_{overlap}}.  
\end{eqalign}
On the event $E_{Lehto,1,2}(A,B)\cap O_{overlap}$, we have the following bounds for $\sigma_{n}^{1},\sigma_{n}^{2}$ 
\begin{eqalign}
 \set{\sigma_{n_i}^{1}\geq g(\delta)}\cap \set{\sigma_{2,r_{i}}\geq g(\delta)},\forall n\in I.
 \end{eqalign}
We can now use the Chernoff-bound to upper bound by
\begin{eqalign}\label{eq:mainequationlehtotermmnsigmanChernoffindc}
\eqref{eq:mainequationlehtotermmnsigmanindc}\leq& \Proba{\sum_{i\in I}\para{\sqrt{m_{1,n_{i}}}+\sqrt{m_{2,r_{i}}}}^{-2}<\frac{\delta}{g(\delta)} N, E_{Lehto,1,2}(A,B)\cap O_{overlap} }\\
\leq &e^{t\frac{\delta}{g(\delta)}N }\Expe{\prod_{i\in I}e^{-t\para{\sqrt{m_{1,n_{i}}}+\sqrt{m_{2,r_{i}}}}^{-2}}\ind{E_{Lehto,1,2}(A,B)\cap O_{overlap}}},
\end{eqalign}
for $t>0$. Below we will prove the following lemma. We will write $E(I):=E_{Lehto,1,2}(A,B)\cap O_{overlap}$ for short and $E_{1}(I):=E_{Lehto,1,2}(A)$ and $E_{2}(I):=E_{Lehto,1,2}(B)$.
\begin{lemma}\label{mndeviationindc}
For the joint moment of $m_{n}$ and $I\subset\set{1,...,N}$ we upper bound
\begin{equation}
 \Exp[\prod_{i\in I}e^{-t_{i}\sqrt{m_{1,n_{i}}}+\sqrt{m_{2,r_{i}}}}\ind{E(I)}] \leq  \sumls{S\subseteq I}e^{-\frac{1}{8}\sum_{i\in S^{c}}t_{i}x_{i}}\prod_{k\in S}x_{k}^{q_{k}}B_{k},
\end{equation}
for parameters $x_{k}\in (0,1)$ that can be optimized and fixed constants $B_{k}>0$.
\end{lemma} 
\noindent This lemma gives similar upper bound as in \cref{eq:mainequationlehtotermmnsigmanChernofflemma} to get
\begin{equation}\label{eq:mainequationlehtotermmnsigmanChernofflemmaindc}
 \eqref{eq:mainequationlehtotermmnsigmanChernoffindc}\leq ce^{t\frac{\delta}{g(\delta)}N }
 \sumls{S\subset I}e^{-\frac{1}{8}\sum_{i\in S^{c}}tx_{i}}\prod_{k\in S}x_{k}^{q_{k}}B_{k}\leq c2^{(1-c_{\sigma})N}e^{-t \delta^{1/2}(1-c_{\sigma})N}.
\end{equation}
This the upper bound can again be made sharper by taking smaller $\delta$ and thus have it be summable in $N$. 
\end{proofs}
\subsection{proof of \sectm{\cref{mndeviationindc}}}
\begin{proofs}
\noindent We start with
\begin{equation}\label{eq:multipointmnindc}
\Expe{e^{-t_{1}\para{\sqrt{m_{1,n_{1}}}+\sqrt{m_{2,r_{1}}}}^{-2}}\cdots e^{-t_{L}\para{\sqrt{m_{n_{L}}^{1}}+\sqrt{m_{r_{L}}^{2}}}^{-2}}\ind{E(I)}}    
\end{equation}
and use the following inequality for $Y_{k}:=\para{\sqrt{m_{1,n_k}}+\sqrt{m_{2,r_k}}}^{-2}$, $B_{k}:=\frac{x_{k}}{4(x_{k}+2)}\geq \frac{x_{k}}{12}$ and generic integrable $X\geq 0$, 
\begin{equation}
\Expe{e^{-t_{k}Y_{k}}X}=\Expe{e^{-t_{k}Y_{k}}X 1_{Y_{k}> B_{k}}}+\Expe{e^{-t_{k}Y_{k}}X 1_{Y_{k}< B_{k}}} \leq   e^{-t_{k}B_{k}} \Expe{X }+\Expe{X 1_{Y_{k}\leq  B_{k}}}.
\end{equation}
We repeat this to get the upper bound
\begin{eqalign}\label{eq:multipointmnsummationindc}
\eqref{eq:multipointmnindc}& \leq \sumls{S\subseteq I}e^{-\frac{1}{12}\sum_{i\in S^{c}}t_{i}x_{i}}\Proba{\bigcap_{k\in S}\set{Y_{k}\leq B_{k}},E(I)}.
\end{eqalign}
where the sum is over all possible subsets $S\subseteq I$. If we have the events
\begin{eqalign}\label{eq:lnindcopies} 
 & \set{L_{1,n_k}(r)\leq  \frac{1}{x_{k}} , ~\forall r\in (a_{n_k}^{0},b_{n_k}^{0})}\cap\set{L_{1,n_k,n_k}\leq  \frac{1}{x_{k}}}\\
 &\cap\set{L_{2,r_{k}}(r)\leq  \frac{1}{x_{k}} , ~\forall r\in (a_{r_{k}}^{0},b_{r_{k}}^{0})}\cap\set{L_{2,r_{k},r_{k}}\leq  \frac{1}{x_{k}}},
\end{eqalign}
then we get
\begin{eqalign}
m_{n_{k}}^{i}=\int_{a_{n_{k}}^{0}}^{b_{n_{k}}^{0}} 1+L_{i,n_{k}}(r)+L_{i,n_{k},n_{k}} \frac{\dr}{P_{n_{k}}^{0}}\leq 1+\frac{2}{x_{k}}=\frac{2+x_{k}}{x_{k}}    
\end{eqalign}
and thus
\begin{eqalign}
 Y_{k}=\para{\sqrt{m_{1,n_k}}+\sqrt{m_{2,r_{k}}}}^{-2}\geq  \frac{x_{k}}{4(x_{k}+2)} = B_{k}.   
\end{eqalign}
So as in \cref{eq:lehtotermerrortail}
\begin{eqalign} 
\Prob \Big[\bigcap_{k\in S}\bigcup_{i=1,2}\set{L_{i,f(i,k)}(r_{*})> \frac{1}{x_{k}} , ~\tforsome~ r_{*}\in (a_{f(i,k)}^{0},b_{f(i,k)}^{0})}\cup\set{L_{i,f(i,k),f(i,k)}\geq  \frac{1}{x_{k}}}\\
,E(I)\Big],
\end{eqalign}
where $f(1,k):=n_{k}$ and $f(2,k):=r_{k}$. So we get the following analogues of \cref{eq:intersectionprobabimkinS}
\begin{eqalign}\label{eq:intersectionprobabimkinSindc}
&\Proba{\bigcap_{k\in S}\set{Y_{k}\leq B_{k}},E(I)}    \\
\leq &\Proba{\bigcap_{k\in S}\bigcup_{\ell\geq 0}\bigcup_{\substack{I\in \mathcal{D}_{k,\ell}\\ I\in \mc{L}_{r}}}\bigcup_{(J_{1}^{k},J_{2}^{k})\in j_{5}(I)}\bigcup_{i=1,2}\set{\delta^{i}_{f(i,k),Q_{H}}(J_{1}^{f(i,k)},J_{2}^{f(i,k)})>\frac{1}{x_{k}} c_{f(i,k),\ell}^{-\rho_{k}}},E(I)},
\end{eqalign}
and of \cref{eq:intersectionprobabimkinSunion}
\begin{eqalign}\label{eq:intersectionprobabimkinSunionindc}
\eqref{eq:intersectionprobabimkinSindc}\lessapprox& \sumls{\ell_{k}\geq 0\\ k\in S}\quad\sumls{I_{k}\in \mathcal{D}_{k,\ell_{k}}\cap \mc{L}_{r}\\ k\in S}\quad\sumls{(J_{1}^{f(i,k)},J_{2}^{f(i,k)})\in j_{5}(I_{i})\\ k\in S}2^{\abs{S}}\sumls{S_{1}\subset S, S_{2}=S\setminus S_{1}}\\
&\Proba{\bigcap_{k\in S_{1}}\set{\frac{Q_{1,H}(\wt{M}_{1,n_{k}}J_{1}^{n_{k}})}{Q_{1,H}(\wt{M}_{1,n_{k}}J_{2}^{n_{k}})}>\frac{c_{n_{k},\ell_{k} }^{-\rho_{k}}}{2x_{k}}},E_{1,5}(I), \bigcap_{k\in S_{2}}\set{\frac{Q_{2,H}(\wt{M}_{2,m_{k}}J_{1}^{m_{k}})}{Q_{2,H}(\wt{M}_{2,m_{k}}J_{2}^{m_{k}})}>\frac{c_{r_{k},\ell_{k} }^{-\rho_{k}}}{2x_{k}}},E_{2,5}(I)}, 
\end{eqalign}
where due to the inner union for the two copies we also have to add an iterated sum over all possible subsets splits $S_{1}\cup S_{2}=S$ and $E_{1,5}(I),E_{2,5}(I)$ are the decoupling events for each copy. Here due to the independence of the two copies and the decoupling for each copy, we get the same ratio-bound
\begin{eqalign}
&\Proba{\bigcap_{k\in S_{1}}\set{\frac{Q_{1,H}(\wt{M}_{1,n_{k}}J_{1}^{n_{k}})}{Q_{1,H}(\wt{M}_{1,n_{k}}J_{2}^{n_{k}})}>\frac{c_{n_{k},\ell_{k} }^{-\rho_{k}}}{2x_{k}}},E_{1,5}(I), \bigcap_{k\in S_{2}}\set{\frac{Q_{2,H}(\wt{M}_{2,m_{k}}J_{1}^{m_{k}})}{Q_{2,H}(\wt{M}_{2,m_{k}}J_{2}^{m_{k}})}>\frac{c_{r_{k},\ell_{k} }^{-\rho_{k}}}{2x_{k}}},E_{2,5}(I)}\\
\leq &c^{\abs{S}} \prod_{k\in S_{1}} \para{x_{k}c_{n_k,\ell_{k}}^{\rho_{k}}}^{q_{k}} \para{\frac{\abs{J_{1}^{n_k}} }{\delta_{n_k}}}^{-\e_{ratio}(q_{k})}\prod_{k\in S_{2}} \para{x_{k}c_{r_{k},\ell_{k}}^{\rho_{k}}}^{q_{k}} \para{\frac{\abs{J_{1}^{r_{k}}} }{\delta_{r_{k}}}}^{-\e_{ratio}(q_{k})},  
\end{eqalign}
where by summing each factor individually we again get the bound $ c^{\abs{S}} \prod_{k\in S} x_{k}^{q_{k}}$.
\end{proofs}
\newpage\section{The overlap event}\label{sec:overlapevent000}
Here we study the deviation term for overlapping annuli.
\subsection{Overlapping annuli}\label{sec:overlappedannulisection}
Here we extract the annuli that overlap. We fix two sets $A,B\subset [N]$ with $\abs{A},\abs{B}\geq c_{*}N$ that satisfy the Lehto events for each copy $E^{1,2}_{Lehto}(A,B)$ and for the same proportionality constant as in \cref{eq:cstarconstant0}
\begin{eqalign}
c_{*} :=\frac{\min(\e_{ratio}-\e,1)}{2}. 
\end{eqalign}
Recall the overlapped annuli event $O_{n,m}^{0}$ from \cref{eq:overlappedannulievent000}. We will obtain a decay on the following event
\begin{eqalign}\label{eq:mainprobabilitytermtotaloverlap001}
O_{overlap}^{c}:=\bigcup_{k,\ell\in [1,D_{N}]}\bigcupls{S\subset A\\ \abs{S}\geq c_{*}N}\bigcap_{n\in S}\bigcap_{m\in I_{n}}\para{O_{n,m}^{0}}^{c}\para{\frac{k}{D_{N}},\frac{\ell}{D_{N}}}, 
\end{eqalign}
for $I_{n}\subseteq B$ defined right below.  The complement of this event implies that there is some $S\subseteq I$ of cardinality $\abs{S}\geq c_{*}N$ so that every top-copy annulus $A_{1,n}(x),n\in S$ intersects with some bottom-copy annulus $A_{2,m}(y)$ for $m\in B$, with centers satisfying $0\leq x-y\leq R_{N} $ as mentioned in \cref{sec:centersalgor}.\\
The problem can arise when $m_{j}=m_{i}$ for $i\neq j$ i.e. when multiple 1-copies $A_{1,n_{i}},A_{1,n_{j}}$ intersect with the \textit{same} 2-copy $A_{2,m_{i}}=A_{2,m_{j}}$. So in the next section we will have to extract a further subset of annuli that don't have overlaps with a common bottom-copy. \\
Due to the independence of the two copies, there is little hope of matching annuli of the same scale i.e. $A_{1,n_{i}}$ and $A_{2,m_{i}}$ overlapping with $n_{i}\approx m_{i}$. The reason being due to the general probabilistic fact that maximum of unbounded \iid variables diverges regardless. So here we make a choice of scales that turns out to accommodate matching of scales $n\leftrightarrow m$. \\
Let
\begin{eqalign}
R_{n}:=n\frac{1+\beta_{1}}{1+\beta_{2}}-\frac{c_{1}}{1+\beta_{2}}\tand L_{n}:= n\frac{1+\beta_{1}}{1+\beta_{2}}-\frac{c_{2}}{1+\beta_{2}}   
\end{eqalign}
for 
\begin{eqalign}
c_{1}:=    \ln\para{1+\frac{\rho_{P}}{\rho_{a}}}+\ln\frac{1}{1-c_{ov}}\tand c_{2}:=\ln\para{\frac{\rho_{b}-\rho_{P}}{\rho_{a}+\rho_{P}}}-\ln(1+c_{ov}).
\end{eqalign}
We use $L_{n}$ to build the following set. Fix $\e_{0}>0$. We split $B=I^{left}_{n}\cup I^{right}_{n}$, where $m\in I^{left}_{n}$ satisfy $m\leq (1-\e_{0})L_{n}$ and $m\in I^{right}_{n}$ satisfy $m> (1+\e_{0})L_{n}$. We pick the largest of the two
\begin{eqalign}\label{eq:epsilonbufferoverlap0}
I_{n}:=\branchmat{I^{left}_{n}&\tcwhen\abs{I^{left}_{n}}\geq\abs{I^{right}_{n}} \\I^{right}_{n}&\tcwhen\abs{I^{left}_{n}}<\abs{I^{right}_{n}} }.    
\end{eqalign}
So for each scale $A_{1,n}$, we will search for candidates $A_{2,m}$ with $m\in I_{n}$.
In the next proposition (combined with \cref{prop:Lehtotermlemmalegsindc} and \cref{lem:uniqueparing}) we will obtain the claimed estimate \cref{Lehtodivergentindc}
\begin{proposition}\label{prop:overlappingannuli}
The complement of the overlap event has the bound
\begin{eqalign}\label{eq:mainprobabilityoverlapeventdecay0}
&\Proba{O_{overlap}^{c},E^{1,2}_{Lehto}}\\
= &\Proba{\bigcup_{k,\ell\in [1,D_{N}]}\bigcupls{S\subset A\\ \abs{S}\geq c_{*}N}\bigcap_{n\in S}\bigcap_{m\in I_{n}}\para{O_{n,m}^{0}}^{c}\para{\frac{k}{D_{N}},\frac{\ell}{D_{N}}},E^{1,2}_{Lehto}}\leq c\rho_{*}^{(1+(\e_{*}-\e))N},\forall N\geq N_{0},
\end{eqalign}
for the same parameters as in \cref{Lehtodivergentindc}.
\end{proposition}
\proofparagraph{Heuristic idea}
The base intervals for $A_{m,M_{2}}(y), m\in B$ are ordered and so the event $\bigcapls{m\in B} O_{n,m}^{c}$ is asking that the base of $A_{n,M_{1}}(x)$ mostly intersects with one the \textit{gap} intervals in-between the $A_{m,M_{2}}(y)$ and with $0,\pm \infty$. In particular, we have the following union event of the right-based and left-base intersections
\begin{eqalign}
&\bigcupls{k=1}^{H}E_{k}^{0,left}\tand \bigcupls{k=1}^{H}E_{k}^{0,right},
\end{eqalign}
for
\begin{eqalign}
E_{k}^{0,right}:=&\set{y+b_{n_{k},M_{2}}\leq x+a_{n,M_{1}}\leq x+b_{n,M_{1}}\leq y+a_{n_{k-1},M_{2}}} \tfor   k=2,...,H-1, \\
E_{H}^{0,right}:=&\set{ x+b_{n,M_{1}}\leq y+a_{n_{H},M_{2}}},   \\
E_{1}^{0,right}:=&\set{y+b_{n_{1},M_{2}}\leq x+a_{n,M_{1}}}, 
\end{eqalign}
and
\begin{eqalign}
E_{k}^{0,left}:=&\set{y-b_{n_{k-1},M_{2}}\leq x-b_{n,M_{1}}\leq x-a_{n,M_{1}}\leq y-a_{n_{k},M_{2}}} \tfor   k=2,...,H-1, \\
E_{H}^{0,left}:=&\set{ y-a_{n_{H},M_{2}}\leq x-b_{n,M_{1}}},   \\
E_{1}^{0,left}:=&\set{x-a_{n,M_{1}}\leq y-b_{n_{1},M_{2}}}. 
\end{eqalign}
As mentioned before, we replace the $x,y$ by using $0\leq x-y\leq R_{N}$, that means that we have the events
\begin{eqalign}
E_{k}^{right}:=&\set{b_{n_{k},M_{2}}-R_{N}\leq a_{n,M_{1}}\leq b_{n,M_{1}}\leq a_{n_{k-1},M_{2}}} \tfor   k=2,...,H-1, \\
E_{H}^{right}:=&\set{ b_{n,M_{1}}\leq a_{n_{H},M_{2}}},   \\
E_{1}^{right}:=&\set{b_{n_{1},M_{2}}-R_{N}\leq a_{n,M_{1}}}. 
\end{eqalign}
These are ordered and so the complement event will be that the bases of $A_{n,M_{1}}(x)$ fall in one of the bases of $A_{m,M_{2}}(y)$. However because we also need an overlap $P_{n,M_{1},m,M_{2}}=\maxp{P_{n,M_{1}},P_{m,M_{2}}}$, we lose the ordering of the endpoints when $P_{n,M_{1}}>P_{m,M_{2}}$, and so we have to carefully split cases.
\begin{proofs}[Proof of \cref{prop:overlappingannuli}]
Here we start the proof.    
\proofparagraph{Overlap event $O_{n,m}$}
We prove the following lemma.
\begin{lemma}\label{lem:intermediateestimate}
Consider the parameters
\begin{eqalign}
R_{1,n,m}:=&\frac{b_{m}-M_{2,m}^{-1}R_{N}}{b_{n}},L_{1,n,m}:=\frac{M_{1,n}}{M_{2,m}},\\
 R_{2,n,m}:=&\frac{b_{m}-M_{2,m}^{-1}R_{N}}{a_{n}+P_{n}}, L_{2,n,m}:=\frac{M_{1,n}}{M_{2,m}},
\end{eqalign}
and the following events
\begin{eqalign}\label{eq:EnvwiththeratioM1overM2}
E_{n,v}:=&E_{n,v}^{2}=  \set{R_{1,n,v}<\frac{M_{1,n}}{M_{2,v}}<L_{2,n,v}},\tfor v\in I_{n}\setminus\set{m_{1},m_{\abs{I_{n}}}} ,   \\
E_{n,m_{1}}:=&\set{\frac{M_{1,n}}{M_{2,m_{\abs{I_{n}}}\in I_{n}}}<L_{2,n,m_{\abs{I_{n}}}\in I_{n}}},   \\
E_{n,m_{\abs{I_{n}}}}:=& \set{R_{1,n,m_{1}}<\frac{M_{1,n}}{M_{2,m_{1}}}}.  
\end{eqalign}
Then 
\begin{eqalign}\label{eq:mainprobabilitytermoverlap0}
\Proba{O_{overlap}^{c},E^{1,2}_{Lehto}}&\leq \Proba{\bigcup_{k,\ell\in [1,D_{N}]}\bigcupls{S\subseteq A\\ \abs{S}\geq c_{*}N}\bigcapls{n\in S} \bigcupls{v\in I_{n}}E_{n,v} },
\end{eqalign}
where for ease of notation we didn't explicitly write the events dependence on the centers.
\end{lemma}
\begin{proof}
We start from the desired event $O_{n,m}^{0}=O_{n,m,right}^{0}\cap O_{n,m,left}^{0}$ from \cref{eq:overlapeventleftrightbaseint} of the annuli $A_{n,M_{1}}(x)$ $A_{m,M_{2}}(y)$ intersecting on both their right base and left base. As mentioned in \cref{sec:centersalgor} too, we will assume 
\begin{eqalign}
0\leq x-y\leq R_{N}    
\end{eqalign}
and so for simplicity we will study a subset event of $O_{n,m}^{0}$
\begin{eqalign}\label{eq:overlapannulieventsplit0}
O_{n,m}^{0}\supseteq O_{n,m}:=&\set{b_{m,M_{2}}-R_{N}> b_{n,M_{1}}> a_{m,M_{2}} +P_{n,M_{1},m,M_{2}}+R_{N},   a_{m,M_{2}}> a_{n,M_{1}}}\\
&\cup\set{b_{m,M_{2}}-R_{N}> b_{n,M_{1}}> a_{n,M_{1}} +P_{n,M_{1},m,M_{2}}+R_{N},   a_{n,M_{1}}> a_{m,M_{2}}}\\
&\cup\set{b_{n,M_{1}}-R_{N}> b_{m,M_{2}}> a_{m,M_{2}} +P_{n,M_{1},m,M_{2}}+R_{N},   a_{m,M_{2}}> a_{n,M_{1}}}\\
&\cup\set{b_{n,M_{1}}-R_{N}> b_{m,M_{2}}> a_{n,M_{1}} +P_{n,M_{1},m,M_{2}}+R_{N},   a_{n,M_{1}}> a_{m,M_{2}}}.
\end{eqalign}
We will also write $O_{n,m}^{0}(x,y),O_{n,m}(x,y)$ to denote the dependence on centers $x,y$.  Therefore, we will get summability for the complement of $O_{n,m}$
\begin{eqalign}\label{eq:mainprobabilityeventoverlaptoRN0}
\Proba{O_{overlap}^{c},E^{1,2}_{Lehto}}=&\Proba{\para{\bigcap_{k,\ell\in [1,D_{N}]}\bigcupls{S\subset A\\ \abs{S}\geq (1-c_{*})N}\bigcap_{n\in S}\bigcup_{m\in I_{n}}O_{n,m}^{0}\para{\frac{k}{D_{N}},\frac{\ell}{D_{N}}}}^{c}}    \\
\leq &\Proba{\bigcup_{k,\ell\in [1,D_{N}]}\bigcupls{S\subset A\\ \abs{S}\geq c_{*}N}\bigcap_{n\in S}\bigcap_{m\in I_{n}}O_{n,m}^{c}\para{\frac{k}{D_{N}},\frac{\ell}{D_{N}}}}.
\end{eqalign}
We cannot have both $b_{m,M_{2}}> b_{n,M_{1}}$ and $a_{n,M_{1}}>a_{m,M_{2}}$ or both $ b_{n,M_{1}}>b_{m,M_{2}}$ and $a_{m,M_{2}}>a_{n,M_{1}}$ because they imply
\begin{eqalign}\label{eq:continuousrandomvariables0}
&\rho_{*}^{m-n}>\frac{M_{1,n}}{M_{2,m}}>\rho_{*}^{m-n}.
\end{eqalign}
So we are left with
\begin{eqalign}\label{eq:overlapannulieventsplit01}
 O_{n,m}=&\set{b_{m,M_{2}}-R_{N}> b_{n,M_{1}}> a_{m,M_{2}} +P_{n,M_{1},m,M_{2}}+R_{N},   a_{m,M_{2}}> a_{n,M_{1}}}\\
&\cup\set{b_{n,M_{1}}-R_{N}> b_{m,M_{2}}> a_{n,M_{1}} +P_{n,M_{1},m,M_{2}}+R_{N},   a_{n,M_{1}}> a_{m,M_{2}}}.
\end{eqalign}    
We further expand over the cases for $P_{n,M_{1},m,M_{2}}:=\maxp{P_{n,M_{1}},P_{m,M_{2}}}$ depending on who is larger than the other
\begin{eqalign}\label{eq:overlapannulieventsplit1}
 O_{n,m}=&\set{b_{m,M_{2}}-R_{N}> b_{n,M_{1}}> a_{m,M_{2}} +P_{n,M_{1}}+R_{N},   a_{m,M_{2}}> a_{n,M_{1}}, P_{n,M_{1}}> P_{m,M_{2}}}\\
&\cup\set{b_{m,M_{2}}-R_{N}> b_{n,M_{1}}> a_{m,M_{2}} +P_{m,M_{2}}+R_{N},   a_{m,M_{2}}> a_{n,M_{1}}, P_{m,M_{2}}> P_{n,M_{1}}}\\
&\cup\set{b_{n,M_{1}}-R_{N}> b_{m,M_{2}}> a_{n,M_{1}} +P_{n,M_{1}}+R_{N},   a_{n,M_{1}}> a_{m,M_{2}}, P_{n,M_{1}}> P_{m,M_{2}}}\\
&\cup\set{b_{n,M_{1}}-R_{N}> b_{m,M_{2}}> a_{n,M_{1}} +P_{m,M_{2}}+R_{N},   a_{n,M_{1}}> a_{m,M_{2}}, P_{m,M_{2}}> P_{n,M_{1}}}.
\end{eqalign}    
Here again, we get the same issue as in \cref{eq:continuousrandomvariables0} for some cases and so we are left with
\begin{eqalign}\label{eq:overlapannulieventsplit2}
 O_{n,m}=&\set{b_{m,M_{2}}-R_{N}> b_{n,M_{1}}> a_{m,M_{2}} +P_{m,M_{2}}+R_{N},   a_{m,M_{2}}> a_{n,M_{1}}, P_{m,M_{2}}> P_{n,M_{1}}}\\
&\cup\set{b_{n,M_{1}}-R_{N}> b_{m,M_{2}}> a_{n,M_{1}} +P_{n,M_{1}}+R_{N},   a_{n,M_{1}}> a_{m,M_{2}}, P_{n,M_{1}}> P_{m,M_{2}}}.
\end{eqalign}    
We write these in terms of the $M_{1,n},M_{2,m}$ random constants
\begin{eqalign}\label{eq:overlapannulieventsplit3}
 O_{n,m}=&\set{ \minp{\frac{a_{m}}{a_{n}},\frac{P_{m}}{P_{n}},\frac{b_{m}-M_{2,m}^{-1}R_{N}}{b_{n}}} >\frac{M_{1,n}}{M_{2,m}}>  \frac{a_{m}+P_{m}+M_{2,m}^{-1}R_{N}}{b_{n}}}\\
&\cup \set{\minp{\frac{a_{n}}{a_{m}},\frac{P_{n}}{P_{m}},\frac{b_{n}-M_{1,n}^{-1}R_{N}}{b_{m}}}>  \frac{M_{2,m}}{M_{1,n}}>  \frac{a_{n}+P_{n}+M_{1,n}^{-1}R_{N}}{b_{m}}}
\end{eqalign}    
or in-terms of the same ratio $\frac{M_{1,n}}{M_{2,m}}$
\begin{eqalign}\label{eq:overlapannulieventsplit40}
 O_{n,m}=&\set{ \minp{\frac{a_{m}}{a_{n}},\frac{P_{m}}{P_{n}},\frac{b_{m}-M_{2,m}^{-1}R_{N}}{b_{n}}} >\frac{M_{1,n}}{M_{2,m}}>  \frac{a_{m}+P_{m}+M_{2,m}^{-1}R_{N}}{b_{n}}}\\
&\cup \set{\frac{b_{m}-M_{2,m}^{-1}R_{N}}{a_{n}+P_{n}}>\frac{M_{1,n}}{M_{2,m}}>\maxp{\frac{a_{m}}{a_{n}},\frac{P_{m}}{P_{n}},\frac{b_{m}+M_{2,m}^{-1}R_{N}}{b_{n}}}}
\end{eqalign}    
We can further simplify since the ratios of those parameters are $\rho_{*}^{m-n}$ and so we just have
\begin{eqalign}\label{eq:overlapannulieventsplit5}
 O_{n,m}= &\set{ R_{1,n,m}:=\frac{b_{m}-M_{2,m}^{-1}R_{N}}{b_{n}} >\frac{M_{1,n}}{M_{2,m}}>  \frac{a_{m}+P_{m}-M_{2,m}^{-1}R_{N}}{b_{n}}=:L_{1,n,m}}\\
 &\cup \set{R_{2,n,m}:=\frac{b_{m}-M_{2,m}^{-1}R_{N}}{a_{n}+P_{n}}>\frac{M_{1,n}}{M_{2,m}}>\frac{b_{m}+M_{2,m}^{-1}R_{N}}{b_{n}}=:L_{2,n,m}}\\
\end{eqalign}
For the second event to be non-empty, we get constraints on $a+P$ versus $b$ and the bounds for $M_{2,m}^{-1}R_{N}$
\begin{eqalign}
&R_{2,n,m}>L_{2,n,m}    \\
\doncl&\frac{1-\para{b_{m}M_{2,m}}^{-1}R_{N}}{1+\para{b_{m}M_{2,m}}^{-1}R_{N}}>\frac{\rho_{a}+\rho_{P}}{\rho_{b}}.
\end{eqalign}
From \cref{prop:scalescomparedtobN+1}, we have 
\begin{eqalign}
c_{ov}^{-1}R_{N}\leq   a_{m,M_{2}}\doncl   a_{m,M_{2}}^{-1}R_{N}\leq c_{ov}<\frac{1}{2}.
\end{eqalign}
So we request
\begin{eqalign}\label{eq:coverlapconstraint}
\frac{1-c_{ov}\frac{\rho_{a}}{\rho_{b}}}{1+c_{ov}\frac{\rho_{a}}{\rho_{b}}}> \frac{\rho_{a}+\rho_{P}}{\rho_{b}},
\end{eqalign}
which we included in the \nameref{def:exponentialchoiceparamindcop}. So combined with $a_{k}+P_{k}<b_{k}$, we have 
\begin{eqalign}
L_{1,n,m}<R_{1,n,m}<L_{2,n,m} <R_{2,n,m}.   
\end{eqalign}
 So the complement event is
\begin{eqalign}
\para{ O_{n,m}}^{c} =& \set{R_{1,n,m}<\frac{M_{1,n}}{M_{2,m}}\tor \frac{M_{1,n}}{M_{2,m}}<L_{1,n,m}}\cap  \set{R_{2,n,m}<\frac{M_{1,n}}{M_{2,m}}\tor \frac{M_{1,n}}{M_{2,m}}<L_{2,n,m}} \\
 =& \set{R_{2,n,m}<\frac{M_{1,n}}{M_{2,m}}}\cup \set{R_{2,n,m}<\frac{M_{1,n}}{M_{2,m}}<L_{1,n,m}}\\
 &\cup \set{\frac{M_{1,n}}{M_{2,m}}<L_{1,n,m}}\cup \set{R_{1,n,m}<\frac{M_{1,n}}{M_{2,m}}<L_{2,n,m}}\\
 =& \set{\frac{M_{1,n}}{M_{2,m}}<L_{1,n,m}}\cup   \set{R_{1,n,m}<\frac{M_{1,n}}{M_{2,m}}<L_{2,n,m}}\cup \set{R_{2,n,m}<\frac{M_{1,n}}{M_{2,m}}}\\
 =:&E_{n,m}^{1}\cup E_{n,m}^{2}\cup E_{n,m}^{3}.
\end{eqalign}
\proofparagraph{The intersection $\bigcapls{m\in I_{n}} $}
By the distributive law we have
\begin{eqalign}\label{eq:overlapannulieventsplit4}
\bigcapls{n\in S}\bigcapls{m\in I_{n}}  \para{E_{n,m}^{1}\cup E_{n,m}^{2}\cup E_{n,m}^{3}}=& \bigcupls{\vec{s}\in \set{1,2,3}^{\abs{I_{n}}}}\bigcapls{m_{i}\in I_{n}}E_{n,m_{i}}^{s_{i}},
\end{eqalign}   
where by $\vec{s}\in \set{1,2,3}^{\abs{I_{n}}}$ we mean that we consider vectors $(s_{1},...,s_{\abs{I_{n}}})$ with each component being $1,2\tor 3$. By the ordering of the bottom-copy constants $a_{m}M_{2,m},P_{m}M_{2,m},b_{m}M_{2,m}$, we have the following for the intersections involving $s_{i}\in \set{1,3}$
\begin{eqalign}
&\bigcapls{m_{i}\in I_{n}\\ s_{i}=1}\set{\frac{M_{1,n}}{M_{2,m_{i}}}<L_{1,n,m_{i}}}=\set{M_{1,n}<M_{2,v_{1}}L_{1,n,v_{1}}}, \\
\tand &\bigcapls{m_{i}\in I_{n}\\ s_{i}=3}\set{R_{2,n,m_i}<\frac{M_{1,n}}{M_{2,m_i}}}=\set{M_{2,v_{3}}R_{2,n,v_{3}}<M_{1,n}}, \\
\end{eqalign}
where $v_{1},v_{3}\in I_{n}$ 
\begin{eqalign}
v_{1}:=&\max \set{m_{i}\in I_{n}:s_{i}=1} ,   \\
v_{3}:=&\min\set{m_{i}\in I_{n}:s_{i}=3}.
\end{eqalign}
In order to simplify these bounds we need ordering across different scales
\begin{eqalign}
&R_{2,m}<L_{1,m-1} \\
\doncl &\frac{\rho_{b}\rho_{*}}{\rho_{a}+\rho_{P}}\frac{1-\para{b_{m}M_{2,m}}^{-1}R_{N}}{1+\para{(a_{m-1}+P_{m-1})M_{2,m-1}}^{-1}R_{N}}<\frac{\rho_{a}+\rho_{P}}{\rho_{b}}\\
\doncl &\frac{1-\para{b_{m}M_{2,m}}^{-1}R_{N}}{1+\para{(a_{m-1}+P_{m-1})M_{2,m-1}}^{-1}R_{N}}<\para{\frac{\rho_{a}+\rho_{P}}{\rho_{b}}}^{2}\frac{1}{\rho_{*}}.
\end{eqalign}
Because the LHS can at most be $1$ we simply request
\begin{eqalign}
&0<\ln\para{\frac{\rho_{a}+\rho_{P}}{\rho_{b}}}^{2}\frac{1}{\rho_{*}}    \\
\doncl &0<(1-2(r_{a}-r_{b}))\ln\frac{1}{\rho_{*}}+2\ln\para{1+\frac{\rho_{P}}{\rho_{a}}},
\end{eqalign}
where we just need $(1-2(r_{a}-r_{b}))>0$ and we have already included this constraint in \nameref{def:exponentialchoiceparamindcop}. So this implies
\begin{eqalign}
v_{1}<v_{3}.    
\end{eqalign}
For the intersection over second set $A^{2}:=\set{m_{i}\in I_{n}:s_{i}=2}$ , we have a double bound 
\begin{eqalign}
E_{n,a,b}^{2}:=\set{M_{2,a} R_{1,n,a}=\maxls{m\in A^{2}}\para{M_{2,m} R_{1,n,m}}\leq M_{1,n}\leq \minls{m\in A^{2}}\para{M_{2,m} L_{2,n,m}}=M_{2,b} L_{2,n,b}},   
\end{eqalign}
where we wrote that the $\max,\min$ are respectively attained at points $a,b\in A^{2}$. We overall have
\begin{eqalign}
 v_{1}< b\leq a\leq v_{3}.
\end{eqalign}
This implies
\begin{eqalign}
   M_{2,v_{3}}R_{1,n,v_{3}}<M_{2,a} R_{1,n,a} < M_{1,n}<M_{2,b} L_{2,n,b}<M_{2,v_{1}}L_{1,n,v_{1}},
\end{eqalign}
where we again used the ordering $L_{2,m}<R_{2,m}<L_{1,m-1}$. So we can simplify the above intersection just by ranging over a single scale $v$ over the extended index set $I_{n}\cup \set{0,\infty}$
\begin{eqalign}\label{eq:overlapannulieventsplit50}
\eqref{eq:overlapannulieventsplit4}=& \bigcupls{v\in I_{n}}\set{R_{1,n,v}<\frac{M_{1,n}}{M_{2,v}}<L_{2,n,v}}\cup \set{\frac{M_{1,n}}{M_{2,m_{\abs{I_{n}}}\in I_{n}}}<R_{1,n,m_{\abs{I_{n}}}\in I_{n}}}\cup \set{L_{1,n,m_{1}}<\frac{M_{1,n}}{M_{2,m_{1}}}}, 
\end{eqalign}   
where $m_{1},m_{\abs{I_{n}}}\in I_{n}$ are the smallest and largest elements respectively. Here we finally obtained the event expected in the heuristic idea of falling between the \textit{gaps}. We observe that
\begin{eqalign}\label{eq:overlapannulieventsplit6}
\eqref{eq:overlapannulieventsplit50}\subseteq& \bigcupls{v\in I_{n}}E_{n,v},
\end{eqalign}  
for $E_{n,v}$ as defined in the \cref{lem:intermediateestimate}. This inclusion together with \cref{eq:mainprobabilityeventoverlaptoRN0} imply \cref{lem:intermediateestimate}. 
\end{proof}

\proofparagraph{The lower bounds}
We now return to the proof of \cref{prop:overlappingannuli}. The next task is to obtain tail bounds to be able to estimate deviations. For $m\in I_{n}\setminus\set{m_{1},m_{\abs{I_{n}}}} $, we rewrite the first event in \cref{eq:overlapannulieventsplit50} 
\begin{eqalign}
\set{R_{1,n,v}<\frac{M_{1,n}}{M_{2,v}}<L_{2,n,v}},    
\end{eqalign}
in terms of the original fields
\begin{eqalign}
U^{1}_{1,n}(\theta_{b_{n}})+\xi_{1}(\theta_{b_{n}})-U^{1}_{2,m}(\theta_{b_{m}})-\xi_{2}(\theta_{b_{m}})-\ln \frac{\tau_{1}(1)}{\tau_{2}(1)}\geq\ln R_{1,n,m}=\ln\frac{b_{m}\delta_{m}^{\beta_{2}}(1-b_{m,M_{2}}^{-1}R_{N})}{b_{n}\delta_{n}^{\beta_{1}}},    
\end{eqalign}
and
\begin{eqalign}
&U^{1}_{2,m}(\theta_{b_{m}})+\xi_{2}(\theta_{b_{m}})-U^{1}_{1,n}(\theta_{b_{n}})-\xi_{1}(\theta_{b_{n}})+\ln \frac{\tau_{1}(1)}{\tau_{2}(1)}\geq\ln\frac{1}{L_{2,n,m}}=\ln\frac{b_{n}\delta_{n}^{\beta_{1}}}{b_{m}\delta_{m}^{\beta_{2}}\para{1+b_{m,M_{2}}^{-1}R_{N}}}.
\end{eqalign}
For $m=m_{1}$ we only have the first inequality and for $m=m_{\abs{I_{n}}}$ only the second inequality. Using \cref{prop:scalescomparedtobN+1}, we have 
\begin{eqalign}
c_{ov}^{-1}R_{N}\leq   a_{m,M_{2}}\doncl   a_{m,M_{2}}^{-1}R_{N}\leq c_{ov}<\frac{1}{2}
\end{eqalign}
and so we lower bound 
\begin{eqalign}\label{eq:lowerboundsoverlap0}
&\ln\frac{b_{m}\delta_{m}^{\beta_{2}}(1-b_{m,M_{2}}^{-1}R_{N})}{b_{n}\delta_{n}^{\beta_{1}}}\geq \ln\frac{b_{m}\delta_{m}^{\beta_{2}}(1-c_{ov})}{b_{n}\delta_{n}^{\beta_{1}}}  =:B_{n,m}^{1},  \\
\tand &\ln\frac{b_{n}\delta_{n}^{\beta_{1}}}{b_{m}\delta_{m}^{\beta_{2}}\para{1+b_{m,M_{2}}^{-1}R_{N}}}\geq \ln\frac{b_{n}\delta_{n}^{\beta_{1}}}{b_{m}\delta_{m}^{\beta_{2}}\para{1+c_{ov}}}=:B_{n,m}^{2}.
\end{eqalign}
We study the sign of those constants: we have
\begin{eqalign}\label{eq:lowerboundsoverlap1}
B_{n,m}^{1}>0\doncl m(1+\beta_{2})\ln\rho_{*}+\ln(1-c_{ov})>&n(1+\beta_{1})\ln\rho_{*}   \\
\doncl  \para{n(1+\beta_{1})-m(1+\beta_{2})}\ln\frac{1}{\rho_{*}}>&\ln\frac{1}{1-c_{ov}}=c_{1},
\end{eqalign}
and 
\begin{eqalign}\label{eq:lowerboundsoverlap2}
B_{n,m}^{2}>0\doncl n(1+\beta_{1})\ln\rho_{*}>&m(1+\beta_{2})\ln\rho_{*}+\ln(1+c_{ov})  \\
\doncl  c_{2}=\ln\frac{1}{1+c_{ov}}>&\para{n(1+\beta_{1})-m(1+\beta_{2})}\ln\frac{1}{\rho_{*}}.
\end{eqalign}
We have that both constants are nonpositive when
\begin{eqalign}\label{eq:bothnonpositivebounds}
m\geq &n\frac{1+\beta_{1}}{1+\beta_{2}}-\frac{c_{1}}{1+\beta_{2}}\tand     m\leq n\frac{1+\beta_{1}}{1+\beta_{2}}-\frac{c_{2}}{1+\beta_{2}}
\end{eqalign}
As mentioned in \cref{eq:epsilonbufferoverlap0}, we especially selected $I_{n}\subset I$ to avoid this situation.\\
We finally define
\begin{eqalign}
B_{n,m}:= \maxp{B_{n,m}^{1},B_{n,m}^{2}}.
\end{eqalign}
For each given set $S\subset A$, define the sets $S_{1}\cup S_{2}=S$ such that $n\in S_{1}$ when $B_{n,m}=B_{n,m}^{1}$, and $S_{2}:=S\setminus S_{1}$. So if $S=\set{s_{1},...,s_{L}}$, then 
\begin{eqalign}
S_{1}:=\set{s_{i_{*}+1},...,s_{L}}\tand S_{2}:=\set{s_{1},...,s_{i_{*}}},    
\end{eqalign}
$s_{i_{*}}$ is the transition scale between the bounds $B^{1},B^{2}$. For $m\in I_{n},n\in S_{1}$, we have
\begin{eqalign}\label{eq:lowerboundsoverlapsav1}
B_{n,m}^{1}\geq  (n\e_{0}(1+\beta_{1})\ln\frac{1}{\rho_{*}}-(1+\e_{0})c_{1}),  
\end{eqalign}
and  for $m\in I_{n},n\in S_{2}$ we have 
\begin{eqalign}\label{eq:lowerboundsoverlapsav2}
B_{n,m}^{2}\geq (n\e_{0}(1+\beta_{1})\ln\frac{1}{\rho_{*}}+(1-\e_{0})c_{2}).
\end{eqalign}
Therefore, $B_{n,m}\geq n\e_{0}(1+\beta_{1})\ln\frac{1}{\rho_{*}}+\maxp{c_{2},c_{1}}$ in either case for $m\in I_{n}$, where we just absorbed the $(1\pm\e_{0})$ in the constants $c_{1},c_{2}$ for notational ease. 

\proofparagraph{Separating the deviations}
Next we separate the deviations. In order to deal with the $\theta_{b_{n}}$, we will use the supremum estimates by inserting the fields evaluated at zero
\begin{eqalign}\label{eq:mainintervaleventoverlapdouble}
&U^{1}_{1,n}(0)-U^{1}_{2,m}(0)+\para{U^{1}_{1,n}(\theta_{b_{n}})-U^{1}_{1,n}(0)}+\xi_{1}(\theta_{b_{n}})\\
&-\para{U^{1}_{2,m}(\theta_{b_{m}})-U^{1}_{2,m}(0)}-\xi_{2}(\theta_{b_{m}})-\ln \frac{\tau_{1}(1)}{\tau_{2}(1)}\geq B_{n,m}^{1},  \\
\tand &-\para{U^{1}_{1,n}(0)-U^{1}_{2,m}(0)}+\para{U^{1}_{2,m}(\theta_{b_{m}})-U^{1}_{2,m}(0)}+\xi_{2}(\theta_{b_{m}})\\
&-\para{U^{1}_{1,n}(\theta_{b_{n}})-U^{1}_{1,n}(0)}-\xi_{1}(\theta_{b_{n}})+\ln \frac{\tau_{1}(1)}{\tau_{2}(1)}\geq B_{n,m}^{2}.
\end{eqalign}
As we vary $m$, the bounds $B_{n,m}^{1},B_{n,m}^{2}$ switch sign and so the only way to always get a positive lower bound is to use the maximum of them. We have that \cref{eq:mainintervaleventoverlapdouble} is contained in the union of the following events
\begin{eqalign}\label{eq:mainprobabilitytermoverlap001}
 &E^{1}_{n,m}:=\set{ \sup_{s\in [0,b_{n}+d_{n}]}\abs{U^{1}_{1,n}(s)-U^{1}_{1,n}(0)}\geq\lambda_{1}B_{n,m} },\\
 &E^{2}_{n,m}:=\set{\sup_{s\in [0,b_{n}+d_{n}]}\abs{\xi_{1}(s)}\geq \lambda_{2} B_{n,m}},\\
&E^{3}_{n,m}:= \set{\sup_{s\in [0,b_{m}+d_{m}]}\abs{U^{1}_{2,m}(s)-U^{1}_{2,m}(0)}\geq\lambda_{3}B_{n,m}^{i}} \tcwhen n\in S_{i},i=1,2,  \\\\
&E^{4}_{n,m}:=\set{  \sup_{s\in [0,b_{m}+d_{m}]}\abs{\xi_{2}(s)}\geq\lambda_{4}B_{n,m}^{i}} \tcwhen n\in S_{i},i=1,2,  \\\\
&E^{5}_{n,m}:=\branchmat{\set{ U^{1}_{1,n}(0)\geq \lambda_{5} B_{n,m}^{1}}& \tcwhen n\in S_{1}\\\set{  -U^{1}_{1,n}(0)\geq \lambda_{5} B_{n,m}^{2}}& \tcwhen n\in S_{2}\\},  \\\\
&E^{6}_{n,m}:=\branchmat{\set{ -U^{1}_{2,m}(0)\geq \lambda_{6} B_{n,m}^{1}}& \tcwhen n\in S_{1}\\\set{  U^{1}_{2,m}(0)\geq \lambda_{6} B_{n,m}^{2}}& \tcwhen n\in S_{2}\\},  \\\\
\tand &E^{7}_{n,m}:=\set{(-1)^{i}\ln \frac{\tau_{1}(1)}{\tau_{2}(1)}\geq \lambda_{7} B_{n,m}}, \tcwhen m\in S_{i},i=1,2,   \\
\end{eqalign}
for $\lambda_{i}\in (0,1)$ summing to one. Returning to \cref{eq:EnvwiththeratioM1overM2} we have the inclusion
\begin{eqalign}
E_{n,m}\subseteq \bigcup_{i=1}^{7}E_{n,m}^{i},   
\end{eqalign}
and so by splitting the deviations we have
\begin{eqalign}
 \Proba{\bigcup_{k,\ell\in [1,D_{N}]}\bigcupls{S\subseteq A\\ \abs{S}\geq c_{*}N}\bigcapls{n\in S} \bigcupls{v\in I_{n}}E_{n,v} }=&\Proba{\bigcup_{k,\ell\in [1,D_{N}]}\set{\sum_{n\in I}\ind{ \bigcupls{m\in I_{n}}E_{n,m}}\geq c_{*}N} } \\
\leq &\Proba{\bigcup_{k,\ell\in [1,D_{N}]}\set{\sum_{n\in I}\ind{ \bigcupls{m\in I_{n}}\bigcup_{i=1}^{7}E_{n,m}^{i} }\geq c_{*}N} } \\
\leq &\Proba{\bigcup_{k,\ell\in [1,D_{N}]}\set{\sum_{i=1}^{7}\sum_{n\in I}\ind{ \bigcupls{m\in I_{n}}E_{n,m}^{i} }\geq c_{*}N} } \\
\leq &\sum_{i=1,2,5}\Proba{\bigcup_{k\in [1,D_{N}]}\set{\sum_{n\in I}\ind{ \bigcupls{m\in I_{n}}E_{n,m}^{i}(k) }\geq q_{i}c_{*}N} } \\
&+\sum_{i=3,4,6}\Proba{\bigcup_{\ell\in [1,D_{N}]}\set{\sum_{n\in I}\ind{ \bigcupls{m\in I_{n}}E_{n,m}^{i}(\ell) }\geq q_{i}c_{*}N} } \\
&+\Proba{\sum_{n\in I}\ind{ \bigcupls{m\in I_{n}}E_{n,m}^{7} }\geq q_{7}c_{*}N} ,
\end{eqalign}
for $q_{i}\in (0,1)$ that sum to one $\sum_{i=1}^{7}q_{i}=1$ and we also denoted the dependence on the centers in $E_{n,m}^{i}(k),E_{n,m}^{i}(\ell)$. In particular,we take $\lambda_{i},q_{i},i=1,2,3,4$ arbitrarily small. And so we are left $\lambda_{5}+\lambda_{6}=1-\e=q_{5}+q_{6}$ for arbitrarily small $\e>0$. So we simply set them and $\e_{0}$ equal to $\lambda_{5}=\lambda_{6}=\frac{1-\e}{2}$. Next we estimate each of those deviations.
\proofparagraph{Deviation for $E^{1}$}
Here we study
\begin{eqalign}\label{eq:u1ndecoupledestimateindc0}
&\Proba{\sum_{n\in I}\ind{\bigcup_{m\in I_{n}} \set{\sup_{s\in [0,b_{n}+d_{n}]}\abs{U^{1}_{1,n}(s)-U^{1}_{1,n}(0)}\geq\lambda_{1} B_{n,m}}}\geq q_{1}(c_{*}N)}\\
=&\Proba{\sum_{n\in I}\ind{\sup_{s\in [0,b_{n}+d_{n}]}\abs{U^{1}_{1,n}(s)-U^{1}_{1,n}(0)}\geq \lambda_{1} u_{n}}\geq q_{1}(c_{*}N)},
\end{eqalign}
where $u_{n}:=n\e_{0}(1+\beta_{1})\ln\frac{1}{\rho_{*}}+\maxp{c_{2},c_{1}}$. We are back to \cref{prop:upperscalesdeviation} with the added difference that $u_{n}$ is actually growing. So we again get a bound $\rho_{*}^{BN}$ for arbitrarily large $B$ and in turn $q_{1}$ can be arbitrarily small.

\proofparagraph{Deviation for $E^{2}$}
Here we repeat as in \cref{sec:comparisonfieldxi} and above to estimate by
\begin{eqalign}\label{eq:u1ndecoupledestimateindc2}
&\Proba{\sum_{n\in I}\ind{\set{\sup_{s\in [0,b_{n}+d_{n}]}\abs{\xi_{1}(s)}\geq\lambda_{2} u_{n}}}\geq q_{2}(c_{*}N)}\\
\leq &\sumls{S\subset A\\ \abs{S}\geq q_{2}(c_{*}N)}\Proba{\supl{t\in [0,b_{m_{S}}+d_{m_{S}}]}\abs{\xi(t)}\geq \frac{1}{\gamma}\lambda_{2}u_{m_{S}}},
\end{eqalign}
for $u_{m_{S}}:=\max_{n\in S}u_{n}$. We use \cref{thm:suprconcen} for $D_{n}:=c\para{b_{n}+d_{n}}^{2/3}$ (as computed in \cref{prop:xicovariancebound}) to get
\begin{eqalign}\label{eqtaile2}
&\Proba{\supl{t\in [0,b_{m_{S}}+d_{m_{S}}]}\abs{\xi(t)}\geq \frac{1}{\gamma}\lambda_{2}u_{m_{S}}}\\
\leq &\Proba{ \sup_{s\in [0,b_{m_{S}}+d_{m_{S}}]}\abs{\xi(s)-\xi(0)}\geq \frac{1}{\gamma} r_{0} \lambda_{2}u_{m_{S}}  }+\Proba{\abs{\xi(0)}\geq  (1-r_{0})\frac{1}{\gamma} \lambda_{2}u_{m_{S}} }\\
\leq &c\expo{-\frac{\para{\lambda_{2}u_{m_{S}}}^{2}}{2\gamma^{2}\para{b_{m_{S}}+d_{m_{S}}}^{\frac{4}{3}}}}+e^{-R  (1-\e_{0})\frac{1}{\gamma}\lambda_{2}u_{m_{S}}}\Expe{e^{R \abs{\xi_{1}(0)}}},
\end{eqalign}
for $r_{0}\in (0,1)$ and arbitrarily large $R>0$ since $\xi$ has all its exponential moments. The first bound has a doubly exponential decay as in \cref{eq:maindeviationxifield3}, so we just bound by the second term which is of the form $\rho_{*}^{BN}$ for arbitrarily large $B$. So $q_{2}$ can be arbitrarily small.

\proofparagraph{Deviation for $E^{3}$}
Here we study
\begin{eqalign}\label{eq:u1ndecoupledestimateindcE30}
&\Proba{\sum_{n\in I}\ind{\bigcup_{m\in I_{n}} \set{\sup_{s\in [0,b_{m}+d_{m}]}\abs{U^{1}_{2,m}(s)-U^{1}_{2,m}(0)}\geq\lambda_{3} B_{n,m}}}\geq q_{3}(c_{*}N)}.
\end{eqalign}
Next we apply distributive law. Let 
\begin{eqalign}
E_{m}(n):=\set{\sup_{s\in [0,b_{m}+d_{m}]}\abs{U^{1}_{2,m}(s)-U^{1}_{2,m}(0)}\geq\lambda_{3} B_{n,m}}.    
\end{eqalign}
We apply distributive law to bring in the intersection
\begin{eqalign}\label{eq:distriblawe6event0}
\bigcap_{n\in S}\bigcup_{m_{k}\in I_{n}}E_{m_{k}}(n)=\bigcup_{\phi}\bigcap_{n\in S}E_{\phi(n)}(n),   
\end{eqalign}
where the union is over all functions $\phi: S\to I_{s_{L}}$ that map $\phi(n)\in I_{n}$. By grouping the duplicates/non-injectivity-points we write
\begin{eqalign}\label{eq:distriblawe6event01}
\eqref{eq:distriblawe6event0}=\bigcupls{O=\set{m_{i}}\subset I_{s_{L}}\\ O\cap I_{s_{i}}\neq \varnothing}\bigcupls{\phi_{O}}\bigcap_{n\in \phi_{O}^{-1}(m_{1})}E_{m_{1}}(n)\cap\cdots \cap \bigcap_{n\in \phi_{O}^{-1}(m_{\abs{O}})}E_{m_{\abs{O}}}(n),   
\end{eqalign}
where for each subset $O\subset I_{s_{L}}$ that intersects all of them $O\cap I_{s_{i}}\neq \varnothing$, we consider the functions $\phi_{O}:S\to O$ that now map $\phi_{O}(n)\in I_{n}\cap O$ and by the inverse notation we mean
\begin{eqalign}
n\in \phi_{O}^{-1}(m_{i}) \doncl m_{i}\in I_{n}.   
\end{eqalign}
Now that we fixed the scale of the field, we can simplify the intersection-events over $A_{i}:=\phi_{O}^{-1}(m_{i})$. In particular over general set $S$ we have
\begin{eqalign}
\bigcap_{n\in S}E_{m_{i}}(n)=\set{\sup_{s\in [0,b_{m_{i}}+d_{m_{i}}]}\abs{U^{1}_{2,m_{i}}(s)-U^{1}_{2,m_{i}}(0)}\geq \lambda_{3} \maxls{n\in S}B_{n,m_{i}}=\lambda_{3} B_{\max S,m_{i}}}=E_{m_{i}}(\max S ),    
\end{eqalign}
and so we have $\bigcap_{n\in A_{i}}E_{m_{i}}(n)=E_{m_{i}}(\max A_{i})$. Therefore, we can simplify the union to be over "maximizers"
\begin{eqalign}\label{eq:distriblawe6event02}
\eqref{eq:distriblawe6event01}=\bigcupls{O=\set{m_{i}}\subset I_{s_{L}}\\ O\cap I_{s_{i}}\neq \varnothing}\bigcupls{A=\set{(v_{i})_{i=1}^{\abs{O}-1},s_{L}}\subseteq S\\ \abs{A}=\abs{O}\\ m_{i}\in I_{v_i}}\bigcap_{i\in [1,\abs{O}]}E_{m_{i}}(v_{i}),   
\end{eqalign}
where we included the maximal element $v_{\abs{O}}:=s_{L}$ in each set $A:=\set{v_{1},...,v_{\abs{O}-1},s_{L} }$ because since we are taking maximum, it will always be present; the rest of the elements in $A$ can repeat and be generic as long as they satisfy $m_{i}\in I_{v_i}$. So once again we will use the maximal element at $i=\abs{O}$ and repeat as in \cref{eq:truncatedsupgaussiantailinc} to bound by
\begin{eqalign}\label{eq:probE3}
\Proba{\sup_{s\in [0,b_{m_{\abs{O}}}+d_{m_{\abs{O}}}]}\abs{U^{1}_{2,m_{\abs{O}}}(s)-U^{1}_{2,m_{\abs{O}}}(0)}\geq \lambda_{3} B_{s_{L},m_{\abs{O}}}}\leq c\expo{-\frac{c}{\rho_{*}}(N\ln(\frac{1}{\rho_{*}}))^{2}},    
\end{eqalign}
where we again used that the maximal element satisfies $s_{L}\geq q_{3}(c_{*}N)$. So we get decay $\rho_{*}^{BN}$ for arbitrarily large $B$ (by making $\rho_{*}$ smaller). So we can take $q_{3}$ arbitrarily small.
\proofparagraph{Deviation for $E^{4}$}
Here we study
\begin{eqalign}\label{eq:u1ndecoupledestimateE4indc2}
&\Proba{\sum_{n\in I}\ind{\bigcup_{m\in I_{n}}\set{\sup_{s\in [0,b_{m}+d_{m}]}\abs{\xi_{2}(s)}\geq\lambda_{4} B_{n,m}}}\geq q_{4}(c_{*}N)}.
\end{eqalign}
We repeat as for \cref{eq:u1ndecoupledestimateindcE30} to get to the analogous \cref{eq:probE3}
\begin{eqalign}\label{eq:probE4}
&\Proba{\sup_{s\in [0,b_{m_{\abs{O}}}+d_{m_{\abs{O}}}]}\abs{\xi_{2}(s)}\geq \lambda_{4} B_{s_{L},m_{\abs{O}}}}\leq c\expo{-\frac{c}{\rho_{*}}(N\ln(\frac{1}{\rho_{*}}))^{2}}.
\end{eqalign}
Here we are in the same situation as in \cref{eqtaile2} and so we again get a bound of the form $\rho_{*}^{BN}$ for arbitrarily large $B$. So we can take $q_{4}$ arbitrarily small.


\proofparagraph{Deviation for $E^{5}$}
Here we study
\begin{eqalign}\label{eq:u1ndecoupledestimateindcE50}
&\Proba{\sum_{n\in I}\ind{\bigcup_{n_{k}\in I_{n}}E_{5}}\geq q_{5}(c_{*}N)}\\
\leq &\sumls{S\subset A\\ \abs{S}\geq q_{5}(c_{*}N)}\Proba{\bigcap_{n\in S_{2}}\set{ -U^{1}_{1,n}(0)\geq \lambda_{5} \minls{n_{k}\in I_{n}}B_{n,m}^{2} }\cap \bigcap_{n\in S_{1}}\set{ U^{1}_{1,n}(0)\geq \lambda_{5} \minls{n_{k}\in I_{n}}B_{n,n_{k}}^{1}}}.
\end{eqalign}
Since the Gaussian tail for the largest scale $s_{L}$ is
\begin{eqalign} \label{eq:Gaussiantailbounde50}
 &\Proba{U^{1}_{1, s_{L}}(0) \geq \lambda_{5}\para{s_{L}\e_{0}(1+\beta_{1})\ln\frac{1}{\rho_{*}}+\maxp{c_{2},c_{1}}}}\\
 \leq& c\expo{-\frac{1}{2\beta_{1}} \lambda_{5}^{2}\para{\e_{0}(1+\beta_{1})+\para{s_{L}\ln\frac{1}{\rho_{*}}}^{-1}\maxp{c_{2},c_{1}}}^{2}s_{L}\ln\frac{1}{\rho_{*}}}\\
 \leq& c \rho_{*}^{\frac{\para{\lambda_{5}\e_{0}(1+\beta_{1})}^{2}}{2\beta_{1}}q_{5}(c_{*}N)},
\end{eqalign}
where we used that $s_{L}$ is the largest element in $S$ with $\abs{S}\geq q_{5}(c_{*}N)$.
\proofparagraph{Deviation for $E^{6}$}
Here we study
\begin{eqalign}\label{eq:u1ndecoupledestimateindcE60}
&\Proba{\sum_{n\in I}\ind{\bigcup_{n_{k}\in I_{n}}E_{6}}\geq q_{6}(c_{*}N)}\\
\leq &\sumls{S\subset A\\ \abs{S}\geq q_{6}(c_{*}N)}\Proba{\bigcap_{n\in S_{1}}\bigcup_{n_{k}\in I_{n}}\set{ -U^{1}_{2,n_{k}}(0)\geq \lambda_{6} B_{n,n_{k}}^{1}}\cap \bigcap_{n\in S_{2}}\bigcup_{n_{k}\in I_{n}}\set{ U^{1}_{2,m}(0)\geq \lambda_{6} B_{n,m}^{2} }}.
\end{eqalign}

 As before, we use the Gaussian tail estimate for the largest scale $s_{L}$ to get the same bound as in \cref{eq:Gaussiantailbounde50} but with $\lambda_{6}$ instead of $\lambda_{5}$.
\proofparagraph{Deviation for $E^{7}$}
Here we repeat as in \cref{eq:u1ndecoupledestimateE4indc2} above to estimate by
\begin{eqalign}\label{eq:u1ndecoupledestimateE7indc2}
&\Proba{\sum_{n\in I}\ind{\set{\abs{\ln \frac{\tau_{1}(1)}{\tau_{2}(1)}}\geq\lambda_{2} u_{n}}}\geq q_{7}(c_{*}N)}\\
\leq &\sumls{S\subset A\\ \abs{S}\geq q_{7}(c_{*}N)}\Proba{\ln \frac{\tau_{1}(1)}{\tau_{2}(1)}\geq \lambda_{7}u_{m_{S}}}+\Proba{-\ln \frac{\tau_{1}(1)}{\tau_{2}(1)}\geq \lambda_{7}u_{m_{S}}},
\end{eqalign}
for $u_{m_{S}}:=\max_{n\in S}u_{n}$. Applying Markov inequality for $p\in \para{0,\minp{\frac{1}{\beta_{1}},\frac{1}{\beta_{2}}}}$, we get
\begin{eqalign}
\Proba{\ln \frac{\tau_{1}(1)}{\tau_{2}(1)}\geq \lambda_{7}u_{m_{S}}}\leq& e^{-p\lambda_{7}u_{m_{S}}}\Expe{\para{ \frac{\tau_{1}(1)}{\tau_{2}(1)}}^{p}}\\
\leq &c \expo{-(\beta_{1,2}^{-1}-\e)\lambda_{7}m_{S}\e_{0}(1+\beta_{1})\ln\frac{1}{\rho_{*}}}    \\
\leq &c \expo{-(\beta_{1,2}^{-1}-\e)\lambda_{7}\e_{0}(1+\beta_{1})\ln\frac{1}{\rho_{*}}q_{7}(c_{*}N)}    
\end{eqalign}
for $\beta_{1,2}^{-1}:=\minp{\frac{1}{\beta_{1}},\frac{1}{\beta_{2}}}$. We note that the total masses don't depend on the center and so we can just insert this deviation without having to multiply by $D_{N}$ in the Borel-Cantelli estimate.

\proofparagraph{Computation of bound for Borel-Cantelli}
In summary, since for the deviations $E^{1},...,E^{4}$ we can be taken arbitrarily small, we just need the summability for the last three deviations
\begin{eqalign}\label{eq:overlapeventBorelCantellisummable}
&\sum_{N\geq N_{0}}D_{N}\rho_{*}^{\frac{\para{\lambda_{5}\e_{0}(1+\beta_{1})}^{2}}{2\beta_{1}}q_{5}(c_{*}N)}\\
+&\sum_{N\geq N_{0}}D_{N}\rho_{*}^{\frac{\para{\lambda_{6}\e_{0}(1+\beta_{1})}^{2}}{2\beta_{1}}q_{5}(c_{*}N)}\\
+&\sum_{N\geq N_{0}}\expo{-(\beta_{1,2}^{-1}-\e)\lambda_{7}\e_{0}(1+\beta_{1})\ln\frac{1}{\rho_{*}}q_{7}(c_{*}N)},   
\end{eqalign}
for $D_{N}=\ceil{\rho_{*}^{-(1+\e_{**}(1-\lambda))N}}$ for some $\lambda\in (0,1)$ and $\e_{**}>0$ from \cref{Lehtodivergent}. The finiteness of the third sum is immediate since the exponent is strictly positive and so we take $\lambda_{7},q_{7}$ arbitrarily small. We also take $\lambda_{i},q_{i},i=1,2,3,4$ arbitrarily small. So we are left $\lambda_{5}+\lambda_{6}=1-\e=q_{5}+q_{6}$ for arbitrarily small $\e>0$. By setting them and $\e_{0}$ equal to $\frac{1-\e}{2}$, we require
\begin{eqalign}\label{eq:overlappednannuliconstraint}
\frac{\para{\lambda_{6}\e_{0}(1+\beta_{1})}^{2}}{2\beta_{1}}q_{5}c_{*}=\frac{1}{64\beta_{1}}(1+\beta_{1})^{2} \frac{\e_{ratio}-\e}{2}>1+\e_{**}(1-\lambda).   
\end{eqalign}
We check the compatibility in \cref{eq:overlappednannuliconstraintindcp} in
 \nameref{def:exponentialchoiceparam}. This finishes the proof of \cref{prop:overlappingannuli}.
\end{proofs}

\newpage\subsection{Unique pairs}\label{sec:uniquepairingforoverlap}
In the previous section we obtained a subset $S$ of size $\abs{S}\geq c_{*}N$ so that each of the top-copy annuli $A_{n_{j},M_{1}}(x),n_{j}\in S$ P-overlaps with some bottom-copy annuli $A_{m_{j},M_{2}}(y)$. As mentioned before, the issue is that $m_{j}=m_{i}$ for $i\neq j$. In this section we  have to extract a further subset $S'\subset S$ of annuli that don't have overlaps with a common bottom-copy and is still large enough $\abs{S'}\geq c_{5}N$, for the $c_{5}$ from \cref{prop:truncateddeviation}. 
\begin{lemma}\label{lem:uniqueparing}
With probability at least $1-c\rho_{*}^{(1+(\e_{*}-\e))N},\forall N\geq N_{0},$ there exists a subset $S'\subset S$, $\abs{S'}\geq c_{5}N$, of top-copy annuli that don't have overlaps with a common bottom-copy.    
\end{lemma}
\begin{proof}
We will use the independence number approach as done in \cref{sec:trunanddisjointannuli}. We say that two nodes $v_{i},v_{j}$ connect when the overlap events $O_{n_{i},m}\cap O_{n_{j},m}$, for $n_{i}\neq n_{j}$, from \cref{eq:overlapannulieventsplit2} and \cref{eq:overlapannulieventsplit50} are true with some common bottom-copy
\begin{eqalign}
 \bigcupls{m\in I}& O_{n_{i},m}\cap O_{n_{j},m}  \\
=\bigcupls{m\in I}&\Big(\set{b_{m,M_{2}}-R_{N}> b_{n_{i},M_{1}}> a_{m,M_{2}} +P_{m,M_{2}}+R_{N},   a_{m,M_{2}}> a_{n_{i},M_{1}}, P_{m,M_{2}}> P_{n_{i},M_{1}}}\\
&\cup\set{b_{n_{i},M_{1}}-R_{N}> b_{m,M_{2}}> a_{n_{i},M_{1}} +P_{n_{i},M_{1}}+R_{N},   a_{n_{i},M_{1}}> a_{m,M_{2}}, P_{n_{i},M_{1}}> P_{m,M_{2}}}\Big)\\
\cap&\Big(\set{b_{m,M_{2}}-R_{N}> b_{n_{j},M_{1}}> a_{m,M_{2}} +P_{m,M_{2}}+R_{N},   a_{m,M_{2}}> a_{n_{j},M_{1}}, P_{m,M_{2}}> P_{n_{j},M_{1}}}\\
&\cup\set{b_{n_{j},M_{1}}-R_{N}> b_{m,M_{2}}> a_{n_{j},M_{1}} +P_{n_{j},M_{1}}+R_{N},   a_{n_{j},M_{1}}> a_{m,M_{2}}, P_{n_{j},M_{1}}> P_{m,M_{2}}}\Big)\\
=\bigcupls{m\in I}&\para{\set{R_{1,n_{i},m}>\frac{M_{1,n_{i}}}{M_{2,m}}>L_{1,n_{i},m}}\cup \set{R_{2,n_{i},m}>\frac{M_{1,n_{i}}}{M_{2,m}}>L_{2,n_{i},m}}}\\
&\cap \para{\set{R_{1,n_{j},m}>\frac{M_{1,n_{j}}}{M_{2,m}}>L_{1,n_{j},m}}\cup \set{R_{2,n_{j},m}>\frac{M_{1,n_{j}}}{M_{2,m}}>L_{2,n_{j},m}}}.
\end{eqalign}
So from here we get the following possibilities
\begin{eqalign}\label{eq:mathparis0}
 \bigcupls{ r_{1},r_{2}\in \set{1,2}} \set{\frac{L_{r_{1},n_{i},m}}{R_{r_{2},n_{j},m}}<\frac{M_{1,n_{i}}}{M_{1,n_{j}}} < \frac{R_{r_{1},n_{i},m}}{L_{r_{2},n_{j},m}}}.
\end{eqalign}
If $n_{i}>n_{j}$, we get that the weakest lower bound is
\begin{eqalign}\label{eq:matchuniquepairstop-copy}
&\ln\frac{M_{1,n_{i}}}{M_{1,n_{j}}} =\xi_{1}(\theta_{b_{n_{i}}})  -\xi_{1}(\theta_{b_{n_{j}}})+U^{1}_{1,n_{j}}(\theta_{b_{n_{i}}})  -U^{1}_{1,n_{j}}(\theta_{b_{n_{j}}})  +U^{n_{j}}_{1,n_{i}}(\theta_{b_{n_{i}}}) \\
>&\ln\frac{L_{1,n_{i},m}}{R_{2,n_{j},m}}+\beta_{1}\ln\frac{\delta_{n_{j}}}{\delta_{n_{i}}}\\
&=\ln \frac{a_{m}+P_{m}-M_{2,m}^{-1}R_{N}}{b_{m}-M_{2,m}^{-1}R_{N}}\frac{a_{n}+P_{n}}{b_{n}}+\beta_{1}\ln\frac{\delta_{n_{j}}}{\delta_{n_{i}}}\\
\geq & (n_{i}-n_{j})(1+\beta_{1})\ln\frac{1}{\rho_{*}}-2(r_{a}-r_{b})\ln\frac{1}{\rho_{*}}+\ln\para{1+\frac{\rho_{P}}{\rho_{a}}-c_{ov}}.
\end{eqalign}
If we have the reverse $n_{j}>n_{i}$, then we the get the same lower bound by using that the largest upper bound in \cref{eq:mathparis0} is $\frac{R_{2,n_{i},m}}{L_{1,n_{j},m}}$. Next, as done in \cref{eq:disjointrandomannuli2}, due to the deviation estimates for the differences we obtain the upper bound 
\begin{eqalign}
  \abs{\xi_{1}(\theta_{b_{n_{i}}})  -\xi_{1}(\theta_{b_{n_{j}}})+U^{1}_{1,n_{j}}(\theta_{b_{n_{i}}})  -U^{1}_{1,n_{j}}(\theta_{b_{n_{j}}})}\leq c\rho_{*}  
\end{eqalign}
and so we have the lower bound
\begin{eqalign}
U^{n_{j}}_{1,n_{i}}(\theta_{b_{n_{i}}})\geq  \para{(n_{i}-n_{j})(1+\beta_{1})-2(r_{a}-r_{b})+\para{\ln\frac{1}{\rho_{*}}}^{-1}\para{\ln\para{1+\frac{\rho_{P}}{\rho_{a}}-c_{ov}}-c\rho_{*}}}\ln\frac{1}{\rho_{*}}.   
\end{eqalign}
Here we get the same lower bound as the one from \cref{sec:trunanddisjointannuli}: $u_{k,m}:=(m-k)r_{u,m,k}\ln\frac{1}{\rho_{*}}$, $r_{u,m,k}:=\beta+1-2\frac{r_{a}-r_{b}}{m-k}>0$ and $m>k$. We ignored the constant term $\para{\ln\frac{1}{\rho_{*}}}^{-1}\para{\ln\para{1+\frac{\rho_{P}}{\rho_{a}}-c_{ov}}-c\rho_{*}}$ because it is small by taking $\rho_{*}$ small and so it doesn't affect the estimates in \cref{sec:trunanddisjointannuli}. Therefore, we indeed get some $c_{5}N$ of unique pairings.
\end{proof}

\newpage\part{Open problems and Appendices}\label{part:researpappend}
\section{Further research directions  }\label{furtherresearchdirections}

\begin{enumerate}

\item\textbf{Extending to all $\gamma\in  (0,\sqrt{2})$}
In \nameref{def:exponentialchoiceparam}, we go over the various constraints for $\beta=\frac{\gamma^{2}}{2}$ and their sources in the article. The biggest sources of reduction were the inverse-ratio moments \cref{prop:Multipointunitcircleandmaximum} in \cite{binder2023inverse} and the \cref{prop:scalescomparedtobN+1}.

    \item \textbf{Analyticity} In  \cite[theorem 5.2]{AJKS}, they also prove that the curve $\Gamma$ is continuous in $\gamma$. This required showing analyticity of GMC in the $\gamma$-variable. This is interesting to study for the inverse too.

    \item  \textbf{Critical case}\\
    Since the current results go up to $\kappa<4$, we still need to test if the Beltrami-approach will work too for coupling $SLE(4)$ with GFF. One major obstacle includes the lack of analogous moments existence. In this work and in the \cite{AJKS} work we used heavily that we have finite moments just above $1<p<\frac{2}{\gamma^{2}}$ (this made many errors summable as in \cite{AJKS}). The critical case with the Beltrami approach was initiated in \cite{tecu2012random}  In the appendix we included a modulus estimate that is inspired from the one in \cite[Theorem 24]{tecu2012random} and that might be of use in developing this overall machinery (especially for the critical case $\gamma=\sqrt{2}$). Finally, it might be necessary to use the generalization of the Lehto approach as developed in \cite[Lemma 7.1]{gutlyanskii2012beltrami} to make some different (random) choice of admissible metric.

    \item \textbf{SLE Loop} In the work in \cite{sheffield2016conformal}, one starts with a Hadamard-coupling of GFF and SLE to obtain a quantum zipper process and in turn the welding result. In the \cite{AJKS}-perspective one starts with the welding but it is still left with relating the resulting Jordan curve $\Gamma$ to an SLE loop as described in \cite[theorem 4.1]{zhan2021sle} and \cite[theorem 1.3]{ang2021integrability}. So the research problem is to obtain an independent second proof of the SLE loop construction in \cite[theorem 1.3]{ang2021integrability} using the \cite{AJKS}-framework.\\
For example, some variation of the Hadamard coupling needs to get rediscovered in the Beltrami equation setting. Perhaps we need more development of partial-welding concepts (\cite{hamilton2002conformal}) and the distributional-extensions for the Dirichlet-energy-coupling in \cite{Viklund_2020}. It is unclear how to characterize the law of the resulting Jordan curve $\Gamma$ from the composition of $\phi_{+}(x)=Q_{1}(x\tau_{1}([0,1]),\phi_{-}(x)=Q_{2}(x\tau_{2}([0,1])$.
By achieving this the community can benefit from having multiple perspectives on the same welding-coupling phenomenon for GFF and SLE. 
\begin{conjecture}
The above Beltrami-welding yields an SLE loop measure and thus an independent second proof of \cite[theorem 1.3]{ang2021integrability}.
\end{conjecture}
\noindent We don't have any developing work on this problem. It seems to require the development of new technology in the overlap of Beltrami equation, partial-welding and Hadamard-coupling.

  \item \textbf{Independence number and conditional independence}\\
 In obtaining enough decoupled annuli we had to study independence number of the interval graph of the inverse. The independent interval graph with \iid endpoints has been studied in \cite{justicz1990random}. It would be interesting if a similar construction of some correlated Poisson process could give exact growth of the independence number.   
    
    \item \textbf{Koebe Circle packing approximation}\\
    From Bishop's work in \cite{bishop2007conformal} we have a generalized welding with "bubbles" at zero logarithmic capacity locations. The map here is bi-Holder so it does satisfy the zero logarithmic capacity constraint. Is there a Lehto-condition analogue in the approach with Koebe circles? From talking to some experts it seems that the machinery for that has not been developed yet ("discretized Lehto condition").
    
    \item \textbf{Wavelet approach}\\
    In the critical welding article \cite{tecu2012random} there is an interesting approach using wavelets and Galton-Watson trees. In particular, they try to approach the critical case $\gamma$ from below by taking a limit of subcritical models. We wish there will be more follow-up works on the use of Galton-Watson tree arguments for controlling the dilatation and the modulus. It seems to be in the right direction if one wants to develop machinery for welding in the critical case. In the appendix we included a modulus estimate that is inspired from the one in \cite[Theorem 24]{tecu2012random} and that might be of use in developing this overall machinery (especially for the critical case $\gamma=\sqrt{2}$).

    \item \textbf{Random planar maps convergence over Beltrami coefficients}
    The Liouville-quantum gravity surfaces have been shown to be Gromov-Hausdorff limits of a type of random Riemann surface called random planar maps (defined here \cite{nachmias2020planar} and see progress here \cite{gwynne2021random,gwynne2021tutte}). It would be interesting to formulate the discrete Lehto condition in circle packing and then a convergence result of discrete Beltrami coefficients to the continuous one of GFF (some references on quasiconformal and circle packing  \cite{schramm2007combinatorically,rodin1987convergence,williams2019constructing}). Similarly, it would be interesting to use Beltrami for the convergence of discrete loops to continuous ones for various models in the $\kappa<4$ range (is there some discrete analogue of Morera's theorem but for the \textit{discrete} Beltrami equation that one can verify and obtain convergence to the continuous one?).

    \item \textbf{General logarithmic fields}\\
    If the field $X(t)$ is locally close to $U(t)+g(t)$, where $U(t)$ is the WNE field and $g(t)$ is some continuous function, then this would still work because the factor $g(t)$ would simply show up as a prefactor, that we then have to multiply our annuli with so as to cancel out:
    \begin{equation}
    Q_{X}(\wt{a})=Q_{U}(e^{-g(\theta_{a})}\wt{a})=Q_{U}(e^{g(\theta_{1})-g(\theta_{a})} a).        
    \end{equation}
    
    \item \textbf{4-dimensional SLE?}\\
    The Quasiconformal theory has analogues in $n\geq 3$ eg, see \cite{martin2013theory}. So it is natural to explore higher dimensional analogous of gluing two hyperspheres with independent Liouville measures on them. This might have some interest in studying the higher dimensional analogues for conformal field theory \cite{poland2019conformal,rattazzi2011central}.
    
    \item \textbf{Stochastic heat equation}\\
    One of the motivations for this work was the flexibility of the machinery. In theory, one can plug into the welding map even the trace of the stochastic heat equation
    \begin{equation}
    h_{t}(x)=\frac{1}{\tau_{t}[0,1]}\tau_{t}[0,x]=\frac{1}{\tau[0,1]}\int_{0}^{x}e^{\thickbar{H}_{t}(s)}\ds    
    \end{equation}
    solve the welding problems and then prove uniform convergence of this "heat flow" to the static problem done in this work since stochastic heat equation has the Gaussian free field as its invariant measure. At least for each fixed $t>0$, one can instead consider the field
    \begin{eqalign}
\xi(x)=H_{t}(x)-U(x)        
    \end{eqalign}
and try to prove continuity and existence of all exponential moments as done for the $\xi$-field in this work. Then most of the above work should go through again (one still needs to work $\tau_{t}$ as we did but there doesn't seem to be any obstacle).\\    
In some numerical works from 2007 to recent time \cite{falkovich2007conformal,puggioni2020conformal}, there are impressions that the zero level sets of parabolic systems could be very close to having conformally invariant behaviour. A likely explanation for these results is probably just the fast convergence of their linearizations the stochastic heat equation (SHE) to its invariant limit the Gaussian free field (GFF) (see \cite{baglioni2013hausdorff} studying Hausdorff dimension via linearization). So one hope is that by getting good rates of convergence of the Beltrami solutions of SHE to that of GFF, one can also  translate these into Gromov-Hausdorff convergence rates for their corresponding  welded-curve solutions.  For many more questions and more recent work on studying the Liouville measure corresponding to the stochastic heat equation field see \cite{dubedat2019stochastic,dubedat2023metric}.   In theory the above time dependent-Beltrami equation with a singular coefficient $\mu$ that depends on the noise is a nonlinear SPDE, that is $\mu$ is singular-like $e^{H(x)}$. This seems closely related to the stochastic quantization works: sine-Gordon work \cite{hairer2016dynamical} and Liouville work in \cite{garban2020dynamical} where they also need handle an exponential noise. Can it fit into the regularity structures framework?

\end{enumerate}
\newpage
\appendix
\section{Inverse GMC properties }
\subsection{Strong Markov property  }
The existence and continuity of the inverse $Q$ follows from the a)non-atomic nature of GMC \cite[theorem 1]{bacry2003log} and b)its continuity and strict monotonicity, which in turn follows from satisfying bi-\Holder over dyadic intervals \cite[theom 3.7]{AJKS}. In \cite{binder2023inverse}, we showed a "buffer"-strong Markov property, an independence in the spirit of strong Markov property (SMP).
\begin{proposition} \label{it:SMP} (\textit{$\delta$-Strong Markov property})
For deterministic $a,t\geq 0$ and $r\geq \delta$, we have the independence:
\begin{equation}\label{deltaSMP}
\etamu{Q^{\delta}(a)+r,Q^{\delta}(a)+r+t}{\delta}\indp Q^{\delta}(a).\tag{$\delta$-SMP}
\end{equation}
We will call it the \textit{$\delta$-Strong Markov property} ($\delta$-SMP). In other words, $Q_{s}^{\delta}\bullet \para{Q_{a}^{\delta}+r}$ is independent of $Q_{a}^{\delta}$.
\end{proposition}
\noindent For the usual GMC we know that its expected value is linear $\Expe{\eta(a,b)}=b-a$.  Using the Markovian-like $\delta$-(SMP) property from before, we obtained a nonlinear relation for the expected value of the inverse.
\begin{proposition}\label{differencetermunshifted}\cite[prop.3.5]{binder2023inverse} We have for $a>0$ and $r\geq \delta $
\begin{eqalign}\label{eq:nonlinearexpect}
\Expe{\eta^{\delta}(Q^{\delta}(a),Q^{\delta}(a)+r)}-r=\Expe{Q^{\delta}(a)}-a&=\int_{0}^{\infty}\Proba{ Q_{R(t)}^{\delta}(a)\leq t \leq  Q^{\delta}(a)}\dt\\
&=\int_{0}^{\infty}\Proba{ \eta^{\delta}(t)\leq a \leq  \eta_{R(t)}^{\delta}(t) }\dt>0.
\end{eqalign}
In particular, for any $a>0$ we have $\Expe{Q^{\delta}(a)}>a$. \end{proposition}
\begin{remark}\label{nontranslationinvar}
This proposition shows that the GMC $\eta$ does \textit{not} satisfy a "strong" translation invariance i.e. $\Expe{\eta(Q(a),Q(a)+r)}\neq  r$. So the same is likely true for $Q(a,a+t)$ 
\begin{equation}
\Expe{Q(a,a+t)}=\int_{0}^{\infty}\Proba{t>\eta^{\delta}(Q^{\delta}(a),Q^{\delta}(a)+r)}\dr\neq \int_{0}^{\infty}\Proba{t>\eta^{\delta}(0,r)}=\Expe{Q(t)}.
\end{equation}
It also shows that $\Expe{Q^{\delta}(a)}$ is a nonlinear function of $a$. 
\end{remark}
\subsection{Comparison formulas }\label{sec:transitionformuals}
One key tool in \cite{AJKS} were the relations between measures from different scales $\eta^{n}$ and $\eta^{m}$ for $m\geq n$:
\begin{equation}\label{eq:scalerelationseta}
A_{I} \eta^{m}(I)\leq \eta^{n}(I)\leq  G_{I} \eta^{m}(I). 
\end{equation}
where $A_{I}:=\inf_{y\in I} e^{\bar{U}^{n}_{m}(y)}$ and $B_{I}:=\sup_{y\in I} e^{\bar{U}^{n}_{m}(y)}$. We will need analogous relations as well in this work we will also upper truncate the inverse.
\begin{proposition}\label{it:shiftscalingdiff}\textit{(Shifting and Scaling between scales a.s.})  Fix $x_{0}\in [0,1]$ and let $Q_{\mu_{x_{0}}}:\Rplus\to \Rplus$ be the inverse of the shifted measure $\mu_{x_{0}}(0,x):=\int_{0}^{x}e^{U(x_{0}+s)}\ds=\mu(x_{0},x_{0}+x)$, then we have the following identity for each increment over $(a,b)$:
\begin{equation}\label{eq:shiftrelationtranslation}
Q_{\mu}(\mu(0,x_{0})+a,\mu(0,x_{0})+b)= Q_{\mu_{x_{0}}}(a,b)=Q_{\mu_{x_{0}}}(b)-Q_{\mu_{x_{0}}}(a).
\end{equation}
For $c,y>0$ we have for $Q=Q_{\eta}$ the relations:
\begin{equation}\label{eq:scalerelinverse}
Q([0,y])=Q^{n}([0,\frac{y}{G_{n}(y)}]) \tand  Q([c,c+y ])=Q^{n}\para{\frac{1}{G_{n}(c)}[c,c+y\frac{G_{n}(c)}{G_{n}(c,c+y)}]},
\end{equation}
where the scaling constants are provided from the integral mean value theorem
\begin{eqalign}
G_{n}(c,c+y)&:=\frac{\eta(Q(c),Q(c)+Q([c,c+y ])}{\eta^{n}(Q(c),Q(c)+Q([c,c+y ])}\\
&=\frac{y}{\eta^{n}(Q(c),Q(c+y) )}\\
&=\expo{\overline{U_{n}}(\theta_{c,c+y} )  }\tforsome \theta_{c,c+y}\in [Q(c),Q(c+y)]    
\end{eqalign}    
and we let $G_{n}(y):=G_{n}(0,y)$. We also have a version for the normalized inverse:
\begin{eqalign}
Q([c,c+y]\eta(1))=Q^{n}\para{\frac{\eta(1)}{G_{n}(c)}\spara{c,c+y\frac{G_{n}(c)}{G_{n}([c,c+y])}}}.
\end{eqalign}
\end{proposition}
\begin{remark}
The \cref{eq:shiftrelationtranslation} allows for a translation invariance. Since we have stationarity $H(x_{0}+s)\eqdis H(s)$, we get
\begin{equation}
Q_{x_{0},\tau}(a,b)\eqdis Q_{\tau}(a,b).  
\end{equation}
\end{remark}
\noindent 
Another key tool in \cite{AJKS} were the relations between measures of the fields $H,V,U$ described in the notations section. We need this here because we compare GMCs on unit circle and real line.
\begin{proposition}\label{it:doubleboundinv}\label{rem:finiteexponentialmomentscomparisonfield}(\textit{Comparison constants between different measures}) 
There exists a continuous process
\begin{eqalign}
\xi(s):=\gamma\liz{\e}\para{H_{\e}(s)-U_{\e}(s)}.    
\end{eqalign}
Moreover the supremum has finite exponential moments over closed intervals $I$
\begin{equation}
\Expe{\expo{\alpha\sup_{s\in I}\xi(s)}}<\infty    
\end{equation}
for $\alpha>0$. The inverses $Q_{H}$ and $Q_{U}$ are related by
\begin{eqalign}
Q_{H}([c,c+y ])=&Q_{U}\para{\frac{1}{G_{H-U}([0,c])}\spara{c,c+y\frac{G_{H-U}([0,c])}{G_{H-U}([c,c+y])}}},
\end{eqalign}
for
\begin{equation}
G_{H-U}([a,b]):=\expo{\xi(\theta_{[a,b]}) } 
 \end{equation}
where $\theta_{[a,b]}$ is some point in$ [Q_{H}(a),Q_{H}(b)]$. Similarly, for deterministic $T>0$ we have
\begin{eqalign}
Q_{H}(a)\bullet T=Q_{U}(aG_{H-U,T}([0,a]))\bullet T ,   
\end{eqalign}
where 
\begin{equation}
G_{H-U,T}([0,a]):=\expo{ \xi(T+\theta_{[0,a]}) } 
 \end{equation}
 for some $\theta_{[0,a]}\in [0,Q_{H}(a)\bullet T]$.
 \end{proposition}

\section{Moment estimates }\label{sec:momenttheorems}
\subsection{Moments bounds of GMC}
\noindent For positive continuous bounded function $g_{\delta}:\Rplus\to \Rplus$ with $g_{\delta}(x)=0$ for all $x\geq \delta$ and uniform bound $M_{g}:=\sup_{\delta\geq 0}\norm{g_{\delta}}_{\infty}$,  we will need the following moment estimates for $\eta^{\delta}=\eta_{g}^{\delta}$ in order to compute the moments of the inverse $Q_{g}^{\delta}$. As mentioned in the notations we need to study more general fields such as $U^{\delta,\lambda}$ in \cref{eq:truncatedscaled}
\begin{equation}
\Expe{U_{ \varepsilon}^{\delta,g }(x_{1} )U_{ \varepsilon}^{  \delta,g }(x_{2} )  }=\left\{\begin{matrix}
\ln(\frac{\delta }{\varepsilon} )-\para{\frac{1 }{\e}-\frac{1}{\delta}}\abs{x_{2}-x_{1}}+g_{\delta}(\abs{x_{2}-x_{1}})&\tifc \abs{x_{2}-x_{1}}\leq \varepsilon\\ 
 \ln(\frac{\delta}{\abs{x_{2}-x_{1}}})-1+\frac{\abs{x_{2}-x_{1}}}{\delta}+g_{\delta}(\abs{x_{2}-x_{1}}) &\tifc \e\leq \abs{x_{2}-x_{1}}\leq \delta\\
  0&\tifc \delta\leq \abs{x_{2}-x_{1}}
\end{matrix}\right.    .
\end{equation}
The following proposition is studying the moment bounds for this truncated case and it is the analogue of \cite[lemma A.1]{david2016liouville} where they study the exact scaling field in 2d.
\begin{proposition}\label{momentseta}\cite[appendix]{binder2023inverse}
We have the following estimates: 
\begin{itemize}
    \item  For $q\in (-\infty,0)\cup [1,\beta^{-1})$ we have
\begin{equation}
\Expe{\para{\eta^{\delta}[0,  t]}^{q}}\leq c_{1}\branchmat{t^{\zeta(q)}\delta^{\frac{\gamma^{2}}{2}(q^{2}-q)} &t\in [0,\delta]\\ t^{q}& t\in [\delta,\infty] },
\end{equation}
where $\zeta(q):=q-\frac{\gamma^{2}}{2}(q^{2}-q)$ is the multifractal exponent and $c_{1}:=e^{\frac{\gamma^{2}}{2}(q^{2}-q)M_{g}}\Expe{\para{\int_{0}^{1}e^{ \overline{\omega}^{1}(x)}\dx}^{q}}$ and $\omega^1$ is the exact-scaling field of height $1$.

\item  When $t,\delta\in [0,1]$ then we have a similar bound with no $\delta$
\begin{equation}
\Expe{\para{\eta^{\delta}[0,  t]}^{q}}\leq c_{1} t^{\zeta(q)},
\end{equation}
but this exponent is not that sharp: if $\delta\leq t\leq 1$ and $q>1$, we have $q>\zeta(q)$ and so the previous bound is better. 

\item Finally, for all $t\in [0,\infty)$ and $q\in [0,1]$, we have
\begin{equation}
\Expe{\para{\eta^{\delta}[0,  t]}^{q}}\leq c_{2,\delta } t^{\zeta(q)}\delta^{\frac{\gamma^{2}}{2}(q^{2}-q)},
\end{equation}
where $c_{2,\delta}:=e^{\frac{\gamma^{2}}{2}(q^{2}-q)M_{g}}\Expe{\para{\int_{0}^{1}e^{ \overline{\omega}^{\delta}(x)}\dx}^{q}}$ where $\omega^{\delta}$ the exact-scaling field of height $\delta$.
\end{itemize}
Furthermore, for the positive moments in $p\in (0,1)$ of lower truncated GMC $\eta_{n}(0,t):=\int_{0}^{t}e^{U_{n}(s)}\ds$, we have the  bound
\begin{equation}\label{lowertrunp01}
\Expe{\para{\eta_{n}(0,t)}^{p}}\lessapprox~t^{\zeta(p)}, \forall t\geq 0.    
\end{equation}
When $p\in(-\infty,0)\cup (1,\beta^{-1})$, we have
\begin{equation}
\Expe{\para{\eta_{n}(0,t)}^{p}}\leq \Expe{\para{\eta(0,t)}^{p}} , \forall t\geq 0.       
\end{equation}
\end{proposition}
We have the following bound in the $L^{2}$-case.
\begin{lemma}\label{lem:l2boundforGMC}
For $\gamma<1$, we have the bound
\begin{equation}
\Expe{\para{\eta^{n}(x)-x}^{2}}\leq \branchmat{ 2x\delta_{n}\frac{\gamma^{2}}{1-\gamma^{2}}& \tcwhen x>\delta_{n}\\ \frac{2}{(1-\gamma^{2})(2-\gamma^{2})}x^{2-\gamma^{2}}\delta_{n}^{\gamma^{2}}-x^{2}& \tcwhen x\leq \delta_{n}}.  
\end{equation}
\end{lemma}
\subsection{ Small ball for \sectm{$\eta$} }
For small ball we can use the following theorem from \cite[theorem 1.1]{nikula2013small} and its application to GMC in \cite[theorem 3.2]{nikula2013small} for the truncated field with $\delta=1$. See also the work in \cite{lacoin2018path}, \cite{garban2018negative} for more small deviation estimates. We will use the notion of stochastic domination:
\begin{equation}
X\succeq Y \doncl \Proba{X\geq x}\geq \Proba{Y\geq x}, \forall x\in \mathbb{R}.    
\end{equation}
\begin{theorem}\cite[theorem 1.1]{nikula2013small}\label{smalldeviationeta}
Let $W$ and $Y$ be positive random variables satisfying stochastic dominating relations
\begin{equation}\label{eq:smoothing-inequality}
Y \succeq W_0 Y_0 + W_1 Y_1,
\end{equation}
where $(Y_0,Y_1)$ is an independent pair of copies of $Y$, independent of $(W_0,W_1)$, and $W_0 \eqdis W_1 \eqdis W$. Suppose further that there exist $\gamma > 1$ and $x' \in ]0,1[$ such that
\begin{equation}\label{eq:w-small-tail}
\Prob(W \leq x) \leq \exp\left( - c (-\log x)^\gamma \right) \quad \textrm{for all } x \leq x'.
\end{equation}
Then for any $\alpha \in [1,\gamma[$ there exists a constant $t_\alpha > 0$ such that for all $t \geq t_\alpha$ we have
\begin{equation}\label{eq:main-estimate}
\E e^{-t Y} \leq \exp\left(- c_\alpha (\log t)^\alpha \right) \quad \textrm{for all } t \geq t_\alpha.
\end{equation}
Therefore, we have the following small deviation estimate for $t\geq t_{\alpha}$ and $u=\frac{M}{t}$ for some $M>0$
\begin{equation}
    \Proba{u\geq Y}\leq e^{M} \exp\left(- c_\alpha (\log \frac{M}{u})^\alpha \right).
\end{equation}
\end{theorem}
Then in \cite[theorem 3.2]{nikula2013small}, it is shown that the above theorem applies to the total mass of the exact-scaling case $Y=\eta_{\omega}^{1}(0,1)$ by checking the two assumptions \cref{eq:smoothing-inequality},\cref{eq:w-small-tail} for $\alpha\in [1,2)$. But here we need it for $Y=\eta_{\omega}^{1}(0,t)$ for $t\geq 0$. Combining this with Kahane's inequality we have the following bound for the Laplace transform.
\begin{proposition}\label{laptraeta}\cite[appendix]{binder2023inverse}
Fix $\alpha\in [1,\beta^{-1})$. We have the following estimate for the Laplace transform of $\eta^{\delta}$ with field $U^{\delta,g}$. Let $r>0$ satisfy $rt\geq t_{\alpha}\geq 1$ from above \cref{smalldeviationeta}. When $t\in  [0,\delta]$ then
\begin{equation}
\Expe{\expo{-r\eta^{\delta}(t)}}\lessapprox     \expo{-\frac{1}{2}\para{\frac{ m_{2}(r,t)}{ 4\sigma_{t}}}^{2}}\vee  \expo{-c_{\alpha}\para{m_{2}(r,t)0.5}^{\alpha}},
\end{equation}
and when $t\in [\delta,\infty)$ then
\begin{equation}
\Expe{\expo{-r\eta^{\delta}(t)}}\lessapprox   \expo{-c_{\alpha}\para{\wt{m}_{2}(r,t)0.5}^{\alpha}},
\end{equation}
where 
\begin{eqalign}
\sigma_{t}^{2}:=\beta(\ln\frac{\delta}{t}+M_{g})&\tand\wt{\sigma}^{2}:=\gamma^{2}M_{g},\\
m_{2}(r,t):=\ln{rt}-\beta(\ln\frac{\delta}{t}+M_{g})>0&\tand \wt{m}_{2}r,t):=\ln{rt}-\beta M_{g}>0.
\end{eqalign}
\end{proposition}

We need corollaries here that simplifies that above bounds for the case of exponential-choices for the parameters and $g=0$. First, we study a simple case for $\delta=\rho_{*}$-scale. 
\begin{corollary}\cite{binder2023inverse}\label{cor:smallballestimateexpocho0}
We fix $\rho\in (0,1)$. We also fix integer $\ell\geq 1$. We set
\begin{eqalign}
r:=\rho_{1}^{-\ell},t:=\rho_{2}^{\ell}    \tand \delta:=\rho,
\end{eqalign}
for $ \rho\geq \rho_{2}>\rho_{1}$ defined as
\begin{equation}
\rho_{1}=\rho^{a_{1} }\tand \rho_{2}=\rho^{a_{2}}.
\end{equation}
The exponents $a_{i}$ need to satisfy $\frac{1}{\ell}<a_{2}<a_{1}$. For small enough $\rho$ we have the bound
\begin{eqalign}
  \Expe{\expo{-r\eta^{\delta}(t)}}\leq  c\rho^{c_{1}\ell},
\end{eqalign}
where
\begin{eqalign}
c_{1}:= \frac{1}{32}\para{\frac{(a_{1}-a_{2}) }{\sqrt{\beta\para{a_{2}-\frac{1}{\ell}}  }   }-\sqrt{\beta\para{a_{2}-\frac{1}{\ell}}  }   }^{2}.  
\end{eqalign}
\end{corollary}

Next we study the more complicated situation of varying scale $\delta=\rho_{*}^{n}$.
\begin{corollary}\cite{binder2023inverse}\label{cor:smallballestimateexpocho}
We fix $\rho\in (0,1),\e\geq 0$. We also fix integers $n\geq k\geq 1$ and $z,\ell\geq 0$. We set
\begin{eqalign}
r:=\rho_{2}^{-n-z}\rho_{1}^{-\ell},t:=\rho_{2}^{k+z+\ell+\e}\rho_{3}^{n-k}    \tand \delta:=\rho_{2}^{n},
\end{eqalign}
for $ \rho_{3}>\rho_{2}>\rho_{1}$ defined as
\begin{equation}
\rho_{3}=\rho^{a_{3}},\rho_{2}=\rho^{a_{2}},\rho_{1}=\rho^{a_{1} }.   
\end{equation}
The exponents $a_{i}$ need to satisfy
\begin{eqalign}
&a_{3}<a_{2}<a_{1}\\
&a_{1}>a_{2}(1+\e)(1+\beta)+\beta z.    
\end{eqalign}
\pparagraph{Case $n=k$ and $\ell\geq 1$}
For small enough $\rho$ we have the bound
\begin{eqalign}
  \Expe{\expo{-r\eta^{\delta}(t)}}\leq  c\rho^{c_{1}\ell},
\end{eqalign}
where
\begin{eqalign}
c_{1}:= \frac{1}{32}\para{\frac{(a_{1}-a_{2}(1+\e) }{\sqrt{\beta\para{a_{2}(1+\e)+z}  }   }-\sqrt{\beta\para{a_{2}(1+\e)+z}  }   }^{2}.  
\end{eqalign}
\pparagraph{Case $n>k$ and $\ell\geq 0$ and $z=0$}
We have
\begin{eqalign}
\Expe{\expo{-r\eta^{\delta}(t)}}\leq c\branchmat{\rho^{\ell c_{2} } & \tcwhen\ell> \frac{a_{2}-a_{3}}{a_{2}}(n-k)\\
\rho^{\ell c_{3,1}+(n-k) c_{3,2} } & \tcwhen 0\leq\ell\leq  \frac{a_{2}-a_{3}}{a_{2}}(n-k)},    
\end{eqalign}
and
\begin{eqalign}
c_{2}:=&\frac{1}{32}\para{\frac{a_{1}-a_{2}(1+\e)(1+\beta )}{\sqrt{\beta a_{2}\para{1+\e}  }   }   }^{2}\\
c_{3,1}:=&(a_{1}-a_{2}(1+\e))\para{\ln\frac{1}{\rho}}^{\alpha-1}\tand c_{3,2}:=(a_{2}-a_{3})\para{\ln\frac{1}{\rho}}^{\alpha-1}.
\end{eqalign}
\end{corollary}

\subsection{Shifted GMC moments}
In this section we study the tail/small-ball and moments for the shifted GMC $\etamu{Q^{\delta}_{a},Q^{\delta}_{a}+t}{\delta}$ for all $t>0$. For large $t$ we can just use \cref{deltaSMP}. 
\begin{corollary}\label{cor:shiftedGMCmoments}
Let $t>0$.
\begin{itemize}    
    \item  For positive moments $p\in [1,\beta^{-1})$ we have
 \begin{eqalign}
\Expe{\para{\etamu{Q^{\delta}_{a},Q^{\delta}_{a}+t}{\delta_{0}}}^{p}} &\lessapprox 
g_{pos}(t,\delta,\delta_{0},p)C_{sup}(a,\delta,\rho),
\end{eqalign}
where
\begin{eqalign}
g_{pos}(t,\delta,\delta_{0},p)=\branchmat{(1+\rho\delta+t)^{p(\beta+\e_{p})} t^{\zeta(p)-p(\beta+\e_{p})}& \tcwhen t\leq 1\\
(1+\rho\delta+t)^{p(\beta+\e_{p})} t^{p-p(\beta+\e_{p})}& \tcwhen t\geq 1}    
\end{eqalign}
and small enough $\e_{p}>0$. 

\item  For positive moments $p\in (0,1)$ and all $t\geq 0$ we have
 \begin{eqalign}
\Expe{\para{\etamu{Q^{\delta}_{a},Q^{\delta}_{a}+t}{\delta_{0}}}^{p}} &\lessapprox 
C_{sup}(a,\delta,\rho)t^{p\para{1-\beta-\e_{p}}},
\end{eqalign}
and small enough $\e_{p}>0$.

 \item  For negative moments we have
\begin{eqalign}\label{shiftednegativeappinfmom}
\Expe{\para{\etamu{Q^{\delta}_{a},Q^{\delta}_{a}+t}{\delta_{0}}}^{-p} } \lessapprox &g_{neg}(t,\delta,p)C_{inf,1}(\alpha,\delta,\delta_{0},\rho)+\para{g_{neg}(t,\delta,p\wt{q}_{11})}^{1/\wt{q}_{11}}C_{inf,2}(\alpha,\delta,\delta_{0},\rho),
\end{eqalign}
where
\begin{eqalign}
g_{neg}(t,\delta,\delta_{0},p)=\branchmat{ t^{\zeta(-p)}\para{\frac{\rho\delta}{t}+2}^{\e_{p}}& \tcwhen t\leq \delta_{0}\\t^{-p}\para{\frac{\rho\delta}{t}+2}^{\e_{p}}& \tcwhen t\geq \delta_{0}}.    
\end{eqalign}
When $\delta=\delta_{0}$, we write $g_{pos}(t,\delta,p):=g_{pos}(t,\delta,\delta,p)$ and $g_{neg}(t,\delta,p):=g_{neg}(t,\delta,\delta,p)$.
\end{itemize}
\end{corollary}
\subsection{Ratio of inverses}
In this section we study the moments of the ratios of $Q_{\eta}$. This the analogous result for $\eta$ is in \cite[lemma 4.4]{AJKS}, where using the lognormal-scaling law one can bound
\begin{equation}
\Expe{\para{\frac{\eta^{\delta}(I)}{\eta^{\delta}(J)}}^{p}  }\leq c\para{\frac{\abs{I}}{\abs{J}}}^{\zeta(p)},    
\end{equation}
for all small enough intervals $I,J$ ie. $\abs{I},\abs{J}<\delta$ that are close and for $1<p<\beta^{-1}$.\\
In the conformal welding part for the inverse, the dilatation $K$ is controlled by ratios $\frac{Q^{\delta}(J)}{Q^{\delta}(I)}$. The case $\abs{I}=\abs{J}$ represents the most singular behaviour of the dilatation as it approaches the real axis. In the following proposition we will study moments of the analogous ratio of the inverse increments. As in \cite[lemma 4.4]{AJKS} we have two separate statements.
\begin{proposition}
\label{prop:inverse_ratio_moments}
We go over the cases of decreasing numerator, equal-length intervals for $Q$ and the same for $Q_{H}$.
\begin{itemize}
    \item Fix  $\beta\in (0,\frac{\sqrt{11}-3}{2})$ (i.e. $\gamma\in (0,\sqrt{\sqrt{11}-3})\approx (0,0.77817...)$) and $\delta\leq 1$ . Assume $\abs{I}=r\delta$ for some $r\in (0,1)$ and $\abs{J}\to 0$. Then there exists $\e_{1},\e_{2}>0$, such that for each $p\in [1,1+\e_{1}]$, one can find $q\in [1,1+\e_{2}]\tand \wt{q}>0$ such that
\begin{eqalign}\label{eq:ratiolpestimates}
  \Expe{\para{\frac{Q^{\delta}\para{J}}{Q^{\delta}\para{I}}}^{p}}\lessapprox \para{\frac{\abs{J}}{\delta}}^{q} r^{-\wt{q}}.
\end{eqalign}

 \item Fix $\gamma<1$ and $\delta\leq 1$. Assume we have two equal-length intervals $J=(a,a+x),I=(b,b+x)\subset [0,1]$ with $b-a=c_{b-a}x$ for $c_{b-a}>1$  and $a=c_{1}\delta,b=c_{2}\delta$ and $x<\delta$. Then we have for all $p\in [1,1+\e_{1}]$ with small enough $\e_{1}>0$, a bound of the form
\begin{eqalign}\label{eq:ratiobound}
\maxp{\Expe{\para{\frac{Q(a,a+x)}{Q(b,b+x)}}^{p} },\Expe{\para{\frac{Q(b,b+x)}{Q(a,a+x)}}^{p} }}\leq c_{a,b,p,\delta }\para{\frac{x}{\delta}}^{-\e_{ratio}(p)},
\end{eqalign}
where $\e_{ratio}(p)\in (0,1)$ can be made arbitrarily small at the cost of larger comparison constant $c_{a,b,p,\delta }$. The constant $c_{a,b,p,\delta }$ is uniformly bounded in $\delta$. 

\item Fix $\gamma<1$.  For the inverse $Q_{H}$ corresponding to the field $H$ in \cref{eq:covarianceunitcircle}, we have the same singular estimate as above for $\delta=1$
\begin{eqalign}\label{eq:ratioboundcircle}
\maxp{\Expe{\para{\frac{Q_{H}(a,a+x)}{Q_{H}(b,b+x)}}^{p} },\Expe{\para{\frac{Q_{H}(b,b+x)}{Q_{H}(a,a+x)}}^{p} }}\leq c_{a,b,p,\delta }x^{-\e_{ratio}(p)}.
\end{eqalign}

\end{itemize}
\end{proposition}

\newpage\subsection{Decoupling}
\noindent Another progress was finding a way to decouple. For $k,m\in [N]$ we consider the complement gap event i.e. the \textit{overlap} event
\begin{equation}
O_{k,m}:=G_{k,m}^{c}:=\ind{Q^{k\wedge m}(a_{k\wedge m})-Q^{m\vee k}(b_{m\vee k}) \leq \delta_{k\vee m}   }.    
\end{equation}
We then consider the random graph $G:=G_{Q}(N)$ with $N$ vertices labeled by $v_{k}:=\set{(Q^{k}(a_{k}),Q^{k}(b_{k})}$ and we say that $v_{k},v_{k+m}$  are connected by edges only when this event happens $G_{k,k+m}^{c}=1$. So the question of having a subsequence of increments that don't intersect is the same as obtaining the largest clique of vertices $S\subset G$ that do not connect to each other, this is called the \textit{independence number $\alpha(G)$ of $G$}.
\begin{theorem}\label{thm:gapeventexistence}\label{indnumbertheorem}
We show that the independence number being less than $c_{gap}N$ decays exponentially in $N$
\begin{equation}
\Proba{\alpha(G)< c_{gap}N}\lessapprox c \rho_{*}^{(1+\e_{*})N},
\end{equation}
for some $\e_{*}>0$, given the following constraint
\begin{eqalign}\label{eq:constraintdecouplingbeta}
(1-\e)\para{\frac{(\beta+1)^{2}}{4\beta}-\frac{(1+\e_{*})}{1-c_{gap}}}-\frac{\beta^{-1}+1}{2}r_{a}> \frac{(1+\e_{*})}{1-c_{gap}}.
\end{eqalign}
\end{theorem}
\begin{remark}
The constraint \cref{eq:constraintdecouplingbeta} has a solution for $\beta<0.171452$  and some small enough $\e,\e_{*},c_{gap}$ using Mathematica.  
\end{remark}
\subsubsection{Decay of overlap }
\noindent In this section we estimate the probability that the gap between $Q^{k+m}(b_{k+m})$ and $Q^{k}(a_{k})$ is small:
\begin{equation}
\Proba{Q^{k}(a_{k})-Q^{k+m}(b_{k+m})< \delta_{k+m}}\to 0\tas m\to +\infty.     
\end{equation}
This says that the decoupling from above becomes asymptotically true in smaller scales. This kind of \textit{asymptotic independence} was crucial in \cite{AJKS}, where even though the annuli were correlated, after removing the upper-scales, the lower scales exhibited decoupling.
\begin{proposition}\label{prop:decayoverlap}\label{rem:exponentialchoices}
We fix $k,m\geq 1$. We let $\delta_{k}:=\rho_{\delta}^{n}$ for some $\rho_{\delta}\leq \rho_{*}$. Then the probability for a small gap can be bounded by:
\begin{eqalign}\label{eq:decayofoverlap}
\Proba{Q^{k}(a_{k})-Q^{k+m}(b_{k+m})<\delta_{k+m}}\lessapprox&  \para{a_k-g_{k,m}-b_{k+m}}^{-2}  \delta_{k}^{p_{1}}b_{k+m}^{2-p_{1}}\\
&+\para{a_{k}-g_{k,m}-b_{k+m}}^{-p_{2}}\delta_{k}^{p_{2}} \para{\frac{\delta_{k+m}}{\delta_{k}}}^{\zeta(p_{2})-p_{2}(\beta+\e_{p_{2}})}\\
=:&e^{gap}_{k,k+m}, 
\end{eqalign}
where $p_{1}\in (0,1),p_{2}\in [1,\beta^{-1})$, and the $\e$'s are arbitrarily small and sequence $g=g_{k,m}:=\rho_{g}\rho_{1}^{k}\rho_{2}^{m}$ for some $\rho_{1}\in [\rho_{\delta},\rho_{*}],\rho_{2}\in (0,1)$. Using the \nameref{def:exponentialchoiceparam} for general $\rho_{\delta}\leq \rho_{*}$, we bound
\begin{eqalign}
\Proba{Q^{k}(a_{k})-Q^{k+m}(b_{k+m})<\delta_{k+m}}\lessapprox &   \para{\rho_{a}- \para{\frac{\rho_{1}}{\rho_{*}}}^{k}\rho_{2}^{m}-\rho_{b}\rho_{*}^{m}}^{-2} \para{\frac{\rho_{\delta}}{\rho_{*}}}^{p_{1}k}\rho_{*}^{(2-p_{1})m}+\frac{1}{\rho_{g}^{p_{2}}} \para{\frac{\rho_{\delta}}{\rho_{1}}}^{p_{2}k}\para{\frac{\rho_{\delta}^{p_{2}\para{1-\beta-\e_{p_{2}}}}}{\rho_{2}^{p_{2}}}}^{m}.
\end{eqalign}
\end{proposition}

\newpage \subsection{ Multipoint and unit circle}\label{sec:multipointmaximum}\label{partmultipointestimates}
\noindent Here we study the multipoint estimate for the inverse $Q_{H}$ on the unit circle. In this section we will use the following notations. \noindent From the \cref{thm:gapeventexistence}, for each $N\geq 1$ we are given a subset $S:=\set{i_{1},\cdots,i_{M}}\subset \set{1,\cdots,N}$ with $M\geq \alpha N$ such that we have large enough gaps
\begin{equation}
G_{S}:=G_{i_{1},\cdots,i_{M}}:=\bigcap_{k\in S}\set{Q_{a_{i_{k}}}^{i_{k}}-Q_{b_{i_{k+1}}}^{i_{k+1}}\geq \delta_{i_{k+1}}}.
\end{equation}
We consider the events
\begin{eqalign}
G_{[1,M]}:=&\set{Q^{i}(a_{i})\geq Q^{n_{i+1}}(b_{n_{i+1}})+\delta_{n_{i+1}}, i\in [1,M]},\\
\Xi_{[1,M]}:=&\set{\supl{(s,t)\in [0,Q^{i}(b_{i})]^{2}}\abs{\xi(t)-\xi(s)}\leq u_{i,\xi} ,i\in [1,M]},\\
U_{[1,M]}:=&\set{\supl{(s,t)\in [0,Q^{i}(b_{i})]^{2}}\abs{U^{1}_{i}(t)-U^{1}_{i}(s)}\leq u_{i,us}, i\in [1,M]},
\end{eqalign}
where we use the shorthand notation $Q^{i}=Q^{n_{i}}$ and consider any sequences $u_{i,\xi},u_{i,us}\in (0,1)$ strictly decreasing to zero as $i\to+\infty$ and interval
\begin{eqalign}
P_{i}:=[-\para{u_{i,\xi}+u_{i,s}},u_{i,\xi}+u_{i,s}].   \end{eqalign}
We will bound the following product
\begin{eqalign}\label{eq:maineventmultiunitcircle}
&\prod\limits_{k\in [1,M]}\para{\frac{Q_{H}(d_{k},d_{k}+y_{k})}{Q_{H}(c_{k},c_{k}+x_{k})}}^{p_{k}}   \ind{ G_{[1,M]} ,\Xi_{[1,M]}, U_{[1,M]}},
\end{eqalign}
where for set $S:=\set{n_{1},...,n_{M}}\subset [1,N]$. In our ensuing work on Lehto-welding we obtained the existence of this set $S$ for these events using deviation estimates. 
\begin{proposition}[Product of ratios]\label{prop:Multipointunitcircleandmaximum}
We fix set $S:=\set{n_{1},...,n_{M}}\subset [N]$.
\begin{itemize}
 \item (Equal length)  Fix $\gamma<1$. Suppose we have equal length intervals $x_{k}=\abs{J_k}=\abs{I_k}\leq \delta_k$ for each $k\in S$. Fix $\delta_k\leq 1$ and intervals $J=(c_k,c_k+x_k),I=(d_k,d_k+x_k)\subset [0,1]$ with $d_k-c_k=c_{d-c,k}x_k$ or $c_k-d_{k}=c_{d-c,k}x_k$ for some $c_{d-c,k}>1$  and $c_k=\lambda_{1,k}\delta_k,d_k=\lambda_{2,k}\delta_k$ for some $\lambda_{i,k}>0$.  Then we have for all $p_k\in [1,1+\e_{k}]$ with small enough $\e_{k}>0$, a bound of the form
\begin{eqalign}\label{eq:decaynumerdenounic}
\Expe{\prod\limits_{k\in [1,M]}\para{\frac{Q_{H}(d_{k},d_{k}+y_{k})}{Q_{H}(c_{k},c_{k}+x_{k})}}^{p_{k}}   \ind{ G_{[1,M]} ,\Xi_{[1,M]}, U_{[1,M]}}}\leq c^{\abs{S}}\prod\limits_{k\in S} \para{\frac{x_{k} }{\delta_{k}}}^{-\e_{ratio}(p_k)},  
\end{eqalign}
where $\e_{ratio}(p_k)>0$ can be made arbitrarily small at the cost of a larger proportionality constant $c$. The constants are uniformly bounded in $x_k$ and also in $\delta_k$.

    \item (Decreasing numerator) Fix  
    \begin{eqalign}
 \beta\in \para{0,\frac{1}{4}\para{1 - 2 \sqrt{2} + 2 \sqrt{\frac{17}{4} + \sqrt{2}}}}\approx (0,0.23874).        
    \end{eqalign}
    Suppose that $\abs{I_{k}}=r_{k}\delta_{k}$ for $r_{k}>0$ and $\abs{J_{k}}\to 0$ for each $k\in S$. Then we have for each $p_{k}\in [1,1+\e_{ratio,1}]$ some  $q_{k}=1+\e_{ratio}>1$ such that
\begin{eqalign}\label{eq:decaynumeratorunic}
\Expe{\prod\limits_{k\in [1,M]}\para{\frac{Q_{H}(d_{k},d_{k}+y_{k})}{Q_{H}(c_{k},c_{k}+x_{k})}}^{p_{k}}   \ind{ G_{[1,M]} ,\Xi_{[1,M]}, U_{[1,M]}}}\leq \prod\limits_{k\in S} \para{\frac{\abs{J_{k}}}{\delta_{k}}}^{q_{k}}.
\end{eqalign}

\end{itemize}

\end{proposition}

\newpage \section{Gaussian fields }

\subsection{Modulus estimate over 1d}
We will need the supremum estimate for Gaussian fields called Borel-Tsirelson-Ibragimov-Sudakov (BTIS) inequality \cite[theorem 4.1.1]{adler2009random},\cite[Lemma 3.3]{AJKS} to get the upper bound.
\begin{theorem}(\cite[Lemma 3.3]{AJKS})\label{thm:suprconcen}
Let $S=\bigcup_{k}I_{k}$ and suppose that the Gaussian field $Y_{t},t\in S$ has $L$-Lipschitz continuous covariance $\Expe{\abs{Y_{t}-Y_{s}}^{2}}\leq L\abs{s-t}$ and $Y_{t_{0}}=0$ with $t_{0}\in S$. Then
\begin{equation}
\Proba{\supl{t\in S}Y_{t}\geq \sqrt{L\abs{S}}u  }\leq c(1+u)e^{-u^{2}/2},   
\end{equation}
for some uniform constant $c$.
\end{theorem}
\noindent We also have bounds for the expected value of the supremum \cite[Theorem 6.5]{van2014probability}. Consider centered Gaussian field $\para{X_{t}}_{t\in T}$ in some index set $T$. Let 
\begin{equation}
d(s,t):=\sqrt{\Expe{\abs{X_{s}-X_{t}}^{2}}}    
\end{equation}
denote a pseudo-metric on $T$ and the \textit{entropy number} $N(T, d, \e)$ equal the minimum number of balls of size $\e>0$ needed to cover $T$ in the $d$-metric.
\pparagraph{Modulus estimate over 2D}
One result for these moduli 
\begin{equation}
\Proba{\supl{(s,t)\in [0,a]^{2}}\abs{U_{\e}^{\delta}(t)-U_{\e}^{\delta}(s)}\geq u}
\end{equation}
is by Talagrand and Ledoux in \cite[lemma 2.1]{talagrand1995hausdorff}\cite[section 11.1]{ledoux2013probability}. Consider set $S$ and Gaussian process $(Z(s)_{s\in S}$. Let $d(s,t):=\sqrt{\Expe{\para{Z(s)-Z(t)}^{2}  }}$, $N_{d}(S,r)$ the smallest number of (open) $d-$balls of radius $r$ needed to cover $S$ and diameter $D:=\sup(d(x,t): x,t\in S)$. Then we have
\begin{lemma}\label{lem:moduest2d}( \cite[lemma 2.1]{talagrand1995hausdorff})
 Given $u>0$, we have
 \begin{equation}
\Proba{\supl{(s,t)\in S^{2}}\abs{Z(t)-Z(s)}\geq u  }\leq \expo{-\frac{1}{K^{2}D}\para{u-K\int_{0}^{D}\sqrt{\ln(N_{d}(S,r)}\dr }^{2}  },     
 \end{equation}
 for fixed constant $K>0$.
\end{lemma}
\noindent For the field $U_{\varepsilon}^{\delta}$ we have the metric
\begin{eqalign}
d(s,t)=&\sqrt{\Expe{|U_{\varepsilon}^{\delta}(s)-U_{\varepsilon}^{\delta}(t)|^{2}}},  
\end{eqalign}
for
\begin{eqalign}
&\Expe{|U_{\varepsilon}^{\delta}(s)-U_{\varepsilon}^{\delta}(t)|^{2}}    \\
=&2\branchmat{\para{\frac{1}{\e}-\frac{1}{\delta}}\abs{t-s}& \tcwhen \abs{t-s}\leq \e\\  \ln\frac{\abs{t-s}}{\e}+1-\frac{\abs{t-s}}{\delta}   & \tcwhen \e\leq \abs{t-s}\leq \delta \\ \ln\frac{\delta}{\e}& \delta\leq \abs{t-s}}.    
\end{eqalign}
For $D=\sup_{s,t}d(s,t)$, we have the bounds $\floor{\frac{D}{r^{2}}}\leq N_{d}([0,a],r)\leq \ceil{\frac{D}{r^{2}}}$ and thus the upper bound
\begin{equation}\label{eq:entropyintegralbound}
\int_{0}^{D}\sqrt{\ln(N_{d}([0,a],r)}\dr\leq D\ln\frac{1}{D}.    
\end{equation}
Therefore, for $u>KD\ln\frac{1}{D}$ the bound on the modulus becomes:
 \begin{equation}\label{eq:fieldmodulusdeviation}
\Proba{\supl{(s,t)\in [0,a]^{2}}\abs{U_{\e}^{\delta}(t)-U_{\e}^{\delta}(s)}\geq u }\leq \expo{-\frac{1}{K^{2}D}\para{u-KD\ln\frac{1}{D} }^{2}  }.     
 \end{equation}
\subsection{Comparison inequalities}
\noindent We will use the Kahane's inequality (eg. \cite[corollary A.2]{robert2010gaussian}).
\begin{theorem}[Kahane Inequality]\label{Kahanesinequality}
Let $\rho$ be a Radon measure on a domain $D\subset\R^{n}$, $X(\cdot)$ and $Y(\cdot)$ be two continuous centred Gaussian fields on $D$, and $F: \Rplus \to \R$ be some smooth function with at most polynomial growth at infinity. For $t \in [0,1]$, define $Z_t(x) = \sqrt{t}X(x) + \sqrt{1-t}Y_t(x)$ and
\begin{eqalign}
\varphi(t) := \EE \left[ F(W_t)\right], \qquad
W_t := \int_D e^{Z_t(x) - \frac{1}{2}\EE[Z_t(x)^2]} \rho(dx).
\end{eqalign}

\noindent Then the derivative of $\varphi$ is given by
\begin{equation}\label{eq:Kahane_int}
\begin{split}
\varphi'(t) & = \frac{1}{2} \int_D \int_D \left(\EE[X(x) X(y)] - \EE[Y(x) Y(y)]\right) \\
& \qquad \qquad \times \EE \left[e^{Z_t(x) + Z_t(y) - \frac{1}{2}\EE[Z_t(x)^2] - \frac{1}{2}\EE[Z_t(y)^2]} F''(W_t) \right] \rho(dx) \rho(dy).
\end{split}
\end{equation}

\noindent In particular, if
\begin{eqalign}
\EE[X(x) X(y)] \le \EE[Y(x) Y(y)] \qquad \forall x, y \in D,
\end{eqalign}

\noindent then for any convex $F: \Rplus \to \R$
\begin{eqalign}\label{eq:Gcomp}
\EE \left[F\left(\int_D e^{X(x) - \frac{1}{2} \EE[X(x)^2]}\rho(dx)\right)\right]
\le
\EE \left[F\left(\int_D e^{Y(x) - \frac{1}{2} \EE[Y(x)^2]}\rho(dx)\right)\right].
\end{eqalign}

\noindent and the inequality is reversed if $F$ is concave instead. In the case of distributional fields $X,Y$, we use this result for the mollifications $X_{\e},Y_{\e}$ and pass to the limit due to the continuity of $F$.
\end{theorem}
\noindent Another comparison inequality for increasing/decreasing functionals is \cite[Theorem 2.15]{boucheron2013concentration}\cite[theorem 3.36]{berestycki2021gaussian}.
\begin{theorem}[FKG inequality]\label{FKGineq}
Let $\set{Z(x)}_{x\in U}$ be an a.s. continuous centred Gaussian field on $U \subset \R^d$ with $\Expe{Z(x)Z(y)} \geq 0$ for all $x, y \in U$. Then, if $f, g$ are two bounded, increasing measurable functions,
\begin{equation}
\Expe{f\para{\set{Z(x)}_{x\in U}}g\para{\set{Z(x)}_{x\in U}}}    \geq \Expe{f\para{\set{Z(x)}_{x\in U}}}\Expe{g\para{\set{Z(x)}_{x\in U}}}    
\end{equation}
and the opposite inequality if one function is increasing and the other one is decreasing
\begin{equation}
\Expe{f\para{\set{Z(x)}_{x\in U}}g\para{\set{Z(x)}_{x\in U}}}    \leq \Expe{f\para{\set{Z(x)}_{x\in U}}}\Expe{g\para{\set{Z(x)}_{x\in U}}}.    
\end{equation}

\end{theorem}

\subsection{Covariance computations}
Here we obtain a modulus bound for the field $\xi$ discussed in \cref{it:doubleboundinv}. We follow the proof of \cite[Lemma 3.5]{AJKS}.
\begin{proposition}\label{prop:xicovariancebound}
For $\delta_{1}>\delta_{2}$ and $z\in [0,\frac{1}{2}]$ we have the following control on the $L^2$-difference for $g:=\para{\frac{z}{2}}^{1/3}$
\begin{eqalign}
\Expe{\para{\xi_{\delta_{2}}^{\delta_{1}}(z)-\xi_{\delta_{2}}^{\delta_{1}}(0)}^{2}}\leq\branchmat{cz^{2/3}&  g\leq \delta_{2}\\ cz^{2/3}& \delta_{2}\leq g\leq \delta_{1}\\2(\delta_{1}^{2}-\delta_{2}^{2})& \delta_{1}\leq g}.    
\end{eqalign}
For each case we have the uniform bound $z^{1/3}\delta_{1}$.    
\end{proposition}
\begin{proofs}
\noindent We use the notation from \cite[Lemma 3.5]{AJKS}. The region is 
\begin{equation}
T_{\delta_{2}}^{\delta_{1}}(x_{0}):=\set{(x,y):\abs{x_{0}-x}\leq \frac{1}{2}\tand   y\in \para{\frac{2}{\pi}\tan(\pi\abs{x_{0}-x}),2\abs{x_{0}-x}   } \cap  [\delta_{2},\delta_{1}]   }.
\end{equation}
Here we study for $z\in [0,\frac{1}{2}]$ the difference
\begin{equation}
\Expe{\para{T_{\delta_{2}}^{\delta_{1}}(z)-T_{\delta_{2}}^{\delta_{1}}(0)}^{2}}\leq \lambda\para{T_{\delta_{2}}^{\delta_{1}}\Delta \para{T_{\delta_{2}}^{\delta_{1}}+z}},    
\end{equation}
where the upper bound denotes the symmetric difference computed wrt to the hyperbolic measure $\dlambda(x,y)=\frac{1}{y^{2}}\dy\dx$. We study this measure using $y-$level sets. As observed in the proof, the line
\begin{equation}
\set{(x,y): y=y_{0}} \cap    \para{T_{\delta_{2}}^{\delta_{1}}\Delta \para{T_{\delta_{2}}^{\delta_{1}}+z}}
\end{equation}
has length bounded by the minimum 
\begin{equation}
2\minp{z, \frac{y_{0}}{2}-\frac{\arctan(\frac{\pi y_{0}}{2})}{\pi}}    
\end{equation}
because either a)the sets $T_{\delta_{2}}^{\delta_{1}},\para{T_{\delta_{2}}^{\delta_{1}}+z}$ are far apart and so we only have the difference of the two graphs or b)the sets are so close that the $z$ is smaller. Then they use the inequality
\begin{equation}
\frac{t}{2}-\frac{\arctan(\frac{\pi t}{2})}{\pi}\leq 2t^{3}    
\end{equation}
to just bound by $2\minp{z,2y_{0}^{3}}$. Therefore,
\begin{eqalign}
\lambda\para{T_{\delta_{2}}^{\delta_{1}}\Delta \para{T_{\delta_{2}}^{\delta_{1}}+z}}
=&\int_{\delta_{2}}^{\delta_{1}} \abs{\set{(x,y): y=y_{0}} \cap    \para{T_{\delta_{2}}^{\delta_{1}}\Delta \para{T_{\delta_{2}}^{\delta_{1}}+z}}}  \frac{\dy_{0} }{y_{0}^{2}}\\
\leq &\int_{\delta_{2}}^{\delta_{1}}2\minp{z,2y_{0}^{3}} \frac{\dy_{0} }{y_{0}^{2}}.
\end{eqalign}
If $\delta_{2}\leq g=\para{\frac{z}{2}}^{1/3}\leq \delta_{1}$, we split in the middle
\begin{equation}
\int_{\delta_{2}}^{g}4y_{0}^{3} \frac{\dy_{0} }{y_{0}^{2}}+\int_{g}^{\delta_{1}}2z\frac{\dy_{0} }{y_{0}^{2}}=2\para{g^{2}-\delta_{2}^{2}}+2z\para{\frac{1}{g}-\frac{1}{\delta_{1}}}\leq cz^{2/3}.
\end{equation}
If $ \delta_{1}\leq g$, then
\begin{eqalign}
\int_{\delta_{2}}^{\delta_{1}}4y_{0}^{3} \frac{\dy_{0} }{y_{0}^{2}}= 2(\delta_{1}^{2}-\delta_{2}^{2}).    
\end{eqalign}
If $ g\leq  \delta_{2}$, then
\begin{eqalign}
2z\int_{\delta_{2}}^{\delta_{1}}\frac{\dy_{0} }{y_{0}^{2}}= 2z(\delta_{2}^{-1}-\delta_{1}^{-1})\leq c z^{2/3}.    
\end{eqalign}
For each case we have the bound $z^{1/3}\delta_{1}$.
    
\end{proofs}
\section{Dyadic Tree moduli estimate }
This appendix is an offshoot result we found while we were studying \cite[Theorem 24]{tecu2012random} and it is nowhere used in this article. We only included it for the benefit of others who want to study the Beltrami-Lehto welding for the critical case $\gamma=\sqrt{2}$. Consider annulus $A=A(1,2)$ of inner diameter $1$ and outer diameter $2$ and the family of paths $\Gamma$ connecting the complement components of annulus $A$. We will study the modulus of $mod(F(A) )$ for a map $F:\C\to \C$ with a dilation $K_{F}$ that diverges on the real line $\R$ but on a fractal set.  
\begin{proposition}\label{singularity_modulus}
Let $F:A(1,2) \to \C $ be a mapping with dilatation $K_{F}(z)$ that satisfies the following.
\begin{itemize}
\item In scales $y\in[-1,-2^{-N}] \cup[2^{-N},1]$ the dilatation is $K=1$,
\item In scales $y\in [-2^{-N},0]\cup [0,2^{-N}]$, there are dendrite sets $\Omega_{left},\Omega_{right}$ on the left and right part of the annulus, formed by dyadic intervals where at scale $2^{-i}$ the dilatation is bounded by 
\begin{equation}\label{eq:survivingintervals}
K(z)\leq K_{i}, \tfor z\in C_{I}, \tif \abs{I}=2^{-i}.    
\end{equation}
\end{itemize}
Then we have $mod(F(A)=\frac{1}{mod(F(\Gamma)}>c,$ where $c^{-1}:=\sum_{i\geq 1}\frac{K_{i}}{S_{i}c_{i,\infty}}$, where $S_{i}$ is the number of "surviving intervals" in the ith scale (i.e. the above \cref{eq:survivingintervals} is true in the corresponding Whitney square) and $S_{i}c_{i,\infty}$ is the proportion of them that have infinite descendants reaching the real line. 
\end{proposition}
\begin{remark}
In order to be able to use this proposition one needs to have concrete knowledge of the number of surviving intervals at each scale.    
\end{remark}
\begin{figure}[ht]
  \includegraphics[scale=1]{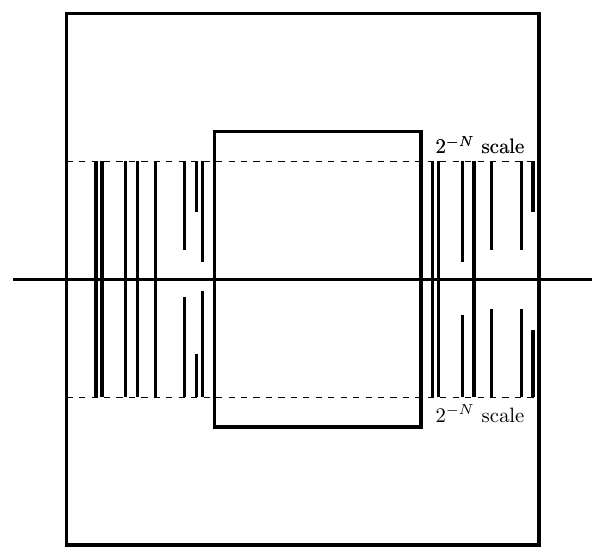}
   \caption{The trees $\Omega_{left},\Omega_{right}$ are reflected across the real line. }
 \end{figure}
\noindent We will need the following standard estimate (eg. \cite{lehto1973quasiconformal}, \cite[section 2.5]{gutlyanskii2012beltrami})
\begin{lemma}
For the modulus $mod(F(\Gamma) )$ can estimate by studying the admissible metrics over annulus $A$
\begin{equation}
\frac{1}{mod(F(A)}=mod(F(\Gamma) )\leq \inf_{\rho\in admis(A)}\iint_{A}\rho^{2}(w)K_{F}(w)dA(w),    
\end{equation}
\end{lemma}
\begin{proof}
\noindent We will next build a metric $\rho:A\to \mb{R}^{+}$ such that
\eq{mod(F(\Gamma)\leq \iint_{A} \rho(w)^{2}K_{F}(w)\dint w)<\infty.}
We need to construct a $\rho(z)$ that will control $K(z)$ as it diverges on the real line. 
First at each scale $2^{-Ni}$ we construct a uniform weight function $f_{y}(x):\mb{R}\to \mb{R}$ for $y\in [2^{-Ni},2^{-N(i-1)}]$:
\eq{f_{y}(x):=\branchmat{ B_{i}&, x\in J\in C_{surv,i} \\ 0 &, otw },}
where we will choose $B_{i}$ at the end based on what we need and $J\in C_{surv,i}$ means that J is one of the surviving intervals at scale $2^{-Ni}$. Each such interval is of length $\abs{J}=2^{-Ni}$ and let $S_{i}$ denote the number of surviving intervals, so $\abs{ C_{surv,i}}=S_{i}2^{-Ni}\leq 1$. We define
\eq{v(x,y):=\branchmat{f_{y}(x)&, (x,y)\in \Omega_{l}\cup \Omega_{r}\\ B_{0} &, (x,y)\in A\setminus ( \Omega_{l}\cup \Omega_{r} \cup \mb{R})\\ \infty & ,y=0   }.} Then we find
\begin{eqalign}
\iint_{A} \rho(w)^{2}K_{F}(w)\dint w = &\sum_{i\geq 1}\int_{2^{-Ni}}^{2^{-N(i-1)}}\int_{1}^{2}f^{2}_{y}(x)K_{F}(x,y)\dx \dy\\
= &\sum_{i\geq 1}\int_{2^{-Ni}}^{2^{-N(i-1)}}\int_{x\in  C_{surv,i} } B_{i}^{2}K_{F}(x,y)\dx \dy\\
\leq  &\sum_{i\geq 1}\int_{2^{-Ni}}^{2^{-N(i-1)}}B_{i}^{2}K_{i} \abs{ C_{surv,i}} \dy\\
=  &\sum_{i\geq 1}2^{-N(i-1)}B_{i}^{2}K_{i}S_{i}2^{-Ni} \\
=  &\sum_{i\geq 1}2^{-N(2i-1)}B_{i}^{2}K_{i}S_{i} \\
\end{eqalign}
Now we check admissibility for $\gamma \in \Gamma$ connecting the two components of $A(1,2)$. 
\eq{\int_{\gamma}\rho(w)|\dint w|=\sum_{i\geq 1}\int_{\gamma\cap \set{2^{-Ni}\leq y\leq 2^{-N(i-1)}}}\rho(w)|\dint w|+\int_{\gamma\cap \set{1\leq y\leq 2}}\rho(w)|\dint w|.}
For the last term we have
\eq{\int_{\gamma\cap \set{1\leq y\leq 2}}\rho(w)|\dint w|=\abs{\gamma\cap \set{1\leq y\leq 2}}.}

\par\textbf{Case 1:} Suppose that $\gamma$ stays in $ \set{1\leq y\leq 2}$, then the length $\abs{\gamma}\geq 1$ since the shortest curve connecting $A(1,2)$ is the length 1 segment.\\

\par\textbf{Case 2:} Suppose that $\gamma$ stays within the scales $\set{2^{-N(i+M)}\leq y\leq 2^{-Ni}}$ for $i\geq 1$. The goal is to build the $\rho-$shortest possible curve that stays between these scales. It suffices to compare it to the piecewise curve passing through each of the following tubes:
\eq{\set{(x,y):2^{-N(i+m)}\leq y\leq 2^{-N(i+m-1)} }\cap \set{(x,y): x\in C_{surv,i+m}}, }
where $m=0,...,M$.

\ben 

\item Each scale has $2^{N}$ children and at the (i+m)th scale there are $S_{i+m}\leq 2^{N}S_{i+m-1}$ surviving Whitney squares whose corresponding intervals have length $2^{-N(i+m)}$. In order to compare them, we bring them all in the bottom  (i+M)th scale: we split the (i+m)th scale squares into $2^{N(M-m)}$ length subsquares such that a (i+M)th scale square form a same $2^{-N(i+M)}$-sized column with all its parent squares. 

\item Since the space between these squares doesn't add any $\rho$-length, we can assume that these $2^{-N(i+M)}$-sized columns are all well separated and treat them individually. In particular, for $j=1,..S_{i+M}$
\eq{Q_{i+m,j}:=\set{(x,y):x\in J_{j,i+M} \tand 2^{-N(i+m)}\leq y\leq 2^{-N(i+m-1)} },}
where $J_{j,i+M} \in C_{surv,i+M}$ is the jth surviving interval in the (i+M)th scale. In other words,the set $\set{Q_{i+m,j}}_{m=1}^{M}$ is the jth column described above and there are $S_{i+M}$ of them. 

\item Next we build the shortest curve. To present the main idea we first work with $M=1$. So the curve can only travel between the parent scale $i$ and the children scale $i+1$. There are $S_{i+1}$ columns.

\ben 

\item If a parent square $Q_{i,j}$ has no child interval, then we pass the curve via the child scale $i+1$. That way the curve doesn't pick up any $\rho-$length.

\item  If a parent square $Q_{i,j}$ has a child interval, then if we pass the curve via the child scale $i+1$ we obtain the length:
\eq{B_{i+1}\abs{\gamma\cap Q_{i,j}}=B_{i+1}2^{-k_{i+1}}. }
If we pass the curve via the parent scale $i$ we obtain the length:
\eq{B_{i}\abs{\gamma\cap Q_{i,j}}=B_{i}2^{-k_{i+1}}. }

\een

\item Therefore, for the total length we have
\eq{S_{i+1}B_{i+1}2^{-k_{i+1}}. }
and by selecting $B_{i+1}\geq \frac{2^{k_{i+1}}}{S_{i+1}}$ we obtain $\rho$-length $\geq 1$. By picking $B_{i}\geq B_{i+1}$ we obtain the same for the second case.

\een 
 For general $M\geq 1$,  we split $Q\in S_{i}$ into $2^{NM}$ subsquares directly aligned with the squares in the $i+M$ scale. From these $2^{NM}\cdot S_{i}$ columns only $S_{i+M}$ of them will have descendants in the $i+M$ scale and so there is some $m\in [1,S_{i}]$ s.t.
\eq{2^{NM}(S_{i}-m)=S_{i+M}.}
\begin{figure}
  \includegraphics[scale=1.5]{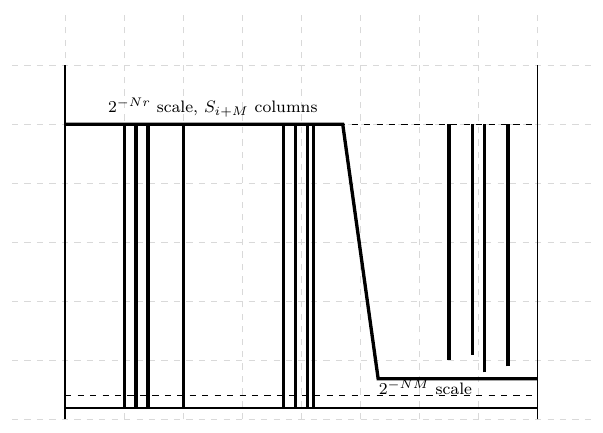}
   \caption{The curve is within the scale $y\in [2^{-NM},2^{-Nr}]$ and it passes under the columns with no elements in the $2^{-NM}$ scale. }
 \end{figure}

\ben 

\item So again we have the path cross through the $(i+M)$ scale if it doesn't contain a square. If there is a square, then we can cross if via any of the scales $(i+m)=i,...,i+M$ to obtain the length
\eq{B_{i+m}2^{-N(i+M)}}
and thus the total length
\eq{ 2^{-N(i+M)}\sum_{k=1}^{S_{i+M}}B_{i+m_{k}},}
for some sequence $m_{k}\in [0,M]$ that minimizes the length.

\item If we take $B_{i+m}$ to be increasing then we find the lower bound
\eq{ 2^{-N(i+M)}S_{i+M}B_{i}=\frac{B_{i}}{2^{Ni}}(S_{i}-m)=:\frac{B_{i}}{2^{Ni}}c_{i,M}S_{i},}
where $c_{i,M}\in [0,1]$. Because M can be arbitrary large, let $c_{i,\infty}$ be the proportion of $S_{i}$ columns that have infinite trees. 

\item We want 
\eq{\frac{B_{i}}{2^{Ni}}c_{i,\infty}S_{i}\geq 1\tand \frac{B_{i}}{2^{Ni}}\to 0\tas i\to \infty,}
the second condition is needed because $S_{i}K_{i}\to \infty$ and so we need a factor that will control $S_{i}K_{i}$ in the $\rho$-area estimate.  Taken together we want 
\eq{c_{i,\infty}S_{i}\to +\infty \tas i\to \infty}
fast enough so that $\frac{B_{i}}{2^{Ni}}c_{i,\infty}S_{i}\geq 1$. 

\een 

\par\textbf{Case 3:} Suppose that the curve $\gamma$  crosses all the scales and intersects the real line but doesn't cross it.\\
\ben 
\item (Single point) If the curve is the real line segment [2,1], then because $\rho|_{[-2,-1]}\equiv \infty$, we obtain $\abs{\gamma}_{\rho}=\infty$. If the intersection with $[-2,-1]$ has positive measure, then again $\abs{\gamma}_{\rho}=\infty$. So suppose that it intersects the real line only at a single point.
\item Also, assume that the curve is within the scale $y\in [2^{-Nr},0]$ for some $r>0$. Suppose that the $(1-c_{r,\infty})S_{r}$ columns that get truncated eventually, reach at most the $2^{-NM}$ scale, below which there are no other other squares. The shortest curve will hit the full tree columns at some scale and then it will cross below the $2^{-NM}$ scale to avoid picking up any $\rho$-length.
\begin{figure}
  \includegraphics[scale=1.5]{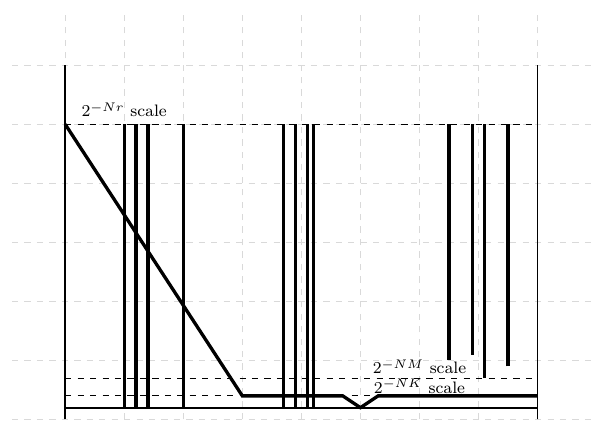}
   \caption{The curve is within the scale $y\in [0,2^{-Nr}]$ and it passes under the $2^{-NM}$ scale. }
 \end{figure}
\item We have the path cross each column once as opposed to travelling down one of them because it will pick less $\rho$-length by doing so. Since there are finitely many columns, there is a smallest scale $2^{-NK}\leq 2^{-NM}$ within which $\gamma$ crosses one of the columns before interesting the real line. Therefore, we returned to case 2, where the curve is between the scales $[2^{-Nr},2^{-NK}]$. 
\een 
\par\textbf{Case 4:} Suppose that the curve crosses the real line. Because the picture of the columns is symmetric across the real line, we return to case 3. 
\par\textbf{Case 5:} Now that we built some of the notation and ideas. suppose that the curve is allowed to travel in $y\in [0,2]$. As above we might have the curve travelling above the scale $2^{-N}$ but then below the lower truncated columns. There are $c_{1,\infty}S_{1}$ columns with infinite descendants and their added length is $c_{1,\infty}S_{1}2^{-N}$. Therefore, for the distance travelled by the curve above the scale $2^{-N}$ we find the lower bound
\eq{\abs{\gamma\cap \{2^{-N}\leq y\leq 2\}}\geq c_{1,\infty}S_{1}2^{-N}.}
As a result, we can set $B_{0}:=c_{1,\infty}S_{1}2^{-N}$ and we find the $\rho$-length
\begin{equation}
B_{0}\abs{\gamma\cap \{2^{-N}\leq y\leq 1\}}\geq 1.    
\end{equation}
All together we obtain the following constraints for weights $B_{i}$ and number surviving intervals $S_{i}$:
\begin{enumerate}
\item To control $K_{i}\to \infty$ we must have $(\frac{B_{i}}{2^{Ni}})^{2}S_{i}\to 0$ as $i\to \infty$. So if $S_{i}\to \infty$, then we require $\frac{B_{i}}{2^{Ni}}\to 0$.
\item On the other hand we require $\frac{B_{i}}{2^{Ni}}c_{i,\infty}S_{i}\geq 1$ for $c_{i,\infty}\in [0,1]$. Since $\frac{B_{i}}{2^{Ni}}\to 0$ we require that $c_{i,\infty}S_{i}\to \infty$. 
\item We used that $B_{i}$ are increasing in obtaining the lower bound
\eq{ 2^{-N(i+M)}\sum_{k=1}^{S_{i+M}}B_{i+m_{k}}\geq  2^{-N(i+M)}S_{i+M}B_{i}=\frac{B_{i}}{2^{Ni}}(S_{i}-m)=:\frac{B_{i}}{2^{Ni}}c_{i,M}S_{i},}
but it appears to be a removable hypothesis if we work with the sum directly. 
\item We need $B_{0}:=c_{1,\infty}S_{1}2^{-N}\neq 0$ on case 5.
\end{enumerate}
Therefore all together, one possible choice is $B_{i}:=\frac{2^{Ni}}{S_{i}c_{i,\infty}}$.
\end{proof}

\printbibliography[title={Whole bibliography}]

\end{document}